\documentclass[a4paper,11pt]{article}
\usepackage[utf8]{inputenc}
\usepackage{mathtools}
\usepackage{amsmath,amsfonts}
\usepackage{amsthm}
\usepackage{amssymb}
\usepackage{multirow}
\usepackage{latexsym}
\usepackage{graphicx}
\usepackage{subcaption}
\usepackage[colorlinks=true,linkcolor=blue,citecolor=blue]{hyperref}{}
\usepackage{xcolor}
\usepackage{algorithm}
\usepackage{algpseudocode}
\usepackage[parfill]{parskip}
\usepackage{placeins}

\usepackage{bbm}

\usepackage[authoryear]{natbib}
\usepackage[margin=1in]{geometry}

\DeclareFontFamily{U}{mathx}{}
\DeclareFontShape{U}{mathx}{m}{n}{<-> mathx10}{}
\DeclareSymbolFont{mathx}{U}{mathx}{m}{n}
\DeclareMathAccent{\widehat}{0}{mathx}{"70}
\DeclareMathAccent{\widecheck}{0}{mathx}{"71}

\allowdisplaybreaks

\newtheorem{theorem}{Theorem}[section]
\newtheorem{lemma}[theorem]{Lemma}
\newtheorem{corollary}[theorem]{Corollary}

\newtheorem{proposition}[theorem]{Proposition}
\newtheorem{remark}[theorem]{Remark}

\newcommand{\argmin}[1]{\underset{#1}{\operatorname{argmin}}}

\newcommand{\IR}{\mathbb{R}}
\newcommand{\IC}{\mathbb{C}}
\newcommand{\IE}{\mathbb{E}}
\newcommand{\IP}{\mathbb{P}}
\newcommand{\Ind}{\mathbbm{1}}
\newcommand{\IN}{\mathbb{N}}

\newcommand{\cR}{{\cal R}}
\newcommand{\LS}{\textrm{LS}}

\renewcommand{\tilde}{\widetilde}
\renewcommand{\epsilon}{\varepsilon}

\renewcommand{\hat}{\widehat}
\renewcommand{\check}{\widecheck}

\newcommand{\bdot}{\boldsymbol{\cdot}}

\begin{document}
	\title{Common Drivers in Sparsely Interacting Hawkes Processes}
	\date{\today}
	\author{Alexander~Fuchs-Kreiss \\ University of Hildesheim, \\ Institute for Mathematics, Mathematics Education and Computer Science Education \\ and \\ Enno Mammen \\ Heidelberg University, Institute of Applied Mathematics \\ and \\ Wolfgang Polonik \\ UC Davis, Department of Statistics}
	\maketitle

	\begin{abstract} 
		We study a multivariate Hawkes process as a model for time-continuous relational event networks. The model does not assume the network to be known, it includes covariates, and it allows for both common drivers, parameters common to all the actors in the network, and also local parameters specific for each actor. We derive rates of convergence for all of the model parameters when both the number of actors and the time horizon tends to infinity. To prevent an exploding network,  sparseness is assumed. We also discuss numerical aspects.
	\end{abstract}

	\section{Introduction}
	\label{sec:intro}
	\addcontentsline{toc}{section}{\protect\numberline{\thesection}Introduction}
	
	In this work we model and analyze a time-continuous relational network.  We develop theory in an asymptotic framework where both the number of actors (nodes) and the time horizon converge to infinity. Our model contains local parameters that are specific to each actor and global parameters that describe the evolution of the whole network. The network structure is estimated through high-dimensional parameters fulfilling sparsity constraints. The theory is complicated by the fact that local and global parameters can be estimated with different rates of convergence, and that bias terms that arise from the estimation of local parameters affect the estimation of the global parameters. Our model allows the network dynamics to depend on covariate processes. This is a further tool for understanding the global structure of the network. However, it makes the network process nonstationary which further adds to the complexity of the mathematical analysis. 
	
	The guiding example of our study is a social network with a follower/friendship structure, where a sender is affecting a set of receivers, her neighbors, consisting of a subset of the actors in the network paying attention to the sender.
	The assumption then is that once an event has been sent, the receivers themselves are getting active (think of re-tweeting).
	{Our model will be able to respect that users become active, that is, send a post, for two reasons: 1) because they saw other posts about the same topic, or 2) because they have been triggered by an event outside the social media platform.
		Our interest is to estimate, in the first case, which users have an effect on which other users.
		Such an influence structure might be related to the follower network\footnote{The follower network might be available, but often it shows only a snapshot at a certain time, and its dynamics are not observed.}, but it will be different because users who follow each other might simply ignore the other's posts.
		Moreover, the algorithm of the social media platform might show posts to a user that are posted by unknown users.
		We aim, hence, to estimate this influence without further assumptions.
		Our model essentially estimates these effects by searching for pairs of users $(i,j)$ where user $j$ is typically active just before user $i$ becomes active.
		However, concluding an influence from such a search is only valid if we control for the possibility that both users react to the same external event.
		This is the second type of effect mentioned above.
		We suppose to observe additional covariate information\footnote{potentially varying in time and the users} for the users.
		The simplest example would be the time in the current time zone of the user that certainly affects the activity pattern of the user.
		Another example would be posts by extremely influential users or the appearance of articles in other media.
		Both could affect users to write about it on social media themselves.}
	
	We will model the activity of the users by counting processes.
	Therefore, the intuitive assumption that we have discussed above states, more formally, that the sender is causing an increase in the activity rate of the receivers.
	To model this exhibiting dynamic structure, we use a certain Hawkes model \citep{HO71}.
	Such ideas have also been pursued by \citet{CZG22}.
	In contrast to related work, we allow for common drivers or (global) covariates along with additional node-wise (local) covariates to impact the current event rates of the actors.
	We argue that this is interesting from a practical perspective for two reasons.
	First, due to our model, the appearance of events at similar time points might no longer be due to the network structure but could be caused by a common driver.
	This is similar in nature to the causal inference literature where a common driver would be a confounding factor.
	In other words, inclusion of such covariates can facilitate a causal interpretation of the network.
	Second, the covariate effects by themselves might be of interest in a practical application.
	Furthermore, we would like to stress that our approach does not assume the network structure to be known, i.e., the set of receivers for each sender needs to be estimated simultaneously with other parameters in the model. \\[-9pt]
	
	{\em The model:} Formally, our model is as follows. For any actor $i\in\{1,...,n\}$ let $t_i^{(1)},t_i^{(2)},...$ be the random time points of events emanating from $i$ in the observation period $[0,T]$, and for each actor define a counting process $N_{n,i} : (-\infty,T] \to \IN_0 = \{0,1,2,\ldots\}$ via
	$$N_{n,i}(t):=\sum_{j=1}^{\infty}\Ind(t_i^{(j)}\leq t),$$ i.e., $N_{n,i}(t)$ denotes the number of events spreading from actor $i$ up to and including time $t$ (where we assume that there were no events before time $0$). 
	We suppose that the multivariate counting process $N_n:=(N_{n,1},....,N_{n,n})$ \citep[cf.][]{ABGK93} forms a multivariate Hawkes process, where the intensity function $\lambda_{n,i}$ of $N_{n,i}$ has the following form: First, for observed time-dependent covariates $X_{n,i}:(-\infty,T]\rightarrow\IR^q$ with $X_{n,i}(t) = 0$ for $t < 0$,  
	the intensity functions of the counting processes $N_{n,i}(t)$ are modeled as 
	\begin{align}
		\label{eq:Hawkes_dynamics}
		\lambda_{n,i}(t)&:=
		\alpha^*_{n,i}\cdot\nu_0(X_{n,i}(t);\beta^*_n)+\sum_{j=1}^nC^*_{n,ij}\int_{-\infty}^{t-}g(t-r;\gamma_n^*)dN_{n,j}(r), 
	\end{align}
	where $\alpha_n^* := (\alpha_{n,1}^*,\ldots,\alpha_{n,n}^*)\in[0,\infty)^n$, $C_n^* = (C_{n,ij}^*)_{ij} \in[0,\infty)^{n\times n}$, and $\theta_n^*:=(\beta_n^*,\gamma_n^*)\in\IR^p \times \IR$ denote the true parameters that we will aim to estimate. The functions $\nu_0:\IR^q\times\IR^p\to[0,\infty)$ and $g:[0,\infty) \times \IR \to[0,\infty)$ are known functions.
	Standard choices are $\nu_0(x;\beta)=\exp(x^T\beta)$ and $g(u;\gamma)= \exp(-\gamma u), \gamma > 0$. The matrix $C_n^*$ is considered an adjacency matrix of a weighted network $G_n = (V_n,E_n)$ with $V_n = \{1,\ldots,n\}$ (the actors) and $E_n \subset V_n \times V_n$. Thus, $C_{n,ij}^* = 0$ means $(i,j) \notin E_n.$
	\begin{figure}[h]
		\centering
		\includegraphics[height=4cm, width=4cm]{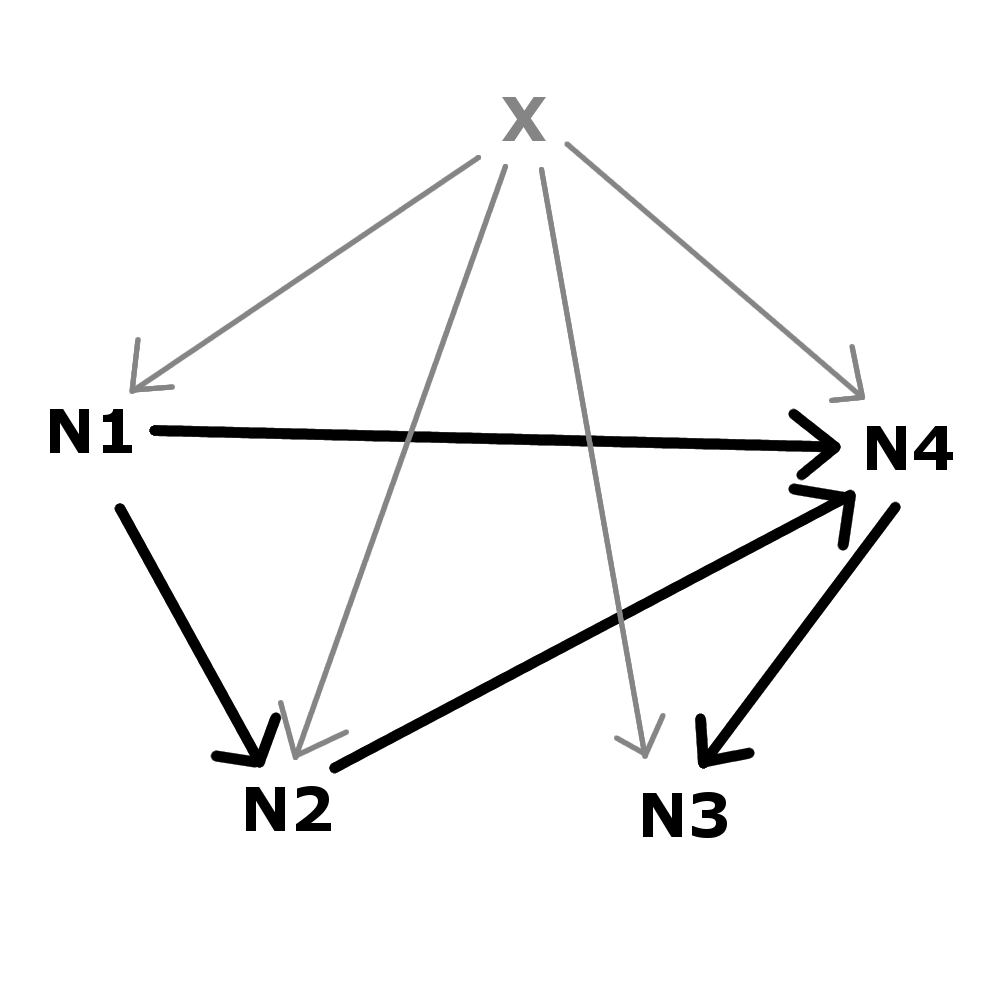}
		\caption{{\small Example of causal DAG for a multivariate process $N=(N_1,...,N_4)$ and confounder $X$.}}
		\label{fig:DAG}
	\end{figure}
	As we will not assume the covariate processes to be stationary, the Hawkes processes considered here are in general not stationary.\\[-9pt]
	
	{\em Contributions:} The contributions of this paper can be summarized as follows:
	\begin{itemize}
		\item We consider a complex time-continuous model that includes possibly non-stationary covariate processes and both global and local parameters. Also the underlying network is not assumed known.
		\item We provide a complete asymptotic analysis with both $n$ and $T$ tending to infinity, showing that the global model parameters can be estimated with a faster rate of convergence as compared to the local parameters. While the latter is not unexpected, the challenge is the bias incurred on the estimates of the local parameters by the penalization approach. This bias also affects the estimation of the global parameters. Without bias correction, the rate of convergence of the global estimator would slow down to the slow rate of the local estimators. To address this problem we apply  a de-biasing procedure for the local estimations and we show that after implementation of this approach the global parameters can be estimated with the rate $O_P((nT)^{-1/2})$ which is the same rate as if the local parameters would be known. For a further discussion of the necessity of the de-biasing step see also Appendix \ref{subsec:debias}. 
		\item We describe the practical benefits of the model, thoroughly discuss numerical aspects, and provide an R-package.
	\end{itemize}
	{\em Discussion of the model:} The general intuition underlying our model is that the larger $\lambda_{n,i}(t)$, the higher is the probability of an event happening in a small neighborhood around time $t$. Hence, an event in process $N_{n,j}$ at a given time $t$ (corresponding to an action of actor $j$) will increase the probability of an action of actor $i$ at time $t$ instantaneously by an amount depending on $C_{n,ij}^*$. The increase declines over time depending on $\gamma$, where the exact rate of decay is determined by $\gamma_n^*$. The mutual excitation among the vertices is hence very explicit in the model and therefore resembles a linear structural causal model as discussed, e.g., in \citet{P00}. Figure \ref{fig:DAG} shows an example of the causal graph induced by a multivariate Hawkes process in which events in $N_1$ have impact on $N_2$ and $N_4$, events in $N_2$ have an impact on $N_4$ and so forth. We have argued that such a causal interpretation of Hawkes processes is useful for interpreting the network, but it is more plausible if common drivers/confounding factors are included in the model as in Figure \ref{fig:DAG}. Ignoring this confounder would prohibit a causal interpretation of the Hawkes structure. In this paper we provide a first step on how to include such confounders in the model. Moreover, since the covariates are actor specific they can also capture a group dependence. Finally, $\alpha_n^*$ allows for heterogeneity in the actors. Hence, the weights $C_{n,ij}^*$ are really only used to model cross-excitation behavior. In contrast to a non-parametric point of view as in \citet{CZG22}, we suggest a parametric approach (in the baseline) which provides a parsimonious model. A non-parametric model might, e.g., set all weights $C_{n,ij}^*$ to zero and use only an actor specific non-parametric baseline to explain the interactions. Moreover, the parametric covariate specific form allows for predictions in settings with different covariates.
	
	Returning to our guiding example of twitter-like social media networks: Suppose our interest is in understanding the structure of information spread originating from one root user $R_0$, and we observe the re-tweets of a tweet of $R_0$. Let $X_{n,i}(t)=X(t)$ for all users other than $R_0$, where $X(t)$ jumps to a higher level if $R_0$ generates an original post and from there it decays. According to our model, the intensity of user $i$ increases once $R_0$ sends a post unless $\alpha_{n,i}^*=0$. Hence, $\alpha_n^*$ may be interpreted as connectedness to the root user $R_0$. If $\alpha_{n,i}^*=0$ for a user $i$, we find that the row $C_{n,i{\boldsymbol \cdot}}^*$ indicates those senders whose retweets can reach $i$.\\[-9pt]
	
	{\em Discussion of related literature. }Recently multivariate counting processes with covariates have been used in applied work, e.g., to model product sales \citep{PMR24}, further examples for applications of multivariate Hawkes processes include \cite{RXSCYvH17} who forecast the popularity of youtube videos and \cite{ZEHRL15} who predict the number of re-tweets of a given tweet. 
	
	Some recent theoretical work on multivariate Hawkes processes can also be found in the 
	literature, see e.g. \citet{HRBR15,YSB21, MM23}, \citet{LC24}. The work \citet{LC24} is perhaps most closely related to our work, but it is using a less general model. In particular, the number of actors is fixed and the Hawkes process considered is stationary.
	The papers \citet{BBGM20,CWS17}
	considered high-dimensional processes with a constant baseline intensity, and 
	\citet{CZG22} studied multivariate Hawkes process models in a high-dimensional setting, where the transferring functions are estimated non-parametrically using B-spline approximations along with a group penalty. In these papers the Hawkes processes for each individual, i.e. the components of the multivariate Hawkes process, are estimated separately. This is possible because, in contrast to our approach, these models do not contain global parameters and no covariate process. 
	In \citet{WKS24} asymptotic theory is developed that allows for the implementation of tests for high-dimensional Hawkes processes.
	\citet{WS21} proposes a deconfounding procedure to
	estimate  networks with only a subset of the nodes being observed.
	Bayesian estimation  of finite dimensional Hawkes processes has been studied in
	\cite{SRR24}. 
	In
	\citet{FXXZG23} 
	latent group memberships of the individuals are assumed which allow dimension reduction by clustering methods. Instead of imposing sparsity on the network, the matrix of excitation kernels can also be assumed to be of low rank as in \cite{LSK17} in order to reduce the number of parameters. In \citet{TL23} the coefficients of the transferring functions are interpreted as a three-way tensor and shrunk by
	low-rank, sparsity, and subgroup constraints. 
	Further examples from this strand of research are \citet{ZRWM21} who study the reconstruction of the influence network in a time discrete domain depending on contexts of the events. 
	\\[-9pt]
	
	The remaining sections are organized as follows. Our estimators are introduced and discussed in section \ref{subsec:estimation}. Our theoretical results are presented in Section~\ref{sec:results}, where stages 1-3 of the estimation procedure are analyzed in Sections~\ref{subsec:single_consistency}-\ref{subsec:consistency}, respectively. In Section~\ref{sec:empirical_results}, we present simulation results and discuss the implementation. Section \ref{sec:conclusion} concludes the paper.\\
	
	Here a brief summary of the different norms and their notation used in this work: For a matrix $A \in \IR^{n\times m}$, we set $\|A\|_1 = \sum_{i,j} |a_{ij}|$, and let $\|A\|_2^2 = \big(\sum_{ij} a_{ij}^2\big)^{1/2}$ denote the $L_2$ or Frobenius norm, and $\|A\|_\infty = \max_{ij} |a_{ij}|$ denotes the maximum norm. The  notation also applies to a vector (either row or column). For a vector-valued function $F: [0,\infty) \to \IR^n$ with $F = (F_1,\ldots,F_n)$ and with $T > 0$, we let $\|F\|^2_T := \sum_{i=1}^n \int_0^T F^2_i(t) dt.$

	\section{Estimation}
	\label{subsec:estimation}
	
	\textbf{Estimation Task:} \emph{Fix $\nu_0$ and $g$. Draw inference about the parameters of the model formulated in \eqref{eq:Hawkes_dynamics} using observations of $N_{n,i}$ and $X_{n,i}$ for $i=1,...,n$. We consider the asymptotic regime where $n\to\infty$ and $T\to\infty,$ simultaneously.}
	
	We will use a least squares approach. To facilitate the estimation of the unknown parameters, the $L_1$-norms of the rows of $C_{n}^*$ will be penalized (see below for a discussion). The global parameters $\beta$ and $\gamma$ are not penalized, resulting in a partially penalized estimation problem. So our rationale is that we prefer models that are primarily explained through the covariates and use effects between the processes only if necessary. The presence of covariates confounds the estimation process which we resolve using a multi-stage procedure.
	
	Formally, the underlying parameter space is defined as follows: Let $0 < K_C, K_\alpha < \infty$ and $\Theta:=K_{\beta}\times K_{\gamma} \subset \IR^p \times \IR$, where $K_\beta$ and $K_\gamma$ are open sets, and define
	\begin{align*}
		\mathcal{H}_n:=\left\{(C,\alpha,\theta)\in[0,K_C)^{n\times n}\times[0,K_\alpha)^n \times\Theta:\,\max_{i\in\{1,...,n\}}\,\|C_{i\bdot}\|_1< \left(\int_0^{\infty}g(t;\gamma)dt\right)^{-1}\right\},
	\end{align*}
	where $\|C_{i\bdot}\|_1$ denotes the $L_1$-norm of the $i$-th row of $C$, and $n \in \IN$.  We discuss in Section \ref{subsec:existence} that Hawkes processes with dynamics described in \eqref{eq:Hawkes_dynamics} indeed exist for $(C_n^*,\alpha_n^*,\theta_n^*)\in\mathcal{H}_n$. The non-negativity constraint in $\mathcal{H}_n$ guarantees that \eqref{eq:Hawkes_dynamics} yields a non-negative intensity function $\lambda_{n,i}$. As an alternative, one could consider non-linear Hawkes processes $N_{n,i}$ with intensity function of the form $\phi(\Psi_{n,i}(t;C_{i\bdot},\alpha_i,\theta))$, where the auxiliary processes $\Psi_{n,i}$ are defined as
	\begin{align}
		\label{eq:approx_model}
		\Psi_{n,i}(t;c,a,\theta):= a\cdot\nu_0(X_{n,i}(t);\beta)+\sum_{j=1}^n c_j\int_{-\infty}^{t-}g(t-r;\gamma)dN_{n,j}(r),
	\end{align}
	with $c\in [0,\infty)^n$, $a\geq0$ and $\theta \in \Theta.$ The function $\phi:\IR\to[0,\infty)$ makes sure that the intensity is non-negative \citep[see, for instance][] {BM96}.
	
	Let
	\begin{align*}
		\Psi_n(t;C,\alpha,\theta)&:=\left(\Psi_{n,1}(t;C_{1\bdot},\alpha_1,\theta),...,\Psi_{n,n}(t;C_{n\bdot},\alpha_n,\theta\right),
	\end{align*}
	and note that with this notation, $\lambda_n(t)=(\lambda_{n,1}(t),...,\lambda_{n,n}(t)) = \Psi_n(t;C^*,\alpha^*,\theta^*).$
	Furthermore, let $\LS_i(C,\alpha,\theta)$ be the least squares criterion defined as
	$$\LS_i(C_{i\bdot},\alpha_i,\theta):=\int_0^T\Psi_{n,i}(t;C_{i\bdot},\alpha_i,\theta)^2dt-2\int_0^T\Psi_{n,i}(t;C_{i\bdot},\alpha_i,\theta)dN_{n,i}(t).$$
	For any stochastic process $F:[0,\infty)\to\IR^n$ with $F=(F_1,...,F_n)$ denote the random path-wise norm as $\|F\|_T^2:=\sum_{i=1}^n\int_0^T|F_i(t)|^2dt$ for $T>0$. With this notation, let
	\begin{align}
		\mathcal{E}(C,\alpha,\theta):&= \left\|\Psi_n(\cdot;C,\alpha,\theta)-\lambda_n\right\|^2_T-\|\lambda_n\|_T^2 \nonumber \\
		&= \sum_{i=1}^n\left(\LS_i(C_{i\bdot},\alpha_i,\theta)+2\int_0^T\Psi_{n,i}(t;C_{i\cdot},\alpha_i,\theta)dM_{n,i}(t)\right). \label{eq:mloss}
	\end{align} 
	Here $M_{n,i}(t) := N_{n,i}(t) - \int_0^t \lambda_{n,i}(s) ds$, and we note that under appropriate conditions, the $M_{n,i}$ are local, square-integrable martingales (see section~\ref{subsec:existence}).  With
	\begin{align*}L_n(C,\alpha,\theta)&:=\frac{1}{nT}\sum_{i=1}^n\textrm{LS}_i(C_{i\bdot},\alpha_i,\theta)+\frac{1}{n}\sum_{i,j=1}^n2\omega_iC_{ij}\\
		&= \frac{1}{n}\sum_{i=1}^n\left(\frac{1}{T}\LS_i(C_{i\bdot},\alpha_i,\theta)+2\omega_i\|C_{i\bdot}\|_1\right)
	\end{align*}
	where $\omega:=(\omega_1,...,\omega_n)\in[0,\infty)^n$ are the tuning parameters for the LASSO penalty (and note that $C$ is constrained to be non-negative), we now define (partially) penalized least-squares type estimators as
	\begin{align}
		(\check{C}_n,\check{\alpha}_n,\check{\theta}_n)	:&= \argmin{(C,\alpha,\theta)\in\mathcal{H}_n}L_n(C,\alpha,\theta) \label{eq:estimator}
	\end{align}
	\begin{remark}
		\label{rem:penalties}
		Penalizing the estimators has practical and theoretical motivations as well as consequences.
		\begin{itemize}
			\item Penalizing rows of $C_n^*$ expresses the idea that each actor only has a limited number of actors she is following. That is, we belief in a sparse information flow.
			\item In order to obtain a non-exploding Hawkes process, it is required that the row sums of $C_n^*$ are bounded (this is explicit in Lemma \ref{lem:goodN}). A penalty is a convenient way to enforce this constraint.
			\item One might believe that many actors are not acting on their own initiative but are only reacting to others. This could be reflected by promoting sparsity in the estimation of $\alpha_n^*$, e.g., by introducing an $L^1$-penalty also on the estimates of $\alpha_n^*$. We will, however, not pursue this in our theory because, if $\alpha_n^*$ is in fact sparse, this would impact the estimation of $\beta_n^*$ because only a few processes would be available for estimation. Therefore, the convergence rate of $\beta_n^*$ must be related to the sparsity. While this situation might be of its own interest and could be part of future research, we will, for simplicity, not assume sparsity in $\alpha_n^*$ and will also not use a penalty when estimating these parameters. Nevertheless, right after the optimization problems \eqref{eq:estimator} and \eqref{eq:split_estimator}, we provide a discussion on how including sparsity in $\alpha_n^*$ would affect the algorithms.
			\item 
			Using a least squares criterion rather than the log-likelihood as objective function allows for efficient numerical computation of the estimators (see section \ref{subsec:LAR} for more on this). Theory about the least squares estimator for Hawkes processes can be found, e.g., in \cite{RBS10,HRBR15}.
		\end{itemize}
	\end{remark}
	It has been noted in the literature that the optimization in \eqref{eq:estimator} with respect to $(C,\alpha)$ can be performed for each $i$ separately because $\LS_i(C,\alpha,\theta)$ is in fact only a function of $C_{i\bdot}$ and $\alpha_i$. More precisely, for any fixed $\theta,$ we denote 
	\begin{equation}
		\label{eq:split_estimator}
		\left(\widehat{C}_{n,i\bdot}(\theta),\widehat{\alpha}_{n,i}(\theta)\right):=\argmin{c\in[0,\infty)^n,a\geq0}\frac{1}{T}\LS_i(c,a,\theta)+2\omega_i\|c\|_1,
	\end{equation}
	where the minimum is taken over all vectors $c$ and numbers $a$ that may appear as $i$-th row of matrices $C$ and $i$-th entry of $\alpha$, for which $(C,\alpha,\theta)\in\mathcal{H}_n$. Then, $(\check{C}_n,\check{\alpha}_n)=(\widehat{C}_n(\check{\theta}_n),\widehat{\alpha}_n(\check{\theta}_n))$. If sparsity in $\alpha_n^*$ is desired, one has to add $2n\omega_{\alpha}a$ in \eqref{eq:split_estimator}. Note that \eqref{eq:split_estimator} implies that, conditionally on $\check{\theta}_n$, we have to solve $n$ individual LASSO problems. This observation is useful for the computations as has also been mentioned in related work, e.g., \cite{CZG22}. However, we emphasize that in our context all these estimation problems are linked by the presence of $\check{\theta}_n$. To address this problem, we suggest the following two-stage estimator.
	\begin{enumerate}
		\item Find $(\widecheck{C}_n,\widecheck{\alpha}_n,\widecheck{\theta}_n)$, the solution of the joint problem \eqref{eq:estimator}.
		\item Compute the de-biased version $\overline{\theta}_n$ of $\widecheck{\theta}_n$.
		\item Find $(\widehat{C}_n,\widehat{\alpha}_n):=(\hat{C}_n(\overline{\theta}_n),\hat{\alpha}_n(\overline{\theta}_n))$, the solution of \eqref{eq:split_estimator}, using the de-biased estimator $\overline{\theta}_n$.
	\end{enumerate}
	The second stage is required  because the first stage penalization approach introduces a bias in the estimation of the local parameters which would  slow down the rate of the global estimator. We already discussed this in the introduction. A further discussion of this point can be found in Appendix \ref{subsec:debias}. To correct for these bias terms we adopt the framework of \citet{vdGBRD14} to our non-linear model to compute a de-biased version $\overline{\theta}_n$ of $\check{\theta}_n$. The third step is suggested by the relation between \eqref{eq:estimator} and \eqref{eq:split_estimator}. This is a natural step because, with $\theta=\overline{\theta}_n$ fixed, problem \eqref{eq:split_estimator} is computationally quick to solve. Moreover, the third stage estimator allows for actor level guarantees (cf. Corollary \ref{cor:ind_conv_rate}), while the first stage estimator is an average over all actors (cf. Proposition~\ref{lem:convergence_rates}). 
	
	{\bf Additional thresholding step.} In our numerical experiments we observed that the cross-validation based selection of the Lasso penalty parameter resulted in estimated networks that were too dense, containing a number of edges with small weights. This motivated us to include a (data driven) thresholding step to threshold the first-stage Lasso estimators of the network coefficients. This turned out to have a positive effect on the performance of the bias correction step for the $\gamma$ parameter (for more details we refer to Section~\ref{sec:empirical_results}). No theoretical results are presented for this modified estimation procedure. While we expect that our technical approach can be adapted to also deal with the thresholded estimator, leading to similar results, we have not pursued this. For further motivation of the thresholded Lasso, we refer to the recent work \citet{Z25}.

	The following sections provide details for the de-biasing step. As already indicated, computational aspects are discussed further in Appendix~\ref{sec:computation}.
	
	\subsection{De-biasing in a Hawkes Model}
	\label{subsec:debiasing}
	In this section, we show how to use the methodology from \cite{vdGBRD14} to de-bias our estimators $(\check{C}_n,\check{\alpha}_n,\check{\theta}_n)$. While the discussion in this section also applies to $\check{C}_n$ and $\check{\alpha}_n$, our main interest is  $\check{\theta}_n$. 
	
	Let $\partial_C L_n(C,\alpha,\theta)$ be the vector of first derivatives of $L_n(C,\alpha,\theta)$ with respect to $C_{ij}$ for all $i,j=1,...,n$, i.e., $\partial_C L_n(C,\alpha,\theta)\in\IR^{n^2}$, where we begin with $C_{1,1},...,C_{1,n}$ and continue with the second row and so forth. $\partial_{\alpha}$ and $\partial_{\theta}$ are defined similarly.
	Denote
	\begin{equation}
		\label{eq:obs_inf}
		\Sigma_n(C,\alpha,\theta):=\frac{1}{nT}\sum_{i=1}^n\begin{pmatrix}
			\partial_{\theta} \\ \partial_{\alpha} \\ \partial_C
		\end{pmatrix}
		\begin{pmatrix}
			\partial_{\theta} \\ \partial_{\alpha} \\ \partial_C
		\end{pmatrix}^\top\textrm{LS}_i(C_{i\bdot},\alpha_i,\theta).
	\end{equation}
	The general motivation behind the de-biasing strategy is very well explained in Sections 2.1 and 3.1 of \citet{vdGBRD14}. The de-biased estimator is defined through the following formula
	\begin{equation}
		\label{eq:debiased_lasso}
		\begin{pmatrix}
			\overline{\theta}_n \\ \overline{\alpha}_n \\ \overline{C}_n
		\end{pmatrix}:=\begin{pmatrix}
			\check{\theta}_n \\  \check{\alpha}_n \\ \check{C}_n
		\end{pmatrix}-\frac{1}{nT}\sum_{i=1}^n\Lambda_n\begin{pmatrix}
			\partial_{\theta} \\ \partial_{\alpha} \\ \partial_C
		\end{pmatrix}\textrm{LS}_i(\check{C}_{n,i\bdot},\check{\alpha}_{n,i},\check{\theta}_n),
	\end{equation}
	where $\Lambda_n\in\IR^{N\times N}, N := p+1+n+n^2$
	is a matrix that we will define shortly. In the proof of the below Theorem \ref{thm:de-biasing}, we show that the motivation from \citet{vdGBRD14} essentially transfers to our setting. We have left to define the matrix $\Lambda_n$. For motivating a choice  for $\Lambda_n$, we note that after neglecting higher order terms in a Taylor expansion of $\textrm{LS}_i(C_{i\bdot},\alpha_i,\theta)$ we get the approximation: \begin{eqnarray*}\begin{pmatrix}
			\overline{\theta}_n \\ \overline{\alpha}_n \\ \overline{C}_n
		\end{pmatrix} - \begin{pmatrix}
			{\theta}^*_n \\ {\alpha}^*_n \\ {C}^*_n
		\end{pmatrix} &\approx& -  \frac{1}{nT}\sum_{i=1}^n\Lambda_n\begin{pmatrix}
			\partial_{\theta} \\ \partial_{\alpha} \\ \partial_C
		\end{pmatrix}\textrm{LS}_i({C}^*_{n,i\bdot},{\alpha}^*_{n,i},{\theta}^*_n) \\ &&- \left ( \Lambda_n \Sigma_n(C_n ^*,\alpha_n^*,\theta_n^*)-I\right )\left (\begin{pmatrix}
			\check{\theta}_n \\ \check{\alpha}_n \\ \check{C}_n
		\end{pmatrix} - \begin{pmatrix}
			{\theta}^*_n \\ {\alpha}^*_n \\ {C}^*_n
		\end{pmatrix}\right).\end{eqnarray*} Thus a natural idea is to choose $\Lambda_n$ such that it
	approximates $\Sigma^{-1}_n(\check{C}_n,\check{\alpha}_n,\check{\theta}_n)$ (even though the latter might not exist). We use a procedure inspired by and very similar to the node-wise Lasso as in \cite{vdGBRD14}. It is, however, different because $\Sigma:=\Sigma_n(\check{C}_n,\check{\alpha}_n,\check{\theta}_n)$ is in our case possibly indefinite. We, therefore, first find $\tilde{\Lambda}_n$ as an approximate inverse of the positive semi-definite matrix $\Sigma^2$ as follows. For $j=1,...,N$, let
	\begin{equation}
		\label{eq:nodewise_lasso}
		v_j:=\argmin{v\in\IR^{N-1}}\left\|\Sigma_{\cdot j}-\Sigma^{\widehat j}v\right\|_2^2+2\sigma_j\|v\|_1,
	\end{equation}
	where $\sigma_j>0$ is a set of tuning parameters, $\Sigma_{\cdot j}$ denotes the $j$-th column of $\Sigma$, and where, for a matrix $A \in \IR^{n \times m}$ we let $A^{\widehat j} \in \IR^{n \times (m-1)}$ denote the submatrix of $A$ obtained by deleting the $j$-th column. 
	Also let $\tau_j:=(\Sigma^2)_{jj}-(\Sigma^2)^{\widehat j}_{j.}v_j$, where $(\Sigma^2)^{\widehat j}_{j.}\in\IR^{1\times(N-1)}$ denotes the $j$-th row of the matrix $(\Sigma^2)^{\widehat j}$.  With this, we define the matrix $\tilde{\Lambda}_n$ as follows ($v_{j,i}$ denotes the $i$-th entry of $v_j$)
	$$\tilde{\Lambda}_n:=\begin{pmatrix}
		\frac{1}{\tau_1}           &         -\frac{v_{1,1}}{\tau_1}           &             -\frac{v_{1,2}}{\tau_1}           & \dots & -\frac{v_{1,N-1}}{\tau_1} \\[6pt]
		-\frac{v_{2,1}}{\tau_2}           &                \frac{1}{\tau_2}           &             -\frac{v_{2,2}}{\tau_2}           & \dots & -\frac{v_{2,N-1}}{\tau_2} \\
		\vdots                    & \vdots                                    & \vdots                                       &       & \vdots                       \\
		-\frac{v_{N,1}}{\tau_{N}} & -\frac{v_{N,2}}{\tau_{N}} &     -\frac{v_{N,3}}{\tau_{N}} & \dots &             \frac{1}{\tau_{N}}
	\end{pmatrix}.$$
	Let $e_j$ denote the $j$-th unit vector. With the same arguments as for (10) in \citet{vdGBRD14}, we obtain for $\Lambda_n:=\tilde{\Lambda}_n\Sigma$
	\begin{equation}
		\label{eq:node_wise_lasso}
		\|\Lambda_{n,j\bdot}\Sigma-e_j^T\|_{\infty}=\|\tilde{\Lambda}_{n,j\bdot}\Sigma^2-e_j\|_{\infty}\leq\frac{\sigma_j}{\tau_j}.
	\end{equation}
	Therefore, $\Lambda_n$ can be thought of as an approximate inverse of $\Sigma$. We refer to the discussion in Section \ref{subsec:lasso_without_intercept} on how to efficiently solve \eqref{eq:nodewise_lasso}. Note that we do not have to compute the complete matrix $\Lambda_n$ if we are just interested, e.g., in $\theta$. In that case, we just have to compute the first $p+1$ rows in \eqref{eq:debiased_lasso} and hence only the first $p+1$ rows of $\Lambda_n$ are required.
	
	\section{Results}
	\label{sec:results}
	Here we present the theoretical results for our estimators by following the structure of the estimation presented in Section \ref{subsec:estimation}: In \ref{subsec:single_consistency}, we study the first stage estimator $(\check{C}_n,\check{\alpha}_n,\check{\theta}_n)$, then the de-biasing in Section \ref{subsec:results_debiasing}, and finally, in Section \ref{subsec:consistency}, we consider the combined estimator. Our set-up is that the processes $N_{n,i}$ and $X_{n,i}$ are adapted w.r.t. right-continuous filtrations $\mathcal{F}_{n,t}$ and the collection $(N_{n,i})_{i=1}^n$ forms a multivariate counting process with intensity functions $\lambda_{n,i}:(-\infty,T]\to[0,\infty)$. Moreover, the processes $X_{n,i}$ are predictable w.r.t. $\mathcal{F}_{n,t}$. 
	
	Appendix \ref{subsec:existence} discusses in more detail the existence of a multivariate Hawkes process as introduced. Here, we state the following assumptions, which guarantee existence and identifiability of the Hawkes process following the dynamics described in \eqref{eq:Hawkes_dynamics}.\\[-9pt]

	\textbf{Assumption (A1)} \emph{The model holds, this is, $(C_n^*,\alpha_n^*,\theta_n^*)\in\mathcal{H}_n$ for all $n$ and the functions $\nu_0$ and $g$ fulfill that for $a_1,a_2\in[0,K_\alpha)$ and $\beta_1,\beta_2\in K_\beta$
		$$\IP(a_1\nu_0(X_{n,i}(t);\beta_1)=a_2\nu_0(X_{n,i}(t);\beta_2)\text{ a.e.})=1\,\Rightarrow\,a_1=a_2\text{ and }\beta_1=\beta_2$$
		and for $c_1,c_2\in[0,K_c)$ and $\gamma_1,\gamma_2\in K_\gamma$
		$$c_1g(\cdot;\gamma_1)=c_2g(\cdot;\gamma_2)\text{ a.e.}\,\Rightarrow\,c_1=c_2\text{ and }\gamma_1=\gamma_2.$$
		Moreover, there exist a  constants $\overline{\nu}_i$ such that $\sup_{\beta \in K_\beta}\nu_0(X_{n,i}(t);\beta)\leq\overline{\nu}_i$ almost surely for all $t > 0,$} and there exists $A > 0$ with $\textrm{supp}(g(\cdot;\gamma))\subseteq[0,A]$ for all $\gamma\in K_{\gamma},$ and 
	$$a_0:=\sup_{n\in\IN}\max_{i=1,...,n}\|C_{n,i\bdot}^*\|_1\int_0^Ag(t;\gamma_n^*)dt<1.$$
	\begin{remark}
		The bounded support of $g(\cdot;\gamma)$ is assumed due to technical reasons because we use results from \cite{HRBR15} that rely on this assumption. We believe that it can be relaxed at the expense of adding further technical arguments. This is not pursued here, however.
	\end{remark}
	
	\subsection{First Stage Consistency}
	\label{subsec:single_consistency}
	In this section, we study
	$$(\check{C}_n,\check{\alpha}_n,\check{\theta}_n):=\argmin{(C,\alpha,\theta)\in\mathcal{H}_n}\frac{1}{n}\sum_{i=1}^n\left(\frac{1}{T}\LS_i(C_{i\bdot},\alpha_i,\theta)+2\omega_i\|C_{i\bdot}\|_1\right).$$
	Below we prove a typical consistency result, Theorem \ref{thm:pre_estimate_consistency}, which holds when the noise can be controlled and a certain \emph{random compatibility condition} holds.
	
	The random compatibility condition defined below depends on the random compatibility constant, defined as follows. For $S \subset\{1,\ldots,n\}$, let $|S|$ denote the number of elements of $S$. \\[-5pt]
	
	\textbf{Random Compatibility Constant }  \emph{For any $C\in\IR^{n\times n}$, $S\subseteq\{1,...,n\}$, and $i\in\{1,...,n\}$, let $C_{iS}\in\IR^{|S|}$ be the vector containing the values $C_{ij}$ for $j\in S$. We define the \emph{random compatibility constant for $S_1,...,S_n\subseteq\{1,...,n\}$, $L>0$,} by
	\begin{align*}
	\phi^2_{\textrm{comp}}&(S_1,...,S_n;L;\mathcal{H}_n) \\
	:&= \inf_{(C,\alpha,\theta) \in \cR_n(S_1,...,S_n;L)}\frac{\frac{1}{nT}\|\Psi_n(\cdot;C,\alpha,\theta)-\lambda_n(\cdot)\|^2_T}{\frac{1}{n}\|\alpha-\alpha_n^*\|_2^2+\frac{1}{n}\sum_{i=1}^n\|C_{iS_i}-C_{n,iS_i}^*\|_2^2+\|\theta-\theta_n^*\|_2^2}
	\end{align*}
	where $\cR_n(S_1,...,S_n;L)\subseteq{\mathcal{H}}_n$ contains all tuples $(C,\alpha,\theta)\in\mathcal{H}_n$ for which the denominator on the right-hand side is non-zero, and
	\begin{align*}
	\frac{1}{n}\sum_{i=1}^n\|C_{iS_i^c}-C_{n,iS_i^c}^*\|_1\leq L\left(\frac{1}{n}\left\|\alpha-\alpha_n^*\right\|_1+\frac{1}{n}\sum_{i=1}^n\|C_{iS_i}-C_{n,iS_i}^*\|_1+\|\theta_n-\theta_n^*\|_1\right).
\end{align*}}
The random compatibility constant is now being used to formulate a restricted eigenvalue condition similar to the one that can be found in Chapter 6.2 in \cite{GB11} or \citet{BYT09} and thus falls in the general class of compatibility conditions.
Let $S_i(C_n^*):=\{j: C_{n,ij}^*\neq0\}$ and $S(\alpha_n^*):=\{j: \alpha_{n,j}^*\neq0\}$ be the active sets of the $i$-th row of a matrix $C_n^*$ and a vector $\alpha_n^*$, respectively.

\vspace{0.2cm}

\textbf{Random Compatibility Condition (RCC)}  \emph{For $L>0$
$$\Omega_{RCC,n}(L,{\mathcal{H}}_n):=\left\{\phi_{\textrm{comp}}(S_1(C_n^*),...,S_n(C_n^*);L;{\mathcal{H}}_n)>0\right\}.$$
For each realization in $\Omega_{RCC,n}(L,{\mathcal{H}}_n)$ we say that the \emph{random compatibility condition (RCC)} holds.}

\vspace{0.1cm}

This assumption is useful once $\IP(\Omega_{RCC,n}(L,\mathcal{H}_n))\to 1.$ We refer to Appendix~\ref{appendic:RCC} where it is shown that in a simplified situation this indeed holds true.

Now we define the important event on which the noise can be controlled as expressed in the following theorem. For given numbers $a_n$, $b_n$, $d_{n,i}$, $e_n, n\in\IN, i=1,...,n$, and with $d_n:=(d_{n,1},...,d_{n,n}),$ define the event
\begin{align}
&\mathcal{T}_n(a_n,b_n,d_n,e_n) \nonumber \\
:&= \,\left\{\sup_{i=1,...,n}\sup_{\overline{\beta}\in K_{\beta}}\frac{2}{T}\left|\int_0^T\nu_0\left(X_{n,i}(t);\overline{\beta}\right)dM_{n,i}(t)\right|\leq a_n\right\} \nonumber \\
&\cap\left\{\sup_{\overline{\beta}\in K_{\beta}}\left|\frac{2}{nT}\sum_{i=1}^n\alpha_{n,i}^*\int_0^T\frac{\nu_0(X_{n,i}(t);\overline{\beta})-\nu_0(X_{n,i}(t);\beta_n^*)}{\left\|\overline{\beta}-\beta_n^*\right\|_1}dM_{n,i}(t)\right|\leq b_n\right\} \nonumber \\
&\bigcap_{i=1}^n\left\{\sup_{j=1,...,n}\sup_{\overline{\gamma}\in K_{\gamma}}\frac{2}{T}\left|\int_0^T\int_0^{t-}g(t-r;\overline{\gamma})dN_{n,j}(r)dM_{n,i}(t)\right|\leq d_{n,i}\right\} \nonumber \\
&\cap\left\{\sup_{\overline{\gamma}\in K_{\gamma}}\left|\frac{2}{nT}\sum_{i,j=1}^nC_{n,ij}^*\int_0^T\int_0^{t-}\frac{g(t-r;\overline{\gamma})-g(t-r;\gamma_n^*)}{|\overline{\gamma}-\gamma_n^*|}dN_{n,j}(r)dM_{n,i}(t)\right|\leq e_n\right\}. \label{eq:def_T}
\end{align}

Note that the integrals in the definition of $\mathcal{T}_n=\mathcal{T}_n(a_n,b_n,d_n,e_n)$ are with respect to the counting process martingales which correspond to the noise in the observations.
Therefore, $\mathcal{T}_n$ can be interpreted as the event under which the noise is bounded by the sequences $a_n$, $b_n$, $d_{n,i}$ and $e_n$.
We will show in Lemma~\ref{lem:lambda_rate} how these sequences have to be chosen for $\mathcal{T}_n$ to have high probability, thereby determining the level of the noise which will later determine the convergence rate.

The following is the main result of this section.
As we will see below, it implies that sequences $\{a_n\}$, $\{b_n\}$, $\{d_{n,i}\}, i=1,...,n$ and $\{e_n\}$ such that the probability of the event $\mathcal{T}_n(a_n,b_n,d_n,e_n)\cap\Omega_{\textrm{RCC},n}(L,{\mathcal{H}}_n)$ tends to one (for $n,T \to \infty$) control the rates of consistency for our parameter estimates.
Our result can be considered an adaptation of Theorem 6.2 in \citet{GB11} to our setting.

\begin{theorem}
\label{thm:pre_estimate_consistency}
Suppose that (A1) holds. 
For $n \in \IN$ let $a_n,b_n,d_{n,i},e_n\in\IR$, and let $\tilde{b}_n\geq\max(b_n,e_n)$. Suppose that $d_{n,i}\leq\omega_i$ and that there is a number $L\in(0,\infty)$ such that
$$L\geq\frac{\max(3\omega_1,...,3\omega_n,2a_n,2\tilde{b}_n)}{\min(\omega_1,...,\omega_n)}.$$
Denote\\[-20pt]
$$\mathcal{L}(C_n^*,\alpha_n^*):=\sqrt{9a_n^2+\frac{64}{n}\sum_{i=1}^n\omega_i^2|S_i(C_n^*)|+9\tilde{b}_n^2(p+1)}.$$
Then, on the event
$$\mathcal{T}_n(a_n,b_n,d_n,e_n)\cap\Omega_{\textrm{RCC},n}(L,\mathcal{H}_n),$$
we have that
\begin{align*}
&\frac{1}{nT}\left\|\Psi_n(\cdot;\check{C}_n,\check{\alpha}_n,\check{\theta}_n)-\lambda_n\right\|_T^2+\frac{2}{n}\sum_{i=1}^n\omega_i\|\check{C}_{n,i\cdot}-C_{n,i\cdot}^*\|_1+\frac{a_n}{n}\|\check{\alpha}_n-\alpha_n^*\|_1\nonumber+\tilde{b}_n\|\check{\theta}_n-\theta_n^*\|_1 \\
&\hspace*{6cm}\leq\frac{\mathcal{L}(C_n^*,\alpha_n^*)^2}{4\phi_{\textrm{comp}}(S_1(C_n^*),...,S_n(C_n^*);L;\mathcal{H}_n)^2}.\\
\end{align*}
\end{theorem}
As already indicated, the above theorem implies bounds on the convergence rate of the estimators. To formulate this, we need the following assumptions:
\\[-5pt]

\textbf{Assumption (A2)} \emph{There exists a (non-random) sequence $c_n\geq1$ with $\max_{i=1,\ldots,n}\|C_{n,i\cdot}^*\|_1\leq c_n$ for all $n\in\IN$. Moreover, with $\overline{\nu}_i$ from (A1) and some $\overline{g}<\infty$}
\begin{align*}
&\sup_{i\in\IN}\|\overline{\nu}_i\|_{\infty}<\infty,\quad \sup_{\overline{\gamma}\in K_{\gamma}}\|g(\cdot;\overline{\gamma})\|_{\infty}\leq \overline{g}.
\end{align*}

\textbf{Assumption (A3)} 
\emph{Suppose that $\nu_0$ and $g$ are continuously differentiable with respect to $\beta$ and $\gamma$, respectively, such that, for deterministic constants $D_{\nu},L_{\nu},D_g,L_g<\infty$, almost surely 
\begin{align*}
&\left\|\frac{d}{d\beta}\nu_0\left(X_{n,i}(t);\beta_1\right)-\frac{d}{d\beta}\nu_0\left(X_{n,i}(t);\beta_2)\right)\right\|_{\infty}\leq D_{\nu}\|\beta_1-\beta_2\|_2, \\
&\sup_{\overline{\beta}\in K_{\beta}}\left\|\frac{d}{d\beta}\nu_0\left(X_{n,i}(t);\overline{\beta}\right)\right\|_2\leq L_{\nu} \\
&\left|\frac{d}{d\gamma}g(r;\gamma_1)-\frac{d}{d\gamma}g(r;\gamma_2)\right|\leq D_g|\gamma_1-\gamma_2|, \sup_{\overline{\gamma}\in K_{\gamma}}\left|\frac{d}{d\gamma}g(r;\overline{\gamma})\right|\leq L_g
\end{align*}
for all $i\in\{1,...,n\}$, $n\in\IN$, all $\beta_1,\beta_2\in K_{\beta}$, all $\gamma_1,\gamma_2\in K_{\gamma}$, all $t\in[0,T]$, and all $r\in[0,A]$.}\\

Using these assumptions, we have the following rates of convergences for our estimators. Its proof is based on Lemma \ref{lem:lambda_rate} and Lemma~\ref{cor:lambda_rate} below and can be found in Section \ref{subsec:proofs_single_consistency}.

\begin{proposition}
\label{lem:convergence_rates}
Suppose that (A1), (A2), and (A3) hold. Let $\alpha_1,...,\alpha_4>0$ such that, as $n,T\to\infty$, 
\begin{align*}
&\frac{c_n^p\log T}{(nT)^{\alpha_1}}+\frac{c_n^{2p}\log(nT)}{(nT)^{\alpha_2}}+\frac{c_n\log(T)}{T^{\alpha_3}}+\frac{c_n\log(nT)}{(nT)^{\alpha_4}}\to0.
\end{align*}
Assume furthermore that there is $\phi_0>0$ such that $\IP(\phi_{\textrm{comp}}(S_1(C_n^*),...,S_n(C_n^*);L;{\mathcal{H}}_n)\geq\phi_0)\to1$ as $n \to \infty$. Then,
\begin{align*}
\frac{1}{n}\left\|\check{C}_n-C_n^*\right\|_1 &= O_P\left(\frac{\log^2(nT)}{\sqrt{T}}s_n\right), \\
\frac{1}{n}\left\|\check{\alpha}_n-\alpha_n^*\right\|_1& =O_P\left(\frac{\log^3(nT)}{\sqrt{T}}s_n\right), \\
\left\|\check{\theta}_n-\theta_n^*\right\|_1& =O_P\left(\frac{\log^2(nT)}{\sqrt{T}}\frac{s_n}{c_n}\right).
\end{align*}
where
$$s_n:=1+\frac{1}{n}\sum_{i=1}^n|S_i(C_n^*)|+c_n^2.$$
\end{proposition}

The derivation of the rates of convergence presented above requires to control the probability of
$$\mathcal{T}_n(a_n,b_n,d_n,e_n)\cap\Omega_{\textrm{RCC},n}(L,\mathcal{H}_n).$$
Note that the probability of $\Omega_{\textrm{RCC},n}(L,\mathcal{H}_n)$ can be increased by reducing the parameter space $\mathcal{H}_n$. It plausibly has a high probability if, e.g., it imposes a non-zero lower bound on $\alpha$ because then changes in $\beta$ are enforced to be visible in the intensity function. 

\begin{remark}
\label{rem:discussion_first_stage_rates}
The above rates resemble classical rates in high-dimensional regression: $s_n$ is the average sparsity of the model, while $T$ corresponds to the number of observations. The different $\log$ factors are due to the concentration inequality for Hawkes processes that we use (rather than classical exponential inequalities for sub-Gaussian random variables in high-dimensional regression tasks).
The differences between the individual rates is due to the fact that the noise, cf. the discussion after Equation~\eqref{eq:def_T}, for the respective quantities is different.
\end{remark}

The following result provides sufficient conditions for  $\IP(\mathcal{T}_n(a_n,b_n,d_n,e_n))\to1$, and it provides corresponding choices for the sequences $\{a_n\} - \{e_n\}.$ It is worth noting, however, that Lemma~\ref{lem:lambda_rate} (on which the proof of Lemma~\ref{cor:lambda_rate} is based), provides more complex (random) choices for the sequences $\{a_n\} - \{e_n\}$ which, however, only depend on known and/or interpretable quantities. This comes in very handy, because in view of Theorem~\ref{thm:pre_estimate_consistency}, we can use the known sequences $d_n,i$ for parameter tuning. This will be discussed further in section~\ref{subsec:cross-validation}.

\begin{lemma}
\label{cor:lambda_rate}
Suppose that (A1), (A2), and (A3) hold. Choose $\alpha_1,...,\alpha_4>0$ as in Proposition \ref{lem:convergence_rates}. Then, there are finite constants $K_a,K_b,K_d,K_e>0$ such that $\IP(\mathcal{T}_n(a_n,b_n,d_n,e_n))\to1$ as $nT\to\infty$,  for
\begin{align*}
a_n& =K_a\frac{\log(nT)}{\sqrt{T}}, \qquad b_n=K_b\frac{\log(nT)}{\sqrt{nT}}, \\
d_{n,i}& =K_d\frac{\log^2(nT)}{\sqrt{T}}, \qquad e_n=K_e\frac{c_n\log^2(nT)}{\sqrt{nT}}.
\end{align*}
\end{lemma}

\subsection{Second Stage: De-biasing}
\label{subsec:results_debiasing}
We first introduce some further notation and assumptions. Recall $\Lambda_n \in \mathbb R^{N\times N}$ from Section \ref{subsec:debiasing} (and the notation $N = p+1+n+n^2$) and let
$$\Lambda_{\theta,n}\in\IR^{(p+1)\times N}$$
denote the first $p+1$ rows of $\Lambda_n$.
Furthermore, for ease of notation, we define
$$S_{n,i}(t):=\begin{pmatrix}
\partial_{\theta} \\ \partial_{\alpha} \\ \partial_C
\end{pmatrix}\Psi_{n,i}(t;C_{n,i\cdot}^*,\alpha_{n,i}^*,\theta_n^*)$$
and 
$$V_n:=\frac{1}{nT}\sum_{i=1}^n \int_0^T\IE\left(S_{n,i}(t)S_{n,i}(t)^T\Psi_{n,i}(t;C_{n,i\cdot}^*,\alpha_{n,i}^*,\theta_n^*)\right)dt.$$
Furthermore, define the random sequence $\check{S}_n$ and a sequence $r_n$ as follows
\begin{align*}
&\check{S}_n:=\max_{i=1,...,n}\left(c_n,\|\check{C}_{n,i\bdot}\|_1\right)\left(c_n+1\right) \\
&\left\|\check{\theta}_n-\theta_n^*\right\|_1+\frac{1}{n}\left\|\check{\alpha}_n-\alpha_n^*\right\|_1+\frac{1}{n}\left\|\check{C}_n-C_n^*\right\|_1= O_P(r_ns_n),
\end{align*}
where $s_n$ is defined as in Proposition \ref{lem:convergence_rates}. Recall also the definitions of $\sigma_j$ and $\tau_j$ from \ref{subsec:debiasing}. \\[-9pt]

\textbf{Assumption (B1)}  \emph{Let $\nu_0$ and $g$ be twice differentiable with Lipschitz continuous second derivatives.}

\textbf{Assumption (B2)} \emph{It holds that
\begin{align*}
n^{\frac{3}{2}}\sqrt{T}\log^2(nT)\check{S}_n\left\|\Lambda_{\theta,n}\right\|_{\infty}r_n^2s_n^2&=o_P(1), \\
\max_{j=1,...,p+1}\frac{1}{\tau_j}&=O_P(1), \textrm{ and} \\
n^{\frac{3}{2}}\sqrt{T}r_ns_n\max_{j=1,...,p+1}\sigma_j& =o(1).
\end{align*}}

\textbf{Assumption (B3)}  \emph{There exists a deterministic sequence of matrices $\Lambda_{0,\theta,n}\in\IR^{(p+1)\times N}$ such that $\Lambda_{0,\theta,n}V_n\Lambda_{0,\theta,n}^T\in\IR^{(p+1)\times(p+1)}$ is positive definite, converges to a positive definite matrix $M_0\in\IR^{(p+1)\times(p+1)}$ and that
$$\left\|\Lambda_{\theta,n}-\Lambda_{0,\theta,n}\right\|_{\infty}=o_P\left(\frac{1}{c_n\log(nT)}\right).$$}

\textbf{Assumption (B4)} 
\emph{For any $a\in\IR^{p+1}$ and with $M_0$ from (B3), 
\begin{align*}
&\frac{1}{(nT)^{\frac{3}{2}}}\sum_{i=1}^n\int_0^T\left|a^T\left(\Lambda_{0,\theta,n}V_n\Lambda_{0,\theta,n}^T\right)^{-\frac{1}{2}}\Lambda_{0,\theta,n}S_{n,i}(t)\right|^3\Psi_{n,i}(t;C_{n,i\cdot}^*,\alpha_{n,i}^*,\theta_n^*)dt=o_P(1);\\
&\Lambda_{0,\theta,n}\frac{1}{nT}\sum_{i=1}^n\int_0^TS_{n,i}(t)S_{n,i}(t)^T\Psi_{n,i}(t;C_{n,i\cdot}^*,\alpha_{n,i}^*,\theta_n^*)dt\Lambda_{0,\theta,n}^T\overset{\IP}{\to}M_0.
\end{align*}}

\begin{theorem}
\label{thm:de-biasing}
Suppose that Assumptions (A1)-(A3) and (B1)-(B4) hold. If, in addition, $\frac{\log (nT)}{\sqrt{n}} = o(1)$, then,
$$\sqrt{nT}\left(\Lambda_{0,\theta,n}V_n\Lambda_{0,\theta,n}^T\right)^{-\frac{1}{2}}\left(\overline{\theta}_n-\theta_n^*\right)\overset{d}{\to}\mathcal{N}(0,4I_{p+1}).$$
\end{theorem}
The proof of the Theorem is presented in Section \ref{subsec:proofs_debiasing}.\\[-9pt]

{\bf Discussion of Assumptions (B1) - (B4).} 

{\sc Assumption (B1)} is a typical assumption that is used to bound remainder terms in second-order Taylor expansions. Formally, we mean that $\beta\mapsto\nu_0(X_{n,i}(t);\beta)$ is twice continuously differentiable and the second derivative (which is random and depends on $i$ and $t$) is Lipschitz continuous where the Lipschitz constant can be chosen deterministically and independently of $i$ and $t$. Since we mainly think of $g$ and $\nu_0$ as exponential functions, we consider it to be not restrictive under the additional assumption of bounded covariates.

{\sc Assumptions (B2) and (B3):} Consider the context of Subsection~\ref{subsec:single_consistency} and suppose that $\IE[\Sigma_n(C_n^*,\alpha_n^*,\theta_n^*)]^{-1}$ exists and is uniformly bounded (entry-wise).

{\bf a)} Using Equation~\eqref{eq:LS_matrix2} and Lemma~\ref{lem:Rn}, it is relatively clear that $\Sigma=\Sigma_n(\check{C}_n,\check{\alpha}_n,\check{\theta}_n)$ converges to $\Sigma_{0,n}:=\IE(\Sigma_n(C_n^*,\alpha_n^*,\theta_n^*))$.
Thus, the deterministic matrix $\Lambda_{0,\theta,n}$ constructed as in Section~\ref{subsec:debiasing} with $\Sigma$ replaced by the deterministic matrix $\Sigma_{0,n}$ can be expected to be, in the limit, close to $\Lambda_{\theta,n}$\footnote{Note that $\Sigma_{0,n}^{-1}$ (if it exists) is not expected to be spare in the rows. A general convergence theory requires therefore a discussion of high-dimensional regression in non-sparse situations. We do not pursue this further, but assume instead that $\Lambda_{\theta,n}$ lies close to a suitable deterministic quantity $\Lambda_{0,\theta,n}$.}.
We assume a very slow convergence rate in Assumption~(B3).
Then, one can expect that under weak conditions, $\|\Lambda_{\theta,n}\|_{\infty}$ remains bounded because the behavior of $\Lambda_{0,\theta,n}$ is guided by $\IE[\Sigma_n(C_n^*,\alpha_n^*,\theta_n^*)]^{-1}$, which we have assumed to be bounded for this discussion.

{\bf b)} Proposition \ref{lem:convergence_rates} shows that we may choose $r_n=\log^3(nT)/\sqrt{T}$ (the worst case rate divided by $s_n$).
Thus, the first requirement of Assumption~(B2) holds if

$$\frac{n^{\frac{3}{2}}\log^8(nT)\check{S}_ns_n^2}{\sqrt{T}} = o_P(1).$$
It seems plausible to assume that $\check{S}_n$ grows very slowly, essentially requiring that $n^3/T\to0$.
Note that we aim to estimate $n+n^2+p+1$ many parameters with $T$ observations per vertex.
Thus, the model is still high-dimensional although from existing theory for high-dimensional estimation one might expect to see the condition $(\log n)/T\to0$.
We believe that the reason for our stronger condition is that, in the first step, each vertex brings its own bias to the Lasso estimator.
Thus, for the estimation of the global parameter $\theta$ these biases add up.
It is unclear to us if this is inherent to the methodology or due to our proof technique.
In any case, in a less high-dimensional regime, the biases are small enough to be handled by the de-biasing scheme which we present here.

{\bf c)} The remaining requirements of Assumption~(B2) are less restrictive as we may choose $\sigma_j$ ourselves.
Note that (unless there are zero columns in $\Sigma$) it is guaranteed that $\tau_j>0$.
More specifically, we have as in \citet{vdGBRD14}
$$\tau_j=\|\Sigma_{\bdot,j}-\Sigma_{\bdot,-j}v_j\|_2^2+\sigma_j\|v_j\|_1.$$
Hence, a small $\sigma_j$ likely yields a small $\tau_j$ because the regression error is also reduced.
Nevertheless, in the less high-dimensional regime, it seems plausible to us that $\tau_j^{-1}$ converges to the corresponding diagonal entry of $\Sigma_{0,n}^{-1}$, which we assume in this discussion to be bounded from above for $j=1,\ldots,p+1$.

{\bf d)} Note, furthermore, for Assumption~(B3), that the matrix $\Lambda_{0,\theta,n}V_n\Lambda_{0,\theta,n}$ is always positive semi-definite.
Thus, we assume in (B3) its invertibility, which, in view of the fact that $p$ is fixed, seems a weak assumption.
The existence of the limit $M_0$ is an assumption that we make for technical convenience, and it can probably be relaxed.
We actually require only that $\Lambda_{0,\theta,n}V_n\Lambda_{0,\theta,n}^T$ interacts well with the expression in the second condition of Assumption~(B4).
However, our formulation with an assumed limit $M_0$ seems to us to be only mildly more restrictive because it refers only to entries corresponding to $\theta$.
It should therefore stabilize if the role of $\theta$ remains the same when $n$ and $T$ increase.

{\sc Assumption (B4):} We discuss firstly the entries of $S_{n,i}(t)$.
For simplicity, we write $\Psi_{n,i}^*:=\Psi_{n,i}(t;C_{n,i}^*,\alpha_{n,i}^*,\theta_n^*)$.
It holds that $\partial_{\beta}\Psi_{n,i}^*$ is a function of $X_{n,i}(t)$ and does not depend on the counting processes.
Furthermore, $\partial_{\gamma}\Psi_{n,i}^*$ depends only on the counting processes but not on $X_{n,i}(t)$.
We also have that $\partial_{\alpha_k}\Psi_{n,i}^*\neq0$ only for $k=i$ and similarly $\partial_{C_{k,l}}\Psi_{n,i}^*\neq0$ only for $k=i$.
In other words $S_{n,i}(t)$ has at most $q+2+n$ non zero entries that are stochastically bounded, cf. Appendix~\ref{subsec:bounds_Hawkes}.

{\bf a)} We can construct the rows of $\Lambda_{0,\theta,n}$ as solutions of an $L^1$-penalized optimization, and hence we expect it to have sparse rows.
Therefore, $\Lambda_{0,\theta,n}S_{n,i}(t)$ has moderately sized entries.
The first part of Assumption~(B4) is hence requiring that $(nT)^{-1/6}$ can counteract its potential growth.

{\bf b)} Recall the definition of $M_0$ from Assumption (B3).
The second part of Assumption~(B4) requires that its sample version converges.
The only difficulty here is that $S_{n,i}$ are of increasing dimension, and we have to argue that the sample version of $V_n$ converges to $V_n$.
Recalling the structure of $S_{n,i}(t)$, we observe that the matrix
$$\sum_{i=1}^n\int_0^TS_{n,i}(t)S_{n,i}(t)^T\Psi_{n,i}(t;C_{n,i}^*,\alpha_{n,i}^*,\theta_n^*)dt$$
is dense, but that only the sub-matrix corresponding to $\theta$ is an actual sum of $n$ terms.
In all other entries, the sum has only one non-zero term.
Thus, the factor $1/n$ likely lets those entries of the matrix vanish.
For the sub-matrix corresponding to $\theta$ we require that averages over functions of the covariates and the Hawkes processes converge to their expectations.
It can be handled by standard techniques.

\subsection{Third Stage Estimation}
\label{subsec:consistency}
In this section, we study the performance of the third stage estimators $\hat{C}_n$ and $\hat{\alpha}_n$ which are obtained from the de-biased estimator $\overline{\theta}_n$. More specifically, we define for every $\theta\in\Theta$
\begin{align}
(\hat{C}_n(\theta),\hat{\alpha}_n(\theta)):&= \argmin{(C,\alpha)\in \mathcal{H}_n(\theta)}\frac{1}{n}\sum_{i=1}^n\left(\frac{1}{T}\LS_i(C_{i\bdot},\alpha_i,\theta)+2\omega_i\|C_{i\bdot}\|_1\right), \label{eq:partial_optim}
\end{align}
where
$$\mathcal{H}_n(\theta) = \{(C,\alpha): (C,\alpha,\theta) \in {\cal H}_n\}.$$
The third stage estimator is, hence, given by $(\hat{C}_n,\hat{\alpha}_n)=(\hat{C}_n(\overline{\theta}_n),\hat{\alpha}_n(\overline{\theta}_n))$. Note that, similarly as in \eqref{eq:split_estimator}, the representation in \eqref{eq:partial_optim} allows conveniently for a term-wise optimization of the sum. This allows us to formulate all results and assumptions separately for each vertex. Specifically, if interest lies on a subset of vertices, it is not required to compute the entire estimator, but one may restrict the rows of $C$ and entries of $\alpha$ to those of interest.

While we keep assuming in the following that the true intensity functions of the observed counting processes have the Hawkes form as in \eqref{eq:Hawkes_dynamics}, we need to acknowledge in this section that the estimate for $\theta_n^*$ is not correct. Therefore, we have to study the estimators for $C_n^*$ and $\alpha_n^*$ that best represent the data under a potentially wrong $\overline{\theta}_n$. Hence, the main result of this paper is an oracle type inequality \citep[similar to, e.g., Theorem 6.2 in ][]{GB11}. The key difference is that we study a random loss. Therefore, we will have to define, for each vertex, a random compatibility constant.

For a given $\theta$, we call a pair $(C,\alpha)$ an oracle in this context if it best approximates the true, unknown intensity functions $\lambda_{n,i}$ among all $(C,\alpha)$ that have the same sparsity structure as $C_n^*$. The following lemma contains a formal definition. Its proof will be given in Section \ref{subsec:proofs_consistency}.
\begin{lemma}
\label{lem:od}
Let (A1) hold and fix $\theta\in\Theta$. There exists $(C_n^*(\theta),\alpha_n^*(\theta),\theta)\in\overline{\mathcal{H}}_n$ (the closure of $\mathcal{H}_n$) such that we have for all $i\in\{1,...,n\}$
$$(C_n^*(\theta),\alpha_n^*(\theta))\in\argmin{(C,\alpha)}\left\|\Psi_{n,i}(\cdot;C_{i\bdot},\alpha_i,\theta)-\lambda_{n,i}\right\|_T^2,$$
where the argmin is taken over all $(C,\alpha)\in\mathcal{H}_n(\theta)$ with $S_i(C)=S_i(C_n^*)$ for all $i=1,...,n$. $(C_n^*(\theta),\alpha_n^*(\theta))$ is called the \emph{oracle for given $\theta$}.
\end{lemma}

Note that $(C_n^*(\theta_n^*),\alpha_n^*(\theta_n^*))=(C_n^*,\alpha_n^*)$. In order to avoid notation overload, we define the individual compatibility constant directly for the true active sets $S_i(C_n^*)$.

{\bf The individual random compatibility constant.} For any $i\in\{1,...,n\}$, $\theta\in\Theta$, and any $L>0,$ we define the \emph{individual random compatibility constant for $i$, $\theta$, and $L$} by
$$\phi_{i,\textrm{comp}}(L;\theta):=\sqrt{\inf_{(C,\alpha) \in \tilde\cR_n(i;\theta,L)}\frac{\frac{1}{T}\|\Psi_{n,i}(\cdot;C_{i\bdot},\alpha_i,\theta)-\Psi_{n,i}(\cdot;C_{n,i\bdot}^*(\theta),\alpha_{n,i}^*(\theta),\theta)\|^2_T}{\|C_{iS_i(C_n^*)}-C_{n,iS_i(C_n^*)}^*(\theta)\|_2^2+|\alpha_i-\alpha_{n,i}^*(\theta)|^2}},$$
where $\tilde{\cR}_n(i;\theta,L) \subseteq \mathcal{H}_n(\theta)$ is the set of all pairs $(C,\alpha) \in \mathcal{H}_n(\theta)$ for which
\begin{align*}
\|C_{n,iS_i(C_n^*)^c}-C_{n,iS_i(C_n^*)^c}^*(\theta)\|_1\leq7\left\|C_{n,iS_i(C_n^*)}-C^*_{n,iS_i(C_n^*)}(\theta)\right\|_1+3L\left|\alpha_i(\theta)-\alpha^*_{n,i}(\theta)\right|.
\end{align*}

\vspace*{0.3cm}
The individual random compatibility constant can be used to formulate a restricted eigenvalue condition similar to the one that can be found in Chapter 6.2 in \cite{GB11} or \citet{BYT09} and thus falls in the general class of compatibility conditions.  Remark~\ref{rem:mot_ind_comp} provides a discussion for a special case in which the individual random compatibility constant can be bounded from below uniformly for all $i=1,...,n$. For $i=1,...,n$ and $a_n,d_{n,i}>0$, define the events
\begin{align*}
\mathcal{T}_n^{(i)}(a_n,d_{n,i}):&= \left\{\sup_{\overline{\beta}\in K_{\beta}}\left|\frac{2}{T}\int_0^T\nu_0(X_{n,i}(t);\overline{\beta})dM_{n,i}(t)\right|\leq a_n\right\} \\
&\cap\left\{\sup_{j=1,...,n}\sup_{\overline{\gamma}\in K_{\gamma}}\left|\frac{2}{T}\int_0^T\int_0^{t-}g(t-r;\overline{\gamma})dN_{n,j}(r)dM_{n,i}(t)\right|\leq d_{n,i}\right\}.
\end{align*}
Using this notation, we can now formulate the following oracle result, the proof of which is similar to that of Theorem 6.2 in \citet{GB11}; we provide the details for our setting for the sake of completeness in Appendix~\ref{sup:consistency}

\begin{theorem}
\label{thm:oracle}
Suppose that (A1) holds and let $\theta$ be an arbitrary $\Theta$-valued random variable. Suppose that $3d_{n,i}\leq\omega_i$ and let $L>\sup_{i=1,...,n}a_n/\omega_i$. For all $i\in\{1,...,n\}$, define (let $x/0:=\infty$ for all $x\in\IR$)
$$\epsilon^*_i(\theta):=\frac{4}{T}\left\|\Psi_{n,i}(\cdot;C^*_{n,i\bdot}(\theta),\alpha_{n,i}^*(\theta),\theta)-\lambda_{n,i}\right\|_T^2+9\cdot\frac{4\omega_i^2|S_i(C_n^*)|+a_n^2}{\phi_{i,\textrm{comp}}^2(L;\theta)},$$
where $(C_n^*(\theta),\alpha_n^*(\theta))$ is the oracle for $\theta$ as in Lemma \ref{lem:od}. Then, on the event $\mathcal{T}_n^{(i)}(a_n,d_{n,i})$,
\begin{align*}
&\frac{1}{T}\left\|\Psi_{n,i}(\cdot;\hat{C}_{n,i\bdot}(\theta),\hat{\alpha}_{n,i}(\theta),\theta)-\lambda_{n,i}\right\|_T^2+2\omega_i\|\hat{C}_{n,i\bdot}(\theta)-C^*_{n,i\bdot}(\theta)\|_1 \\
&\qquad+2a_n|\hat{\alpha}_n(\theta)_{n,i}-\alpha^*_{n,i}(\theta)|\leq2\epsilon^*_i(\theta).
\end{align*}
\end{theorem}

To derive convergence rates from Theorem \ref{thm:oracle}, we have to assume a lower bound on $\phi_{i,\textrm{comp}}(L;\theta)$. Since  the following is a random condition that applies to individuals, we call it the \emph{Individual Random Compatibility Condition}.\\[-9pt]

\noindent
\textbf{Individual Random Compatibility Condition (IRCC)}  \emph{For $L>0$, $\phi_0>0$, and $U\subseteq\Theta$, we define the event
$$\Omega_{IRCC}^{(i)}(L,\phi_0;U):=\left\{\inf_{\theta\in U}\phi_{i,\textrm{comp}}(L;\theta)>\phi_0\right\}.$$
For each realisation in $\Omega_{IRCC}^{(i)}(L,\phi_0;U)$ we say that the \emph{individual random compatibility condition (IRCC)} holds.}

\begin{corollary}
\label{cor:ind_conv_rate}
Let Assumptions (A1)-(A3), and (B1)-(B4) hold. Let $a_n$ and $d_{n,i}$ be as in Lemma \ref{lem:lambda_rate}, with the constants chosen such that the right hand side of \eqref{eq:probT_bound} converges to zero, and set $\omega_i:=3\sqrt{\log(nT)}d_{n,i}$. Suppose that there are constants $\phi_0$, $r$, $\kappa_1$, $\kappa_2$, $K_1$, $K_2>0$ such that
\begin{align*}
&\IP\left(\bigcap_{i=1}^n\Omega_{\textrm{IRCC}}^{(i)}\left(L,\phi_0;B_r(\theta_n^*)\right)\right)\to1, \\
&\IP\left(\kappa_1\sqrt{\frac{\log(nT)}{T}}\leq a_n\leq K_1\frac{\log(nT)}{\sqrt{T}}\right)\to1, \\
&\IP\left(\kappa_2\sqrt{\frac{\log(nT)}{T}}\leq d_{n,i}\leq K_2\frac{\log^2(nT)}{\sqrt{T}}\textrm{ for all }i=1,...,n\right)\to1,
\end{align*}
where $B_r(\theta_n^*)\subseteq\Theta$ denotes a ball of radius $r$ centered around $\theta_n^*$. Suppose furthermore that
\begin{equation}
\label{eq:growth_condition}
\sup_{i=1,...,n}\frac{(1+|S_i(C_n^*)|)(c_n^2+\|C_{n,i\bdot}^*(\overline{\theta}_n)\|_1^2)^2\log(nT)}{n}=O_P(1).
\end{equation}
Then,
\begin{align*}
&\frac{1}{T}\left\|\Psi_{n,i}(\cdot;\hat{C}_{n,i\bdot},\hat{\alpha}_{n,i},\overline{\theta}_n)-\lambda_{n,i}\right\|_2^2+2\omega_i\|\hat{C}_{n,i\bdot}-C^*_{n,i\bdot}\|_1+2a_n|\hat{\alpha}_{n,i}-\alpha^*_{n,i}| \\
&\qquad\qquad=O_P(1)\cdot\left((1+|S_i(C_n^*)|)\frac{\log^5(nT)}{T}\right), \\
&\|\hat{C}_{n,i\bdot}-C^*_{n,i\bdot}\|_1=O_P(1)\cdot\left((1+|S_i(C_n^*)|)\frac{\log^4(nT)}{\sqrt{T}}\right), \\
&|\hat{\alpha}_{n,i}-\alpha^*_{n,i}|=O_P(1)\cdot\left((1+|S_i(C_n^*)|)\frac{\log^{\frac{9}{2}}(nT)}{\sqrt{T}}\right).
\end{align*}
\end{corollary}
The proof is straight-forward conceptually, but has some technical difficulties that we deal with in Section \ref{subsec:proofs_consistency}.
Note that the conditions on $a_n$ and $d_{n,i}$ are only restrictive about the lower bound.
The upper bounds can be established on $\Omega_{\mathcal{N}}$ for $\mathcal{N}=6\mathcal{N}_0\log n$.
Inspecting the formulas for $d_{n,i}$ and $a_n$ in Lemma \ref{lem:lambda_rate} shows that the upper bounds are driven by the stochastic integral over $[0,T]$ with respect to $N_{n,i}$.
We exemplify the motivation for the lower bound by discussing $a_n$.
By considering the definition of $a_n$ from Lemma~\ref{lem:lambda_rate}, we obtain that (where $K$ is a constant that changes from appearance to appearance but is independent of $T$ and $n$)
$$a_n\geq K\sqrt{\hat{V}_{a,\mathcal{T}}^{\mu}\log(nT)}\geq\sqrt{\hat{V}_{a,\mathcal{T}}^{\mu}\log(nT)}\geq\sqrt{\int_0^T\nu_0(X_{n,i}(t);\beta_n^*)^2dN_{n,i}(t)\frac{\log(nT)}{T^2}}.$$
Therefore, the lower bound on $a_n$ stated in Corollary~\ref{cor:ind_conv_rate} requires that $N_{n,i}([0,T])=O_P(T)$, i.e., that $T$ is indeed the size of the observation window.
We consider this assumption rather weak.
Equation~\eqref{eq:growth_condition} is also a weak requirement because by sparsity $|S_i(C_n^*)|\ll n$ and $\|C_{n,i\bdot}^*\|_1\leq c_n$ is restricted via Assumption (A2).

\section{Empirical Results}
\label{sec:empirical_results}
In this section, we will firstly discuss computational challenges of the estimator and how we can still obtain a computationally feasible estimator. We give the ideas here and postpone the technical details to Appendix \ref{sec:computation}. Afterwards, in Section \ref{subsec:cross-validation}, we will suggest a method for how to choose the tuning parameters. Finally, we provide a simulation study in Section \ref{subsec:simulations}.

\subsection{Implementation}
We have written the R-package \emph{Hausal}\footnote{Available for download on \url{https://github.com/akreiss/Hausal}}. It contains functions for simulation and estimation of our model. In the R-code, a penalty for $\alpha_n^*$ can be specified, but we discuss the algorithm for the case without penalty on $\alpha_n^*$. In this section, we make some remarks about the computational challenges. In Lemma \ref{lem:LAR_statement} in Appendix \ref{subsec:LAR}, we show that the optimization problems \eqref{eq:estimator} and \eqref{eq:split_estimator} can be reformulated as classical least-squares problems. Therefore, the very efficient LAR algorithm (cf. \cite{EHJT04}) can be used to optimize the LASSO penalized criterion with respect to $C$ and $\alpha$ if the other parameters are fixed. As another complication, we note that the least squares form provided by Lemma \ref{lem:LAR_statement} does not include an intercept. However, the LAR algorithm requires in this case that the input data is centralized, cf. Algorithm 5.1 in \cite{HTW15}. Therefore, the reformulation provided in Lemma \ref{lem:LAR_statement} cannot be directly used. We show in Lemma \ref{lem:design} in Appendix \ref{subsec:lasso_without_intercept} how an arbitrary least-squares problem can be stated in the required form.

Note that the minimization with respect to $\theta$ depends on the exact model formulation and cannot be treated generally. But since $\theta$ is low dimensional, optimization in $\theta$ can be performed with reasonable efficiency by using standard solvers. Another problem is that the criterion function is not convex in $\theta$. We approach this problem by repeatedly solving the optimization problem with random starting values. In a subsequent step, the best of these runs is then refined. 

Overall, we can compute solutions of \eqref{eq:split_estimator} using Algorithm \ref{alg:Calpha}. Being able to solve this, we can compute the solution to \eqref{eq:estimator} as described in Algorithm \ref{alg:theta}.

\begin{algorithm}
\caption{Compute $(\hat{C}_n(\theta),\hat{\alpha}_n(\theta))$ as solutions of \eqref{eq:split_estimator}}\label{alg:Calpha}
\begin{algorithmic}
\State \textbf{Input}: $\theta\in\Theta$ and $\textrm{tol}>0$
\State \textbf{Output:} Solution to \eqref{eq:split_estimator} for all $i=1,...,n$
\State Set $\alpha=1\in\IR^n$
\State Set $C=0\in\IR^{n\times n}$
\While{Progress in $\alpha$ and $C$ is less than $\textrm{tol}$}
\State Optimize \eqref{eq:split_estimator} for all $i=1,...,n$ with respect to $C$ keeping $\alpha$ and $\theta$ fixed
\State Optimize \eqref{eq:split_estimator} for all $i=1,...,n$ with respect to $\alpha$ keeping $C$ and $\theta$ fixed
\State Compute progress as the change in $\alpha$ and $C$ compared to the previous values
\EndWhile
\State \Return $(C,\alpha)$
\end{algorithmic}
\end{algorithm}

\begin{algorithm}
\caption{Compute $(\check{C}_n,\check{\alpha},\check{\theta}_n)$ as solution of \eqref{eq:estimator}}\label{alg:theta}
\begin{algorithmic}
\State \textbf{Input}: $K\in\IN$ and $\textrm{tol}_1,\textrm{tol}_2,\textrm{tol}_3>0$
\State \textbf{Output:} Solution to \eqref{eq:estimator}
\For{$j\gets 1$ to $K$}
\State Compute random value $\theta_1\in\Theta$
\State Use convex optimization with tolerance $\textrm{tol}_2$ starting from $\theta_1$ to optimize $\theta\mapsto\frac{1}{n}\sum_{i=1}^n\left(\frac{1}{T}\textrm{LS}_i(\hat{C}_{n,i\bdot}(\theta),\hat{\alpha}_{n,i}(\theta),\theta)+2\omega_i\|\hat{C}_{n,i\cdot}(\theta)\|_1\right)$, where $(\hat{C}_{n,i}(\theta),\hat{\alpha}_{n,i}(\theta))$ is computed using Algorithm \ref{alg:Calpha} with $\textrm{tol}=\textrm{tol}_1$. 
\EndFor
\State $\theta_2\gets$ Optimizer from above that yields the lowest value of the criterion function
\State Perform another convex optimization as above starting from $\theta_2$ with tolerance $\textrm{tol}_3$.
\State $\check{\theta}_n\gets$ Optimizer from previous step
\State \Return $(\hat{C}_n(\check{\theta}_n),\hat{\alpha}_n(\check{\theta}_n),\check{\theta}_n)$
\end{algorithmic}
\end{algorithm}

We have noted, and will give more details in Section~\ref{subsec:simulations}, that the first stage estimator as described above tends to produce too dense networks $C_n$ when using the cross-validation procedure as outlined in the next section.
This, in turn, seems to have a negative effect on the de-biasing.
Our ad-hoc suggestion is therefore to threshold the estimator $\check{C}_n$ before applying the de-biasing such that more entries of $\check{C}_n$ equal zero.
In our simulations below, we use maximum distance to the cord to find the threshold\footnote{We plot the non-negative estimates of $\check{C}_n$ in decreasing order uniformly spaced in the interval $[0,1]$.
Then, we find the point that lies furthest away from the line $h(x)=1-x$ to determine the point where the estimates start to level around zero.}.

\subsection{Tuning Parameter Choice}
\label{subsec:cross-validation}

Choosing $\omega\in[0,\infty)^n$ is of course challenging as this are $n$ parameters to choose. However, we note that in \eqref{eq:split_estimator}, for any given $\theta$, it is possible to choose $\omega_i$ independently of each other. While this is true for \eqref{eq:split_estimator}, it does not hold for \eqref{eq:estimator}, where we also optimize over $\theta$. Nevertheless, we take the independence in \eqref{eq:split_estimator} as motivation for a cross-validation procedure that searches for the tuning parameters in parallel.

Furthermore, as discussed before Lemma~\ref{cor:lambda_rate}, the results from Lemma \ref{lem:lambda_rate} can be used to compute theoretically guided tuning parameters. On the other hand, when conducting our simulation studies we realized that (possibly due to a sub-optimal nature of the constants) these choices yielded very high penalties such that the resulting networks were too sparse. We suggest therefore the following cross-validation procedure. Split the time interval $[0,T]$ in the training time $[0,S]$ and the testing time $(S,T]$ for some $S<T$. Then fit the model on the interval $[0,S]$ for a certain choice of $\omega$. The quality of the fit is then evaluated individually for each vertex by computing $\textrm{LS}_i$ on the time interval $(S,T]$ for each $i=1,...,n$. This procedure is repeated for several choices of $\omega$ and the cross-validated $\omega$ is given by selecting individually for each vertex $i$ that choice of $\omega_i$ that yielded the lowest least-squares error. Algorithm \ref{alg:CV} gives a schematic overview of the procedure.

\begin{algorithm}
\caption{Cross-Validation}\label{alg:CV}
\begin{algorithmic}
\State \textbf{Input}: $M\in\IN$, $S<T$
\State \textbf{Output:} $\omega$
\State Compute $\omega_0$ from Lemma \ref{lem:lambda_rate} based on $[0,T]$
\State Compute estimate $(C_0,\alpha_0,\theta_0)$ using $\omega=\omega_0$
\For{$m\gets 1$ to $M$}
\State Compute $(\textrm{LS}_1,...,\textrm{LS}_n)$ based on $[S,T]$ and $(C_{m-1},\alpha_{m-1},\theta_{m-1})$
\State Compute $\omega_{m,i}$ for each $i=1,...,n$ based on $\textrm{LS}_i$ (and potentially previous values, but ignoring $\textrm{LS}_j$ for $j\neq i$), e.g., by Golden-Section Search
\State Compute new estimates $(C_m,\alpha_m,\theta_m)$ based on $[0,S]$ using $\omega=\omega_m$

\EndFor
\For{$i\gets$ to $n$}
\State $\omega_i\gets$ That $(\omega_{m,i})_{m=1,...,M}$ that yielded the lowest value of $\textrm{LS}_i$ in the previous for-loop
\EndFor
\State \Return $\omega$
\end{algorithmic}
\end{algorithm}

\subsection{Simulation Study}
\label{subsec:simulations}

\subsubsection{Data Generating Process}
\label{subsubsec:DGP}
We consider $n=10$ vertices which are observed over $T=32$ \emph{days}.
The baseline intensities depend on $q=1$ covariate, which we update hourly, that is, $X_{n,i}(t)$ is piecewise constant on intervals that have length $1/24$.
The values of the covariates is the same for all vertices and serves as the \emph{common-driver}.
To create this covariate process, we define firstly a mean process $\mu:[0,T]\to\IR$ that is, similarly as $X_i$, piecewise constant on segments of length $1$ hour.
We choose randomly $8$ time points at which an exponential decay is started that increases $\mu$ instantaneously by $0.8$, has a decay rate of $0.5$, and lasts for $10$ days.
Let $t_j$ denote the beginning of the $j$-th segment, we define
\begin{equation}
\label{eq:covariate_def}
X_i(t):=1+|\mu(t_j)+Z_j|,\text{ where }Z_j\sim\mathcal{N}(0,0.05^2).
\end{equation}
The resulting covariate function is shown in Figure~\ref{fig:common_driver}.
The function  $\nu_0(x;\beta)$ is chosen as $\exp(x\beta)$ with $\beta:=1$.

\begin{figure}
\centering
\begin{subfigure}[t]{0.48\textwidth}
\centering
\includegraphics[width=\textwidth]{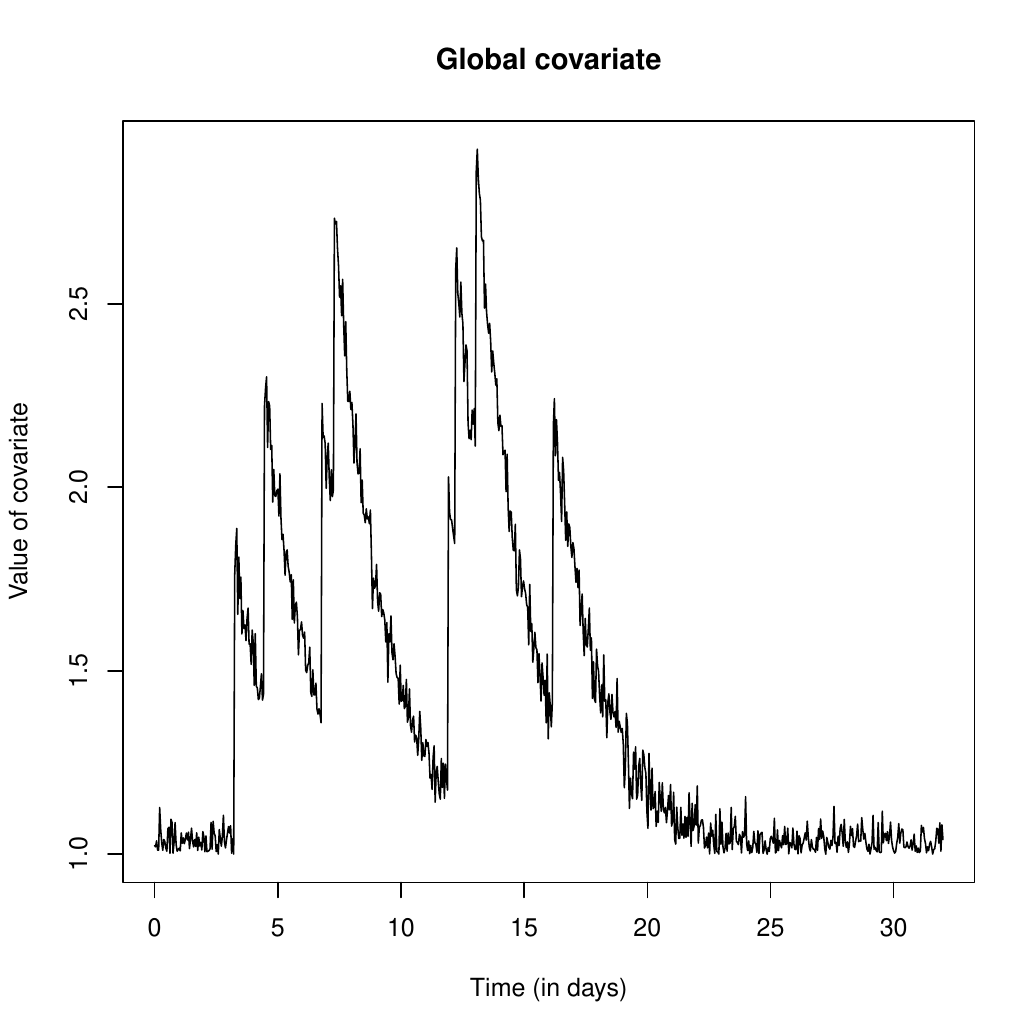}
\caption{Global covariate function $X_i(t)$ that is identical for all $i=1,...,n$.}
\label{fig:common_driver}
\end{subfigure}
\hfill
\begin{subfigure}[t]{0.48\textwidth}
\centering
\includegraphics[width=\textwidth]{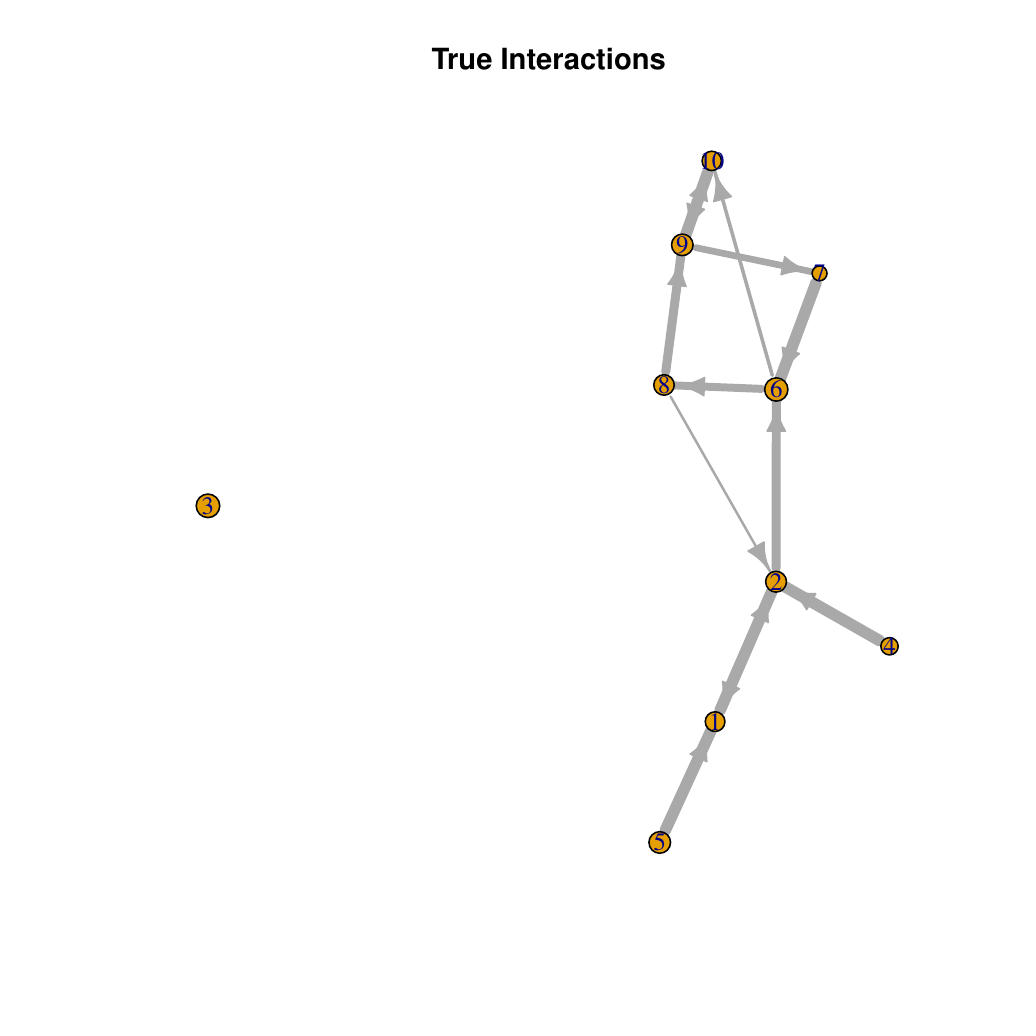}
\caption{True interactions used in the simulation; vertex sizes are proportional to $\alpha_{n,i}^*$ and edge widths are proportional of $C_{n,ij}^*$.}
\label{fig:true_interactions}
\end{subfigure}
\caption{Visualization of the parameters used in the data generating process.}
\label{fig:dgp}
\end{figure}

The individual baseline intensities $\alpha_{n,i}$ are selected uniformly at random from the interval $[0.25/e,0.5/e]$\footnote{The division by $e$ compensates for the additional $+1$ term in Equation~\eqref{eq:covariate_def}}.
Its exact values are shown in Figure~\ref{fig:alpha} in Appendix~\ref{sup:additional_simulation}.
The interaction network is created as a realization of a stochastic block model.
The initial influences are selected by a power law distribution with quadratic decay and, in a second step, they are rescaled such that each vertex has an out-degree of $0.8$ or $0$ depending on whether it has outgoing edges or not\footnote{This corresponds to normalizing the row sums of $C_n$ to $0.8$.}.
The resulting influence network is shown in Figure~\ref{fig:true_interactions}.
Finally, we let $g(t;\gamma)=\gamma\exp(-\gamma t)$ and $\gamma_n^*=-4\cdot\log(0.1)\approx9.21$, which means that initial effects of new events (as described by $C_n$) decay after $6$ hours to $10\%$ of their initial value.
Overall, the model comprises $n^2=100$ network parameters, $n=10$ individual activities, $q=1$ covariate parameters and $1$ exponential decay parameter.
Thus, we have to find $112$ parameters.

\subsubsection{Simulation Results}
We compute $N=300$ realizations of the model presented in Section~\ref{subsubsec:DGP}.
In order to illustrate the importance of including common drivers, we fit two models in each realization: The \emph{full} model as described in Section~\ref{subsubsec:DGP} and a \emph{slim-oracle} model that contains no covariates but assumes the true value of $\gamma_n^*$ to be known.
Note that the \emph{slim-oracle} fits a multivariate Hawkes with constant baseline function.
Therefore, it provides a comparison with a benchmark estimator that ignores common drivers.
To select the tuning parameter, for computational reasons, we use the cross-validation from Section~\ref{subsec:cross-validation} only for one data set in the full model and use the same values for $\omega$ throughout the whole simulation.
We use box-constrained optimization for $\beta$ and $\gamma$ with bounds $[-1,3]$ for $\beta$ and $[4.6,15]$\footnote{The corresponding times until an initial effect reduces to $10\%$ are given by about $3$ and $12$ hours, respectively.} for $\gamma$.

We discuss firstly the results of the estimation for $\beta_n^*$ and $\gamma_n^*$ in the first stage and using the original de-biasing without thresholding.
These estimations are only reasonable to discuss in the full model, where these parameters are considered unknown.
Table \ref{tab:beta_gamma} shows the results.
We measure the error of estimation in terms of bias in mean and median (difference of the mean/median od the estimates from the Monte Carlo replicates and the true value), standard deviation of the Monte Carlo replicate (SD), median absolute deviation of the estimators from the true value (MAD), and the root of the mean squared error (RMSE).
The estimate for $\beta_n^*$ is already in the first stage relatively precise in terms of bias. We conjecture that the estimate of $\beta_n^*$ has a small correlation with the estimates of $C_{n,ij}^*$ because $\beta_n^*$ determines the baseline intensity and $C_{n,ij}^*$ drives the decay of the intensity which are qualitatively different features of the fit. This may explain why the biases of the estimates of $C_{n,ij}^*$ do not disturb strongly the estimation of $\beta_n^*$. Because of the small size of the bias of the estimate of $\beta_n^*$ no improvement of bias can be expected by the debiasing. The good news is that the de-biasing is, correctly, not affecting the first stage estimators.
Moreover, SD, MAD and RMSE stay almost the same after de-biasing.   

When it comes to estimation of $\gamma_n^*$, the first stage exhibits a  visible bias of $3.02$ that is slightly  reduced by the de-biasing to $2.94$ while keeping SD.
This leads to a reduction of the RMSE.
The median based quantities (median bias and MAD) stay on the same level.
To understand the de-biasing better, we show in Figure~\ref{fig:DebiasGamma} the effect of the de-biasing on each individual estimator.
Even though the effect of the de-biasing is small, we see a clear positive effect.

\begin{table}
	\footnotesize
\centering
\caption{Estimation results for $\beta_n^*$ and $\gamma_n^*$ using the original de-biasing}
\label{tab:beta_gamma}
\begin{tabular}{l|rr|rrr||rr|rrr}
\hline
& \multicolumn{5}{c}{$\beta_n^*$} & \multicolumn{5}{c}{$\gamma_n^*$} \\
& \multicolumn{2}{c}{Bias} & \multicolumn{3}{c}{} & \multicolumn{2}{c}{Bias} & \multicolumn{3}{c}{} \\
De-biased & Mean & Median & SD & MAD & RMSE & Mean & Median & SD & MAD & RMSE \\ 
\hline
No          & 0.08 & 0.04 & 0.39 & 0.19 & 0.39 & 3.02 & 2.86 & 1.42 & 2.86 & 3.33 \\ 
Original     & 0.08 & 0.04 & 0.39 & 0.19 & 0.40 & 2.94 & 2.87 & 1.41 & 2.87 & 3.26 \\ 
Thresholded & 0.19 & 0.14 & 0.35 & 0.20 & 0.40 & 0.09 & 0.00 & 1.42 & 0.92 & 1.42 \\ 
\hline
\end{tabular}
\end{table}

\begin{figure}
\centering
\includegraphics[width=0.5\linewidth]{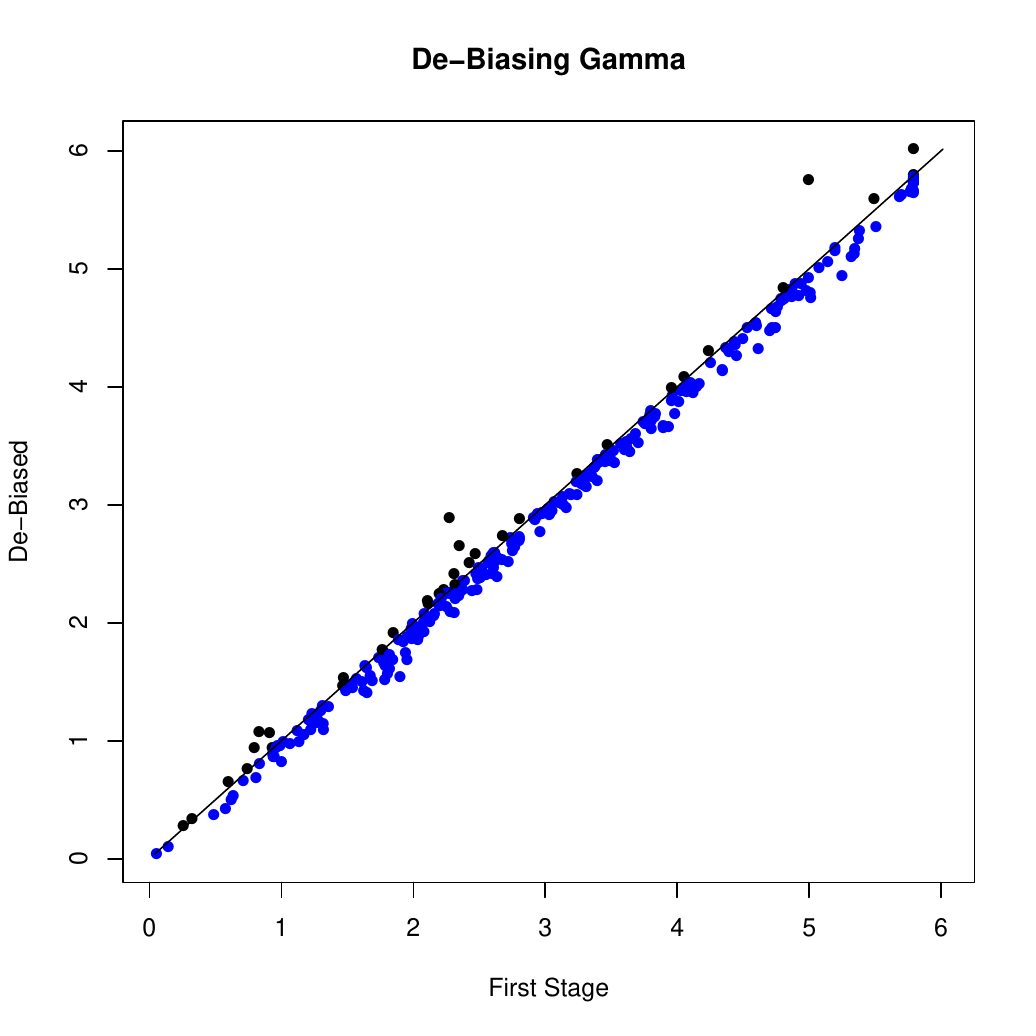}
\caption{The x-axis shows the absolute error of the first stage estimator, while the y-axis corresponds to the relative error of the de-biased estimator (original formulation), both for $\gamma_n^*$. Points are colored blue if the de-biasing has a positive effect on the first-stage estimator and black otherwise.}
\label{fig:DebiasGamma}
\end{figure}

As already mentioned in Section~\ref{sec:computation}, we observe that the first stage estimator produces relatively dense networks.
Therefore, we investigate the behavior of an adjusted estimator which thresholds some entries of $\check{C}_n$ to zero before proceeding to the de-biasing.
Figure~\ref{fig:debias} and Table \ref{tab:beta_gamma} show that the adjusted procedure removes the bias from estimation of $\gamma_n^*$ almost entirely. Furthermore, the histogram of its estimator has a symmetric shape and fits well with the normal limit. 
There is a slight increase in the bias for estimating $\beta_n^*$ which seems to be an effect of additional biases of the estimators of $C_{n,ij}^*$ caused by thresholding.
However, when comparing the histograms, this difference does not seem to be very pronounced. In summary we think that the observed behaviour of the estimators of $\gamma_n^*$ is well explained by our asymptotic theory whereas estimation of $\beta_n^*$ shows a qualitatively quite different picture because already the first stage estimator is nearly unbiased.
We continue to investigate the estimators based on the adjusted de-biasing.

\begin{figure}
\centering
\begin{subfigure}[t]{0.9\textwidth}
\centering
\includegraphics[height=9cm,width=0.8\textwidth]{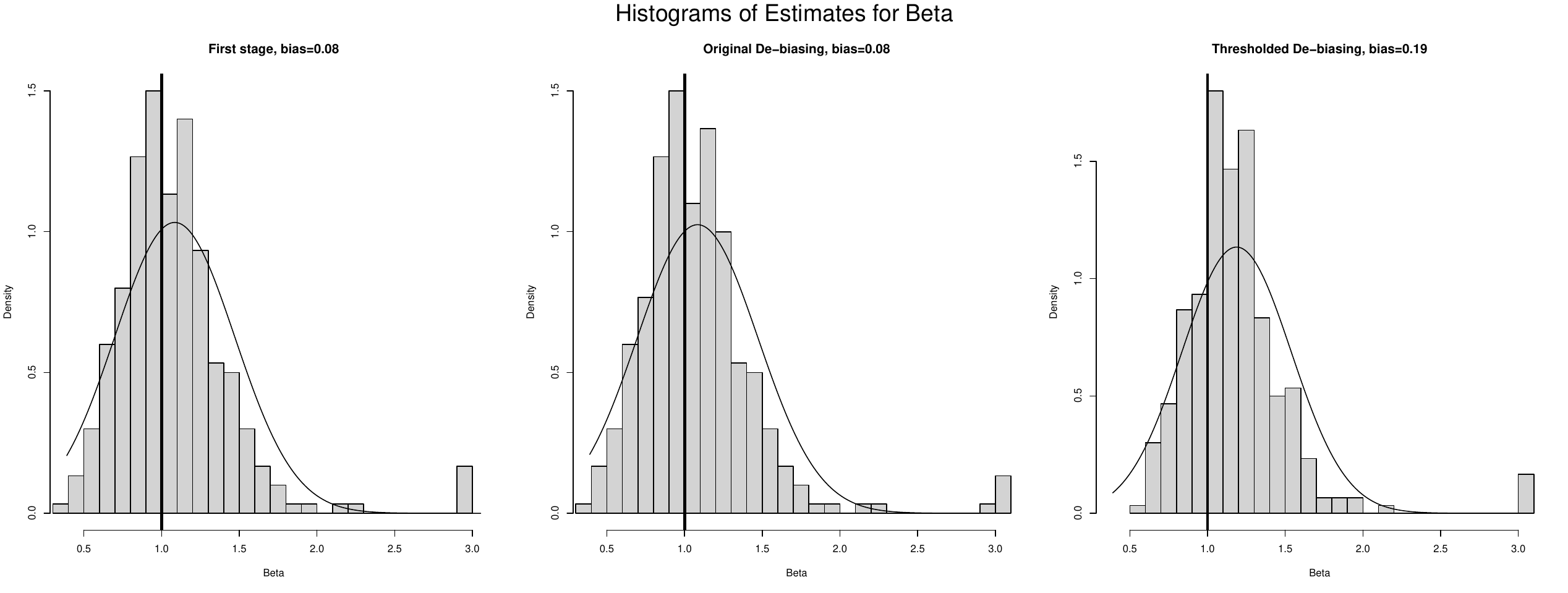}
\caption{Estimation of $\beta_n^*$}
\label{fig:debias:beta}
\end{subfigure}
\begin{subfigure}[t]{0.9\textwidth}
\centering
\includegraphics[width=\textwidth]{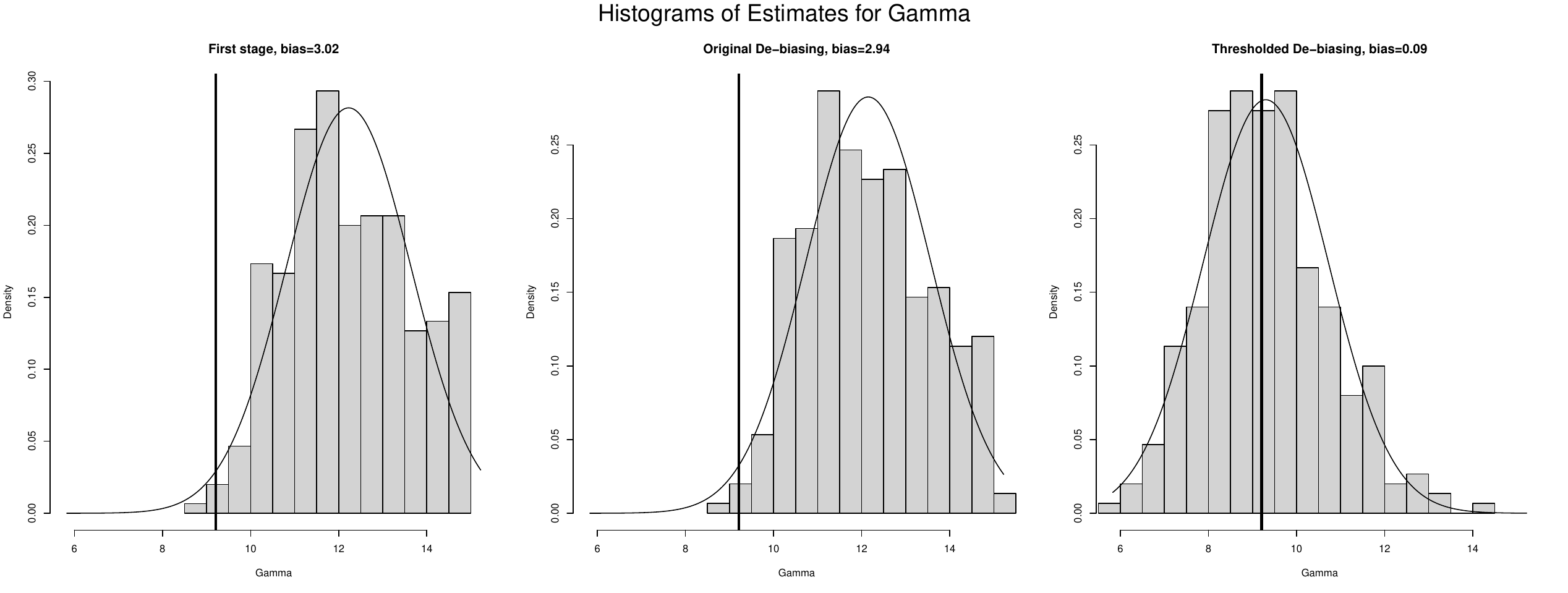}
\caption{Estimation of $\gamma_n^*$}
\label{fig:debias:gamma}
\end{subfigure}
\caption{Histograms of estimators for $\beta_n^*$ and $\gamma_n^*$ in the first stage and after de-biasing without and with thresholding; the vertical line shows the true values and the normal distribution is chosen to align the data for illustration.}
\label{fig:debias}
\end{figure}

We turn next to the estimation of the network itself.
Figure~\ref{fig:detections} shows the percentage of non-zero estimations per edge in the full model before and after the de-biasing step and in the slim-oracle model.
All models detect the existing edges well.
However, the first stage estimates and the slim-oracle produce more false positives than the 3rd stage estimator.
In the full model, we see that the networks after de-biasing of $\theta$ seem to be generally less dense compared to the first stage estimation (cf. Figure~\ref{fig:sparsity} in Appendix~\ref{sup:additional_simulation} for histograms of the number of detected edges).
Table~\ref{tab:confusion} shows in more detail that all three estimators find most of the edges reliably.
The number of true negatives (not detected non-edges) is non-negligible in all cases, but the 3rd stage estimator is shows advantages.
Therefore, the 3rd stage estimator is preferred when it comes to estimation of the network.
Since the slim-oracle estimator produces more false positives (compared to the full estimators), we conclude that omission of the common driver induces some spurious effects even when using the true $\gamma_n^*$\footnote{Recall that the slim estimator is of oracle type in the sense that it uses knowledge of the correct $\gamma_n^*$}.

\begin{figure}
\centering
\includegraphics[width=\textwidth]{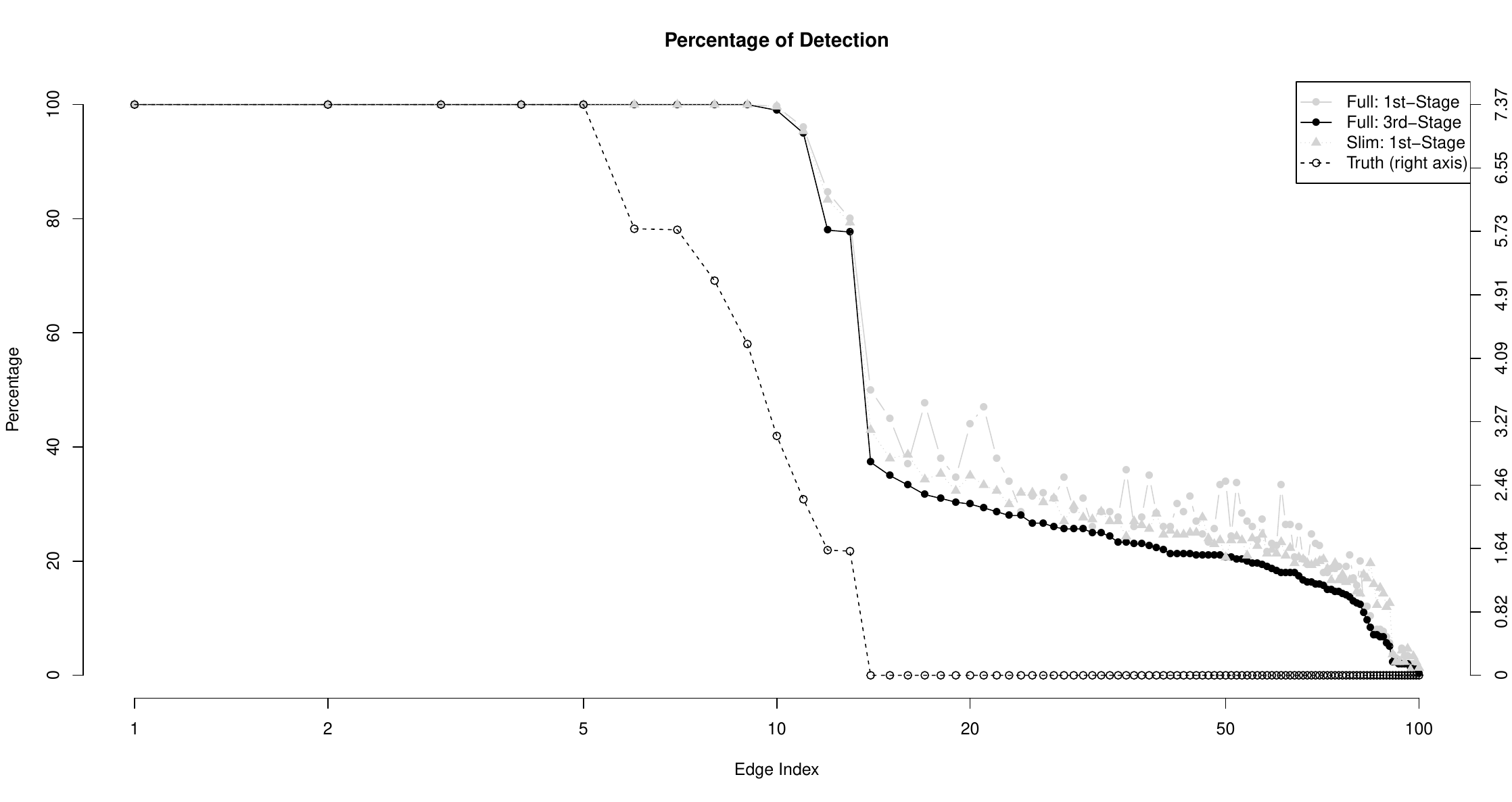}
\caption{Percentage of detections per edge in all considered scenarios (de-biasing with previous thresholding). Edges $(i,j)$ are ordered decreasingly according to $C_{n,ij}^*$, cf. dashed line (right axis). The x-axis shows indices according to this ordering. The solid lines show the percentage of Monte Carlo replicates in which the corresponding edge was detected (left axis).}
\label{fig:detections}
\end{figure}

\begin{table}
	\footnotesize
\centering
\caption{Confusion matrices for the different scenarios (on average per simulation). There are $13$ true edges and $87$ non-edges. Abbreviations: 1st-St, first stage estimator; 3rd-St, third stage estimators (with and without thresholding as indicated); Detect., detected edges; Not detect., not detected edges.}
\label{tab:confusion}
\begin{tabular}{l||rr|rr|rr|rr}
\hline
& \multicolumn{6}{c}{Full Model} & \multicolumn{2}{c}{Slim Model} \\
& \multicolumn{2}{c}{1st-St} & \multicolumn{2}{c}{3rd-St (Original)} & \multicolumn{2}{c}{3rd-St (Thresholded)} & \multicolumn{2}{c}{Slim: 1st-St} \\
Edges & Detect. & Not detect. & Detect. & Not detect. & Detect. & Not detect. & Detect. & Not detect. \\ 
\hline
Present & 12.60 &  0.40 & 12.60 &  0.40 & 12.50 &  0.50 & 12.58 &  0.42 \\ 
Absent  & 20.14 & 66.86 & 20.04 & 66.96 & 15.49 & 71.51 & 18.60 & 68.40 \\ 
\hline
\end{tabular}
\end{table} 

In Table \ref{tab:C} we show the detailed estimation results for the non-zero entries of the matrix $C_n$.
Figure~\ref{fig:avgC} in Appendix~\ref{sup:additional_simulation} shows a visual account of the estimation of all entries of the matrix $C_n$.
Most estimates of $C_n$ are downwards biased, which is expected due to the LASSO penalty.
We observe that in the full model the 3rd stage seems to experience a larger bias.
This observations indicates that for estimation of the actual effect, the first stage estimator is useful.
However, given the high-number of false positives of the first stage estimator, the 3rd stage estimator seems to be more appropriate for model selection.
In addition, the 3rd stage estimates show a lower fluctuation in terms of standard deviation.
Recall, moreover, that we have (for computational reasons) not recomputed the tuning parameter after the de-biasing stage.
A more refined tuning parameter selection method combined with further methods for model selection in high dimensions, e.g., de-biasing for $C_n$ or post selection inference, might lead to a better fit.
Finally, Figure~\ref{fig:Cmse} in Appendix~\ref{sup:additional_simulation} shows the root mean squared error of the estimation of all of $C_n$.
After de-biasing, the RMSE reduced compared to the first stage, i.e., the higher accuracy in model selection seems to outweigh the slightly higher bias in the size of the influence.
Moreover, the 3rd stage estimator achieves RMSEs of the same order as the slim-oracle.

\begin{table}
\centering
\caption{Estimation performance of the non-zero entries of $C_n^*$}
\label{tab:C}
\begin{tabular}{l||rrr|rrr||rrr}
\hline
& \multicolumn{6}{c}{Full} & \multicolumn{3}{c}{Slim} \\
& \multicolumn{3}{c}{1st Stage} & \multicolumn{3}{c}{3rd Stage} & \multicolumn{3}{c}{1st Stage} \\
Edge & Bias & SD & RMSE & Bias & SD & RMSE & Bias & SD & RMSE  \\ 
\hline
$(2, 1)$ & -0.48 & 0.94 & 1.05 & -0.80 & 0.76 & 1.11 & -0.70 & 0.75 & 1.02 \\ 
$( 5, 1)$ &  0.24 & 1.55 & 1.56 & -0.92 & 1.39 & 1.67 & -0.81 & 1.17 & 1.43 \\ 
$( 1, 2)$ & -0.08 & 1.50 & 1.50 & -1.10 & 1.30 & 1.70 & -0.97 & 1.17 & 1.51 \\ 
$( 4, 2)$ &  0.14 & 1.56 & 1.57 & -0.92 & 1.37 & 1.65 & -0.83 & 1.22 & 1.47 \\ 
$( 8, 2)$ & -0.48 & 1.03 & 1.14 & -0.75 & 0.86 & 1.14 & -0.67 & 0.85 & 1.08 \\ 
$( 2, 6)$ & -0.15 & 1.46 & 1.46 & -0.92 & 1.23 & 1.53 & -0.78 & 1.17 & 1.40 \\ 
$( 7, 6)$ & -0.02 & 1.43 & 1.43 & -1.07 & 1.28 & 1.67 & -0.97 & 1.15 & 1.50 \\ 
$( 9, 7)$ & -0.08 & 1.38 & 1.38 & -0.75 & 1.17 & 1.39 & -0.60 & 1.14 & 1.29 \\ 
$( 6, 8)$ & -0.08 & 1.54 & 1.54 & -0.88 & 1.35 & 1.61 & -0.75 & 1.24 & 1.45 \\ 
$( 8, 9)$ &  0.05 & 1.50 & 1.50 & -0.83 & 1.30 & 1.54 & -0.73 & 1.18 & 1.38 \\ 
$(10, 9)$ & -0.12 & 1.60 & 1.60 & -1.19 & 1.37 & 1.82 & -1.08 & 1.19 & 1.61 \\ 
$( 6,10)$ & -0.10 & 1.18 & 1.18 & -0.53 & 1.01 & 1.14 & -0.44 & 0.97 & 1.06 \\ 
$( 9,10)$ & -0.37 & 1.20 & 1.26 & -0.92 & 1.04 & 1.39 & -0.78 & 0.96 & 1.24 \\ 
\hline
\end{tabular}
\end{table}

Figure~\ref{fig:alpha} in Appendix~\ref{sup:additional_simulation} shows boxplots for the estimates of $\alpha_n$.
In the slim model, the individual activity of the vertices, i.e., $\alpha_n$ is over-estimated.
In the first-stage full model the estimates seem to be relatively accurate, while, in the 3rd stage full model, the estimation appear even better in terms of bias and standard deviation.
Detailed results are shown in Table~\ref{tab:alpha}.
In the slim model, the too high estimates for $\alpha_n$ are potentially a consequence from omitting the common driver. 

\begin{table}
\centering
\caption{Estimation Results for $\alpha_n^*$}
\label{tab:alpha}
\begin{tabular}{l|rrr|rrr||rrr|rrr}
\hline
& \multicolumn{6}{c}{Full} & \multicolumn{3}{c}{Slim} \\
& \multicolumn{3}{c}{1st Stage} & \multicolumn{3}{c}{3rd Stage} & \multicolumn{3}{c}{1st Stage}\\
Vertex & Bias & SD & RMSE & Bias & SD & RMSE & Bias & SD & RMSE  \\ 
\hline
$\alpha_{n, 1}^*$ & 0.04 & 0.13 & 0.14 & -0.02 & 0.10 & 0.10 & 0.47 & 0.33 & 0.57 \\ 
$\alpha_{n, 2}^*$ & 0.07 & 0.14 & 0.16 & 0.01 & 0.10 & 0.10 & 0.63 & 0.37 & 0.72 \\
$\alpha_{n, 3}^*$ & -0.01 & 0.11 & 0.11 & -0.04 & 0.08 & 0.09 & 0.53 & 0.18 & 0.56 \\
$\alpha_{n, 4}^*$ & 0.04 & 0.12 & 0.13 & -0.02 & 0.09 & 0.09 & 0.41 & 0.33 & 0.52 \\ 
$\alpha_{n, 5}^*$ & 0.05 & 0.14 & 0.15 & -0.02 & 0.11 & 0.11 & 0.54 & 0.34 & 0.63 \\ 
$\alpha_{n, 6}^*$ & 0.06 & 0.15 & 0.16 & -0.01 & 0.11 & 0.11 & 0.62 & 0.36 & 0.71 \\ 
$\alpha_{n, 7}^*$ & 0.04 & 0.12 & 0.13 & -0.01 & 0.09 & 0.09 & 0.39 & 0.30 & 0.49 \\ 
$\alpha_{n, 8}^*$ & 0.06 & 0.15 & 0.16 & -0.01 & 0.10 & 0.10 & 0.58 & 0.37 & 0.68 \\ 
$\alpha_{n, 9}^*$ & 0.07 & 0.15 & 0.17 & 0.00 & 0.10 & 0.10 & 0.61 & 0.37 & 0.72 \\ 
$\alpha_{n,10}^*$ & 0.05 & 0.14 & 0.14 & -0.01 & 0.10 & 0.10 & 0.49 & 0.33 & 0.59 \\ 
\hline
\end{tabular}
\end{table}

In Appendix~\ref{subsec:large_network}, we present a similar analysis for a larger network that is observed over a shorter time horizon.
The results are qualitatively similar.
In particular, we also observe a relatively well estimated $\beta_n^*$ in the first stage, and a substantial effect of the de-biasing on the estimate for $\gamma_n^*$.
However, the initial estimation of $\gamma_n^*$ seems to be less accurate.
In terms on network reconstruction, we also observe for the larger network that the 3rd stage estimator has fewer false positives.
Interestingly, the most reliably detected edges are not necessarily those which have the highest influence suggesting that other quantities also effect the accuracy of detection.

\section{Conclusion}
\label{sec:conclusion}
In this paper, we have studied a high-dimensional Hawkes model that incorporates covariates in the baselines. These covariates can serve as common drivers to all Hawkes processes. We study a regime in which both, the size of the network and the length of the observation period, go to infinity. This introduces mathematical challenges because the different parameters of the model can be estimated with different convergence rates. To obtain the correct rate, we have suggested to use the de-biasing technique from \cite{vdGBRD14}. Our results show that obtaining the fast convergence rate for $\theta$ using the de-biasing is possible under some assumptions. These assumptions, unfortunately, restrict the level of sparsity. If one is only interested in the slower convergence rate, the sparsity is allowed to be much lower. An interesting future research direction would therefore be to understand if this restriction in the sparsity is an artifact from the Lasso and the proof techniques that we have applied here, or if the lower sparsity for the fast convergence rates is indeed a theoretical boundary.

In our simulation study, we provide some first suggestions why it is important to incorporate covariates that affect the entire network (common drivers). In our example, we see that ignoring such common drivers leads to less accurate results about the estimation of the influence structure.

Further research directions include the discussion of covariates in the excitation kernels $g$. In our simulations, estimation of $\gamma$ appears to be a difficult task even though our model assumes this excitation to be the same for all interactions. Nevertheless, it is arguably quite interesting for applications to assume that the interactions might be heterogeneous. It is then of interest to learn these interactions. A more direct question, would be to generalize the results to multiple observations, where one observes several multivariate Hawkes processes with the same set of parameters but potentially different covariates. \\

\begin{appendix}
\noindent
\textbf{Appendix.} \\[10pt]
Appendix \ref{sec:computation} discusses details about the implementation of the procedure presented in this paper. Appendix \ref{sec:Hawkes} contains theoretical results about Hawkes processes that are useful for our theory. Afterwards, Appendix \ref{sec:proofs} presents proofs that were left out in the main paper. In Appendix \ref{subsec:debias}, we give a heuristic discussion for the necessity of the de-biasing step. And in Appendix \ref{appendic:RCC}, we discuss the random compatibility conditions in a special case. In Appendix~\ref{sup:HawkesBounds}, we present additional details for the proofs of Section \ref{subsec:bounds_Hawkes}. Appendix~\ref{sup:single_consistency} contains the proof of Theorem \ref{thm:pre_estimate_consistency} and further details for the proof of Lemma \ref{lem:lambda_rate}. Appendix~\ref{sup:results_debiasing} contains proofs needed in Section \ref{subsec:results_debiasing}. Appendix~\ref{sup:consistency} provides more technical details for the proofs of results from Section \ref{subsec:consistency}; and, finally, Appendix~\ref{sup:emp} presents additional simulation results.

\newpage

\section{Algorithmic and computational considerations}
\label{sec:computation}
\subsection{Using the LAR algorithm}
\label{subsec:LAR}
In this Section we show how the optimization problems \eqref{eq:estimator} and \eqref{eq:split_estimator} can be reformulated in order to use the LAR algorithm (cf. \cite{EHJT04}). The LAR algorithm is designed to optimize $\|Y-\delta {\mathbf 1}-X\gamma\|_2^2+\lambda\|\gamma\|_1$ in $(\gamma,\delta)\in\IR^q \times \IR$ for a given vector $Y\in\IR^n$ and a given matrix $X\in\IR^{n\times q}$. Here ${\mathbf 1}$ denotes an $n$-dimensional vector of all $1$'s.  We refer to this set-up as the pure least squares linear model. Neither \eqref{eq:estimator} nor \eqref{eq:split_estimator} is in this form if we consider all three parameters $(C,\alpha,\theta)$ at once. Therefore, we keep two parameters fixed and minimize with respect to the third parameter. For $\theta=(\beta,\gamma)\in\Theta$ and $i=1,...,n$, we define matrices $V_n(\beta),A_n(\gamma),G_n(\theta),\Gamma_n(\gamma)\in\IR^{n\times n}$ via
\begin{align*}
V_{n,ii}(\beta)&:= \int_0^T\nu_0(X_{n,i}(t);\beta)^2dt,\quad V_{n,ij}=0\textrm{ if }i\neq j, \\
\Gamma_{n,ij}(\gamma)&:= \int_0^T\int_{-\infty}^{t-}g(t-s;\gamma)dN_{n,i}(s)\cdot\int_{-\infty}^{t-}g(t-r;\gamma)dN_{n,j}(r)dt, \\
G_{n,ij}(\theta)&:= \int_0^T\nu_0(X_{n,i}(t);\beta)\int_{-\infty}^{t-}g(t-s;\gamma)dN_{n,j}(s)dt, \\
A_{n,ij}(\gamma)&:= \int_0^T\int_{-\infty}^{t-}g(t-s;\gamma)dN_{n,j}(s)dN_{n,i}(t),
\end{align*}
and a vector $v_n(\beta)\in\IR^n$ via
$$v_{n,i}(\beta):=\int_0^T\nu_0(X_{n,i}(t);\beta)dN_{n,i}(t).$$
Using these definitions we can rewrite
\begin{align*}
\int_0^T\Psi_{n,i}(t;c,a,\theta)^2dt&= a^2V_{n,ii}(\beta)+c\Gamma_n(\gamma)c^T+2acG_{n,i\bdot}(\theta)^T \textrm{ and}\\
\int_0^T\Psi_{n,i}(t;c,a,\theta)dN_{n,i}(t)&= av_{n,i}(\beta)+cA_{n,i\bdot}(\gamma)^T.
\end{align*}
This implies
\begin{align}
\LS_i(C,a,\theta)&= a^2V_{n,ii}(\beta)+c\Gamma_n(\gamma)c^T+2acG_{n,i\bdot}(\theta)^T-2av_{n,i}(\beta)-2cA_{n,i\bdot}(\gamma)^T,  \label{eq:LS_matrix1} \\
\sum_{i=1}^n\LS_i(C_{n,i\bdot},\alpha_i,\theta)&= \alpha^TV_n(\beta)\alpha+\textrm{tr}\left(C_n\Gamma_n(\gamma)C_n^T\right)+2\alpha^T\textrm{diag}\left(C_nG_n(\theta)^T\right) \nonumber \\
&\quad\quad-2\alpha^Tv_n(\beta)-2\textrm{tr}\left(C_nA_n(\gamma)^T\right), \label{eq:LS_matrix2}
\end{align}
where $\textrm{tr}(A)$ denotes the trace of the matrix $A$ and $\textrm{diag}(A)$ the diagonal of the matrix $A$ written as column vector. The following result serves as a definition of $\Gamma_n^{-1}$ in case $\Gamma_n$ is not invertible.
\begin{lemma}
\label{lem:gamma_sym}
The matrix $\Gamma_n(\gamma)$ is positively semi-definite and symmetric. It can hence be written as $\Gamma_n(\gamma)=M\Lambda M^T$ for a diagonal matrix $\Lambda=\textrm{diag}(\lambda_1,...,\lambda_r,0,...,0)$ with $\lambda_1\geq...\geq\lambda_r>0$ and $r\leq n$ and an orthogonal matrix $M$. Denote,
$$\Lambda^{\frac{1}{2}}:=\textrm{diag}(\sqrt{\lambda_1},...,\sqrt{\lambda}_r,0,...,0)\textrm{ and }\Lambda^{-\frac{1}{2}}:=\textrm{diag}\left(\frac{1}{\sqrt{\lambda_1}},...,\frac{1}{\sqrt{\lambda}_r},0,...,0\right)$$
as well as $\Gamma_n(\gamma)^{\frac{1}{2}}:=M\Lambda^{\frac{1}{2}}M^T$ and $\Gamma_n(\gamma)^{-\frac{1}{2}}:=M\Lambda^{-\frac{1}{2}}M^T$. It holds that $\Gamma_n(\gamma)^{\frac{1}{2}}\Gamma_n(\gamma)^{\frac{1}{2}}=\Gamma_n(\gamma)$.
\end{lemma}
\begin{proof}
Let $x\in\IR^n$ be arbitrary. Then,
\begin{align*}
	x^T\Gamma_n(\gamma)x&= \int_0^T\left(\sum_{j=1}^n\int_{-\infty}^{t-}g(t-r;\gamma)dN_{n,j}(r)x_j\right)^2dt\geq0.
\end{align*}
The remaining statements are easy to check.
\end{proof}
\begin{remark}
Let $w_n(t):=\left(\int_{-\infty}^{t-}g(t-r;\gamma)dN_{n,j}(r)\right)_{j=1,...,n}$. The above proof shows that $\Gamma_n(\gamma)$ is rank deficient if there is $x\in\IR^n\setminus\{0\}$ such that $w_n(t)^Tx=0$ for all $t\in[0,T]$. A typical reason why this could happen is if for some $i$, $w_{n,i}(t)=0$ for all $t\in[0,T]$, i.e., if $N_{n,i}$ does not jump at all. Recall that counting process with probability $1$ jump at different times. Hence, $w_{n,i}(t)\neq w_{n,j}(t)$ for all $i\neq j$ and all $t$ with probability one if $w_{n,i}(t)\neq0$ and $w_{n,j}(t)\neq0$. Furthermore, since $w_n(t)$ changes whenever there is a jump in any of the processes $N_{n,i}$, we regard other situations in which $\Gamma_n(\gamma)$ is rank deficient as highly unlikely.
\end{remark}
The next lemma shows that if some process not jumping is the only reason for rank deficiency of $\Gamma_n(\gamma)$, we can rewrite the optimization tasks \eqref{eq:estimator} and \eqref{eq:split_estimator} in the form required by the LAR algorithm.
\begin{lemma}
\label{lem:LAR_statement}
Assume that $V_{n,ii}(\beta)>0$ for all $i=1,...,n$. Suppose without loss of generality that $N_{n,j}$ does not jump for $j\in L$ with $L=\{n_0+1,...,n\}$ or $L=\emptyset$. Let $\overline{\Gamma}_n(\gamma)$ denote the upper left $n_0\times n_0$ block of $\Gamma_n(\gamma)$. If $\overline{\Gamma}_n(\gamma)$ is positive definite, it holds that (below $c\in\IR^n$ is a row-vector)
\begin{align}
	\LS_i(c,a,\theta)-\left\|\Gamma_n(\gamma)^{-\frac{1}{2}}\left(A_{n,i\bdot}(\gamma)^T-aG_{n,i\bdot}(\theta)^T\right)-\Gamma_n(\gamma)^{\frac{1}{2}}c^T\right\|_2^2&=f(a,\theta), \label{eq:easyC} \\
	\sum_{i=1}^n\LS_i(C_{n,i\bdot},\alpha_i,\theta)-\left\|V_n(\beta)^{-\frac{1}{2}}\left(v_n(\beta)-\textrm{diag}(C_nG_n(\theta)^T)\right)-V_n(\beta)^{\frac{1}{2}}\alpha\right\|_2^2&=f(C_n,\theta), \label{eq:easyNU}
\end{align}
where $f(a,\theta)$ is a not further specified function of $a$ and $\theta$, which is independent of $c$, and similarly $f(C_n,\theta)$ is not a function of $\alpha$.
\end{lemma}
\begin{proof}
Clearly, $\Gamma_n(\gamma)$ equals $0$ outside the upper left block $\overline{\Gamma}_n(\gamma)$. By assumption $\overline{\Gamma}_n=WDW^T$, for some orthogonal matrix $W\in\IR^{n_0\times n_0}$ and $D$ a $n_0\times n_0$ diagonal matrix with positive entries. It follows that $M$ from Lemma \ref{lem:gamma_sym} is a block-diagonal matrix with $W$ in its upper left and a $(n-n_0)\times(n-n_0)$ identity matrix in the bottom right. Similarly, $\Lambda$ equals $D$ in the upper right and equals $0$ everywhere else. Using these considerations it is direct to compute that $\Gamma_n(\gamma)^{-\frac{1}{2}}\Gamma_n(\gamma)^{\frac{1}{2}}=I_{n_0,n}$, where $I_{n_0,n}$ is a diagonal matrix with $n_0$ $1$s in the first entries of the diagonal and $n-n_0$ $0$s at the end.

We show the statements now by showing that the corresponding derivatives equal zero: Using the above argument and \eqref{eq:LS_matrix1},
\begin{align*}
	&\frac{d}{dc}\eqref{eq:easyC} \\
	&= 2\Gamma_n(\gamma)c^T+2aG_{n,i\bdot}(\theta)^T-2A_{n,i\bdot}(\gamma)^T+2\left(I_{n_0,n}\left(A_{n,i\bdot}(\gamma)^T-aG_{n,i\bdot}(\theta)^T\right)-\Gamma_n(\gamma)c^T\right) \\
	&= 2a\left(I_n-I_{n_0,n}\right)G_{n,i\bdot}(\theta)^T-2\left(I_n-I_{n_0,n}\right)A_{n,i\bdot}(\gamma)^T=0
\end{align*}
because $A_{n,ij}(\gamma)=0$ for all $j\geq n_0+1$ and $G_{n,ij}(\theta)=0$ for all $j\geq n_0+1$.

In a similar fashion, using \eqref{eq:LS_matrix2}, we get
\begin{align*}
	&\frac{d}{d\alpha}\eqref{eq:easyNU} \\
	&= 2V_n(\beta)\alpha+2\textrm{diag}(C_nG_n(\theta)^T)-2v_n(\beta)+2\left(v_n(\beta)-\textrm{diag}(C_nG_n(\theta)^T)-V_n(\beta)\alpha\right)=0.
\end{align*}
This completes the proof.
\end{proof}

Equation \eqref{eq:easyC} shows that optimization in $c$ can be understood as a pure least squares linear model with $\delta=0$. Equation \eqref{eq:easyNU} shows the same about optimization in $\alpha$. Moreover, we can see from \eqref{eq:easyNU} that finding the baseline parameter $\alpha$ is (unsurprisingly) the same problem as regressing constants on the number of events after subtracting the events induced by self-excitement. The equation \eqref{eq:easyC} is less easy to interpret. Lastly we note that, if the parameter $\alpha$ is left unpenalized (i.e. $\omega_\alpha=0$), we can compute the minimizer for fixed $C$ and $\theta$ by solving $\partial_{\alpha}\sum_{i=1}^n\LS_i(\hat{C}_{n,i\bdot},\alpha,\hat{\theta}_n)=0$ to be $\alpha=V_n(\hat{\beta}_n)^{-1}\left(v_n(\hat{\beta}_n)-\textrm{diag}\left(\hat{C}_nG_n\left(\hat{\theta}_n\right)^T\right)\right)$.

Note finally that, in order to use the LAR algorithm for an efficient computation of the Lasso estimator, we have to either include an intercept $\delta$ that remains unpenalized or provide centralized data, cf. Algorithm 5.1 in \cite{HTW15}. Since neither is the case in \eqref{eq:easyC} and \eqref{eq:easyNU}, we need to make a further argument. This will be given in the next subsection.

\subsection{Lasso with intercept}
\label{subsec:lasso_without_intercept}
Suppose we are interested in solving
\begin{equation}
\label{eq:pls}
\hat{\gamma}_n:=\argmin{\gamma\in\IR^p}\|Y-X\gamma\|_2^2+\lambda\|\gamma\|_1=\argmin{\gamma}\,Y^TY+\gamma^TX^TX\gamma-2Y^TX\gamma+\lambda\|\gamma\|_1
\end{equation}
efficiently for general $Y\in\IR^n$ and $X\in\IR^{n\times p}$. The highly efficient LAR algorithm, cf. Algorithm 5.1 in \cite{HTW15}, requires that
\begin{align} 
&\frac{1}{N}\sum_{i=1}^nY_i=0 \textrm{ and } \label{eq:LAR_cond1} \\
&\frac{1}{N}\sum_{i=1}^nX_{ij}=0 \textrm{ for each } j=1,...,p. \label{eq:LAR_cond2}
\end{align}
We show in this section how LAR can be used in other scenarios as well. Since we believe that this discussion might be of general interest, we discuss this problem generally and therefore the notation in this subsection is independent of the rest of the paper. However, this observation might have been made elsewhere. Typically, if \eqref{eq:LAR_cond1} and \eqref{eq:LAR_cond2} do not hold, one centralizes them. This essentially means that one assumes an unpenalized intercept. In our scenario (and possibly in others too), the intercept is a more complicated part of the model (in our case the baseline intensity). Therefore it cannot simply be removed manually and it is also penalized. We describe in this section what to do in such a case. The following result is a preparation for the main statement.

\begin{lemma}
\label{lem:xtilde}
Let
$$E:=\begin{pmatrix}
	\frac{\sqrt{6}}{6} & -\frac{\sqrt{2}}{2}\\
	\frac{\sqrt{6}}{6} & \frac{\sqrt{2}}{2}
\end{pmatrix} \textrm{ and }v:=\begin{pmatrix}
	-\frac{\sqrt{6}}{3} & 0
\end{pmatrix}.$$
For any even $m\in\IN$, $m\geq2$ define the matrices (all blocks $E$ and $v$ appear $m/2$ times below)
$$\tilde{X}_m:=\begin{pmatrix}
	E &   &       \\
	& \ddots &     \\
	&        & E \\ 
	v &        &   \\
	& \ddots &   \\
	&        & v  
\end{pmatrix}\in\IR^{\frac{3m}{2}\times m}.$$
For odd $m\in\IN$, $m\geq2$ define (all blocks $E$ and $v$ appear $(m-1)/2$ times below) 
$$\tilde{X}_m:=\begin{pmatrix}
	E &        &   &                     \\
	& \ddots &   &                     \\
	&        & E &                     \\
	&        &   & \frac{\sqrt{2}}{2}  \\
	v &        &   &                     \\
	& \ddots &   &                     \\
	&        & v &                     \\
	&        &   & -\frac{\sqrt{2}}{2}\end{pmatrix}\in\IR^{\frac{3m+1}{2}\times m}.$$
It holds for all $m\geq 2$ that $\tilde{X}_m^T\tilde{X}_m=I_m$ and $\sum_i\tilde{X}_{m,ij}=0$ for all $j=1,...,m$, where $I_m$ denotes the $m\times m$ identity matrix.
\end{lemma}
\begin{proof}
Let $k\in\{1,...,m\}$. Then,
\begin{align*}
	\left[\tilde{X}_m^T\tilde{X}_m\right]_{k,k}=\sum_{j}\tilde{X}_{m,jk}^2=\left\{\begin{array}{cc}
		\left(\frac{\sqrt{6}}{6}\right)^2+\left(\frac{\sqrt{6}}{6}\right)^2+\left(-\frac{\sqrt{6}}{3}\right)^2  & \textrm{ if }k\textrm{ odd and }k<\frac{3m+1}{2}  \\
		\left(\frac{-\sqrt{2}}{2}\right)^2+\left(\frac{\sqrt{2}}{2}\right)^2  & \textrm{ if }k\textrm{ even}  \\
		\left(\frac{\sqrt{2}}{2}\right)^2+\left(-\frac{\sqrt{2}}{2}\right)^2  & \textrm{ if }k=\frac{3m+1}{2}\textrm{ and }m\textrm{ odd}  \\
	\end{array}\right\}=1
\end{align*}
as required. Moreover, for $k\neq l$, we obtain $\left[\tilde{X}_m^T\tilde{X}_m\right]_{k,l}=\sum_{j}\tilde{X}_{m,jk}\tilde{X}_{m,jl}=0$ if $|k-l|\geq2$. If $k$ is even and $l=k+1$, the statement remains true because then the columns $k$ and $l$ cover different blocks. Similarly, if $k$ is odd and $l=k-1$, the statement remains true. Finally, it also holds if $l=m$ or $k=m$, when $m$ is odd. So we have left to check the situation that $k<m$ is odd and $l=k+1$ (and the case that $k\leq m$ is even and $l=k-1$ but this works in the same way)
\begin{align*}
	\left[\tilde{X}_m^T\tilde{X}_m\right]_{m,k(k+1)}=\sum_{j}\tilde{X}_{m,jk}\tilde{X}_{m,j(k+1)}=-\frac{\sqrt{6}}{6}\cdot\frac{\sqrt{2}}{2}+\frac{\sqrt{6}}{6}\cdot\frac{\sqrt{2}}{2}=0.
\end{align*}
This proves $\tilde{X}_m^T\tilde{X}_m=I_m$. The property $\sum_i\tilde{X}_{m,ij}=0$ for $j=1,...,m$ is easy to check.

\end{proof}
With the help of the above lemma, we prove the main result of this section, which provides the desired representation of \eqref{eq:pls} in which \eqref{eq:LAR_cond1} and \eqref{eq:LAR_cond2} hold:
\begin{lemma}
\label{lem:design}
Let $Y,X,n,p$ be as in \eqref{eq:pls} and let $\tilde{X}_m$ be defined as in Lemma \ref{lem:xtilde} for $m=n$. Define $\overline{Y}:=\tilde{X}_mY$ and $\overline{X}:=\tilde{X}_mX$. Then, $\overline{Y}$ and $\overline{X}$ fulfill \eqref{eq:LAR_cond1} and \eqref{eq:LAR_cond2} and
$$\overline{\gamma}_n:=\argmin{\gamma\in\IR^p}\left\|\overline{Y}-\overline{X}\gamma\right\|_2^2+\lambda\|\gamma\|_1$$
equals $\hat{\gamma}_n$ from \eqref{eq:pls}.
\end{lemma}
\begin{proof}
Using Lemma \ref{lem:xtilde}, we see that
\begin{align*}
	\left\|\overline{Y}-\overline{X}\gamma\right\|_2^2&=\overline{Y}^T\overline{Y}+\gamma^T\overline{X}^T\overline{X}\gamma-2\overline{Y}^T\overline{X}\gamma=Y^T\tilde{X}_m^T\tilde{X}_mY+\gamma^TX^T\tilde{X}_m^T\tilde{X}_mX\gamma-2Y^T\tilde{X}_m^T\tilde{X}_mX \\
	&= Y^TY+\gamma^TX^TX\gamma-2Y^TX.
\end{align*}
Comparing this with \eqref{eq:pls} and since the penalty remains unchanged, it is clear that $\hat{\gamma}_n=\overline{\gamma}_n$. It remains to check \eqref{eq:LAR_cond1} and \eqref{eq:LAR_cond2}. We have by Lemma \ref{lem:xtilde} for all $j=1,...,p$
\begin{align*}
	&\sum_i\overline{Y}_i=\sum_i\sum_{k=1}^n\tilde{X}_{m,ik}Y_k=\sum_{k=1}^nY_k\sum_i\tilde{X}_{m,ik}=0, \\
	&\sum_i\overline{X}_{ij}=\sum_i\sum_{k=1}^n\tilde{X}_{m,ik}X_{kj}=\sum_{k=1}^nX_{kj}\sum_i\tilde{X}_{m,ik}=0,
\end{align*}
and the proof is complete.
\end{proof}

Lemma \ref{lem:design} can now be used to formulate a pure least squares problem which is equivalent to our estimation problem. This least squares problem, in turn, can be solved effectively by using the LAR algorithm.

\section{Useful results about Hawkes processes}
\label{sec:Hawkes}
\subsection{Existence and identifiability of the Hawkes process in our model}
\label{subsec:existence}

\begin{lemma}
\label{lem:existence}
Under Assumption (A1), for any choice of $n\in\IN$ and parameters $(C,\alpha,\theta)\in\mathcal{H}_n,$ there is a unique multivariate counting process $N_n=(N_{n,1},...,N_{n,n})$ with intensity functions given by Equation~\eqref{eq:Hawkes_dynamics} with $(C_n^*,\alpha_n^*,\theta_n^*)=(C,\alpha,\theta)$ and $N_{n,i}((-\infty, 0)) =0$. It holds that $\IE(N_{n,i}([a,b]))<\infty$ for any finite interval $[a,b]$ and all $i=1,...,n$. In particular, the processes $M_{n,i}(t):=N_{n,i}(t)-\int_0^t\lambda_{n,i}(s)ds$ are local, square-integrable martingales.

Moreover, the model is identified in the sense that processes with different parameters in Equation~\eqref{eq:Hawkes_dynamics} have different distributions.
\end{lemma}

The proof of this lemma is using the cluster representation which is as follows (see also \cite{HO71}).  Consider a multivariate Hawkes process $N=(N_1,...,N_n)$ with intensity functions
$$\lambda_{i}(t):=\nu_i+\sum_{j=1}^n\int_{-\infty}^{t-}h_{t,i,j}(t-r)dN_j(r)$$
for a vector $\nu\in[0,\infty)^n$ and deterministic functions $h_{t,i,j}:[0,\infty)\to[0,\infty)$ that we allow to depend on $t$. Consider independent counting processes $N_{b,i}$ with constant intensities $\nu_i$ for $i=1,...,n$. Let $\mathcal{W}(0):=(\mathcal{W}_1(0),...,\mathcal{W}_n(0))$ be a vector of given, discrete sets $\mathcal{W}_i(0)\subseteq\IR$. Having defined the sets $\mathcal{W}(K-1)$ the random sets $\mathcal{W}(K)$ are defined as follows:
$$\mathcal{W}_i(K):=\bigcup _{j=1}^n \bigcup_{t\in \mathcal{W}_j(K-1)}\left(t+\{\textrm{jumps of }N_{t,i,j}\}\right),$$
where in every term of the union the process $N_{t,i,j}$ is a different independent copy of a counting process with intensity $h_{t,i,j}(\cdot)$. The sets $\mathcal{W}_i(K)$ are called events of $i$ in generation $K$. It can be shown that if $\mathcal{W}(0)$ contains the sets of jump points of the process $N_{b,i}$, then it holds that
$$N_i\sim \bigcup_{K=0}^{\infty}\mathcal{W}_i(K),$$
see \cite{BMM15}. We now use the cluster representation to prove Lemma \ref{lem:existence}. 

\begin{proof}[Proof of Lemma \ref{lem:existence}]
Theorem 7 in \citet{BM96} states that there is a unique multivariate Hawkes process $\overline{N}$ (with finite expectations) with intensity given by \eqref{eq:Hawkes_dynamics},  where $\nu_0(X_{n,i}(t);\beta)$ is replaced by $\overline{\nu}_i$ if the largest eigenvalue of the matrix
$A:=C$ is strictly less than $1$. This, in turn, holds by the following argument: Suppose that $\mu\in\IC$ and $v = (v_1,\ldots,v_n) \in\IC^n\setminus\{0\}$ are an eigenpair of $C$, i.e., $Cv=\mu v$. Let $k\in\{1,...,n\}$ be such that $|v_k|=\|v\|_{\infty}$. Then (since $v\neq0$, $|v_k|>0$),
$$|\mu|\|v\|_{\infty}=|\mu v_k|=\left|\left(Cv\right)_k\right|\leq \sum_{j=1}^nC_{k,j}|v_j|\leq\|C_{k\cdot}\|_1\|v\|_{\infty} \le \|\nu\|_\infty,$$
i.e. $|\mu| < 1$.

The process $\overline{N}$ has stationary increments and the expected number of events on any finite interval is finite. Note that the stationary process has cluster representation with $h_{t,i,j}(s)=g(s;\gamma)$ and $\nu_i=\alpha_{0,i}\overline{\nu}_i$. From this cluster representation of $\overline{N}$ we can easily construct the cluster representation of a process $N_n$ which has dynamics described by \eqref{eq:Hawkes_dynamics}: In a first step one removes all $t<0$ from $\mathcal{W}_i(0)$, then, in a second step, since $\alpha_{0,i}\nu_0(X_{n,i}(t);\beta)\leq\nu_i$, we can thin the process $N_{b,i}$ from the cluster representation such that it has time-varying intensity $\alpha_{0,i}\nu_0(X_{n,i}(t);\beta)$. Then in the construction of $\mathcal{W}_i(K)$ one removes the sets
$$\left(t+\{\textrm{jumps of }N_{t,i,j}\}\right)$$
for those $t$ that have been removed from $\mathcal{W}_j(K-1)$. Thus, by thinning of the stationary process $\overline{N}$, we obtain a process $N_n$ with dynamics \eqref{eq:Hawkes_dynamics}. Because $N_n$ is bounded by $\overline{N}$, $N_n$ has also finite expectations. Furthermore, because the cluster representation of the process is uniquely defined, cf. discussion after Proposition~6 in \cite{BMM15}, we have that the process is unique.

For the second part of the statement, let $N_n$ and $\tilde{N}_n$ be two processes with dynamics described in Equation~\eqref{eq:Hawkes_dynamics} for two parameter sets $(C_n^*,\alpha_n^*,\theta_n^*)$ and $(C_n^{**},\alpha_n^{**},\theta_n^{**})$, respectively, such that $(C_n^*,\alpha_n^*,\theta_n^*)\neq(C_n^{**},\alpha_n^{**},\theta_n^{**})$.
The assumptions imply that the intensity functions of some of the processes $N_{b,i}$ or $N_{t,i,j}$ used in the cluster representation of $N_n$ have a different intensity function of their counterpart in the cluster representation of $\tilde{N}_n$.
Therefore, the cluster representations of $\tilde{N}$ and $\tilde{N}_n$ are different.
The above mentioned uniqueness of the cluster representation, in turn, implies that $N_n$ and $\tilde{N}_n$ must have different distributions.
\end{proof}

\subsection{Bounds on increments of Hawkes processes}
\label{subsec:bounds_Hawkes}
In this section, we prove that Hawkes processes have almost uniformly bounded increments (up to a $\log$-factor). Let $N_{n,i}$ for $i=1,...,n$ be Hawkes processes with dynamics \eqref{eq:Hawkes_dynamics}. For given $A>0$ and $\mathcal{N}>0$, let
$$\Omega_{\mathcal{N}}:=\left\{\forall k\in\IN_0\cap\left[0,\frac{T}{A}\right]:\, |N_{n,i}\left[kA,(k+1)A\right)|\leq\frac{\mathcal{N}}{2} \textrm{ for all } i\in\{1,...,n\}\right\}.$$
\begin{lemma}
\label{lem:omega_lemma}
Suppose that (A1) holds. There are $\epsilon\in(0,1)$ and $r>0$ such that $|e^x-1|\leq\frac{\epsilon}{a_0}|x|$ for all $|x|\leq r$. Let $a:=r(1-\epsilon)$. Then, any $\mathcal{N}_0>0$, it holds with $\mathcal{N}:=6\mathcal{N}_0\cdot \log Tn$, that
$$\IP\left(\overline{\Omega}_{\mathcal{N}}^c\right)\leq\left(2\exp\left(\frac{A\|\overline{\nu}\|_{\infty}\|\alpha\|_{\infty}\cdot r\epsilon^2}{a_0(1-\epsilon)}\right)+\exp\left(\|\overline{\nu}\|_{\infty} A\|\alpha\|_{\infty}\left(e^a-1\right)\right)\right)\frac{T+A}{AT}(nT)^{1-a\mathcal{N}_0}.$$
\end{lemma}
For the proof of Lemma \ref{lem:omega_lemma}, fix $l\in\{1,...,n\}$, and consider the cluster representation introduced in Section \ref{subsec:existence} with initial sets $\mathcal{W}_i(0)$ as follows: $\mathcal{W}_l(0)$ contains exactly one element and $\mathcal{W}_j(0)=\emptyset$ for $j\neq l$. For $i\in\{1,...,n\}$, denote $W_i^l(K):=\left|\bigcup_{k=0}^K\mathcal{W}_i(k)\right|$ and $W_i^l:=W_i^l(\infty)$. Repeating this construction for each $l\in\{1,...,n\}$ yields random vectors $W^l(k)=\left(W_1^l(k),...,W_n^l(k)\right)\in\IN^n$ for $l=1,..,n$ and $k\in\IN\cup\{\infty\}$.
We first state and prove a slightly stronger version of Lemma 1 and Proposition 2 in \citet{HRBR15} that fits our setting.
\begin{lemma}
\label{lem:goodN}
Let Assumption~(A1) hold and let $N_i, i=1,...,n$ be Hawkes processes with intensity functions
$$
\lambda_{n,i}(t):= 
\alpha^*_{n,i}\cdot\overline{\nu}_i+\sum_{j=1}^nC^*_{n,ij}\int_{-\infty}^{t-}g(t-r;\gamma_n^*)dN_{n,j}(r),$$
where $\overline{\nu}_i$ is defined in Assumption~(A1) and $(C_n^*,\alpha_n^*,\gamma_n^*)=(C,\alpha,\gamma)$ arbitrary. Then, with $\epsilon, r$ as in Lemma~\ref{lem:omega_lemma}, for all $s\in[0,\infty)^n$ with $\|s\|_1\leq r(1-\epsilon)$ it holds that
\begin{align}
	&\sum_{l=1}^n\left|\log\IE\left(e^{s^TW^l}\right)\right|\leq r, \label{eq:lem11} \\
	&\sum_{k=0}^{\infty}\sum_{l=1}^n\left|\IE\left(e^{s^T(W^l-W^l(k))}\right)-1\right|\leq\frac{r\epsilon^2}{a_0(1-\epsilon)}. \label{eq:lem12}
\end{align}
\end{lemma}
The proof takes many ideas of the proofs of Lemma 1 and Proposition 2 in \citet{HRBR15}. However, for our purposes, we need to refine some of the arguments in order to prove a stronger bound on the sum of the expectations. For completeness, we provide the proof in Appendix~\ref{sup:HawkesBounds}.

\begin{proof}[Proof of Lemma \ref{lem:omega_lemma}]
In this proof, we use again ideas from Proposition 2 of \citet{HRBR15}. Note that $\IP(\Omega_{\mathcal{N}})\geq\IP(\overline{\Omega}_{\mathcal{N}})$, where $\overline{\Omega}_{\mathcal{N}}$ is defined as $\Omega_{\mathcal{N}}$ but for the Hawkes process defined in Lemma~\ref{lem:goodN}. By a thinning argument, it is clear that these new processes can be used to form a stationary upper bound of the original processes. In the following, all counting processes are understood to be the stationary upper bounds. By stationarity, we have
\begin{align}
	\IP\left(\overline{\Omega}_{\mathcal{N}}^c\right)\leq&\sum_{k=1}^{\lceil\frac{T}{A}\rceil}\sum_{i=1}^n\IP\left(|N_{n,i}[kA,(k+1)A)|>\frac{\mathcal{N}}{2}\right) \nonumber \\
	&\leq \frac{T+A}{A}\sum_{i=1}^n\IP\left(|N_{n,i}[-A,0)|>\frac{\mathcal{N}}{2}\right). \label{eq:u_bound}
\end{align}
Recall that, in the cluster representation, $N_{b,j}$ denotes the $0$-th generation events and that we collected them in $W_j(0)$. Consider now such a $0$-th generation event from the interval $[-(k+1)A,-kA)$. Since $g(\cdot;\gamma)$ is supported on $[0,A]$, we conclude that any $1$st generation event must have occurred by time $-(k-1)A$. Continuing this argument inductively explains that every $0$-th generation event from the interval $[-(k+1)A,-kA)$ can only spawn offspring events in the interval $[-A,0)$ from the $k$-th generation onwards. Therefore, the number of events that a $0$-th generation event in process $N_{b,l}$ in the interval $[-(k+1)A,-kA)$ can produce in the process $N_{n,i}$ in the interval $[-A,0)$ is upper bounded (in distribution) by $N_i^l(k):=W_i^l-W_i^l(k-1)$ for $k\geq1$ and $N_i^l(0):=W_i^l$. We abuse notation, and denote by $N_i^l(k)$ a collection of jointly independent random variables with the aforementioned marginal distributions. Denote by $N_{i,m}^l(k)$ iid copies of $N_i^l(k)$ indexed by $m$. Since every event must be a member of a cluster that eventually originates from a $0$-th generation event, and since the clusters evolve independently, we conclude that for each $x\geq0$
\begin{align}
	&\IP(|N_{n,i}[-A,0)|\geq x)
	\leq\IP\left(\sum_{k=0}^{\infty}\sum_{l=1}^n\sum_{m=1}^{|N_{b,l}[-(k+1)A,-kA)|}N_{i,m}^l(k)\geq x\right) \nonumber \\
	&\leq\IP\Bigg(\sum_{k=1}^{\infty}\sum_{l=1}^n\sum_{m=1}^{|N_{b,l}[-(k+1)A,-kA)|}N_{i,m}^l(k) \nonumber \\
	&\qquad+\sum_{l=1}^n\sum_{m=1}^{|N_{b,l}[-A,0)|}\left(N_{i,m}^l(0)-W_i^l(0)\right)+|N_{b,i}[-A,0)|\geq x\Bigg) \nonumber \\
	&\leq\IP\left(\sum_{k=1}^{\infty}\sum_{l=1}^n\sum_{m=1}^{|N_{b,l}[-(k+1)A,-kA)|}N_{i,m}^l(k)\geq\frac{x}{3}\right)+\IP\left(\sum_{l=1}^n\sum_{m=1}^{|N_{b,l}[-A,0)|}\left(N_{i,m}^l(0)-W_i^l(0)\right)\geq\frac{x}{3}\right) \nonumber \\
	&\qquad+\IP\left(|N_{b,i}[-A,0)|\geq\frac{x}{3}\right) \nonumber \\
	&\leq 2\IP\left(\sum_{k=1}^{\infty}\sum_{l=1}^n\sum_{m=1}^{|N_{b,l}[-(k+1)A,-kA)|}N_{i,m}^l(k)\geq\frac{x}{3}\right)+\IP\left(|N_{b,i}[-A,0)|\geq\frac{x}{3}\right) \nonumber \\
	&\leq e^{-\frac{ax}{3}}\left(2\IE\left(e^{a\sum_{k=1}^{\infty}\sum_{l=1}^n\sum_{m=1}^{|N_{b,l}[-(k+1)A,-kA)|}N_{i,m}^l(k)}\right)+\IE\left(e^{a|N_{b,i}[-A,0)|}\right)\right)
	\label{eq:u_bound_split}
\end{align}
The penultimate inequality holds because $N_{i,m}^l(0)-W_i^l(0)$ has the same distribution as $N_{i,m}^l(1)$, and $N_{b,l}[-(k+1)A,-kA)$ are independent. We begin with bounding the first expectation from (\ref{eq:u_bound_split}). Below, we make use of the independence of the involved random variables and the fact that $N_{b,l}[-(k+1)A,-kA)$ is Poisson distributed with rate $A\overline{\nu}_i\alpha_i$
\begin{align*}
	&\log\IE\left(e^{a\sum_{k=1}^{\infty}\sum_{l=1}^n\sum_{m=1}^{|N_{b,l}[-(k+1)A,-kA)|}N_{i,m}^l(k)}\right) \\
	&= \sum_{k=1}^{\infty}\sum_{l=1}^n\log\IE\left(\IE\left(e^{a N_{i,m}^l(k)}\right)^{|N_{b,l}[-(k+1)A,-kA)|}\right)=\sum_{k=1}^{\infty}\sum_{l=1}^nA\overline{\nu}_i\alpha_i\left(\IE\left(e^{a N_{i,m}^l(k)}\right)-1\right) \\
	&\leq A\|\overline{\nu}\|_{\infty}\|\alpha\|_{\infty}\frac{r\epsilon^2}{a_0(1-\epsilon)},
\end{align*}
by \eqref{eq:lem12} of Lemma \ref{lem:goodN}, which we apply here with $s\in[0,\infty)^n$ with $s_i=a=r(1-\epsilon)$ and $s_j=0$ for $j\neq i$ because $N_{i,m}^l(k)\sim W_i^l-W_i^l(k-1)$. Since, $|N_{b,i}[-A,0)|$ follows a Poisson distribution with parameter $A\overline{\nu}_i\alpha_i$, we get for the second expectation in \eqref{eq:u_bound_split} that
$$\log\IE\left(e^{a|N_{b,i}[-A,0)|}\right)=A\alpha_i\overline{\nu}_i\left(e^a-1\right)\leq A\|\overline{\nu}\|_{\infty}\|\alpha\|_{\infty}\left(e^a-1\right).$$
Using the previous two displays in \eqref{eq:u_bound_split}, we find
\begin{equation}
	\label{eq:exp_bound_Hawes}
	\IP(|N_{n,i}[-A,0)|\geq x)\leq e^{-\frac{ax}{3}}\left(2\exp\left(\frac{A\|\overline{\nu}\|_{\infty}\|\alpha\|_{\infty}\cdot r\epsilon^2}{a_0(1-\epsilon)}\right)+\exp\left(A\|\overline{\nu}\|_{\infty}\|\alpha\|_{\infty}\left(e^a-1\right)\right)\right).
\end{equation}
Plugging the above with $x=\mathcal{N}/2$ in \eqref{eq:u_bound} yields
\begin{align*}
	\IP\left(\overline{\Omega}_{\mathcal{N}}^c\right)\leq&\frac{T+A}{A}\sum_{i=1}^n\left(e^{-\frac{a\mathcal{N}}{6}}\left(2\exp\left(A\|\overline{\nu}\|_{\infty}\|\alpha\|_{\infty}\frac{r\epsilon^2}{a_0(1-\epsilon)}\right)+\exp\left(A\|\overline{\nu}\|_{\infty}\|\alpha\|_{\infty}\left(e^a-1\right)\right)\right)\right) \\
	\leq&\left(2\exp\left(A\|\overline{\nu}\|_{\infty}\|\alpha\|_{\infty}\frac{r\epsilon^2}{a_0(1-\epsilon)}\right)+\exp\left(A\|\overline{\nu}\|_{\infty}\|\alpha\|_{\infty}\left(e^a-1\right)\right)\right)\frac{T+A}{AT}(nT)^{1-a\mathcal{N}_0}.
\end{align*}
The proof is complete.
\end{proof}

\section{Proofs}
\label{sec:proofs}

\subsection{Proofs of Section \ref{subsec:single_consistency}}
\label{subsec:proofs_single_consistency}
The proof of Theorem~\ref{thm:pre_estimate_consistency} is given in the Appendix. The proof of Proposition~\ref{lem:convergence_rates} is based on Lemma~\ref{cor:lambda_rate}, which in turn follows from Lemma~\ref{lem:lambda_rate} that we formulate and prove first:

\begin{lemma}
\label{lem:lambda_rate}
Suppose that (A1),(A2), and (A3) hold. Let $\mathcal{N}_0,\alpha_1,\alpha_2,\alpha_3,\alpha_4>0$ be arbitrary, and let $\mu\in(0,3)$ be such that $\mu>\phi(\mu),$ where $\phi(u)=e^u-u-1$. Denote
\begin{align*}
	\hat{V}_{a,\mathcal{T}}^{\mu}:&= \underset{\overline{\beta}\in K_{\beta}}{\sup_{i=1,...,n}}\frac{16\mu\int_0^T\nu_0(X_{n,i}(t);\overline{\beta})^2dN_{n,i}(t)}{(\mu-\phi(\mu))T^2}+\frac{16\|\overline{\nu}\|_{\infty}^2\left(\log(n)+p\log T+\alpha_1\log(nT)\right)}{(\mu-\phi(\mu))T^2},
\end{align*}
\begin{align*}
	\hat{V}_{b,\mathcal{T}}^{\mu}:&= \underset{\beta_2: \|\beta_2\|_1=1}{\sup_{\beta_1\in K_{\beta}}}\frac{16K_{\alpha}^2\mu\sum_{i=1}^n\int_0^T\left(\int_0^1\frac{d}{d\beta}\nu_0\left(X_{n,i}(t);(1-s)\beta_n^*+s\beta_1\right)^T\beta_2ds\right)^2dN_{n,i}(t)}{(\mu-\phi(\mu))n^2T^2} \\
	&\qquad\qquad+\frac{16K_{\alpha}^2L_{\nu}^2(2p+\alpha_2)\log(nT)}{(\mu-\phi(\mu))n^2T^2}, \\
	\hat{V}_{d,\mathcal{T},i}^{\mu}:&= \underset{\overline{\gamma}\in K_{\gamma}}{\sup_{j\in\{1,...,n\}}}\frac{16\mu\int_0^T\left(\int_0^{t-}g(t-r;\overline{\gamma})dN_{n,j}(r)\right)^2dN_{n,i}(t)}{(\mu-\phi(\mu))T^2} \\
	&\qquad\qquad+\frac{567\overline{g}^2\mathcal{N}_0^2\log^2(nT)\cdot\left(\log n+\log(nT)+\alpha_3\log T\right)}{(\mu-\phi(\mu))T^2},
\end{align*}
\begin{align*}
	\hat{V}_{e,\mathcal{T}}^{\mu}:&= \sup_{\overline{\gamma}\in K_{\gamma}}\frac{\mu}{\mu-\phi(\mu)} \\
	&\qquad\times\sum_{i=1}^n\int_0^T\Bigg(\sum_{j=1}^nC_{n,ij}^*\frac{4\int\limits_0^{t-}\int\limits_0^1\frac{d}{d\gamma}g(t-r;s\overline{\gamma}+(1-s)\gamma_n^*)dsdN_{n,j}(r)}{nT}\Bigg)^2 dN_{n,i}(t) \\
	&\qquad\qquad+\frac{576L_g^2\mathcal{N}_0^2(1+\alpha_4)c_n^2\log^3(nT)}{(\mu-\phi(\mu))n^2T^2}.
\end{align*}
Let $\Omega_{\mathcal{N}}$ by as in Lemma \ref{lem:omega_lemma} with $\mathcal{N}:=6\mathcal{N}_0\log(nT)$ where ${\cal N}_0$ is arbitrary, and denote
\begin{align*}
	a_n:&= \Bigg(2\sqrt{\hat{V}_{a,\mathcal{T}}^{\mu}\left(\log n+p\log T+\alpha_1\log(nT)\right)} \\
	&\qquad\qquad+\frac{4\|\overline{\nu}\|_{\infty}\left(\log n+p\log T+\alpha_1\log(nT)\right)}{3T}\Bigg)\Ind_{\Omega_{\mathcal{N}}}, \\
	b_n:&= \left(2\sqrt{\hat{V}_{b,\mathcal{T}}^{\mu}(2p+\alpha_2)\log(nT)}+\frac{4K_{\alpha}L_{\nu}(2p+\alpha_2)\log(nT)}{3nT}\right)\Ind_{\Omega_{\mathcal{N}}}, \\
	d_{n,i}:&= \Bigg(2\sqrt{\hat{V}_{d,\mathcal{T},i}^{\mu}\left(\log n+\log(nT)+\alpha_3\log T\right)} \\
	&\qquad\qquad+\frac{24\overline{g}\mathcal{N}_0 \log(nT)\cdot\left(\log n+\log(nT)+\alpha_3\log T\right)}{3T}\Bigg)\Ind_{\Omega_{\mathcal{N}}}, \\
	e_n:&= \left(2\sqrt{\hat{V}_{e,\mathcal{T}}^{\mu}(1+\alpha_4)\log(nT)}+\frac{24L_g\mathcal{N}_0(1+\alpha_4)c_n\log^2(nT)}{3nT}\right)\Ind_{\Omega_{\mathcal{N}}}.
\end{align*}
Then, there are constants $a,c_{\Omega},c_1,c_2,c_3,c_4>0$, where $a$ depends only on $a_0$ from (A1), such that
\begin{align}
	\IP(\mathcal{T}_n(a_n,b_n,d_n,e_n)^c)&\leq c_1\frac{c_n^p\log T}{(nT)^{\alpha_1}}+c_2\frac{c_n^{2p}\log(nT)}{(nT)^{\alpha_2}}+c_3\frac{c_n\log(T)}{T^{\alpha_3}} \nonumber \\
	&\qquad+c_4\frac{\log(nT)(1+c_n)}{(nT)^{\alpha_4}}+4c_{\Omega}(Tn)^{1-a\mathcal{N}_0}. \label{eq:probT_bound}
\end{align}
\end{lemma}

\begin{proof}[Proof of Lemma \ref{lem:lambda_rate}]
The main tool for the  proof is the exponential inequality from Theorem 3 in \citet{HRBR15}, and we use ideas from the proofs of Theorem 2 and Proposition 2 in \citet{HRBR15}. Note that for any $t\in[0,T]$ the interval $[t,t+A)$ can be covered by two intervals of the form $[kA,(k+1)A)$. Then, we have on $\Omega_{\mathcal{N}}$ that $N_{n,i}[t,t+A)\leq \mathcal{N}$ for all $t\in[0,T]$ and $n\in\IN$. Thus, we conclude that on $\Omega_{\mathcal{N}}$ for all $i,j=1,...,n$
\begin{align}
	\sup_{\overline{\beta}\in K_{\beta}}\int_0^T\nu_0(X_{n,i}(t);\overline{\beta})^2dN_{n,i}(t)&\leq\left(\frac{T}{A}+1\right)\overline{\nu}_i^2\mathcal{N}, \label{eq:V0} \\
	\sup_{t\in[0,T]}\sup_{\overline{\gamma}\in K_{\gamma}}\left|\int_0^{t-}g(t-r;\overline{\gamma})dN_{n,j}(r)\right|&\leq\overline{g}\mathcal{N}. \label{eq:Bj}
\end{align}
We fix now $\mathcal{N}:=6\mathcal{N}_0\cdot\log(Tn)$ and study the four parts of $\mathcal{T}_n$ separately. In each part, a term $\IP(\Omega_{\mathcal{N}}^c)$ will appear. This term is always bounded using Lemma \ref{lem:omega_lemma}, which we may apply by Assumptions~(A1) and (A2) for each $n\in\IN$ with the same choice of $a_0,r,\epsilon$, yielding the last term in \eqref{eq:probT_bound}. Define furthermore the stopping time
$$\tau_n:=T\wedge \inf\left\{t\in[0,\infty):\exists i,j\in\{1,...,n\}, \sup_{\overline{\gamma}\in K_{\gamma}}\left|\int_0^{t-}g(t-r;\overline{\gamma})dN_{n,j}(r)\right|>\overline{g}\mathcal{N}\right\}.$$
we have that $\tau_n=T$ on $\Omega_{\mathcal{N}}$.

\underline{Part involving $a_n$}: We can use a classical \emph{chaining light} argument as follows. Let $K_{\beta,n,\eta}\subseteq K_{\beta}$ be finite such that for any $\overline{\beta}\in K_{\beta}$ there exists $P_n(\overline{\beta})\in K_{\beta,n,\eta}$ such that $\|\overline{\beta}-P_n(\overline{\beta})\|\leq\eta$. It is possible to choose $K_{\beta,n,\eta}$ such that $|K_{\beta,n,\eta}|\leq K_0\eta^{-p}$ for some constant $K_0>0$. Define a constant $c_1''$ such that for all $n$ and $T$
$$2L_{\nu}\left(\frac{\mathcal{N}(T+A)}{TA}+K_{\alpha}\|\overline{\nu}\|_{\infty}+c_n\overline{g}\mathcal{N}\right)\leq c_1''c_n\log(nT).$$
This definition is, at this point, unmotivated. Later in the proof, it will turn out that, with $c_1''$ as above, the following choice of $\eta>0$ is the right one for our purposes
$$\eta:=\frac{4\|\overline{\nu}\|_{\infty}\left(\log n+p\log T+\alpha_1\log(nT)\right)}{c_1''\sqrt{\mu-\phi(\mu)}Tc_n\log(nT)}.$$
These definitions will be important later. We first note that
\begin{align}
	&\IP\left(\sup_{i=1,...,n}\sup_{\overline{\beta}\in K_{\beta}}\frac{2}{T}\left|\int_0^T\nu_0(X_{n,i}(t);\overline{\beta})dM_{n,i}(t)\right|\geq a_n\right) \nonumber \\
	&\leq\IP\left(\sup_{i=1,...,n}\sup_{\overline{\beta}\in K_{\beta}}\frac{2}{T}\left|\int_0^T\nu_0(X_{n,i}(t);\overline{\beta})dM_{n,i}(t)\right|\geq a_n,\Omega_{\mathcal{N}}\right)+\IP(\Omega_{\mathcal{N}}^c) \nonumber \\
	&\leq\IP\Bigg(\sup_{i=1,...,n}\sup_{\overline{\beta}\in K_{\beta}}\Bigg|\frac{2}{T}\int_0^T\nu_0(X_{n,i}(t);\overline{\beta})dM_{n,i}(t) \nonumber \\
	&\qquad\qquad-\frac{2}{T}\int_0^T\nu_0(X_{n,i}(t);P_n(\overline{\beta}))dM_{n,i}(t)\Bigg|\geq \frac{a_n}{2},\Omega_{\mathcal{N}}\Bigg) \label{eq:an_cont} \\
	&+\IP\left(\sup_{i=1,...,n}\sup_{\overline{\beta}\in K_{\beta,n,\eta}}\frac{2}{T}\left|\int_0^T\nu_0(X_{n,i}(t);\overline{\beta})dM_{n,i}(t)\right|\geq \frac{a_n}{2},\Omega_{\mathcal{N}}\right)+\IP(\Omega_{\mathcal{N}}^c). \label{eq:an_emp_proc}
\end{align}
The probability in \eqref{eq:an_emp_proc} can be bounded from above using union bound:
\begin{align}
	&\IP\left(\sup_{i=1,...,n}\sup_{\overline{\beta}\in K_{\beta,n,\eta}}\frac{2}{T}\left|\int_0^T\nu_0(X_{n,i}(t);\overline{\beta})dM_{n,i}(t)\right|\geq \frac{a_n}{2},\Omega_{\mathcal{N}}\right) \nonumber \\
	&\leq nK_0\eta^{-p}\sup_{i=1,...,n}\sup_{\overline{\beta}\in K_{\beta,n,\eta}}\IP\left(\frac{2}{T}\left|\int_0^T\nu_0(X_{n,i}(t);\overline{\beta})dM_{n,i}(t)\right|\geq \frac{a_n}{2},\Omega_{\mathcal{N}}\right). \label{eq:an_ub}
\end{align}
We handle this term by applying Theorem 3 of \citet{HRBR15} for a uni-variate process, i.e., in the notation of \citet{HRBR15}, $M=1$. We keep the notation as in the theorem for the convenience of the reader. We have $H(t)=\frac{4}{T}\nu_0(X_{n,i}(t);\overline{\beta})$, which is uniformly bounded in absolute value by $B:=\frac{4\overline{\nu}_i}{T}$. The integral conditions are hence true and we consider the constant stopping time $\tau=T$. We furthermore define for any $x>0$
$$\hat{V}^{\mu}_a:=\frac{\mu}{\mu-\phi(\mu)}\int_0^T\frac{16\nu_0(X_{n,i}(t);\overline{\beta})^2}{T^2}dN_{n,i}(t)+\frac{x}{\mu-\phi(\mu)}\cdot\frac{16\overline{\nu}_i^2}{T^2}.$$
On $\Omega_{\mathcal{N}}$, it holds by \eqref{eq:V0} that
$$w:=\frac{16x\overline{\nu}_i^2}{T^2(\mu-\phi(\mu))}\leq\hat{V}^{\mu}_a\leq\frac{16\mu\overline{\nu}_i^2\mathcal{N}(T+A)}{T^2A(\mu-\phi(\mu))}+\frac{16x\overline{\nu}_i^2}{T^2(\mu-\phi(\mu))}=:v.$$
We hence obtain from Theorem 3 of \citet{HRBR15} for $\epsilon=1$ that, by letting $x=\log n+p\log T+\alpha_1\log(nT)$, that there exists $c_1'>0$ such that
\begin{align*}
	&\IP\left(\int_0^T\frac{2}{T}\nu_0(X_{n,i}(t);\overline {\beta})dM_{n,i}(t)\geq \frac{a_n}{2}, \Omega_{\mathcal{N}}\right) \\
	&\leq\IP\Bigg(\int_0^T\frac{4}{T}\nu_0(X_{n,i}(t);\overline{\beta})dM_{n,i}(t)\geq 2\sqrt{\hat{V}^{\mu}_ax}+\frac{Bx}{3}, w\leq\hat{V}_a^{\mu}\leq v, \\
	&\qquad\qquad\sup_{t\in[0,\tau]}\frac{4\nu_0(X_{n,i}(t);\overline{\beta})}{T}\leq\frac{4\overline{\nu}_i}{T}\Bigg) \\
	&\leq 2\left(\log_2\left(\frac{v}{w}\right)+1\right)e^{-x}=2\left(\log_2\left(1+\frac{\mu\mathcal{N}(T+A)}{Ax}\right)+1\right)e^{-x}\leq c_1'\log T \cdot n^{-1}T^{-p}(nT)^{-\alpha_1},
\end{align*}
where we used $\mathcal{N}=6\mathcal{N}_0\log(nT)$ in the last inequality. Combining the above with \eqref{eq:an_ub} yields
\begin{align*}
	&\IP\left(\sup_{i=1,...,n}\sup_{\overline{\beta}\in K_{\beta,n,\eta}}\frac{2}{T}\left|\int_0^T\nu_0(X_{n,i}(t);\overline{\beta})dM_{n,i}(t)\right|\geq \frac{a_n}{2},\Omega_{\mathcal{N}}\right) \\
	&\leq nK_0\eta^{-p}2c_1'\log T \cdot n^{-1}T^{-p}(nT)^{-\alpha_1} \\
	&\leq nK_02c_1'\log T \cdot n^{-1}T^{-p}(nT)^{-\alpha_1} \\
	&\qquad\times\frac{(c_1'')^p(\mu-\phi(\mu))^{\frac{p}{2}}T^pc_n^p\log(nT)^p}{4^p\|\overline{\nu}\|_{\infty}^p\left(\log n+p\log T+\alpha_1\log(nT)\right)^p}\leq c_1\frac{c_n^p\log T}{(nT)^{\alpha_1}}
\end{align*}
for a suitable choice of $c_1>0$.

We turn now to \eqref{eq:an_cont}. For ease of notation, we denote below $d|M_{n,i}|(t):=dN_{n,i}(t)+\lambda_{n,i}(t)dt$. On the event $\Omega_{\mathcal{N}}$, the Lipschitz continuity of $\nu_0$ (expressed as differentiability and bounded derivative in (A3)) together with the properties of $N_{n,i}$ on $\Omega_{\mathcal{N}}$ implies
\begin{align*}
	&\sup_{i=1,...,n}\sup_{\overline{\beta}\in K_{\beta}}\left|\frac{2}{T}\int_0^T\nu_0(X_{n,i}(t);\overline{\beta})dM_{n,i}(t)-\frac{2}{T}\int_0^T\nu_0(X_{n,i}(t);P_n(\overline{\beta}))dM_{n,i}(t)\right| \\
	&\leq \sup_{i=1,...,n}\sup_{\overline{\beta}\in K_{\beta}}\frac{2}{T}\int_0^T\left|\nu_0(X_{n,i}(t);\overline{\beta})-\nu_0(X_{n,i}(t);P_n(\overline{\beta}))\right|d|M_{n,i}|(t) \\
	&\leq \sup_{i=1,...,n}\frac{2L_{\nu}\eta}{T}\Bigg(\int_0^TdN_{n,i}(t)+\int_0^T\alpha_{n,i}^*\nu_0(X_{n,i}(t);\beta_n^*) \\
	&\qquad\qquad+\sum_{j=1}^nC_{n,ij}^*\int_0^{t-}g(t-r;\gamma_n^*)dN_{n,j}(r)dt\Bigg) \\
	&\leq \sup_{i=1,...,n}\frac{2L_{\nu}\eta}{T}\left(\frac{\mathcal{N}(T+A)}{A}+TK_{\alpha}\overline{\nu}_i+T\|C_{n,i\cdot}^*\|_1\overline{g}\mathcal{N}\right) \\
	&\leq 2L_{\nu}\left(\frac{\mathcal{N}(T+A)}{AT}+K_{\alpha}\|\overline{\nu}\|_{\infty}+c_n\overline{g}\mathcal{N}\right)\eta\leq c_1''c_n\log(nT)\eta,
\end{align*}
where $c_1''$ was chosen such that the last inequality holds. By choice of $\eta$, we obtain on $\Omega_{\mathcal{N}}$
$$c_1''c_n\log(nT)\eta\leq\sqrt{\sup_{i=1,..,n}wx}\leq\sqrt{\hat{V}_{a,\mathcal{T}}^{\mu}x}\leq\frac{a_n}{2}.$$
This, finally, implies $\eqref{eq:an_cont}=0$.

The arguments for the remaining parts work similarly, and we defer the details to Appendix~\ref{subsec:details_proof_lambda_rate}
.
\end{proof}

\begin{proof}[Proof of Lemma~\ref{cor:lambda_rate}]
Choose $\mathcal{N}_0$ such that $a\mathcal{N}_0>1$ ($a$ as in Lemma \ref{lem:omega_lemma}). Then, Lemma \ref{lem:omega_lemma} implies $\IP(\Omega_{\mathcal{N}})\to1$ for $\mathcal{N}:=6\mathcal{N}_0\log(nT)$. Note furthermore that $\IP(\mathcal{T}_n(a_n,b_n,d_n,e_n))$ increases if we increase $a_n,b_n,d_n,e_n$. Therefore, it is sufficient to find deterministic upper bounds on the sequences $a_n,b_n,d_{n,i},e_n$ provided by Lemma \ref{lem:lambda_rate}. Due to the formulas given there, we may restrict to the case when $\Omega_{\mathcal{N}}$ holds true. Below, $K$ is a constant that may change from appearance to appearance, but it never depends on quantities that change with $n$. Using the bounds on the $\hat{V}^{\mu}$ expressions from the proof of Lemma \ref{lem:lambda_rate}, we obtain
\begin{align*}
	a_n&\leq K\left(\sqrt{\left(\frac{\log(nT)}{T}+\frac{\log(nT)}{T^2}\right)\log(nT)}+\frac{\log(nT)}{T}\right)\leq K\frac{\log(nT)}{\sqrt{T}}, \\
	b_n&\leq K\left(\sqrt{\left(\frac{\log(nT)}{nT}+\frac{\log(nT)}{n^2T^2}\right)\log(nT)}+\frac{\log(nT)}{nT}\right)\leq K\frac{\log(nT)}{\sqrt{nT}}, \\
	d_{n,i}&\leq K\left(\sqrt{\left(\frac{\log^3(nT)}{T}+\frac{\log^3(nT)}{T^2}\right)\log(nT)}+\frac{\log^2(nT)}{T}\right)\leq K\frac{\log^2(nT)}{\sqrt{T}}, \\
	e_n&\leq K\Bigg(\sqrt{\left(\frac{\log^3(nT)c_n^2}{nT}+\frac{\log^3(nT)c_n^2}{n^2T^2}\right)\log(nT)}+\frac{c_n\log^2(nT)}{nT}\Bigg) \\
	&\leq K\frac{\log^2(nT)\cdot c_n}{\sqrt{nT}}.
\end{align*}
Note that, importantly, the constant $K$ for bounding $d_{n,i}$ does not depend on $i$. These bounds imply the statement since $\IP(\Omega_{\mathcal{N}})\to1$ by Lemma \ref{lem:omega_lemma}.
\end{proof}

Now we are ready to present the proof of Proposition~\ref{lem:convergence_rates}.

\begin{proof}[Proof of Proposition \ref{lem:convergence_rates}]
In the situation of Lemma~\ref{cor:lambda_rate}, the requirements of Theorem \ref{thm:pre_estimate_consistency} are fulfilled if, for a suitable $K_{\tilde{b}}>0$, we choose 
\begin{align*}
	\tilde{b}_n& =K_{\tilde{b}}\frac{\log^2(nT)c_n}{\sqrt{T}}\geq\frac{\max\left(K_b\log(nT),K_e\log^2(nT)c_n\right)}{\sqrt{nT}}, \\
	\omega_i& =K_d\frac{\log^2(nT)}{\sqrt{T}}.
\end{align*}
Moreover, using these definitions and the statements of Lemma~\ref{cor:lambda_rate}, we obtain that
\begin{align*}
	\mathcal{L}(C_n^*,\alpha_n^*)^2& = O_P\left(\frac{\log^2(nT)}{T}+\frac{\log^4(nT)}{T}\cdot\frac{1}{n}\sum_{i=1}^n|S_i(C_n^*)|+\frac{\log^4(nT)}{T}c_n^2\right) \\
	& = O_P\left(\frac{\log^4(nT)}{T}\left(1+\frac{1}{n}\sum_{i=1}^n|S_i(C_n^*)|+c_n^2\right)\right) = O_P\Big(\frac{\log^4(nT)}{T}s_n\Big).
\end{align*}
Using the statement of Theorem \ref{thm:pre_estimate_consistency}, we conclude if $\phi_{\textrm{comp}}(S_1(C_n^*),...,S_n(C_n^*);L;\mathcal{H}_n)$ is uniformly bounded from below,
\begin{align*}
	\frac{1}{n}\left\|\check{C}_n-C_n^*\right\|_1 &= O_P\left(\frac{\log^2(nT)}{\sqrt{T}}s_n\right), \\
	\frac{1}{n}\left\|\check{\alpha}_n-\alpha_n^*\right\|_1& =O_P\left(\frac{\log^3(nT)}{\sqrt{T}}s_n\right), \\
	\left\|\check{\theta}_n-\theta_n^*\right\|_1& =O_P\Bigg(\frac{\log^2(nT)}{\sqrt{T}}\frac{s_n}{c_n}\Bigg).
\end{align*}
\end{proof}

\subsection{Proofs of Section \ref{subsec:results_debiasing}}
\label{subsec:proofs_debiasing}
Let,
$$J:=\begin{pmatrix}
1 &        &               & 0      & \cdots & 0  \\
& \ddots &               & \vdots & \ddots & \vdots \\
&        & 1             & 0      & \cdots & 0
\end{pmatrix}\in\IR^{(p+1)\times N}$$
denote the first $p+1$ rows of $I_N$ (the identity matrix). We continue by listing the derivatives of $\Psi_{n,i}(t;C_{i\bdot},\alpha_i,\theta)$ for $i,k,x,y\in\{1,...,n\}$
\begin{align*}
\partial_{\alpha_k}\Psi_{n,i}(t;C_{i\bdot},\alpha_i,\theta)&= \nu_0(X_{n,i}(t);\beta)\Ind(i=k), \\
\partial_{C_{xy}}\Psi_{n,i}(t;C_{i\bdot},\alpha_i,\theta)&= \int_0^{t-}g(t-r;\gamma)dN_{n,y}(r)\Ind(i=x), \\
\partial_{\beta}\Psi_{n,i}(t;C_{i\bdot},\alpha_i,\theta)&= \alpha_i\nu_0'(X_{n,i}(t);\beta), \\
\partial_{\gamma}\Psi_{n,i}(t;C_{i\bdot},\alpha_i,\theta)&= \sum_{j=1}^nC_{ij}\int_0^{t-}g'(t-r;\gamma)dN_{n,j}(r),
\end{align*}
where $\nu_0'$ and $g'$ denote the first derivatices of $\nu_0$ resp. $g$ with respect to $\beta$ resp. $\gamma$. Many second order derivatives equal zero. We list here only the non-zero derivatives for $i,k,x,y\in\{1,...,n\}$
\begin{align*}
\partial_{\beta}\partial_{\alpha_k}\Psi_{n,i}(t;C_{i\bdot},\alpha_i,\theta)&= \nu_0'(X_{n,i}(t);\beta)\Ind(i=k), \\
\partial_{\gamma}\partial_{C_{xy}}\Psi_{n,i}(t;C_{i\bdot},\alpha_i,\theta)&= \int_0^{t-}g'(t-r);\gamma)dN_{n,y(r)}\Ind(i=x), \\
\partial_{\beta}\partial_{\beta}\Psi_{n,i}(t;C_{i\bdot},\alpha_i,\theta)&= \alpha_0\nu_0''(X_{n,i}(t);\beta), \\
\partial_{\gamma}\partial_{\gamma}\Psi_{n,i}(t;C_{i\bdot},\alpha_i,\theta)&= \sum_{j=1}^nC_{ij}\int_0^{t-}g''(t-r;\gamma)dN_{n,j}(r),
\end{align*}
where $\nu_0''$ and $g''$ denote the second derivatives of $\nu_0$ resp. $g$ with respect to $\beta$ resp. $\gamma$.

We present now two technical lemmas, which we prove in Appendix~\ref{sup:results_debiasing}
.

\begin{lemma}
\label{lem:Rn}
Let (A1) - (A3), and (B1) hold. Let $(C_1,\alpha_1,\theta_1), (C_2,\alpha_2,\theta_2) \in {\cal H}_n$ be arbitrary vectors of (random) parameters. Let $R_n$ be an $N \times N$ matrix such that each row $R_{n,a\bdot}$ of $R_n$ equals a row of $\Sigma_n(C,\alpha,\theta)$ for some $(C,\alpha,\theta) \in {\cal H}_n$ that can vary with $a$ and satisfies $(C_1,\alpha_1,\theta_1) \le (C,\alpha,\theta) \le (C_2,\alpha_2,\theta_2)$. Then, if $\frac{\log (nT)}{n} = o(1)$, we have
\begin{align*}
	&\frac{\max_{a,b}\left|R_{n,ab}-\Sigma_{n,ab}(C_1,\alpha_1,\theta_1)\right|}{\frac{1}{n}\|\alpha_1-\alpha_2\|_1+\frac{1}{n}\|C_1-C_2\|_1+\left(\frac{1}{n}\sum_{i=1}^n\|C_{1,i\bdot}\|_1+1\right)\|\theta_1-\theta_2\|_1} \\
	&\hspace*{5cm}= O_P\left(\max_{i=1,...,n}(\|C_{1,i\bdot}\|_1,\|C_{2,i\bdot}\|_1,1)\log^2(nT)\right).
\end{align*}
\end{lemma}

\begin{lemma}
\label{lem:Dassump}
Let (A1), (A2) and (A3) hold, and suppose that $\log(nT)/\sqrt{nT}\to0$. Then (with $\vee$ denoting maximum),
$$\left\|\frac{1}{\sqrt{nT}}\sum_{i=1}^n \int_0^T\begin{pmatrix}
	\partial_{\theta} \\ \partial_{\alpha} \\ \partial_C
\end{pmatrix}\Psi_{n,i}(t;C_{n,i\cdot}^*,\alpha_{n,i}^*,\theta_n^*)dM_{n,i}(t)\right\|_{\infty}=O_P\left(c_n\log nT\right).$$
\end{lemma}

\begin{proof}[Proof of Theorem \ref{thm:de-biasing}]
We use the KKT conditions \citep[e.g. see][]{B95}, to characterize $(\check{C}_n,\check{\alpha}_n,\check{\theta}_n)$. More precisely, define the Lagrange function
\begin{align*}
	&L:\mathcal{H}_n\times[0,\infty)^{n^2+n}\to\IR, \\
	&((C,\alpha,\theta),(\mu_C,\mu_{\alpha}))\mapsto\frac{1}{n}\sum_{i=1}^n\left(\frac{1}{T}\textrm{LS}_i(C_{i\bdot},\alpha_i,\theta)+2\omega_i\|C_{i\bdot}\|_1\right) \\
	&\qquad\qquad+\sum_{i,j=1}^n\mu_{C,ij}(-C_{ij})+\sum_{i=1}^n\mu_{\alpha,i}(-\alpha_i).
\end{align*}
Since $\Theta$ is an open set and since the inequality constraint in $\mathcal{H}_n$ is a strict inequality, only the non-negativity constraints are potentially binding. Moreover, the KKT conditions imply the existence of  $(\mu_C,\mu_{\alpha})$ with
$$L((C,\alpha,\theta),(\mu_C,\mu_{\alpha}))\geq L((\check{C}_n,\check{\alpha}_n,\check{\theta}_n),(\mu_C,\mu_{\alpha}))$$
for all $(C,\alpha,\theta)$ which fulfill all constraints of $\mathcal{H}_n$ other than potentially the non-negativity. Since this is an open set, we may conclude that
\begin{align}
	0&= \frac{1}{nT}\sum_{i=1}^n\begin{pmatrix}
		\partial_{\theta} \\ \partial_{\alpha} \\ \partial_C
	\end{pmatrix}\textrm{LS}_i(\check{C}_{n,i\bdot},\check{\alpha}_{n,i},\check{\theta}_n)+\frac{1}{n}\begin{pmatrix}
		0_{p+1} \\ 0_n \\ 2\omega_{1}1_n \\ \vdots \\ 2\omega_{n}1_n
	\end{pmatrix}+\begin{pmatrix}
		0_{p+1} \\ -\mu_{\alpha} \\ -\mu_C
	\end{pmatrix}, \label{eq:debias1}
\end{align}
where $0_p$ and $1_p$ are vectors of zeros or ones of size $p$, and $\mu_C\in\IR^{n^2}$ is the vector
$$(\mu_{C,11},...,\mu_{C,1n},\mu_{C,21},...,\mu_{2n},...,\mu_{C,n1},...,\mu_{C,nn}).$$
Moreover, $\mu_{C,ij}\check{C}_{n,ij}=0$ and $\mu_{\alpha,i}\check{\alpha}_{n,i}=0$ for all $i,j=1,...,n$. We may rearrange \eqref{eq:debias1} to obtain
\begin{equation}
	\begin{pmatrix}
		0_{p+1} \\ -\mu_{\alpha} \\ \frac{2\omega_{1}}{n}1_n-\mu_{C,1\bdot} \\ \vdots \\ \frac{2\omega_{n}}{n}1_n-\mu_{C,n\bdot}
	\end{pmatrix}=-\frac{1}{nT}\sum_{i=1}^n\begin{pmatrix}
		\partial_{\theta} \\ \partial_{\alpha} \\ \partial_C
	\end{pmatrix}\textrm{LS}_i(\check{C}_{n,i\bdot},\check{\alpha}_{n,i},\check{\theta}_n). \label{eq:debias2}
\end{equation}
Recall that $(C_n^*,\alpha_n^*,\theta_n^*)$ denotes the true parameters. We then have using a Taylor expansion,
\begin{align*}
	&\frac{1}{nT}\sum_{i=1}^n\begin{pmatrix}
		\partial_{\theta} \\ \partial_{\alpha} \\ \partial_C
	\end{pmatrix}\textrm{LS}_i(\check{C}_{n,i\bdot},\check{\alpha}_{n,i},\check{\theta}_n) \\
	&= \underbrace{\frac{1}{nT}\sum_{i=1}^n\begin{pmatrix}
			\partial_{\theta} \\ \partial_{\alpha} \\ \partial_C
		\end{pmatrix}\textrm{LS}_i(C_{n,i\bdot}^*,\alpha_{n,i}^*,\theta_n^*)}_{=:W_n}+\underbrace{\frac{1}{nT}\sum_{i=1}^n\begin{pmatrix}
			\partial_{\theta} \\ \partial_{\alpha} \\ \partial_C
		\end{pmatrix}\begin{pmatrix}
			\partial_{\theta} \\ \partial_{\alpha} \\ \partial_C
		\end{pmatrix}^\top\textrm{LS}_i(\check{C}_{n,i\bdot},\check{\alpha}_{n,i},\check{\theta}_n)}_{=\, \Sigma_n(\check{C}_n,\check{\alpha}_n,\check{\theta}_n)}\cdot\begin{pmatrix}
		\check{\theta}_n-\theta_n^* \\  \check{\alpha}_n-\alpha_n^* \\ \check{C}_n-C_n^*
	\end{pmatrix} \\
	&\qquad+\left(R_n-\Sigma_n(\check{C}_n,\check{\alpha}_n,\check{\theta}_n)\right)\begin{pmatrix}
		\check{\theta}_n-\theta_n^* \\  \check{\alpha}_n-\alpha_n^* \\ \check{C}_n-C_n^*
	\end{pmatrix},
\end{align*}
where $R_n$ is a matrix obtained from $\Sigma_n$ by replacing $\check{C}_n,\check{\alpha}_n,\check{\theta}_n$ in each row by a different intermediate point. Putting the above in \eqref{eq:debias1} and multiplying with a (for now) arbitrary matrix $\Lambda_n$ yields
\begin{align*}
	0&= -\Lambda_nW_n-\Lambda_n\Sigma_n(\check{C}_n,\check{\alpha}_n,\check{\theta}_n)\cdot\begin{pmatrix}
		\check{\theta}_n-\theta_n^* \\  \check{\alpha}_n-\alpha_n^* \\ \check{C}_n-C_n^*
	\end{pmatrix}-\Lambda_n\left(R_n-\Sigma_n(\check{C}_n,\check{\alpha}_n,\check{\theta}_n)\right)\begin{pmatrix}
		\check{\theta}_n-\theta_n^* \\  \check{\alpha}_n-\alpha_n^* \\ \check{C}_n-C_n^*
	\end{pmatrix} \\
	&\qquad-\Lambda_n\begin{pmatrix}
		0_{p+1} \\ -\mu_{\alpha} \\ \frac{2\omega_{1}}{n}1_n-\mu_{C,1\bdot} \\ \vdots \\ \frac{2\omega_{n}}{n}1_n-\mu_{C,n\bdot}
	\end{pmatrix}.
\end{align*}
Adding $(\check{\theta}_n-\theta_n^*,\check{\alpha}_n-\alpha_n^*,\check{C}_n-C_n^*)^\top$ on both sides, using \eqref{eq:debias2} and the definition of the debiased Lasso estimator $\overline \theta_n$ given in (\ref{eq:debiased_lasso}) and multiplying by $\sqrt{nT}$, we obtain 
\begin{align}
	\sqrt{nT}&\left(\overline{\theta}_n-\theta_n^*\right)\nonumber\\
	&=\sqrt{nT}\left(\begin{pmatrix}
		\check{\theta}_n \\  \check{\alpha}_n \\ \check{C}_n
	\end{pmatrix}-{nT}\sum_{i=1}^n\Lambda_n\begin{pmatrix}
		\partial_{\theta} \\ \partial_{\alpha} \\ \partial_C
	\end{pmatrix}\textrm{LS}_i(\check{C}_{n,i\bdot},\check{\alpha}_{n,i},\check{\theta}_n)-\begin{pmatrix}
		\theta_n^* \\  \alpha_n^* \\ C_n^*
	\end{pmatrix}\right) \nonumber \\
	&= -\sqrt{nT}\Lambda_nW_n+\sqrt{nT}\left(I_N-\Lambda_n\Sigma_n(\check{C}_n,\check{\alpha}_n,\check{\theta}_n)\right)\cdot\begin{pmatrix}
		\check{\theta}_n-\theta_n^* \\  \check{\alpha}_n-\alpha_n^* \\ \check{C}_n-C_n^*
	\end{pmatrix} \nonumber \\[4pt]
	&\qquad\qquad\qquad-\sqrt{nT}\Lambda_n\left(R_n-\Sigma_n(\check{C}_n,\check{\alpha}_n,\check{\theta}_n)\right)\begin{pmatrix}
		\check{\theta}_n-\theta_n^* \\  \check{\alpha}_n-\alpha_n^* \\ \check{C}_n-C_n^*
	\end{pmatrix}. \label{eq:debiasing_expansion}
\end{align}
We firstly argue that the last two terms converge to zero. We begin with the second term. Note that each row of $R_n$ contains the remainder term in a first-order Taylor expansion. Therefore, for each row $j$ of $R_n$, there are $\tilde{C}_n^{(j)}$, $\tilde{\alpha}_n^{(j)}$, $\tilde{\theta}_n^{(j)}$ which lie on the connecting lines between $C_n^*$ and $\check{C}_n$, $\alpha_n^*$ and $\check{\alpha}_n$, $\theta_n^*$ and $\check{\theta}_n$, respectively, such that
	$$R_{n,j\cdot}=\Sigma_{n,j\cdot}\left(\tilde{C}^{(j)}_n,\tilde{\alpha}^{(j)}_n,\tilde{\theta}^{(j)}_n\right).$$
	Then, $R_n$ is as in Lemma \ref{lem:Rn}. Therefore, Lemma \ref{lem:Rn} implies that
	\begin{align*}
		&\left\|\left(R_n-\Sigma_n\left(\check{C}_n,\check{\alpha}_n,\check{\theta}_n\right)\right)\begin{pmatrix}
			\check{\theta}_n-\theta_n^* \\ \check{\alpha}_n-\alpha_n^* \\ \check{C}_n-C_n^*
		\end{pmatrix}\right\|_{\infty}\leq\max_{a,b}\left|R_{n,ab}-\Sigma_{n,ab}\left(\check{C}_n,\check{\alpha}_n,\check{\theta}_n\right)\right|\cdot\left\|\begin{pmatrix}
			\check{\theta}_n-\theta_n^* \\ \check{\alpha}_n-\alpha_n^* \\ \check{C}_n-C_n^*
		\end{pmatrix}\right\|_1 \\ \\
		&\leq O_P(n\log^2(nT))\max_{i=1,...,n}\left(\|C_{n,i\bdot}^*\|_1,\|\check{C}_{n,i\bdot}\|_1,1\right)\left(\frac{1}{n}\sum_{i=1}^n\|C_{n,i\bdot}^*\|_1+1\right) \\
		&\hspace*{6cm}\times\left(\frac{1}{n}\|C_n^*-\check{C}_n\|_1+\frac{1}{n}\|\alpha_n^*-\check{\alpha}_n\|_1+\|\theta_n^*-\check{\theta}_n\|_1\right)^2 \\
		&= O_P\left(n\log^2(nT)\check{S}_nr_n^2s_n^2\right),
	\end{align*}
	where $\check{S}_n$ and $r_n$ are defined as in Section \ref{subsec:debiasing}. We hence get
	\begin{align}
		&\left\|\sqrt{nT}\Lambda_{\theta,n}\left(R_n-\Sigma_n(\check{C}_n,\check{\alpha}_n,\check{\theta}_n)\right)\begin{pmatrix}
			\check{\theta}_n-\theta_n^* \\ \check{\alpha}_n-\alpha_n^* \\ \check{C}_n-C_n^*
		\end{pmatrix}\right\|_{\infty} \nonumber \\
		&\leq \sqrt{nT}\left\|\Lambda_{\theta,n}\right\|_{\infty}\left\|\left(R_n-\Sigma_n(\check{C}_n,\check{\alpha}_n,\check{\theta}_n)\right)\begin{pmatrix}
			\check{\theta}_n-\theta_n^* \\ \check{\alpha}_n-\alpha_n^* \\ \check{C}_n-C_n^*
		\end{pmatrix}\right\|_{\infty} \nonumber \\
		&= O_P\left(n^{\frac{3}{2}}\sqrt{T}\log^2(nT)\check{S}_n\left\|\Lambda_{\theta,n}\right\|_{\infty}r_n^2s_n^2\right)=o_P(1), \label{eq:term2}
	\end{align}
	where the last step is using Assumption (B2). Furthermore, by \eqref{eq:node_wise_lasso}
	\begin{align}
		&\left\|\sqrt{nT}(J-\Lambda_{\theta,n}\Sigma_n(\check{C}_n,\check{\alpha}_n,\check{\theta}_n))\begin{pmatrix}
			\check{\theta}_n-\theta_n^* \\ \check{\alpha}_n-\alpha_n^* \\ \check{C}_n-C_n^*
		\end{pmatrix}\right\|_{\infty} \nonumber \\
		&\leq \sqrt{nT}\underset{b=1,...,p+1+n+n^2}{\max_{a=1,...,p+1}}\left|J_{a,b}-\left[\Lambda_{\theta,n}\Sigma_n(\check{C}_n,\check{\alpha}_n,\check{\theta}_n)\right]_{ab}\right|\left\|\begin{pmatrix}
			\check{\theta}_n-\theta_n^* \\ \check{\alpha}_n-\alpha_n^* \\ \check{C}_n-C_n^*
		\end{pmatrix}\right\|_1 \nonumber \\
		&= O_P\left(n^{\frac{3}{2}}\sqrt{T}r_ns_n\right)\max_{j=1,...,p+1}\frac{\sigma_j}{\tau_j}=o_P(1), \label{eq:term1}
	\end{align}
	with the last equality again using Assumption (B2). Plugging \eqref{eq:term1} and \eqref{eq:term2} in \eqref{eq:debiasing_expansion} yields
	\begin{align}
		\sqrt{nT}\left(\overline{\theta}_n-\theta_n^*\right)&= -\sqrt{nT}\Lambda_{\theta,n}W_n+o_P(1). \label{eq:rootn_expansion}
	\end{align}
	We hence need to study the asymptotic behaviour of $-\sqrt{nT}\Lambda_{\theta,n}W_n$. Recall that $M_{n,i}$ are the counting process martingales. We have by definition of $W_n$ and $M_{n,i}$
	\begin{align}
		-\sqrt{nT}\Lambda_{\theta,n}W_n&= -\Lambda_{\theta,n}\frac{1}{\sqrt{nT}}\sum_{i=1}^n\begin{pmatrix}
			\partial_{\theta} \\ \partial_{\alpha} \\ \partial_C
		\end{pmatrix}\textrm{LS}_i(C_{n,i}^*,\alpha_{n,i}^*,\theta_n^*) \nonumber \\
		&= -\Lambda_{\theta,n}\frac{1}{\sqrt{nT}}\sum_{i=1}^n\Bigg(\int_0^T2\begin{pmatrix}
			\partial_{\theta} \\ \partial_{\alpha} \\ \partial_C
		\end{pmatrix}\Psi_{n,i}(t;C_{n,i\cdot}^*,\alpha_{n,i}^*,\theta_n^*)\Psi_{n,i}(t;C_{n,i\cdot}^*,\alpha_{n,i}^*,\theta_n^*)dt \nonumber \\
		&\qquad-2\int\begin{pmatrix}
			\partial_{\theta} \\ \partial_{\alpha} \\ \partial_C
		\end{pmatrix}\Psi_{n,i}(t;C_{n,i\cdot}^*,\alpha_{n,i}^*,\theta_n^*)dN_{n,i}(t)\Bigg) \nonumber \\
		&= \Lambda_{\theta,n}\frac{1}{\sqrt{nT}}\sum_{i=1}^n \int_0^T2\begin{pmatrix}
			\partial_{\theta} \\ \partial_{\alpha} \\ \partial_C
		\end{pmatrix}\Psi_{n,i}(t;C_{n,i\cdot}^*,\alpha_{n,i}^*,\theta_n^*)dM_{n,i}(t). \label{eq:theta}
	\end{align}
	We then have
	\begin{align*}
		&\left\|\left(\Lambda_{\theta,n}-\Lambda_{0,\theta,n}\right)\frac{1}{\sqrt{nT}}\sum_{i=1}^n \int_0^T2\begin{pmatrix}
			\partial_{\theta} \\ \partial_{\alpha} \\ \partial_C
		\end{pmatrix}\Psi_{n,i}(t;C_{n,i\cdot}^*,\alpha_{n,i}^*,\theta_n^*)dM_{n,i}(t)\right\|_{\infty} \\
		&\leq \left\|\Lambda_{\theta,n}-\Lambda_{0,\theta,n}\right\|_{\infty}\left\|\frac{1}{\sqrt{nT}}\sum_{i=1}^n \int_0^T2\begin{pmatrix}
			\partial_{\theta} \\ \partial_{\alpha} \\ \partial_C
		\end{pmatrix}\Psi_{n,i}(t;C_{n,i\cdot}^*,\alpha_{n,i}^*,\theta_n^*)dM_{n,i}(t)\right\|_{\infty} \\
		&= O_P\left(\left\|\Lambda_{\theta,n}-\Lambda_{0,\theta,n}\right\|_{\infty}c_n\log nT\right) \\
		&= o_P(1),
	\end{align*}
	by Assumption (B3) and Lemma \ref{lem:Dassump}. Using the above in \eqref{eq:theta} shows that
	\begin{align*}
		-\sqrt{nT}\Lambda_{\theta,n}W_n&= \frac{1}{\sqrt{nT}}\sum_{i=1}^n \int_0^T2\Lambda_{0,\theta,n}\begin{pmatrix}
			\partial_{\theta} \\ \partial_{\alpha} \\ \partial_C
		\end{pmatrix}\Psi_{n,i}(t;C_{n,i\cdot}^*,\alpha_{n,i}^*,\theta_n^*)dM_{n,i}(t)+o_P(1).
	\end{align*}
	Note that $\Lambda_{0,\theta,n}V_n\Lambda_{0,\theta,n}^T$ is positive definite by Assumption (B3) and hence the matrix
	$$(\Lambda_{0,\theta,n}V_n\Lambda_{0,\theta,n}^T)^{-\frac{1}{2}}$$
	is well defined. Using the previous display in \eqref{eq:rootn_expansion}, we obtain
	\begin{align*}
		&\sqrt{nT}\left(\Lambda_{0,\theta,n}V_n\Lambda_{0,\theta,n}^T\right)^{-\frac{1}{2}}\left(\overline{\theta}_n-\theta_n^*\right) \\
		&= \frac{2}{\sqrt{nT}}\sum_{i=1}^n \int_0^T\left(\Lambda_{0,\theta,n}V_n\Lambda_{0,\theta,n}^T\right)^{-\frac{1}{2}}\Lambda_{0,\theta,n}S_{n,i}(t)dM_{n,i}(t)+o_P(1)
	\end{align*}
	again by Assumption (B3), where we recall that
	$$S_{n,i}(t):=\begin{pmatrix}
		\partial_{\theta} \\ \partial_{\alpha} \\ \partial_C
	\end{pmatrix}\Psi_{n,i}(t;C_{n,i\cdot}^*,\alpha_{n,i}^*,\theta_n^*).$$
	Let $a\in\IR^p$ be arbitrary. We show now that
	\begin{equation}
		\label{eq:asymp_norm1}
		\frac{1}{\sqrt{nT}}\sum_{i=1}^n\int_0^Ta^T\left(\Lambda_{0,\theta,n}V_n\Lambda_{0,\theta,n}^T\right)^{-\frac{1}{2}}\Lambda_{0,\theta,n}S_{n,i}(t)dM_{n,i}(t)\overset{d}{\to}\mathcal{N}(0,\|a\|_2^2).
	\end{equation}
	The main tool for the proof is Rebolledo's Martingale Central Limit Theorem (Theorem II.5.1 in \cite{ABGK93}). For an application of this theorem note first that $$
	s \to \mathcal{M}^{n,T} (s) = \int_0^{sT}\frac{1}{\sqrt{nT}}\sum_{i=1}^na^T\left(\Lambda_{0,\theta,n}V_n\Lambda_{0,\theta,n}^T\right)^{-\frac{1}{2}}\Lambda_{0,\theta,n}S_{n,i}(t)dM_{n,i}(t) $$ is a local square integrable martingale. We apply Rebolledo's Martingale Central Limit Theorem for one single point $\mathcal {T}_0=\{1\}$ and verify (2.5.1) and (2.5.3) in Theorem II.5.1 in \cite{ABGK93} for this choice of $\mathcal {T}_0$. This implies \eqref{eq:asymp_norm1}. For a check of (2.5.1) we consider
	the quadratic variation process:
	\begin{align*}
		&\left\langle\frac{1}{\sqrt{nT}}\sum_{i=1}^n\int_0^Ta^T\left(\Lambda_{0,\theta,n}V_n\Lambda_{0,\theta,n}^T\right)^{-\frac{1}{2}}\Lambda_{0,\theta,n}S_{n,i}(t)dM_{n,i}(t)\right\rangle \\
		&= \frac{1}{nT}\sum_{i=1}^n\int_0^T\left(a^T\left(\Lambda_{0,\theta,n}V_n\Lambda_{0,\theta,n}^T\right)^{-\frac{1}{2}}\Lambda_{0,\theta,n}S_{n,i}(t)\right)^2\Psi_{n,i}(t;C_{n,i\cdot}^*,\alpha_{n,i}^*,\theta_n^*)dt \\
		&= a^T\left(\Lambda_{0,\theta,n}V_n\Lambda_{0,\theta,n}^T\right)^{-\frac{1}{2}} \\
		&\hspace*{3cm}\times\Lambda_{0,\theta,n}\frac{1}{nT}\sum_{i=1}^n\int_0^TS_{n,i}(t)S_{n,i}(t)^T\Psi_{n,i}(t;C_{n,i\cdot}^*,\alpha_{n,i}^*,\theta_n^*)dt\Lambda_{0,\theta,n}^T \\
		&\qquad\qquad\times\left(\Lambda_{0,\theta,n}V_n\Lambda_{0,\theta,n}^T\right)^{-\frac{1}{2}}a \\
		\overset{\IP}{\to}&a^TM_0^{-\frac{1}{2}}M_0M_0^{-\frac{1}{2}}a=\|a\|_2^2,
	\end{align*}
	by Assumptions (B3) and (B4). Let now $\epsilon>0$ be arbitrary. The martingale
	\begin{align*}
		&\frac{1}{\sqrt{nT}}\sum_{i=1}^n\int_0^Ta^T\left(\Lambda_{0,\theta,n}V_n\Lambda_{0,\theta,n}^T\right)^{-\frac{1}{2}}\Lambda_{0,\theta,n}S_{n,i}(t) \\
		&\qquad\qquad\times\Ind\left(\left|\frac{1}{\sqrt{nT}}a^T\left(\Lambda_{0,\theta,n}V_n\Lambda_{0,\theta,n}^T\right)^{-\frac{1}{2}}\Lambda_{0,\theta,n}S_{n,i}(t)\right|>\epsilon\right)dM_{n,i}(t)
	\end{align*}
	contains all jumps of size larger than $\epsilon$. For a check of (2.5.3) in Theorem II.5.1 in \cite{ABGK93} we have to prove  that its quadratic variation converges to zero:
	\begin{align*}
		&\Bigg\langle\frac{1}{\sqrt{nT}}\sum_{i=1}^n\int_0^Ta^T\left(\Lambda_{0,\theta,n}V_n\Lambda_{0,\theta,n}^T\right)^{-\frac{1}{2}}\Lambda_{0,\theta,n}S_{n,i}(t) \\
		&\qquad\times\Ind\left(\left|\frac{1}{\sqrt{nT}}a^T\left(\Lambda_{0,\theta,n}V_n\Lambda_{0,\theta,n}^T\right)^{-\frac{1}{2}}\Lambda_{0,\theta,n}S_{n,i}(t)\right|>\epsilon\right)dM_{n,i}(t)\Bigg\rangle \\
		&= \frac{1}{nT}\sum_{i=1}^n\int_0^T\left(a^T\left(\Lambda_{0,\theta,n}V_n\Lambda_{0,\theta,n}^T\right)^{-\frac{1}{2}}\Lambda_{0,\theta,n}S_{n,i}(t)\right)^2 \\
		&\qquad\times\Ind\left(\left|\frac{1}{\sqrt{nT}}a^T\left(\Lambda_{0,\theta,n}V_n\Lambda_{0,\theta,n}^T\right)^{-\frac{1}{2}}\Lambda_{0,\theta,n}S_{n,i}(t)\right|>\epsilon\right)\Psi_{n,i}(t;C_{n,i\cdot}^*,\alpha_{n,i}^*,\theta_n^*)dt \\
		&\leq \frac{1}{\epsilon(nT)^{\frac{3}{2}}}\sum_{i=1}^n\int_0^T\left|a^T\left(\Lambda_{0,\theta,n}V_n\Lambda_{0,\theta,n}^T\right)^{-\frac{1}{2}}\Lambda_{0,\theta,n}S_{n,i}(t)\right|^3\Psi_{n,i}(t;C_{n,i\cdot}^*,\alpha_{n,i}^*,\theta_n^*)dt \\
		&= o_P(1),
	\end{align*}
	by Assumption (B4). This proves \eqref{eq:asymp_norm1} by Rebolledo's Martingale Central Limit Theorem (see above). This, in turn, implies the statement because $a\in\IR^p$ was arbitrary.
\end{proof}

\subsection{Proofs of Section \ref{subsec:consistency}}
\label{subsec:proofs_consistency}
\begin{proof}[Proof of Lemma \ref{lem:od}]
	Let $i$ be fixed. By continuity, the minimum in the definition of the oracle is attained on $\overline{\mathcal{H}}_n(\theta)\subseteq\overline{\mathcal{H}}_n$. Choose an arbitrary minimizer and denote it by $C^{(i)},\alpha^{(i)}$. Note that $C^{(i)}_{j\bdot}$ and $\alpha^{(i)}_j$ for $j\neq i$ are irrelevant when it comes to computation of the objective function. This holds for all $i$. Thus, by selecting the $i$-th row of $C^{(i)}$ and the $i$-th entry of $\alpha^{(i)}$, we can merge all these minimizers into a single $(C_n^*(\theta),\alpha_n^*(\theta))$ at which the minimum is attained for all $i$ simultaneously.
\end{proof}

Before proving Corollary \ref{cor:ind_conv_rate}, we present a technical lemma, which we prove in Appendix~\ref{sup:consistency}
.
\begin{lemma}
	\label{lem:diff}
	Let Assumptions (A1), (A2) and (A3) hold. We have that
	\begin{align*}
		&\left\|\begin{pmatrix}
			\alpha_{n,i}^*(\overline{\theta}_n)-\alpha_{n,i}^*(\theta_n^*) \\ C_{n,i\bdot}^*(\overline{\theta}_n)-C_{n,i\bdot}^*(\theta_n^*)
		\end{pmatrix}\right\|_1 \\
		&= O_P(1)\cdot\left(\frac{1+|S_i(C_n^*)|}{\phi_{i,\textrm{comp}}(L;\theta_n^*)^2}\log^2(nT)\cdot\left\|\overline{\theta}_n-\theta_n^*\right\|_2\left(c_n^2+\|C_{n,i\bdot}^*(\overline{\theta}_n)\|_1^2\right)\right).
	\end{align*}
\end{lemma}

\begin{proof}[Proof of Corollary \ref{cor:ind_conv_rate}]
	We argue next that we may restrict to the event that both of the following are true:
	\begin{enumerate}
		\item The statement of Theorem \ref{thm:oracle} applies for a suitable $L>0$.
		\item $\inf_{i=1,...,n}\phi_{i,\textrm{comp}}(L;\overline{\theta}_n)\geq\phi_0$
	\end{enumerate}
	It is enough to argue that each individual statement holds with probability converging to 1.
	
	\underline{As for (1):}  Since $\omega_i=3\sqrt{\log(nT)}d_{n,i}$, we have by assumption that $a_n/\omega_i=O_P(1)$ uniformly and, hence, there is $L>0$ such that $\IP(L\geq a_n/\omega_i\textrm{ for all }i=1,...,n)\to1$: Thus, $L>0$ as required in Theorem \ref{thm:oracle} can be found with high probability. \\
	Let, in addition to $a_n$ and $d_{n,i}$, also $b_n, e_n$ be chosen as in Lemma \ref{lem:lambda_rate}. Since
	$$\mathcal{T}_n(a_n,b_n,d_n,e_n)\subseteq\mathcal{T}_n^{(i)}(a_n,d_{n,i}),$$
	we have $\IP(\bigcap_{i=1}^n\mathcal{T}_n^{(i)}(a_n,d_{n,i}))\to1$ by Lemma \ref{lem:lambda_rate}. \\
	On the intersection of these events, which has probability converging to 1, the statement of Theorem \ref{thm:oracle} applies. \\
	\underline{As for (2):} By assumption, $\IP\left(\bigcap_{i=1}^n\Omega_{\textrm{IRCC}}^{(i)}\left(L,\phi_0;B_r(\theta_n^*)\right)\right)\to1$. Since, by Theorem \ref{thm:de-biasing}, $\overline{\theta}_n-\theta_n^*\overset{\IP}{\to}0$, we have that $\overline{\theta}_n\in B_r(\theta_n^*)$ with probability converging to one for any fixed $r>0$. These together imply the statement.
	
	For proving the rate of convergence, we may thus restrict to the event that both statements are true. For $\theta=\overline{\theta}_n$, we conclude, hence, from Theorem \ref{thm:oracle} that
	\begin{equation}
		\label{eq:ind_bound}
		\frac{1}{T}\left\|\Psi_{n,i}(\cdot;\hat{C}_{n,i\bdot},\hat{\alpha}_{n,i},\overline{\theta}_n)-\lambda_{n,i}\right\|_T^2+2\omega_i\|\hat{C}_{n,i\bdot}-C^*_{n,i\bdot}(\overline{\theta}_n)\|_1+2a_n|\hat{\alpha}_{n,i}-\alpha^*_{n,i}(\overline{\theta}_n)|\leq2\epsilon^*_i(\overline{\theta}_n).
	\end{equation}
	Keeping in mind that $C_n^*=C_n^*(\theta_n^*)$, we obtain for the right hand side
	\begin{align}
		\epsilon_i^*(\overline{\theta}_n)&\leq\frac{4}{T}\left\|\Psi_{n,i}\left(\cdot;C_{n,i}^*(\overline{\theta}_n),\alpha_{n,i}^*(\overline{\theta}_n),\overline{\theta}_n\right)-\Psi_{n,i}\left(\cdot;C_{n,i}^*(\theta_n^*),\alpha_{n,i}^*(\theta_n^*),\theta_n^*\right)\right\|_T^2 \nonumber \\
		&\qquad\qquad\qquad\qquad\qquad\qquad\qquad\qquad\qquad\qquad+9\frac{4\omega_i^2|S_i(C_n^*)|+a_n^2}{\phi_0^2}. \label{eq:eps_over_bound}
	\end{align}
	We bound the first expression. Note that by Lemma \ref{lem:diff}, since $\phi_{i,\textrm{comp}}(L;\theta_n^*)\geq\phi_0>0$, and by Theorem \ref{thm:de-biasing} we have
	\begin{align}
		&\left\|\begin{pmatrix}
			\alpha_{n,i}^*(\overline{\theta}_n)-\alpha_{n,i}^*(\theta_n^*) \\ C_{n,i\bdot}^*(\overline{\theta}_n)-C_{n,i\bdot}^*(\theta_n^*)
		\end{pmatrix}\right\|_1 \nonumber \\
		&\hspace*{0.5cm}= O_P(1)\cdot\left(\frac{1+|S_i(C_n^*)|}{\phi_{i,\textrm{comp}}(L;\theta_n^*)^2}\log^2(nT)\cdot\left\|\overline{\theta}_n-\theta_n^*\right\|_2\left(c_n^2+\|C_{n,i\bdot}^*(\overline{\theta}_n)\|_1^2\right)\right) \nonumber \\
		&\hspace*{0.5cm}= O_P(1)\cdot\left(\frac{1+|S_i(C_n^*)|}{\sqrt{nT}}\log^2(nT)\left(c_n^2+\|C_{n,i\bdot}^*(\overline{\theta}_n)\|_1^2\right)\right). \label{eq:opt_bound}
	\end{align}
	To make use of this bound, let
	\begin{align*}
		v_i(t;\theta):&= \begin{pmatrix}
			\nu_0(X_{n,i}(t);\beta), & \int_0^{t-}g(t-r;\gamma)dN_{n,1}(r), & \dots &, \int_0^{t-}g(t-r;\gamma)dN_{n,n}(r)
		\end{pmatrix}^T.
	\end{align*}
	so that
	$$\Psi_{n,i}(t;c,a,\theta)=v_i(t;\theta)^T\binom{a}{c}.$$
	Moreover, on the event $\Omega_{\mathcal{N}}$ with $\mathcal{N}=\mathcal{N}_0\log (nT)$, where $\mathcal{N}_0>0$ is chosen such that $\IP(\Omega_{\mathcal{N}})\to1$ by Lemma \ref{lem:omega_lemma}, we can make use of the boundedness and differentiability assumptions from (A2) and (A3) to obtain the bounds
	\begin{equation}
		\label{eq:v_bound}
		\sup_{t\in[0,T]}\|v_i(t;\overline{\theta}_n)-v_i(t;\theta_n^*)\|_{\infty}\leq R_1\|\overline{\theta}_n-\theta_n^*\|\log(nT),\,\sup_{t\in[0,T]}\|v_i(t;\theta_n^*)\|_{\infty}\leq R_2\log(nT),
	\end{equation}
	where $R_1,R_2>0$ are suitable constants. Finally, since $(\alpha_n^*(\overline{\theta}_n),C_n^*(\overline{\theta}_n))\in\mathcal{H}_n(\overline{\theta}_n)$ and $\overline{\theta}_n\in\Theta$, we have that $\|C_{n,i\bdot}^*(\overline{\theta}_n)\|_1\leq|S_i(C_n^*)|K_C$. This insight together with \eqref{eq:opt_bound}, \eqref{eq:v_bound}, and Theorem \ref{thm:de-biasing} shows that
	\begin{align*}
		&\frac{1}{T}\left\|\Psi_{n,i}\left(\cdot;C_{n,i}^*(\overline{\theta}_n),\alpha_{n,i}^*(\overline{\theta}_n),\overline{\theta}_n\right)-\Psi_{n,i}\left(\cdot;C_{n,i}^*(\theta_n^*),\alpha_{n,i}^*(\theta_n^*),\theta_n^*\right)\right\|_T^2 \\
		&= \frac{1}{T}\left\|v_i(t;\overline{\theta}_n)^T\begin{pmatrix}
			\alpha_{n,i}^*(\overline{\theta}_n) \\ C_{n,i\bdot}^*(\overline{\theta}_n) \end{pmatrix}-v_i(t;\theta_n^*)^T\begin{pmatrix}
			\alpha_{n,i}^*(\theta_n^*) \\ C_{n,i\bdot}^*(\theta_n^*) \end{pmatrix}\right\|_T^2 \\
		&= \frac{1}{T}\left\|\left(v_i(t;\overline{\theta}_n)-v_i(t;\theta_n^*)\right)^T\begin{pmatrix}
			\alpha_{n,i}^*(\overline{\theta}_n) \\ C_{n,i\bdot}^*(\overline{\theta}_n) \end{pmatrix}-v_i(t;\theta_n^*)^T\begin{pmatrix}
			\alpha_{n,i}^*(\theta_n^*)-\alpha_{n,i}^*(\overline{\theta}_n) \\ C_{n,i\bdot}^*(\theta_n^*)-C_{n,i\bdot}^*(\overline{\theta}_n) \end{pmatrix}\right\|_T^2 \\
		&\leq \frac{2}{T}\left(\left\|\left(v_i(t;\overline{\theta}_n)-v_i(t;\theta_n^*)\right)^T\begin{pmatrix}
			\alpha_{n,i}^*(\overline{\theta}_n) \\ C_{n,i\bdot}^*(\overline{\theta}_n) \end{pmatrix}\right\|_T^2+\left\|v_i(t;\theta_n^*)^T\begin{pmatrix}
			\alpha_{n,i}^*(\theta_n^*)-\alpha_{n,i}^*(\overline{\theta}_n) \\ C_{n,i\bdot}^*(\theta_n^*)-C_{n,i\bdot}^*(\overline{\theta}_n) \end{pmatrix}\right\|_T^2\right) \\
		&\leq \frac{2}{T}\Bigg(T\sup_{t\in[0,T]}\|v_i(t;\overline{\theta}_n)-v_i(t;\theta_n^*)\|_{\infty}^2\left\|\begin{pmatrix}
			\alpha_{n,i}^*(\overline{\theta}_n) \\ C_{n,i\bdot}^*(\overline{\theta}_n) \end{pmatrix}\right\|_1^2 \\
		&\qquad\qquad\qquad+T\sup_{t\in[0,T]}\|v_i(t;\theta_n^*)\|_{\infty}^2\left\|\begin{pmatrix}
			\alpha_{n,i}^*(\theta_n^*)-\alpha_{n,i}^*(\overline{\theta}_n) \\ C_{n,i\bdot}^*(\theta_n^*)-C_{n,i\bdot}^*(\overline{\theta}_n) \end{pmatrix}\right\|_1^2\Bigg) \\
		&= O_P(1)\cdot\left(\frac{(1+|S_i(C_n^*)|)^2\log^6(nT)}{nT}(c_n^2+\|C_{n,i\bdot}^*\|_1^2)^2\right).
	\end{align*}
	Using the above and the bounds on $\omega_i$ an $a_n$ in \eqref{eq:eps_over_bound} yields
	\begin{align}
		&\epsilon_i^*(\overline{\theta}_n) \nonumber \\
		&= O_P(1)\cdot\left(\frac{(1+|S_i(C_n^*)|)^2\log^6(nT)}{nT}(c_n^2+\|C_{n,i\bdot}^*\|_1^2)^2\right)+O_P(1)\cdot\left((1+|S_i(C_n^*)|)\frac{\log^5(nT)}{T}\right) \nonumber \\
		& =O_P(1)\left((1+|S_i(C_n^*)|)\frac{\log^5(nT)}{T}\left(\frac{(1+|S_i(C_n^*)|)\log(nT)}{n}(c_n^2+\|C_{n,i\bdot}^*\|_1^2)^2+1\right)\right) \nonumber \\
		&=  O_P(1)\cdot\left((1+|S_i(C_n^*)|)\frac{\log^5(nT)}{T}\right), \label{eq:eps_star_bound}
	\end{align}
	where we used the growth condition \eqref{eq:growth_condition} in the last step. Recall that $C_n^*=C_n^*(\theta_n^*)$ and $\alpha_n*=\alpha_n^*(\theta_n^*)$. Then, using \eqref{eq:opt_bound} together with the assumptions on $a_n$ and $\omega_i=3\sqrt{\log(nT)}d_{n,i}$ and Condition~\eqref{eq:growth_condition}, we obtain
	\begin{align*}
		&2\omega_i\left\|\hat{C}_{n,i\bdot}-C_{n,i\bdot}^*\right\|_1+2a_n\left|\hat{\alpha}_{n,i}-\alpha_{n,i}^*\right| \\
		&\leq 2\omega_i\left\|\hat{C}_{n,i\bdot}-C_{n,i\bdot}^*(\overline{\theta}_n)\right\|_1+2\omega_i\left\|C_{n,i\bdot}^*(\overline{\theta}_n)-C_{n,i\bdot}^*(\theta_n^*)\right\|_1+2a_n\left|\hat{\alpha}_{n,i}-\alpha_{n,i}^*(\overline{\theta}_n)\right| \\
		&\qquad\qquad\qquad\qquad\qquad+2a_n\left|\alpha_{n,i}^*(\overline{\theta}_n)-\alpha_{n,i}^*(\theta_n^*)\right| \\
		&\leq 2\omega_i\left\|\hat{C}_{n,i\bdot}-C_{n,i\bdot}^*(\overline{\theta}_n)\right\|_1+2a_n\left|\hat{\alpha}_{n,i}-\alpha_{n,i}^*(\overline{\theta}_n)\right| \\
		&\qquad\qquad+O_P(1)\cdot\left(\frac{\max(\omega_i,a_n)(1+|S_i(C_n^*)|)\log^2(nT)(c_n^2+\|C_{n,i\bdot}^*(\overline{\theta}_n)\|_1^2)}{\sqrt{nT}}\right) \\
		&= 2\omega_i\left\|\hat{C}_{n,i\bdot}-C_{n,i\bdot}^*(\overline{\theta}_n)\right\|_1+2a_n\left|\hat{\alpha}_{n,i}-\alpha_{n,i}^*(\overline{\theta}_n)\right| \\
		&\qquad\qquad+O_P(1)\cdot\left(\frac{(1+|S_i(C_n^*)|)\log^{\frac{9}{2}}(nT)(c_n^2+\|C_{n,i\bdot}^*(\overline{\theta}_n)\|_1^2)}{\sqrt{n}T}\right) \\
		&= 2\omega_i\left\|\hat{C}_{n,i\bdot}-C_{n,i\bdot}^*(\overline{\theta}_n)\right\|_1+2a_n\left|\hat{\alpha}_{n,i}-\alpha_{n,i}^*(\overline{\theta}_n)\right|+O_P(1)\cdot\left(\frac{(1+|S_i(C_n^*)|)^{\frac{1}{2}}\log^4(nT)}{T}\right).
	\end{align*}
	Using the above estimate together with \eqref{eq:ind_bound} and \eqref{eq:eps_star_bound} yields finally
	\begin{align*}
		&\frac{1}{T}\left\|\Psi_{n,i}(\cdot;\hat{C}_{n,i\bdot},\hat{\alpha}_{n,i},\overline{\theta}_n)-\lambda_{n,i}\right\|_2^2+2\omega_i\|\hat{C}_{n,i\bdot}-C^*_{n,i\bdot}\|_1+2a_n|\hat{\alpha}_{n,i}-\alpha^*_{n,i}| \\
		&\leq \frac{1}{T}\left\|\Psi_{n,i}(\cdot;\hat{C}_{n,i\bdot},\hat{\alpha}_{n,i},\overline{\theta}_n)-\lambda_{n,i}\right\|_2^2+2\omega_i\|\hat{C}_{n,i\bdot}-C^*_{n,i\bdot}(\overline{\theta}_n)\|_1+2a_n|\hat{\alpha}_{n,i}-\alpha^*_{n,i}(\overline{\theta}_n)| \\
		&\qquad\qquad+O_P(1)\cdot\left(\frac{(1+|S_i(C_n^*)|)^{\frac{1}{2}}\log^4(nT)}{T}\right) \\
		&\leq 2\epsilon^*_i(\overline{\theta}_n)+O_P(1)\cdot\left(\frac{(1+|S_i(C_n^*)|)^{\frac{1}{2}}\log^4(nT)}{T}\right) \\
		&= O_P(1)\cdot\left((1+|S_i(C_n^*)|)\frac{\log^5(nT)}{T}\right)+O_P(1)\cdot\left(\frac{(1+|S_i(C_n^*)|)^{\frac{1}{2}}\log^4(nT)}{T}\right),
	\end{align*}
	which we wanted to prove. The stated convergence rates for $\hat{C}_{n,i\bdot}$ and $\hat{\alpha}_{n,i}$ are then direct consequences.
\end{proof}

\section{Some heuristics on the debiasing step}
\label{subsec:debias}

Here we give a heuristic discussion about why a debiasing step is needed to get an estimator of the global parameter $\theta=( \beta, \gamma)$ with convergence rate $O_P( (nT)^{-1/2})$ that allows for an asymptotic distribution theory for the estimator of $\theta$. To simplify the discussion we consider the closely related model 
\begin{align}
	& Y_{n,i,t} = \alpha^*_{n,i} \nu_0(X_{n,i,t}) + \sum_{j=1}^n C_{n,ij}^* g(\gamma_n^* ,Z_{n,ij,t}) + \varepsilon_{n,i,t} \qquad (i=1,...,n; t=1,...,T), \label{eq:heurmodel1}\end{align}
where $(X_{n,i,t},Z_{n,ij,t} : j=1,...,n)$ are some covariates, $g$ is a known function and
$\varepsilon_{n,i,t}$ are mean zero variables that are independent of $(X_{n,i,t},Z_{n,ij,t} : j=1,...,n)$. To simplify the discussion even further we omit the first term on the right hand side of \eqref{eq:heurmodel1} and we arrive at the model
\begin{align}
	& Y_{n,i,t} =  \sum_{j=1}^n C_{n,ij}^* g(\gamma_n^*, Z_{n,ij,t}) + \varepsilon_{n,i,t} \qquad (i=1,...,n; t=1,...,T).\label{eq:heurmodel2}\end{align}
In this model we have a global parameter $\gamma_n^*$ and local parameters $C_{n,ij}^*$. We now give arguments why in this model a debiasing step is needed to get an estimator of $\gamma_n^*$ with rate $O_P( (nT)^{-1/2})$.
The LASSO estimator of the parameters $C_{n,i,j}^*$ and $\gamma_n^*$ is given as the argmin $(\hat C_{n,ij}, \hat \gamma_n)$ over $(C_{ij}, \gamma)$ of 
$$\sum_{i=1}^n\sum_{t=1}^T \left (Y_{n,i,t} -  \sum_{j=1}^n C_{ij} g(\gamma, Z_{n,ij,t})\right )^2+ 2 \sum_{i=1}^n \omega_i \| C_{i,.}\|_1.$$
For a discussion of $\hat C_{n,i,j}$  we now assume that $\hat C_{n,i,j}$ is asymptotically equivalent to the oracle estimator $\tilde C_{n,i,j}$ which makes use of knowing the unknown parameter $\gamma_n^*$. More precisely,  $\gamma_n^*$ is defined as the minimizer of 
$$\sum_{i=1}^n\sum_{t=1}^T \left (Y_{n,i,t} -  \sum_{j=1}^n C_{ij} g(\gamma_n^*, Z_{n,ij,t})\right )^2+ 2 \sum_{i=1}^n \omega_i \| C_{i.}\|_1,$$ where $\gamma$ is replaced by the true value $\gamma_n^*$. The oracle estimator $\tilde C_{n,i,j}$ is a LASSO estimator in a high-dimensional linear model. We will make use below of the fact that  its asymptotic theory is well understood.

Furthermore, we assume that $\hat \gamma_n$ can be approximated by the minimizer $\tilde \gamma_n$ of $$\sum_{i=1}^n\sum_{t=1}^T \left (Y_{n,i,t} -  \sum_{j=1}^n \tilde C_{n,ij} g(\gamma, Z_{n,ij,t})\right )^2+ 2 \sum_{i=1}^n \omega_i \| \tilde C_{n,i\bdot}\|_1,$$ where $C_{ij}$ is replaced by the oracle estimators  $\tilde C_{n,ij}$. That bias and standard deviation are of the same order up to logarithmic terms is caused by optimal choices of penalty constants which balance these two terms.

We now consider the estimators $(\hat C_{n,ij}, \hat \gamma_n)$ assuming that the above approximations are valid and that for each $1 \leq i \leq n$ the parameters $C^*_{n,i\bdot}=(C^*_{n,ij}: j=1,...,n)$ are sparse with $\big|\{j:C^*_{n,ij} \not = 0\}\big|$ bounded by a fixed constant. Then, up to a logarithmic factor, the bias of $\tilde  C_{n,i\bdot} = (\tilde  C_{n,ij}: j=1,...,n)$ is of order $O(T^{-1/2})$. Furthermore the variance is of order $O(T^{-1}) $. 

Using expansions of the function $\gamma \to g(\gamma,z))$ about $\gamma = \gamma_n^*$ we obtain that, approximately,  $\tilde \gamma_n$ is given as the sum of functions of $\tilde C_{n,i\bdot}$. The summands are independent because 
the tuples $\tilde C_{n,i,.}$ are independent and for each $i$ given  as the minimizer of 
$$\sum_{t=1}^T \left (Y_{n,i,t} -  \sum_{j=1}^n C_{ij} g(\gamma_n^*, Z_{n,ij,t})\right )^2+ 2  \omega_i \| C_{i\bdot}\|_1.$$
We conclude that up to a logarithmic factor $\tilde \gamma_n$ has a bias of order $O(T^{-1/2})$ whereas the variance is of order $O((nT)^{-1}) $. Because the size of the bias is not of order $O((nT)^{-1/2}) $
the estimator $\tilde \gamma_n$ does not achieve the rate $O_P((nT)^{-1/2}) $. 

The situation changes if we replace $\tilde  C_{n,i\bdot}$ by bias corrected estimators as discussed e.g. in \citet{vdGBRD14}. Then the bias of the estimators is up to logarithmic terms of order $O(T^{-1})$ which results in a bias of order $O(T^{-1})$ for the estimator of $\gamma$ which is of order $o((nT)^{-1/2}) $ as long as $T/n \to \infty$. Our discussion on the sufficiency of this rate  heavily depends on the assumption that the approximations of the debiased estimator $\hat \gamma_n$ by $\tilde \gamma_n$ are valid for this rate of convergence. In the proof of Theorem \ref{thm:de-biasing} we use the stronger assumption $T/n^3 \to \infty$ which may be avoided at the cost of more complex calculations.

\section{Discussion of the random compatibility condition}
\label{appendic:RCC}
We argue in this section that (RCC) can be expected to hold in a slightly simplified setup. Specifically, we assume for now that the infimum in the definition of $\phi_{\text{comp}}^2$ is only taken over $C$ and that $\theta=\theta_n^*$ and $\alpha=\alpha_n^*$ are fixed to the true parameters. Moreover, we assume that $\alpha_n^*\geq\alpha_0>0$ is uniformly bounded from below. First, we rewrite
$$\frac{1}{nT}\left\|\Psi_n(\cdot;C,\alpha_n^*,\theta_n^*)-\lambda_n(\cdot)\right\|_T^2=\frac{1}{nT}\left\langle(C-C_n^*)g(\cdot;\gamma_n^*),(C-C_n^*)g(\cdot;\gamma_n^*)\right\rangle_T,$$
where, for matrix valued functions $A,B:[0,\infty)\to\IR^{n\times n}$,
$$\langle A,B\rangle_T:=\sum_{i=1}^n\int_0^T\sum_{j=1}^n\int_0^{t-}A_{ij}(t-r)dN_{n,j}(r)\sum_{j=1}^n\int_0^{t-}B_{ij}(t-r)dN_{n,j}(r)dt.$$
We apply next Proposition 3 of \cite{HRBR15}. For the convenience of the reader, we reformulate the above in the notation of \cite{HRBR15}. Let to this end,
$$Z_C(N_n):=\frac{1}{nT}\sum_{i=1}^n\left(\sum_{j=1}^n(C_{ij}-C_{n,ij}^*)\int_{-A}^{0-}g(-r;\gamma_n^*)dN_{n,j}(r)\right)^2,$$
and let $\mathcal{S}_t$ denote the shift-operator such that
$$\frac{1}{nT}\langle(C-C_n^*)g(\cdot;\gamma_n^*),(C-C_n^*)g(\cdot;\gamma_n^*)\rangle_T=\int_0^T\left(Z_C\circ\mathcal{S}_t\right)(N_n)dt.$$
It is now direct to see that
\begin{align}
	Z_C(N_n)\leq&\frac{c_n^2}{T}\|g(\cdot;\gamma_n^*)\|_{\infty}^2\max_{j=1,...,n}|N_{n,j}[-A,0)|^2\leq\frac{c_n^2}{T}\|g(\cdot;\gamma_n^*)\|_{\infty}^2(1+|N_n[-A,0)|^2). \label{eq:ZCbound}
\end{align}
In other words, the condition of Proposition 3 holds for $b=\frac{c_n^2}{T}\|g(\cdot;\gamma_n^*)\|_{\infty}^2$ and $\eta=2$. Furthermore, we have to study the quantity
$$\sigma_n^2:=\IE\left(\left(Z_C(N_n)-\IE(Z_C(N_n))\right)^2\Ind(|N_n[-A,0)|\leq C_3\log T)\right),$$
where $C_3>0$ will be chosen later. Noe that, by Equation~\eqref{eq:exp_bound_Hawes}, we have, for specific choices of $a,c_0>0$, that for all $x>0$ $\IP(|N_{n,i}[-A,0)|\geq x)\leq c_0\exp\left(-\frac{ax}{3}\right)$. Using this result and Equation~\eqref{eq:ZCbound}, we conclude that
\begin{align*}
	&\IE(Z_C(N_n)^2\Ind(|N_n[-A,0)|\leq C_3\log T))\leq\frac{c_n^4\log^4 T}{T^2}\|g(\cdot;\gamma_n^*)\|_{\infty}^4C_3^4 \\
	&\IE(Z_C(N_n))\leq\frac{c_n^2}{T}\|g(\cdot;\gamma_n^*)\|_{\infty}^2\IE\left(\max_{j=1,...,n}|N_{n,j}[-A,0)|^2\right) \\
	& =\frac{c_n^2}{T}\|g(\cdot;\gamma_n^*)\|_{\infty}^2\int_0^{\infty}\IP\left(\max_{j=1,...,n}|N_{n,j}[-A,0)|\geq \sqrt{t}\right)dt \\
	\leq&\frac{c_n^2}{T}\|g(\cdot;\gamma_n^*)\|_{\infty}^2\left(\frac{9}{a^2}\log^2n+n\int_{\frac{9}{a^2}\log^2n}^{\infty}c_0e^{-\frac{a\sqrt{t}}{3}}dt\right) \\
	& =\frac{c_n^2}{T}\|g(\cdot;\gamma_n^*)\|_{\infty}^2\left(\frac{9}{a^2}\log^2n-n\frac{18c_0}{a^2}e^{-\sqrt{x}}(\sqrt{x}+1)\big|_{\log^2n}^{\infty}\right) \\
	& =\frac{c_n^2}{T}\|g(\cdot;\gamma_n^*)\|_{\infty}^2\left(\frac{9}{a^2}\log^2n+n\frac{18c_0}{a^2}\frac{1}{n}(\log n+1)\right).
\end{align*}
The above implies that
\begin{align*}
	\sigma_n^2& \leq 2\left(\IE\left(Z_C(N_n)^2\Ind(|N_n[-A,0)|\leq C_3\log T)\right)+\IE(Z_C(N_n))^2\right) \\
	& \leq 2\frac{c_n^4}{T^2}\|g(\cdot;\gamma_n^*)\|_{\infty}^4\left(C_3^4\log^4 T+\left(\frac{9}{a^2}\log^2 n+\frac{18c_0}{a^2}(\log n+1\right)^2\right) \\
	& \approx \frac{c_n^4}{T^2}(\log^4T+\log^4n),
\end{align*}
where $\approx$ is hiding some constants. In particular, when assuming
$$\frac{c_n^2}{\sqrt{T}}\left(\log^2T+\log^2 n\right)\log^{\frac{3}{2}}T\to0,$$
we have that $\sigma_n\sqrt{T\log^3T}\to0$. Since also $b_n\log^4T\to0$, we may conclude from Proposition~3 of \cite{HRBR15} that 
$$\IP\left(\sup_{C\in[0,K_C)^{n\times n}}\left|\frac{1}{nT}\left\|\Psi_n(\cdot;C,\alpha_n^*,\theta_n^*)-\lambda_n(\cdot)\right\|_T^2-T\IE(Z_C(N_n))\right|>\epsilon\right)\to0,$$
for all $\epsilon>0$. In the notation of Section~5.1 in \citet{HRBR15}, we have that
\begin{align*}
	T\IE\left(Z_C(N_n)\right)=&\frac{1}{n}\sum_{i=1}^n\IE\left(\left(\sum_{j=1}^n(C_{ij}-C_{n,ij}^*)\int_{-A}^{0-}g(-r;\gamma_n^*)dN_{n,j}(r)\right)^2\right) \\
	=:&\frac{1}{n}\sum_{i=1}^nQ\left((C-C_n^*)_{i\bdot}g(\cdot;\gamma_n^*),(C-C_n^*)_{i\bdot}g(\cdot;\gamma_n^*)\right).
\end{align*}
Now, using Proposition~4 of \cite{HRBR15}, we get for some $\zeta>0$ that
\begin{align*}
	&T\IE(Z_C(N_n)) \\
	& =\frac{1}{n}\sum_{i=1}^nQ\left((C-C_n^*)_{i\bdot}g,(C-C_n^*)_{i\bdot}g\right) \\
	&\geq \zeta\frac{1}{n}\sum_{i=1}^n\|(C-C_n^*)_{i\bdot}\|_2^2\|g(\cdot;\gamma_n^*)\|_2^2\geq\zeta\|g(\cdot;\gamma_n^*)\|_2^2\frac{1}{n}\sum_{i=1}^n\|C_{iS_i(C_n^*)}-C_{n,iS_i(C_n^*)}^*\|_2^2.
\end{align*}
Thus, in this special case, we can prove that
$$\phi_{\text{comp}}(S_1(C_n^*),...,S_n(C_n^*);L;\mathcal{H}_n)\geq\frac{\zeta}{2}\|g(\cdot;\gamma_n*)\|_2^2>0$$
with probability converging to one.

\begin{remark}
	\label{rem:mot_ind_comp}
	The discussion above can be transferred directly to individual compatibility constant defined in Section~\ref{subsec:single_consistency}. First note that, in the special case discussed here, we obtain
	\begin{align*}
		&\frac{1}{T}\left\|\Psi_{n,i}(\cdot;C_i,\alpha_{n,i}^*,\theta_n^*)-\Psi_{n,i}(\cdot;C_{n,i\cdot}^*(\theta_n^*);\alpha_{n,i}^*(\theta_n)^*)\right\|_T^2 \\
		& =\frac{1}{T}\left\|\Psi_{n,i}(\cdot;C_i,\alpha_{n,i}^*,\theta_n^*)-\lambda_{n,i}(\cdot)\right\|_T^2 \\
		& =\frac{1}{T}\int_0^T\left(\sum_{j=1}^n\left(C_{ij}-C_{n,ij}^*\right)\int_0^{t-}g(t-r;\gamma_n^*)dN_j(r)\right)^2dt=\frac{1}{T}\int_0^T(Z_{i,C}\circ\mathcal{S}_t)(N_n)dt,
	\end{align*}
	where
	$$Z_{C,i}(N_n):=\frac{1}{T}\left(\sum_{j=1}^n\left(C_{ij}-C_{n,ij}^*\right)\int_{-A}^{0-}g(t-r;\gamma_n^*)dN_j(r)\right)^2.$$
	Hence, we may repeat the above arguments by replacing $Z_C(N_n)$ by $Z_{C,i}(N_n)$ to obtain that
	\begin{align*}
		&\frac{1}{T}\left\|\Psi_{n,i}(\cdot;C_i,\alpha_{n,i}^*,\theta_n^*)-\Psi_{n,i}(\cdot;C_{n,i\cdot}^*(\theta_n^*);\alpha_{n,i}^*(\theta_n)^*)\right\|_T^2\geq\frac{1}{2}T\IE(Z_{i,C}(N_n)) \\
		& =\frac{1}{2}Q((C-C_n^*)_{i\cdot}g(\cdot;\gamma_n^*),(C-C_n^*)_{i\cdot}g(\cdot;\gamma_n^*))
		\geq\frac{\zeta}{2}\|g(\cdot;\gamma_n^*)\|_2^2
	\end{align*}
	simultaneously for all $i=1,...,n$ with probability converging to one. This implies that
	$$\IP(\cap_{i=1}^n\Omega_{\text{IRCC}}(L,\phi_0;B_r(\theta_n*)))\to1$$
	with $\phi_0:=\sqrt{\frac{\zeta}{2}}\|g(\cdot;\gamma_n^*)\|_2$.
\end{remark}

\section{Proofs of Section \ref{subsec:bounds_Hawkes}}
\label{sup:HawkesBounds}
\begin{proof}[Proof of Lemma \ref{lem:goodN}]
	The proof takes many ideas from the proofs of Lemma 1 and Proposition 2 in \citet{HRBR15}. However, for our purposes, we need to refine some of the arguments in order to prove a stronger bound on the sum of the expectations. To this end, we need to keep track of the constants and use the sparsity assumption in order to obtain a uniform result which holds for all $n\in\IN$.  In the interest of completeness we give the full proof here. Denote for all $i,l=1,...,n$ by
	$$K_i^l(0):=W_i^l(0) \textrm{ and }K_i^l(k):=W_i^l(k)-W_i^l(k-1) \textrm{ for } k\geq1$$
	the number of events in the $k$-th generation in a cluster which has started from an initial event in the $l$-th process. Let $K^l(k):=(K_1^l(k),...,K_n^l(k))$ and define the logarithm of the Laplace transform of $K^l(1)$ by
	$$\phi_{n,l}(s):=\log\IE\left(e^{s^TK^l(1)}\right).$$
	Let finally $\phi_n:=(\phi_{n,1},...,\phi_{n,n})$. Let $a := \int_0^{\infty}g(t;\gamma)dt $. We note that $K^l_j(1)$ for $j=1,...,n$ equals the number of events of a counting process with intensity function $C_{n,jl}g(t;\gamma)$ and hence $K^l_j(1)$ is Poisson distributed with rate $C_{n,jl}$. Moreover, for any $l$, the variables $K_1^l(1),...,K_n^l(1)$ are independent. It is therefore direct to compute that
	\begin{equation}
		\label{eq:def_phi}
		\phi_{n,l}(s)=\sum_{j=1}^nC_{n,jl}\left(e^{s_j}-1\right).
	\end{equation}
	Then, we obtain the following estimate for all $s\in\IR^n$ with $\|s\|_1\leq r$ (recall the definition of $r$ and $\epsilon$, and note that in particular $|s_j|\leq r$ for all $j$)
	\begin{equation}
		\label{eq:phi_contracts}
		\|\phi_n(s)\|_1\leq \sum_{l,j=1}^na C_{n,jl}\left|e^{s_j}-1\right|\leq a_0\sum_{j=1}^n\left|e^{s_j}-1\right|\leq\epsilon\|s\|_1.
	\end{equation}
	Thus, all $\phi_n$ are contractions with the same constant $\epsilon$. Moreover, we prove next that for any $p\geq1$ and all $s\in\IR^n$ we have
	\begin{equation}
		\label{eq:rec_K}
		\IE\left(e^{s^TK^l(p)}\Big|K^l(p-1),...,K^l(0)\right)=e^{\phi_n(s)^TK^l(p-1)}.
	\end{equation}
	\begin{proof}[Proof of \eqref{eq:rec_K}]
		We bring the following intuition in formulas: \emph{The distribution of the number of new events in the $p$-th generation of the $i$-th process can be obtained by starting for each event in any of the processes $j=1,...,n$ in the $(p-1)$-th generation a new counting process with intensity $C_{n,ij}g(t;\gamma)$ and summing all their event numbers (which are iid copies of $K_i^j(1)$)}. To write this in formulas, let $K_{i,r}^j(1)$ be independent copies of $K_i^j(1)$. Then, the following equality holds in distribution
		$$K_i^l(p)=\sum_{j=1}^n\sum_{r=1}^{K_j^l(p-1)}K_{i,r}^j(1).$$
		From this we obtain
		\begin{align*}
			&\IE\left(e^{s^TK^l(p)}\Big|K^l(p-1),...,K^l(0)\right)=\prod_{j=1}^n\prod_{r=1}^{K_j^l(p-1)}\IE\left(e^{\sum_{i=1}^ns_iK_{i,r}^j(1)}\Big|K^l(p-1),...,K^l(0)\right) \\
			&= \prod_{j=1}^n\prod_{r=1}^{K_j^l(p-1)}e^{\phi_{n,j}(s)}=\prod_{j=1}^ne^{\phi_{n,j}(s)K_j^l(p-1)}=e^{\phi_n(s)^TK^l(p-1)}.
		\end{align*}
	\end{proof}
	We consider now the following recursively defined functions $g^{(p)}:\IR^n\to\IR^n$
	$$g^{(0)}(s):=s \textrm{ and } g^{(p)}(s)=s+\phi_n\left(g^{(p-1)}(s)\right), \textrm{ for } p\geq1.$$
	In the next step we prove the following relation for all $k,m\in\IN_0$ and $s\in\IR^n$ with $k\geq0$, $m\geq1$, and $k+m\geq2$:
	\begin{align}
		&\IE\left(e^{-s^T(W^l(k)-W^l(k+m-2))+g^{(1)}(s)^TK^l(k+m-1)}\Big|K^l(k),...,K^l(0)\right) \nonumber \\
		&\qquad\qquad=\IE\left(e^{g^{(m-1)}(s)^TK^l(k+1)}\Big|K^l(k),...,K^l(0)\right). \label{eq:red}
	\end{align}
	\begin{proof}[Proof of \eqref{eq:red}]
		We show \eqref{eq:red} via induction over $m\geq1$. For $m=1$ we have by the definitions of $K^l(k)$, $g^{(p)}$ and by \eqref{eq:rec_K} for all $k\geq1$
		\begin{align*}
			&\IE\left(e^{s^TK^l(k+1)}\Big|K^l(k),...,K^l(0)\right)=e^{\phi_n(s)^TK^l(k)}=e^{\left(-s+g^{(1)}(s)\right)^TK^l(k)} \\
			&= \IE\left(e^{-s^T(W^l(k)-W^l(k-1))+g^{(1)}(s)^TK^l(k)}\Big|K^l(k),....,K^l(0)\right).
		\end{align*}
		The case $k=0$ and $m=2$ is trivial. Thus, the induction start is complete. In the induction step, we suppose $k\geq1$ first. Suppose now that \eqref{eq:red} holds for some $m-1$, i.e. $m\geq2$. We have then by definition of $g^{(p)}$, \eqref{eq:rec_K}, the induction hypothesis and definition of $K^l(k)$ for all $k\geq1$
		\begin{align*}
			&\IE\left(e^{g^{(m-1)}(s)^TK^l(k+1)}\Big|K^l(k),...,K^l(0)\right) \\
			&= \IE\left(e^{s^TK^l(k+1)}e^{\phi_n(g^{(m-2)}(s))^TK^l(k+1)}\Big|K^l(k),...,K^l(0)\right) \\
			&= \IE\left(e^{s^TK^l(k+1)}\IE\left(e^{g^{(m-2)}(s)^TK^l(k+2)}\Big|K^l(k+1),...,K^l(0)\right)\Big|K^l(k),...,K^l(0)\right) \\
			&= \IE\Bigg(e^{s^TK^l(k+1)}\IE\left(e^{-s^T(W^l(k+1)-W^l(k+m-2))+g^{(1)}(s)^TK^l(k+m-1)}\Big|K^l(k+1),...,K^l(0)\right) \\
			&\quad\qquad\qquad\qquad\Big|K^l(k),...,K^l(0)\Bigg) \\
			&= \IE\left(e^{-s^T(W^l(k+1)-K^l(k+1)-W^l(k+m-2))+g^{(1)}(s)^TK^l(k+m-1)}\Big|K^l(k),...,K^l(0)\right) \\
			&= \IE\left(e^{-s^T(W^l(k)-W^l(k+m-2))+g^{(1)}(s)^TK^l(k+m-1)}\Big|K^l(k),...,K^l(0)\right)
		\end{align*}
		and the induction is complete. It remains to consider the case $k=0$ which can be proven by similar arguments.
	\end{proof}
	By the tower rule, \eqref{eq:red} holds also without the conditions in the expectations. Let now either $k=0$ and $p\geq2$ or $p>k\geq1$ for $p,k\in\IN_0$. Then, we can apply \eqref{eq:red} with $m=p-k$ to obtain (use first the definition ok $K^l(p)$, in the second equality \eqref{eq:rec_K}, and in the third equality \eqref{eq:red})
	\begin{align*}
		\IE\left(e^{s^T(W^l(p)-W^l(k))}\right)&=\IE\left(e^{-s^TW^l(k)+s^TW^l(p-1)}\IE\left(e^{s^TK^l(p)}\Big|K^l(p-1),...,K^l(0)\right)\right) \\
		&=\IE\left(e^{-s^T(W^l(k)-W^l(p-2))+g^{(1)}(s)^TK^l(p-1)}\right)=\IE\left(e^{g^{(p-k-1)}(s)^TK^l(k+1)}\right).
	\end{align*}
	By applying \eqref{eq:rec_K} repeatedly, we can continue
	\begin{align*}
		&=\IE\left(\IE\left(e^{g^{(p-k-1)}(s)^TK^l(k+1)}\Big|K^l(k),...,K^l(1)\right)\right) \\
		&=\IE\left(e^{\phi_n(g^{(p-k-1)}(s))^TK^l(k)}\right)=...=\IE\left(e^{\phi_n^{\circ (k+1)}(g^{(p-k-1)}(s))^TK^l(0)}\right)=e^{\left[\phi_n^{\circ (k+1)}(g^{(p-k-1)}(s))\right]_l},
	\end{align*}
	where $\phi_n^{\circ(k)}:=\phi_n\circ\dots\circ\phi_n$ $k$-times. By manually checking the case $p=1$ and $k=0$ we conclude that the following equality holds for all $p>k\geq0$
	\begin{equation}
		\label{eq:phi_equality}
		\IE\left(e^{s^T\left(W^l(p)-W^l(k)\right)}\right)=e^{\left[\phi_n^{\circ(k+1)}\left(g^{(p-k-1)}(s)\right)\right]_l}.
	\end{equation}
	In particular, we obtain for $k=0$
	$$\log\IE\left(e^{s^TW^l(p)}\right)=s_l+\phi_n\left(g^{(p-1)}(s)\right)_l=g^{(p)}(s)_l.$$
	By using \eqref{eq:phi_contracts} it can be shown that $\|g^{(p)}(s)\|_1\leq r$ for all $p\in\IN$ if $\|s\|_1\leq r(1-\epsilon)$. Note that for each fixed $i$ the sequence $W_i^l(K)$ is increasing as $K$ grows. Hence, $W_i^l=\lim_{K\to\infty}W_i^l(K)\in\IN_0\cup\{\infty\}$ exists. Let $s\in[0,\infty)^n$ be arbitrary. By the monotone convergence theorem we conclude that
	\begin{align}
		x_l(s)&:=\log\IE\left(e^{s^T W^l}\right)=\log\IE\left(\lim_{K\to\infty}e^{s^TW^l(K)}\right)=\lim_{K\to\infty}\log\IE\left(e^{s^TW^l(K)}\right)=\lim_{K\to\infty}g^{(K)}(s)_l. \label{eq:Wlim}
	\end{align}
	Thus, we get for $\|s\|_1\leq r(1-\epsilon)$,
	$$\sum_{l=1}^n\left|\log\IE\left(e^{s^TW^l}\right)\right|=\|x(s)\|_1=\lim_{K\to\infty}\left\|g^{(K)}(s)\right\|_1\leq r.$$
	Thus, we have proven \eqref{eq:lem11}.
	
	In order to prove \eqref{eq:lem12}, we firstly note \eqref{eq:lem11} implies for $\|s\|_1\leq r(1+\epsilon)$ with $s\in[0,\infty)^n$
	$$\log\IE\left(e^{s^T(W^l-W^l(k))}\right)\leq\log\IE\left(e^{s^TW^l}\right)\leq r$$
	because $s\in[0,\infty)$ and monotonicity of $W_i^l(k)$ in $k$ imply that the expectations are larger or equal than one and hence, the logarithms are non-negative. Thus, by definition of $\epsilon,r$
	$$\left|\IE\left(e^{s^T(W^l-W^l(k))}\right)-1\right|\leq\frac{\epsilon}{a_0}\log\IE\left(e^{s^T(W^l-W^l(k))}\right).$$
	We therefore prove in the following that
	$$\sum_{k=0}^{\infty}\sum_{l=1}^n\log\IE\left(e^{s^T(W^l-W^l(k))}\right)\leq\frac{r\epsilon}{1-\epsilon}.$$
	We use monotone convergence, \eqref{eq:phi_equality}, and continuity of $\phi_n$ to obtain
	\begin{align*}
		&\sum_{l=1}^n\log\IE\left(e^{s^T(W^l-W^l(k))}\right)=\sum_{l=1}^n\lim_{p\to\infty}\log\IE\left(e^{s^T(W^l(p)-W^l(k))}\right) \\
		&= \sum_{l=1}^n\left[\phi_n^{\circ(k+1)}\left(\lim_{p\to\infty} g^{(p-k-1)}(s)\right)\right]_l=\sum_{l=1}^n\left[\phi_n^{\circ(k+1)}(x(s))\right]_l\leq \left\|\phi_n^{\circ (k+1)}(x(s))\right\|_1.
	\end{align*}
	Note that $\|\phi^{\circ k}_n(x(s))\|_1\leq \epsilon^k r$ by \eqref{eq:phi_contracts} and since $\|x(s)\|_1\leq r$ as we have just proven in \eqref{eq:lem11}. Hence,
	$$\sum_{k=0}^{\infty}\sum_{l=1}^n\log\IE\left(e^{s^T(W^l-W^l(k))}\right)\leq\sum_{k=0}^{\infty}\epsilon^{k+1}r=\frac{\epsilon r}{1-\epsilon}.$$
	Thus, the proof of the lemma is complete.
\end{proof}

\section{Proofs of Section \ref{subsec:single_consistency}}
\label{sup:single_consistency}
\begin{proof}[Proof of Theorem \ref{thm:pre_estimate_consistency}]
	The proof runs along the same lines of the proof of Theorem 6.2 in \citet{GB11}. However, for completeness we repeat it here in our setting. Define
	\begin{align*}
		\check{P}_n:=\frac{1}{n}\sum_{i=1}^n\omega_i\|\check{C}_{n,i\bdot}\|_1 \textrm{ and } P^*_n:=\frac{1}{n}\sum_{i=1}^n\omega_i\|C_{n,i\bdot}^*\|_1.
	\end{align*}
	With this, we obtain by the definition of $(\check{C}_n,\check{\alpha}_n,\check{\theta}_n)$,
	\begin{align}
		&\frac{1}{nT}\mathcal{E}(\check{C}_n,\check{\alpha}_n,\check{\theta}_n)+2\check{P}_n \nonumber \\
		&= \frac{1}{nT}\sum_{i=1}^n\left(\textrm{LS}_i(\check{C}_{n,i\bdot},\check{\alpha}_{n,i},\check{\theta}_n)+2\int_0^T\Psi_{n,i}(t;\check{C}_{n,i\bdot},\check{\alpha}_{n,i},\check{\theta}_n)dM_{n,i}(t)\right)+2\check{P}_n \nonumber \\
		&\leq \frac{1}{nT}\sum_{i=1}^n\left(\textrm{LS}_i(C_{n,i\bdot}^*,\alpha_{n,i}^*,\theta_n^*)+2\int_0^T\Psi_{n,i}(t;\check{C}_{n,i\bdot},\check{\alpha}_{n,i},\check{\theta}_n)dM_{n,i}(t)\right)+2P_n^* \nonumber \\
		&= \frac{1}{nT}\mathcal{E}(C_n^*,\alpha_n^*,\theta_n^*)+2P_n^* \nonumber \\
		&\quad\quad+\sum_{i=1}^n\Bigg(\frac{2}{nT}\int_0^T\Psi_{n,i}(t;\check{C}_{n,i\bdot},\check{\alpha}_{n,i},\check{\theta}_{n,i})dM_{n,i}(t) \nonumber\\
		&\qquad\qquad-\frac{2}{nT}\int_0^T\Psi_{n,i}(t;C_{n,i\bdot}^*,\alpha_{n,i}^*,\theta_n^*)dM_{n,i}(t)\Bigg) \label{eq:pre_basic_inequality}
	\end{align}
	On the event $\mathcal{T}_n(a_n,b_n,d_n,e_n)$, defined in \eqref{eq:def_T}, we have
	\begin{align*}
		&\left|\frac{2}{nT}\sum_{i=1}^n\left(\int_0^T\Psi_{n,i}(t;\check{C}_{n,i\bdot},\check{\alpha}_{n,i},\check{\theta}_{n,i})dM_{n,i}(t)-\int_0^T\Psi_{n,i}(t;C_{n,i\bdot}^*,\alpha_{n,i}^*,\theta_n^*)dM_{n,i}(t)\right)\right| \\
		&\leq \left|\frac{2}{nT}\sum_{i=1}^n\int_0^T\left(\check{\alpha}_{n,i}-\alpha_{n,i}^*\right)\nu_0(X_{n,i}(t);\check{\beta}_n)dM_{n,i}(t)\right| \\
		&+\left|\frac{2}{nT}\sum_{i=1}^n\int_0^T\alpha_{n,i}^*\left(\nu_0(X_{n,i}(t);\check{\beta}_n)-\nu_0(X_{n,i}(t);\beta_n^*)\right)dM_{n,i}(t)\right| \\
		&+\left|\frac{2}{nT}\sum_{i=1}^n\sum_{j=1}^n\int_0^T\int_0^{t-}g(t-r;\check{\gamma}_n)dN_{n,j}(r)dM_{n,i}(t)\cdot\left(\check{C}_{n,ij}-C_{n,ij}^*\right)\right| \\
		&+\left|\frac{2}{nT}\sum_{i=1}^n\sum_{j=1}^n\int_0^T\int_0^{t-}\left(g(t-r;\check{\gamma}_n)-g(t-r;\gamma_n^*)\right)dN_{n,j}(r)dM_{n,i}(t)\cdot C_{n,ij}^*\right| \\
		&\leq \frac{2}{T}\sup_{i=1,...,n}\sup_{\overline{\beta}\in K_{\beta}}\left|\int_0^T\nu_0(X_{n,i}(t);\overline{\beta})dM_{n,i}(t)\right|\cdot\frac{1}{n}\left\|\check{\alpha}_n-\alpha_n^*\right\|_1 \\
		&+\left|\frac{2}{nT}\sum_{i=1}^n\alpha_{n,i}^*\int_0^T\frac{\nu_0(X_{n,i}(t);\check{\beta}_n)-\nu_0(X_{n,i}(t);\beta_n^*)}{\left\|\check{\beta}_n-\beta_n^*\right\|_1}dM_{n,i}(t)\right|\cdot\left\|\check{\beta}_n-\beta_n^*\right\|_1 \\
		&+\frac{1}{n}\sum_{i=1}^n\sup_{j=1,...,n}\left|\frac{2}{T}\int_0^T\int_0^{t-}g(t-r;\check{\gamma}_n)dN_{n,j}(r)dM_{n,i}(t)\right|\cdot\left\|\check{C}_{n,i\cdot}-C_{n,i\cdot}^*\right\|_1 \\
		&+\left|\frac{2}{nT}\sum_{i=1}^n\sum_{j=1}^n\int_0^T\int_0^{t-}\frac{g(t-r;\check{\gamma}_n)-g(t-r;\gamma_n^*)}{\left|\check{\gamma}_n-\gamma_n^*\right|}dN_{n,j}(r)dM_{n,i}(t)\cdot C_{n,ij}^*\right|\cdot\left|\check{\gamma}_n-\gamma_n^*\right| \\
		&\leq \frac{a_n}{n}\left\|\check{\alpha}_n-\alpha_n^*\right\|_1+b_n\left\|\check{\beta}_n-\beta_n^*\right\|_1+\frac{1}{n}\sum_{i=1}^nd_{n,i}\left\|\check{C}_{n,i\cdot}-C_{n,i\cdot}^*\right\|_1+e_n|\check{\gamma}_n-\gamma_n^*| \\
		&= \frac{a_n}{n}\left\|\check{\alpha}_n-\alpha_n^*\right\|_1+\frac{1}{n}\sum_{i=1}^nd_{n,i}\left\|\check{C}_{n,i\cdot}-C_{n,i\cdot}^*\right\|_1+\tilde{b}_n\|\check{\theta}_n-\theta_n^*\|_1,
	\end{align*}
	where $\tilde{b}_n\geq\max(b_n,e_n)$. Using this in \eqref{eq:pre_basic_inequality} yields on the event $\mathcal{T}_n(a_n,b_n,d_n,e_n)$ the following {\em basic inequality:}
	\begin{align}
		\frac{1}{nT}\mathcal{E}(\check{C}_n,\check{\alpha}_n,\check{\theta}_n)+2\check{P}_n&\leq \frac{1}{nT}\mathcal{E}(C_n^*,\alpha_n^*,\theta_n^*)+2P_n^* \nonumber \\
		&+\frac{a_n}{n}\left\|\check{\alpha}_n-\alpha_n^*\right\|_1+\frac{1}{n}\sum_{i=1}^nd_{n,i}\left\|\check{C}_{n,i\cdot}-C_{n,i\cdot}^*\right\|_1+\tilde{b}_n\|\check{\theta}_n-\theta_n^*\|_1. \label{eq:new_basicinequality}
	\end{align}
	Denote for any $S\subseteq\{1,...,n\}$
	\begin{align*}
		P_{n,i,S}^*:&= \frac{\omega_i}{n}\|C_{n,iS}^*\|_1,\quad\check{P}_{n,i,S}:=\frac{\omega_i}{n}\|\check{C}_{n,iS}\|_1.
	\end{align*}
	Note that $P_{n,i,S_i^c(C_n^*)}^*=0$. Then, we obtain (use \eqref{eq:new_basicinequality} in the second inequality, the reverse triangle inequality in the last line, and $d_{n,i}\leq\omega_i$ several times):
	\begin{align}
		&\frac{2}{nT}\mathcal{E}(\check{C}_n,\check{\alpha}_n,\check{\theta}_n)+2\sum_{i=1}^n\check{P}_{n,i,S_i^c(C_n^*)} \nonumber \\
		&\leq \frac{2}{nT}\mathcal{E}(\check{C}_n,\check{\alpha}_n,\check{\theta}_n)+4\sum_{i=1}^n\check{P}_{n,i,S_i^c(C_n^*)}-\frac{2}{n}\sum_{i=1}^nd_{n,i}\|\check{C}_{n,iS_i^c(C_n^*)}\|_1 \nonumber \\
		&= \frac{2}{nT}\mathcal{E}(\check{C}_n,\check{\alpha}_n,\check{\theta}_n)+4\check{P}_n-4\sum_{i=1}^n\check{P}_{n,i,S_i(C_n^*)}-\frac{2}{n}\sum_{i=1}^nd_{n,i}\|\check{C}_{n,iS_i^c(C_n^*)}\|_1 \nonumber \\
		&\leq \frac{2}{nT}\mathcal{E}(C_n^*,\alpha_n^*,\theta_n^*)+4P_n^*+\frac{2a_n}{n}\left\|\check{\alpha}_n-\alpha_n^*\right\|_1+\frac{2}{n}\sum_{i=1}^nd_{n,i}\left\|\check{C}_{n,i\cdot}-C_{n,i\cdot}^*\right\|_1+2\tilde{b}_n\|\check{\theta}_n-\theta_n^*\|_1 \nonumber \\
		&-4\sum_{i=1}^n\check{P}_{n,i,S_i(C_n^*)}-\frac{2}{n}\sum_{i=1}^nd_{n,i}\|\check{C}_{n,iS_i^c(C_n^*)}\|_1 \nonumber \\
		&= \frac{2}{nT}\mathcal{E}(C_n^*,\alpha_n^*,\theta_n^*)+4P_n^*-4\sum_{i=1}^n\check{P}_{n,i,S_i(C_n^*)} \nonumber \\
		&+\frac{2a_n}{n}\left\|\check{\alpha}_n-\alpha_n^*\right\|_1+\frac{2}{n}\sum_{i=1}^nd_{n,i}\|\check{C}_{n,iS_i(C_n^*)}-C_{n,iS_i(C_n^*)}^*\|_1+2\tilde{b}_n\|\check{\theta}_n-\theta_n^*\|_1 \nonumber \\
		&\leq \frac{2}{nT}\mathcal{E}(C_n^*,\alpha_n^*,\theta_n^*)+4\sum_{i=1}^nP_{n,i,S_i(C_n^*)}^*-4\sum_{i=1}^n\check{P}_{n,i,S_i(C_n^*)} \nonumber \\
		&+\frac{2a_n}{n}\left\|\check{\alpha}_n-\alpha_n^*\right\|_1+\frac{2}{n}\sum_{i=1}^n\omega_i\|\check{C}_{n,iS_i(C_n^*)}-C_{n,iS_i(C_n^*)}^*\|_1+2\tilde{b}_n\|\check{\theta}_n-\theta_n^*\|_1 \nonumber \\
		&\leq \frac{2}{nT}\mathcal{E}(C_n^*,\alpha_n^*,\theta_n^*) \nonumber \\
		&+\frac{2a_n}{n}\left\|\check{\alpha}_n-\alpha_n^*\right\|_1+\frac{6}{n}\sum_{i=1}^n\omega_i\|\check{C}_{n,iS_i(C_n^*)}-C_{n,iS_i(C_n^*)}^*\|_1+2\tilde{b}_n\|\check{\theta}_n-\theta_n^*\|_1 \nonumber
	\end{align}
	By the definition of $\mathcal{E}$ and since $\Psi_{n,i}(t;C_{n,i\bdot}^*,\alpha_{n,i}^*,\theta_n^*)=\lambda_{n,i}(t)$ this is equivalent to
	\begin{align}
		&\frac{2}{nT}\left\|\Psi_n(\cdot;\check{C}_n,\check{\alpha}_n,\check{\theta}_n)-\lambda_n\right\|_T^2+2\sum_{i=1}^n\check{P}_{n,i,S_i^c(C_n^*)} \nonumber \\
		&\leq \frac{2a_n}{n}\left\|\check{\alpha}_n-\alpha_n^*\right\|_1+\frac{6}{n}\sum_{i=1}^n\omega_i\|\check{C}_{n,iS_i(C_n^*)}-C_{n,iS_i(C_n^*)}^*\|_1+2\tilde{b}_n\|\check{\theta}_n-\theta_n^*\|_1. \label{eq:new_newbi}
	\end{align}
	Replacing the definition of $\check{P}_{n,i,S_i^c(C_n^*,\alpha_n^*)}$, the above implies in particular that
	\begin{align*}
		&\frac{1}{n}\sum_{i=1}^n\omega_i\|\check{C}_{n,iS_i^c(C_n^*)}-C_{n,iS_i^c(C_n^*)}^*\|_1 \\
		&\leq \frac{a_n}{n}\left\|\check{\alpha}_n-\alpha_n^*\right\|_1+\frac{3}{n}\sum_{i=1}^n\omega_i\|\check{C}_{n,iS_i(C_n^*)}-C_{n,iS_i(C_n^*)}^*\|_1+\tilde{b}_n\|\check{\theta}_n-\theta_n^*\|_1.
	\end{align*}
	Since we are on $\Omega_{RCC,n}(L,H_n)$ and $L\geq\max(3\omega_1,...,3\omega_n,a_n,\tilde{b}_n)/\min(\omega_1,...,\omega_n)$, we can apply the assertion of the random compatibility condition. More precisely, this means by using the Cauchy-Schwarz inequality that
	\begin{align}
		&\frac{3a_n}{n}\left\|\check{\alpha}_n-\alpha_n^*\right\|_1+\frac{8}{n}\sum_{i=1}^n\omega_i\|\check{C}_{n,iS_i(C_n^*)}-C_{n,iS_i(C_n^*)}^*\|_1+3\tilde{b}_n\|\check{\theta}_n-\theta_n\|_1 \nonumber  \\
		&\leq \mathcal{L}(C_n^*,\alpha_n^*)\sqrt{\frac{1}{n}\|\check{\alpha}_n-\alpha_n^*\|^2+\frac{1}{n}\sum_{i=1}^n\|\check{C}_{n,iS_i(C_n^*)}-C_{n,iS_i(C_n^*)}^*\|^2+\|\check{\theta}_n-\theta_n^*\|^2} \nonumber \\
		&\leq \frac{\mathcal{L}(C_n^*,\alpha_n^*)}{\phi_{\textrm{comp}}(S_1(C_n^*),...,S_n(C_n^*);L;\mathcal{H}_n)}\frac{1}{\sqrt{nT}}\|\Psi_n(\cdot;\check{C}_n,\check{\alpha}_n,\check{\theta}_n)-\lambda_n\|_T \nonumber \\
		&\leq \frac{\mathcal{L}(C_n^*,\alpha_n^*)^2}{4\phi_{\textrm{comp}}(S_1(C_n^*),...,S_n(C_n^*);L;\mathcal{H}_n)^2}+\frac{1}{nT}\|\Psi_n(\cdot;\check{C}_n,\check{\alpha}_n,\check{\theta}_n)-\lambda_n\|_T^2. \label{eq:estimator_bound}
	\end{align}
	Thus, we obtain from \eqref{eq:new_newbi}
	\begin{align*}
		&\frac{2}{nT}\left\|\Psi_n(\cdot;\check{C}_n,\check{\alpha}_n,\check{\theta}_n)-\lambda_n\right\|_T^2 \\
		&\qquad+\frac{2}{n}\sum_{i=1}^n\omega_i\|\check{C}_{n,i\cdot}-C_{n,i\cdot}^*\|_1+\frac{a_n}{n}\|\check{\alpha}_n-\alpha_n^*\|_1+\tilde{b}_n\|\check{\theta}_n-\theta_n^*\|_1 \nonumber \\
		&\leq \frac{3a_n}{n}\left\|\check{\alpha}_n-\alpha_n^*\right\|_1+\frac{8}{n}\sum_{i=1}^n\omega_i\|\check{C}_{n,iS_i(C_n^*)}-C_{n,iS_i(C_n^*)}^*\|_1+3\tilde{b}_n\|\check{\theta}_n-\theta_n^*\|_1 \\
		&\leq \frac{\mathcal{L}(C_n^*,\alpha_n^*)^2}{4\phi_{\textrm{comp}}(S_1(C_n^*),...,S_n(C_n^*);L;\mathcal{H}_n)^2}+\frac{1}{nT}\|\Psi_n(\cdot;\check{C}_n,\check{\alpha}_n,\check{\theta}_n)-\lambda_n\|_T^2.
	\end{align*}
	which in turn implies
	\begin{align*}
		&\frac{1}{nT}\left\|\Psi_n(\cdot;\check{C}_n,\check{\alpha}_n,\check{\theta}_n)-\lambda_n\right\|_T^2+\frac{2}{n}\sum_{i=1}^n\omega_i\|\check{C}_{n,i\cdot}-C_{n,i\cdot}^*\|_1+\frac{a_n}{n}\|\check{\alpha}_n-\alpha_n^*\|_1\nonumber+\tilde{b}_n\|\check{\theta}_n-\theta_n^*\|_1 \\
		&\leq \frac{\mathcal{L}(C_n^*,\alpha_n^*)^2}{4\phi_{\textrm{comp}}(S_1(C_n^*),...,S_n(C_n^*);L;\mathcal{H}_n)^2}.
	\end{align*}
\end{proof}

\subsection{Further details to the proof of Lemma \ref{lem:lambda_rate}}
\label{subsec:details_proof_lambda_rate}
\underline{Part involving $b_n$:}
We firstly note that by differentiability of $\nu_0$
\begin{align*}
	&\frac{2}{nT}\sum_{i=1}^n\alpha_{n,i}^*\int_0^T\frac{\nu_0(X_{n,i}(t);\overline{\beta})-\nu_0(X_{n,i}(t);\beta_n^*)}{\big\|\overline{\beta}-\beta_n^*\big\|_1}dM_{n,i}(t) \\
	&= \frac{2}{nT}\sum_{i=1}^n\alpha_{n,i}^*\int_0^T\int_0^1\frac{d}{d\beta}\nu_0\big(X_{n,i}(t);(1-s)\beta_n^*+s\overline{\beta}\big)^T\frac{\big(\overline{\beta}-\beta_n^*\big)}{\big\|\overline{\beta}-\beta_n^*\big\|_1}dsdM_{n,i}(t)
\end{align*}
Note that, by convexity, $(1-s)\beta_n^*+s\overline{\beta}\in K_{\beta}$ for all $s\in[0,T]$. We can use a standard \emph{chaining light} argument as follows: Let $K_{\beta,n,\eta}$ be as in the part involving $a_n$. Let similarly $L^1_{\eta}\subseteq\{\beta\in\IR^p:\|\beta\|_1=1\}$ be a finite set such that for each $\beta\in\IR^p$ with $\|\beta\|_1=1$ there is $Q_n(\beta)\in L^1_{\eta}$ such that $\|\beta-Q_n(\beta)\|\leq\eta$. By adjusting $K_0>0$, it is possible to choose $K_{\beta,n,\eta}$ and $L^1_{\eta}$ such that $|K_{\beta,n,\eta}|,|L^1_{\eta}|\leq K_0\eta^{-p}$. As before, we define $\eta$ at this point without motivation. Choose first $c_2''$ such that
\begin{align*}
	&2K_{\alpha}\big(\frac{\mathcal{N}(T+A)}{AT}+K_{\alpha}\|\overline{\nu}\|_{\infty}+c_n\overline{g}\mathcal{N}\big)\big(\frac{D_{\nu}}{2}+L_{\nu}\big)\leq c_2''c_n\log(nT).
\end{align*}
Let $\eta:=\frac{4K_{\alpha}L_{\nu}(2p+\alpha_2)}{c_2''\sqrt{\mu-\phi(\mu)}c_nnT}$. We conclude that
\begin{align}
	&\IP\Bigg(\sup_{\overline{\beta}\in K_{\beta}}\Big|\frac{2}{nT}\sum_{i=1}^n\alpha_{n,i}^*\int_0^T\frac{\nu_0(X_{n,i}(t);\overline{\beta})-\nu_0(X_{n,i}(t);\beta_n^*)}{\big\|\overline{\beta}-\beta_n^*\big\|_1}dM_{n,i}(t)\Big|>b_n\Bigg) \nonumber \\
	&= \IP\Bigg(\sup_{\overline{\beta}\in K_{\beta}}\Big|\sum_{i=1}^n\frac{2\alpha_{n,i}^*}{nT}\int_0^T\int_0^1\frac{d}{d\beta}\nu_0\big(X_{n,i}(t);(1-s)\beta_n^*+s\overline{\beta}\big)^T\frac{\big(\overline{\beta}-\beta_n^*\big)}{\big\|\overline{\beta}-\beta_n^*\big\|_1}dsdM_{n,i}(t)\Big|>b_n\Bigg) \nonumber \\
	&\leq \IP\Bigg(\underset{\beta_2\in\IR^p:\|\beta_2\|_1=1}{\sup_{\beta_1\in K_{\beta},}}\Big|\sum_{i=1}^n\frac{2\alpha_{n,i}}{nT}^*\int_0^T\int_0^1\frac{d}{d\beta}\nu_0\big(X_{n,i}(t);(1-s)\beta_n^*+s\beta_1\big)^T\beta_2dsdM_{n,i}(t)\Big|>b_n\Bigg) \nonumber \\
	&\leq \scalebox{0.95}{$\displaystyle{\IP\Bigg(\underset{\beta_2\in\IR^p:\|\beta_2\|_1=1}{\sup_{\beta_1\in K_{\beta},}}\Big|\sum_{i=1}^n\frac{2\alpha_{n,i}}{nT}^*\int_0^T\int_0^1\frac{d}{d\beta}\nu_0\big(X_{n,i}(t);(1-s)\beta_n^*+s\beta_1\big)^T\beta_2dsdM_{n,i}(t)\Big|>b_n, 
			\Omega_{\mathcal{N}}\Bigg)}$}\nonumber\\
	&\hspace*{10cm}+\IP(\Omega_{\mathcal{N}}^c) \nonumber \\
	&\leq \IP\Bigg(\underset{\beta_2\in\IR^p:\|\beta_2\|_1=1}{\sup_{\beta_1\in K_{\beta},}}\Big|\sum_{i=1}^n\frac{2\alpha_{n,i}}{nT}^*\int_0^T\int_0^1\frac{d}{d\beta}\nu_0\big(X_{n,i}(t);(1-s)\beta_n^*+s\beta_1\big)^T\beta_2 \nonumber \\
	&\qquad\qquad-\frac{d}{d\beta}\nu_0\big(X_{n,i}(t);(1-s)\beta_n^*+sP_n(\beta_1)\big)^TQ_n(\beta_2)dsdM_{n,i}(t)\Big|>\frac{b_n}{2},\Omega_{\mathcal{N}}\Bigg) \label{eq:bn_cont} \\
	&+\scalebox{0.95}{$\displaystyle{\IP\Bigg(\underset{\beta_2\in L^1_{\eta}}{\sup_{\beta_1\in K_{\beta,n,\eta},}}\Big|\sum_{i=1}^n\frac{2\alpha_{n,i}}{nT}^*\int_0^T\int_0^1\frac{d}{d\beta}\nu_0\big(X_{n,i}(t);(1-s)\beta_n^*+s\beta_1\big)^T\beta_2dsdM_{n,i}(t)\Big|>\frac{b_n}{2},\Omega_{\mathcal{N}}\Bigg)}$} \nonumber \\
	&\hspace*{10cm}+\IP(\Omega_{\mathcal{N}}^c) \label{eq:bn:emp_proc}
\end{align}
We begin with the probability in \eqref{eq:bn:emp_proc}. By a union bound argument, we have
\begin{align}
	&\IP\Bigg(\underset{\beta_2\in L^1_{\eta}}{\sup_{\beta_1\in K_{\beta,n,\eta},}}\Big|\sum_{i=1}^n\frac{2\alpha_{n,i}}{nT}^*\int_0^T\int_0^1\frac{d}{d\beta}\nu_0\left(X_{n,i}(t);(1-s)\beta_n^*+s\beta_1\right)^T\beta_2dsdM_{n,i}(t)\Big|>\frac{b_n}{2},\Omega_{\mathcal{N}}\Bigg) \nonumber \\
	&\leq \scalebox{0.89}{$\displaystyle{K_0^2\eta^{-2p}\sup_{\beta_1\in K_{\beta,n,\eta}\atop \beta_2\in L^1_{\eta}} 
			\IP\left(\left|\frac{2}{nT}\sum_{i=1}^n\alpha_{n,i}^*\int_0^T\int_0^1\frac{d}{d\beta}\nu_0\left(X_{n,i}(t);(1-s)\beta_n^*+s\beta_1\right)^T\beta_2dsdM_{n,i}(t)\right|>\frac{b_n}{2},\Omega_{\mathcal{N}}\right).}$} \label{eq:bn_ub}
\end{align}
The above probability can be handled by Theorem 3 of \citet{HRBR15}. We use again the same notation from the original paper for ease of the reader. We apply this time the multivariate version with $M=n$ and
$$H_i(t)=\frac{4\alpha_{n,i}^*}{nT}\int_0^1\frac{d}{d\beta}\nu_0\left(X_{n,i}(t);(1-s)\beta_n^*+s\beta_1\right)^T\beta_2ds.$$
We check now the conditions of the previously mentioned theorem. $|H_i(t)|\leq \frac{4K_{\alpha}}{nT}L_{\nu}=:B$ and the integral conditions are fulfilled. We consider the constant stopping time $\tau=T$. Define next
\begin{align*}
	\hat{V}_b^{\mu}:&= \frac{\mu}{\mu-\phi(\mu)}\sum_{i=1}^n\int_0^T\frac{16(\alpha_{n,i}^*)^2}{n^2T^2}\left(\int_0^1\frac{d}{d\beta}\nu_0\left(X_{n,i}(t);(1-s)\beta_n^*+s\beta_1\right)^T\beta_2ds\right)^2dN_{n,i}(t) \\
	&\qquad+\frac{16K_{\alpha}^2L_{\nu}^2x}{n^2T^2(\mu-\phi(\mu))}.
\end{align*}
On $\Omega_{\mathcal{N}}$, we have
$$w:=\frac{16K_{\alpha}^2L_{\nu}^2x}{n^2T^2(\mu-\phi(\mu))}\leq\hat{V}_b^{\mu}\leq\frac{16\mu K_{\alpha}^2L_{\nu}^2\mathcal{N}(T+A)}{nT^2A(\mu-\phi(\mu))}+\frac{16K_{\alpha}^2L_{\nu}^2x}{n^2T^2(\mu-\phi(\mu))}=:v.$$
Hence, application of Theorem 3 of \citet{HRBR15} shows for $\epsilon=1$ and $x=(2p+\alpha_2)\log(Tn)$
\begin{align*}
	&\IP\left(\frac{2}{nT}\sum_{i=1}^n\alpha_{n,i}^*\int_0^T\int_0^1\Big[\frac{d}{d\beta}\nu_0\left(X_{n,i}(t);(1-s)\beta_n^*+s\beta_1\right)^T\beta_2\Big]dsdM_{n,i}(t)>\frac{b_n}{2},\Omega_{\mathcal{N}}\right) \\
	&\leq \IP\left(\sum_{i=1}^n\int_0^TH_i(t)dM_{n,i}(t)>2\sqrt{\hat{V}_b^{\mu}x}+\frac{Bx}{3},w\leq\hat{V}_b^{\mu}\leq v,\sup_{i,t}|H_i(t)|\leq B\right) \\
	&\leq 2\left(\log_2\left(\frac{\mu\mathcal{N}n(T+A)}{Ax}+1\right)+1\right)e^{-x}\leq c_2'\log(nT)(nT)^{-(2p+\alpha_2)}
\end{align*}
for a suitable constant $c_2'>0$. Combining the above with \eqref{eq:bn_ub} yields
\begin{align*}
	&\IP\Bigg(\underset{\beta_2\in L^1_{\eta}}{\sup_{\beta_1\in K_{\beta,n,\eta},}}\Bigg|\sum_{i=1}^n\frac{2\alpha_{n,i}^*}{nT}\int_0^T\int_0^1\frac{d}{d\beta}\nu_0\left(X_{n,i}(t);(1-s)\beta_n^*+s\beta_1\right)^T\beta_2dsdM_{n,i}(t)\Bigg|>\frac{b_n}{2},\Omega_{\mathcal{N}}\Bigg) \\
	&\leq 2K_0^2\eta^{-2p}c_2'\log(nT)(nT)^{-(2p+\alpha_2)}=c_2\frac{c_n^{2p}\log(nT)}{(nT)^{\alpha_2}}
\end{align*}
for a suitable choice of $c_2$.

We turn now to \eqref{eq:bn_cont}. For ease of notation, we denote below $d|M_{n,i}(t)|:=dN_{n,i}(t)+\lambda_{n,i}(t)dt$. We note that, by Lipschitz continuity of $\frac{d}{d\beta}\nu_0$, on $\Omega_{\mathcal{N}}$
\begin{align*}
	&\underset{\beta_2\in\IR^p:\|\beta_2\|_1=1}{\sup_{\beta_1\in K_{\beta},}}\Bigg|\frac{2}{nT}\sum_{i=1}^n\alpha_{n,i}^*\int_0^T\int_0^1\frac{d}{d\beta}\nu_0\left(X_{n,i}(t);(1-s)\beta_n^*+s\beta_1\right)^T\beta_2 \\
	&\hspace*{3cm}-\frac{d}{d\beta}\nu_0\left(X_{n,i}(t);(1-s)\beta_n^*+sP(\beta_1)\right)^TQ_n(\beta_2)dsdM_{n,i}(t)\Bigg| \\
	&= \underset{\beta_2\in\IR^p:\|\beta_2\|_1=1}{\sup_{\beta_1\in K_{\beta},}}\Bigg|\frac{2}{nT}\sum_{i=1}^n\alpha_{n,i}^*\int_0^T\int_0^1\Big(\frac{d}{d\beta}\nu_0\left(X_{n,i}(t);(1-s)\beta_n^*+s\beta_1\right) \\
	&\hspace*{5cm}-\frac{d}{d\beta}\nu_0\left(X_{n,i}(t);(1-s)\beta_n^*+sP(\beta_1)\right)\Big)^T\beta_2 \\
	&\qquad\qquad\qquad-\frac{d}{d\beta}\nu_0\left(X_{n,i}(t);(1-s)\beta_n^*+sP(\beta_1)\right)^T(Q_n(\beta_2)-\beta_2)dsdM_{n,i}(t)\Bigg| \\
	&\leq \underset{\beta_2\in\IR^p:\|\beta_2\|_1=1}{\sup_{\beta_1\in K_{\beta},}}\frac{2K_{\alpha}}{nT}\sum_{i=1}^n\int_0^T\int_0^1\Big\|\frac{d}{d\beta}\nu_0\left(X_{n,i}(t);(1-s)\beta_n^*+s\beta_1\right) \\
	&\hspace*{5cm}-\frac{d}{d\beta}\nu_0\left(X_{n,i}(t);(1-s)\beta_n^*+sP(\beta_1)\right)\Big\|_{\infty}\|\beta_2\|_1 \\
	&\qquad\qquad\qquad+\Big\|\frac{d}{d\beta}\nu_0\left(X_{n,i}(t);(1-s)\beta_n^*+sP(\beta_1)\right)\Big\|\|Q_n(\beta_2)-\beta_2\|dsd|M_{n,i}|(t) \\
	&\leq \frac{2K_{\alpha}}{nT}\sum_{i=1}^n\int_0^T\frac{D_{\nu}}{2}\eta+L_{\nu}\eta d|M_{n,i}|(t) \\
	&= \frac{2K_{\alpha}}{nT}\sum_{i=1}^n\int_0^T\frac{D_{\nu}}{2}\eta+L_{\nu}\eta dN_{n,i}(t)+\frac{2K_{\alpha}}{nT}\sum_{i=1}^n\int_0^T\left(\frac{D_{\nu}}{2}\eta+L_{\nu}\eta\right) \\
	&\hspace*{3cm}\times\Big(\alpha_{n,i}^*\nu_0(X_{n,i}(t);\beta_n^*)+\sum_{j=1}^nC_{n,ij}^*\int_0^{t-}g(t-r;\gamma_n^*)dN_{n,j}(r)\Big)dt \\
	&\leq 2K_{\alpha}\left(\frac{\mathcal{N}(T+A)}{AT}+K_{\alpha}\|\overline{\nu}\|_{\infty}+\sup_{i=1,..,n}\|C_{n,i\cdot}^*\|_1\overline{g}\mathcal{N}\right)\left(\frac{D_{\nu}}{2}+L_{\nu}\right)\eta \\
	&\leq 2K_{\alpha}\left(\frac{\mathcal{N}(T+A)}{AT}+K_{\alpha}\|\overline{\nu}\|_{\infty}+c_n\overline{g}\mathcal{N}\right)\left(\frac{D_{\nu}}{2}+L_{\nu}\right)\eta\leq c_2''c_n\log(nT)\eta,
\end{align*}
by choice of $c_2''$, \eqref{eq:Bj}, and the fact that the number of jumps of each $N_{n,i}$ on $[0,T]$ is on $\Omega_{\mathcal{N}}$ bounded by $\mathcal{N}(T+A)/A$. By choice of $\eta$, we have that
$$c_2''c_n\log(nT)\eta\leq\sqrt{wx}\leq\sqrt{\hat{V}_b^{\mu}x}\leq\frac{b_n}{2}.$$
Hence, $\eqref{eq:bn_cont}=0$.

\underline{Part involving $d_{n,i}$:} Let $K_{\gamma,n,\eta}$ be a finite, discrete grid covering $K_{\gamma}$ such that for each $\gamma\in K_{\gamma}$, there is $P_n(\gamma)\in K_{\gamma,n,\eta}$ such that $|\gamma-P_n(\gamma)|\leq\eta$. It is possible to choose $K_{\gamma,n,\eta}$ such that $|K_{\gamma,n,\eta}|\leq K_1\eta^{-1}$. Let $c_{3,i}''>0$ be such that
$$2L_g\mathcal{N}\left(\frac{(T+A)\mathcal{N}}{AT}+K_{\alpha}\overline{\nu}_i+c_n\overline{g}\mathcal{N}\right)\leq c_{3,i}''\log(nT)^2c_n$$
and 
$$\eta_i:=\frac{24\overline{g}\mathcal{N}_0\left(\log n+\log(nT)+\alpha_3\log T\right)}{\sqrt{\mu-\phi(\mu)}c_{3,i}''\log(nT)c_nT}.$$
We now have by the standard union bound and a \emph{chaining light} argument
\begin{align}
	&\IP\Big(\exists i\in\{1,...,n\}: \sup_{j=1,...,n}\sup_{\overline{\gamma}\in K_{\gamma}}\frac{2}{T}\Bigg|\int_0^T\int_0^{t-}g(t-r;\overline{\gamma})dN_{n,j}(r)dM_{n,i}(t)\Bigg|> d_{n,i}\Big) \nonumber \\
	&\leq n^2\sup_{i,j=1,...,n}\IP\Big(\sup_{\overline{\gamma}\in K_{\gamma}}\frac{2}{T}\left|\int_0^T\int_0^{t-}g(t-r;\overline{\gamma})dN_{n,j}(r)dM_{n,i}(t)\right|> d_{n,i},\Omega_{\mathcal{N}}\Big)+\IP(\Omega_{\mathcal{N}}^c) \nonumber \\
	&\leq \underset{i,j=1,...,n}{n^2\sup}\IP\Bigg(\sup_{\overline{\gamma}\in K_{\gamma}}\frac{2}{T}\Big|\int_0^T\int_0^{t-}g(t-r;\overline{\gamma})-g(t-r;P_n(\overline{\gamma}))dN_{n,j}(r)dM_{n,i}(t)\Big|>\frac{d_{n,i}}{2},\Omega_{\mathcal{N}}\Bigg) \label{eq:dn_cont} \\
	&+\underset{i,j=1,...,n}{n^2\sup}\IP\left(\sup_{\overline{\gamma}\in K_{\gamma,n,\eta_i}}\frac{2}{T}\left|\int_0^T\int_0^{t-}g(t-r;\overline{\gamma})dN_{n,j}(r)dM_{n,i}(t)\right|>\frac{d_{n,i}}{2},\Omega_{\mathcal{N}}\right)+\IP(\Omega_{\mathcal{N}}^c) \label{eq:dn_emp_proc}
\end{align}
For \eqref{eq:dn_emp_proc}, we can continue using union bound
\begin{align}
	&n^2\sup_{i,j=1,...,n}\IP\left(\sup_{\overline{\gamma}\in K_{\gamma,n,\eta_i}}\frac{2}{T}\left|\int_0^T\int_0^{t-}g(t-r;\overline{\gamma})dN_{n,j}(r)dM_{n,i}(t)\right|>\frac{d_{n,i}}{2},\Omega_{\mathcal{N}}\right) \nonumber \\
	&\leq n^2K_1\eta_i^{-1}\sup_{i,j=1,...,n}\sup_{\overline{\gamma}\in K_{\gamma,n,\eta_i}}\IP\left(\frac{2}{T}\left|\int_0^T\int_0^{t-}g(t-r;\overline{\gamma})dN_{n,j}(r)dM_{n,i}(t)\right|>\frac{d_{n,i}}{2},\Omega_{\mathcal{N}}\right). \label{eq:dn_emp_proc2}
\end{align}
We can apply now again Theorem 3 of \citet{HRBR15} in its univariate form, that is, with $M=1$. Let
$$H(t):=\frac{4}{T}\int_0^{t-}g(t-r;\overline{\gamma})dN_{n,j}(r)\Ind(t\leq\tau_n).$$
We have $|H(t)|\leq \frac{4\overline{g}\mathcal{N}}{T}=:B$ which implies the integral conditions (see also the proof of Theorem 2 of \citet{HRBR15}). Define $x:=\log n+\log(nT)+\alpha_3\log T$ and
\begin{align*}
	\hat{V}_d^{\mu}& := \frac{16\mu}{T^2(\mu-\phi(\mu))}\int_0^{\tau_n}\left(\int_0^{t-}g(t-r;\overline{\gamma})dN_{n,j}(r)\right)^2dN_{n,i}(t)+\frac{16\overline{g}^2\mathcal{N}^2x}{T^2(\mu-\phi(\mu))} \\
	\hat{V}_{d,0}^{\mu}&:= \frac{16\mu}{T^2(\mu-\phi(\mu))}\int_0^T\left(\int_0^{t-}g(t-r;\overline{\gamma})dN_{n,j}(r)\right)^2dN_{n,i}(t)+\frac{16\overline{g}^2\mathcal{N}^2x}{T^2(\mu-\phi(\mu))}.
\end{align*}
We bound on $\Omega_{\mathcal{N}}$, using \eqref{eq:Bj},
$$w:=\frac{16\overline{g}^2\mathcal{N}^2x}{T^2(\mu-\phi(\mu))}\leq\hat{V}_d^{\mu}\leq\frac{16\mu}{T^2(\mu-\phi(\mu))}\frac{T+A}{A}\overline{g}^2\mathcal{N}^3+\frac{16\overline{g}^2\mathcal{N}^2x}{T^2(\mu-\phi(\mu))}=:v.$$
Using Theorem 3 of \citet{HRBR15} we hence have for $\epsilon=1$
\begin{align*}
	&\IP\left(\frac{2}{T}\int_0^T\int_0^{t-}g(t-r;\overline{\gamma})dN_{n,j}(r)dM_{n,i}(t)>\frac{d_{n,i}}{2},\Omega_{\mathcal{N}}\right) \\
	&\leq \IP\left(\frac{4}{T}\int_0^T\int_0^{t-}g(t-r;\overline{\gamma})dN_{n,j}(r)dM_{n,i}(t)>2\sqrt{\hat{V}_{d,0}^{\mu}x}+\frac{Bx}{3},\Omega_{\mathcal{N}}\right) \\
	&= \IP\left(\frac{4}{T}\int_0^T\int_0^{t-}g(t-r;\overline{\gamma})dN_{n,j}(r)dM_{n,i}(t)>2\sqrt{\hat{V}_d^{\mu}x}+\frac{Bx}{3},\Omega_{\mathcal{N}}\right) \\
	&\leq \IP\Bigg(\frac{4}{T}\int_0^T\int_0^{t-}g(t-r;\overline{\gamma})dN_{n,j}(r)dM_{n,i}(t)>2\sqrt{\hat{V}_d^{\mu}x}+\frac{Bx}{3},w\leq\hat{V}_d^{\mu}\leq v, \\
	&\qquad\qquad\qquad\sup_{t\leq\tau_n}|H(t)|\leq B\Bigg) \\
	&\leq 2\left(\log_2\left(\frac{\mu(T+A)\mathcal{N}}{Ax}+1\right)+1\right)e^{-x}=c_3'\log T\cdot n^{-2}T^{-(1+\alpha_3)}
\end{align*}
for a suitable constant $c_3'$. Using the above in \eqref{eq:dn_emp_proc2} yields by the definition of $\eta_i$.
\begin{align*}
	&n^2\sup_{i,j=1,...,n}\IP\left(\sup_{\overline{\gamma}\in K_{\gamma,n,\eta_i}}\frac{2}{T}\left|\int_0^T\int_0^{t-}g(t-r;\overline{\gamma})dN_{n,j}(r)dM_{n,i}(t)\right|>\frac{d_{n,i}}{2},\Omega_{\mathcal{N}}\right)  \\
	&\leq n^2K_1\max_{i=1,...,n}\frac{\sqrt{\mu-\phi(\mu)}c_{3,i}''\log(nT)c_nT}{24\overline{g}\mathcal{N}_0\left(\log n+\log(nT)+\alpha_3\log T\right)}2c_3'\log T\cdot n^{-2}T^{-(1+\alpha_3)} \\
	&\leq c_3\frac{c_n\log T}{T^{\alpha_3}}
\end{align*}
for a suitable $c_3>0$.

To bound \eqref{eq:dn_cont}, we use Lipschitz continuity of $\gamma$. Denote again $d|M_{n,i}|(t):=dN_{n,i}(t)+\lambda_{n,i}(t)dt$. We have
\begin{align*}
	&\IP\left(\sup_{\overline{\gamma}\in K_{\gamma}}\frac{2}{T}\left|\int_0^T\int_0^{t-}g(t-r;\overline{\gamma})-g(t-r;P_n(\overline{\gamma}))dN_{n,j}(r)dM_{n,i}(t)\right|>\frac{d_{n,i}}{2},\Omega_{\mathcal{N}}\right) \\
	&\leq \IP\left(\frac{2}{T}\int_0^T\int_{t-A}^{t-}L_g\eta_i dN_{n,j}(r)d|M_{n,i}|(t)>\frac{d_{n,i}}{2},\Omega_{\mathcal{N}}\right) \\
	&\leq \IP\left(\frac{2}{T}\int_0^TL_g\eta_i\mathcal{N}d|M_{n,i}|(t)>\frac{d_{n,i}}{2},\Omega_{\mathcal{N}}\right) \\
	&= \IP\Big(\frac{2}{T}\int_0^TL_g\eta_i\mathcal{N}dN_{n,i}(t) 
	+\frac{2}{T}\int_0^TL_g\eta_i\mathcal{N}\Big(\alpha_{n,i}^*\nu_0(X_{n,i}(t);\beta_n^*)\, +\\
	&\hspace{4.5cm}\sum_{j=1}^nC_{n,ij}^*\int_0^{t-}g(t-r;\gamma_n^*)dN_j(r)\Big)dt>\frac{d_{n,i}}{2},\Omega_{\mathcal{N}}\Big) \\
	&\leq \IP\Bigg(\frac{2L_g\eta_i(T+A)\mathcal{N}^2}{AT}+2L_g\eta_i\mathcal{N}\left(K_{\alpha}\overline{\nu}_i+\|C_{n,i\bdot}\|_1\overline{g}\mathcal{N}\right)>\frac{d_{n,i}}{2},\Omega_{\mathcal{N}}\Bigg) \\
	&\leq \IP\Bigg(2L_g\mathcal{N}\left(\frac{(T+A)\mathcal{N}}{AT}+K_{\alpha}\overline{\nu}_i+c_n\overline{g}\mathcal{N}\right)\eta_i>\frac{d_{n,i}}{2},\Omega_{\mathcal{N}}\Bigg) \\
	&\leq \IP\left(c_{3,i}''\log^2(nT)c_n\eta_i>\frac{d_{n,i}}{2},\Omega_{\mathcal{N}}\right).
\end{align*}
By definition of $\eta_i$, we have, however $c_{3,i}''\log^2(nT)c_n\eta_i\leq\sqrt{wx}\leq d_{n,i}/2$ on $\Omega_{\mathcal{N}}$ and hence the above probability equals $0$.

\underline{Part involving $e_n$:} Let $K_{\gamma,\eta,n}$ be the same grid as in the part involving $d_{n,i}$. Let $c_4''>0$ be chosen such that
\begin{align*}
	&\frac{D_g\mathcal{N}c_n}{T}\left(\frac{T+A}{A}\mathcal{N}+TK_{\alpha}\|\overline{\nu}\|_{\infty}+c_n\overline{g}\mathcal{N}T\right)\leq c_4''c_n\left(1+c_n\right)\log^2(nT).
\end{align*}
We define then
$$\eta:=\frac{24L_g\mathcal{N}_0(1+\alpha_4)}{\sqrt{\mu-\phi(\mu)}c_4''(1+c_n)nT}.$$
We compute using the fundamental theorem of calculus
\begin{align}
	&\IP\Big(\sup_{\overline{\gamma}\in K_{\gamma}}\left|\frac{2}{nT}\sum_{i,j=1}^nC_{n,ij}^*\int_0^T\int_0^{t-}\frac{g(t-r;\overline{\gamma})-g(t-r;\gamma_n^*)}{|\overline{\gamma}-\gamma_n^*|}dN_{n,j}(r)dM_{n,i}(t)\right|>e_n\Big) \nonumber \\
	&= \scalebox{0.9}{$\displaystyle{\IP\Bigg(\sup_{\overline{\gamma}\in K_{\gamma}}\left|\sum_{i,j=1}^n\frac{2C_{n,ij}^*}{nT}\int_0^T\int_0^{t-}\frac{\int_0^1\frac{d}{d\gamma}g(t-r;s\overline{\gamma}+(1-s)\gamma_n^*)ds(\overline{\gamma}-\gamma_n^*)}{|\overline{\gamma}-\gamma_n^*|}dN_{n,j}(r)dM_{n,i}(t)\right| >e_n\Bigg)}$} \nonumber \\
	&\leq \scalebox{0.9}{$\displaystyle{\IP\Bigg(\underset{\sigma\in\{-1,1\}}{\sup_{\overline{\gamma}\in K_{\gamma}}}\left|\sum_{i,j=1}^nC_{n,ij}^*\int_0^T\frac{2\sigma\int_0^{t-}\int_0^1\frac{d}{d\gamma}g(t-r;s\overline{\gamma}+(1-s)\gamma_n^*)ds dN_{n,j}(r)}{nT}dM_{n,i}(t)\right|>e_n,\Omega_{\mathcal{N}}\Bigg)}$}\nonumber \\ &\hspace*{10cm}+\IP(\Omega_{\mathcal{N}}^c) \nonumber \\
	&\leq \scalebox{0.9}{$\displaystyle{\IP\Bigg(\sup_{\overline{\gamma}\in K_{\gamma}}\left|\sum_{i,j=1}^nC_{n,ij}^*\int_0^T\frac{2\int_0^{t-}\int_0^1\frac{d}{d\gamma}g(t-r;s\overline{\gamma}+(1-s)\gamma_n^*)dsdN_{n,j}(r)}{nT}dM_{n,i}(t)\right|>e_n, \Omega_{\mathcal{N}}\Bigg)}$}\nonumber
	\\
	&\leq \IP\Bigg(\sup_{\overline{\gamma}\in K_{\gamma}}\Bigg|\frac{2}{nT}\sum_{i,j=1}^nC_{n,ij}^*\int_0^T\int_0^{t-}\int_0^1\Big[\frac{d}{d\gamma}g(t-r;s\overline{\gamma}+(1-s)\gamma_n^*) \nonumber \\
	&\qquad\qquad\qquad\qquad-\frac{d}{d\gamma}g(t-r;sP_n(\overline{\gamma})+(1-s)\gamma_n^*)ds\Big]dN_{n,j}(r)dM_{n,i}(t)\Bigg|>\frac{e_n}{2},\Omega_{\mathcal{N}}\Bigg) \label{eq:en_cont} \\
	&+\scalebox{0.9}{$\displaystyle{\IP\Bigg(\sup_{\overline{\gamma}\in K_{\gamma,n,\eta}}\left|\sum_{i,j=1}^nC_{n,ij}^*\int_0^T\frac{2\int_0^{t-}\int_0^1\frac{d}{d\gamma}g(t-r;s\overline{\gamma}+(1-s)\gamma_n^*)dsdN_{n,j}(r)}{nT}dM_{n,i}(t)\right| >\frac{e_n}{2},\Omega_{\mathcal{N}}\Bigg)}$}\nonumber\\
	&\hspace*{10cm}+\IP(\Omega_{\mathcal{N}}^c). \label{eq:en_emp_proc}
\end{align}
We study \eqref{eq:en_emp_proc} by using the union bound
\begin{align}
	&\scalebox{0.95}{$\displaystyle{\IP\Bigg(\sup_{\overline{\gamma}\in K_{\gamma,n,\eta}}\left|\sum_{i,j=1}^nC_{n,ij}^*\int_0^T\frac{2\int_0^{t-}\int_0^1\frac{d}{d\gamma}g(t-r;s\overline{\gamma}+(1-s)\gamma_n^*)dsdN_{n,j}(r)}{nT}dM_{n,i}(t)\right|>\frac{e_n}{2}, \Omega_{\mathcal{N}}\Bigg)}$} \nonumber \\
	&\leq \scalebox{0.9}{$\displaystyle{\frac{K_1}{\eta}\sup_{\overline{\gamma}\in K_{\gamma,n,\eta}}\IP\Bigg(\left|\sum_{i,j=1}^nC_{n,ij}^*\int_0^T\frac{2\int_0^{t-}\int_0^1\frac{d}{d\gamma}g(t-r;s\overline{\gamma}+(1-s)\gamma_n^*)dsdN_{n,j}(r)}{nT}dM_{n,i}(t)\right| >\frac{e_n}{2}, \Omega_{\mathcal{N}}\Bigg).}$} \label{eq:en_emp_proc2}
\end{align}
The above can be handled by using Theorem 3 in \citet{HRBR15} in its multivariate version. We let to this end,
$$H_i(t):=\frac{4}{nT}\sum_{j=1}^nC_{n,ij}^*\int_0^{t-}\int_0^1\frac{d}{d\gamma}g(t-r;s\overline{\gamma}+(1-s)\gamma_n^*)dsdN_{n,j}(r)\Ind(t\leq\tau_n).$$
By definition of $\tau_n$, the integral condition is fulfilled. Moreover,
$$\sup_{i=1,...,n}\sup_{t\in[0,\tau_n]}|H_i(t)|\leq\frac{4L_g\mathcal{N}\sup_{i=1,...,n}\|C_{n,i\cdot}^*\|_1}{nT}\leq\frac{4L_g\mathcal{N}c_n}{nT}=:B.$$
Define next for $x=(1+\alpha_4)\log nT$
\begin{align*}
	&\scalebox{0.9}{$\displaystyle{\hat{V}_e^{\mu} := \frac{\mu}{\mu-\phi(\mu)}\sum_{i=1}^n\int_0^{\tau_n}\left(\sum_{j=1}^nC_{n,ij}^*\frac{4\int_0^{t-}\int_0^1\frac{d}{d\gamma}g(t-r;s\overline{\gamma}+(1-s)\gamma_n^*)dsdN_{n,j}(r)}{nT}\right)^2dN_{n,i}(t)}$} \nonumber \\ 
	&\hspace*{7cm} +\scalebox{0.95}{$\displaystyle{\frac{16L_g^2\mathcal{N}^2c_n^2x}{(\mu-\phi(\mu))n^2T^2}.}$}
\end{align*}
On $\Omega_{\mathcal{N}}$ it holds that $\tau_n=T$ and hence
\begin{align*}
	w:&= \frac{16L_g^2\mathcal{N}^2c_n^2x}{(\mu-\phi(\mu))n^2T^2}\leq \hat{V}_e^{\mu} \\
	&\leq\frac{\mu}{\mu-\phi(\mu)}\frac{16L_g^2\mathcal{N}^3(T+A)}{An^2T^2}\sum_{i=1}^n\|C_{n,i\cdot}^*\|_1^2+\frac{16L_g^2\mathcal{N}^2c_n^2x}{(\mu-\phi(\mu))n^2T^2} \\
	&\leq\frac{\mu}{\mu-\phi(\mu)}\frac{16L_g^2\mathcal{N}^3(T+A)c_n^2}{AnT^2}+\frac{16L_g^2\mathcal{N}^2c_n^2x}{(\mu-\phi(\mu))n^2T^2}=:v.
\end{align*}
Thus, Theorem 3 of \citet{HRBR15} implies for $\epsilon=1$
\begin{align*}
	&\IP\Bigg(\sum_{i,j=1}^nC_{n,ij}^*\int_0^T\frac{2\int_0^{t-}\int_0^1\frac{d}{d\gamma}g(t-r;s\overline{\gamma}+(1-s)\gamma_n^*)dsdN_{n,j}(r)}{nT}dM_{n,i}(t)>\frac{e_n}{2},\Omega_{\mathcal{N}}\Bigg) \\
	&\leq \IP\left(\sum_{i=1}^n\int_0^{\tau_n}H_i(t)dM_{n,i}(t)>2\sqrt{\hat{V}_e^{\mu}x}+\frac{Bx}{3},w\leq\hat{V}_e^{\mu}\leq v,\sup_{i=1,...,n}\sup_{t\in[0,\tau_n]}|H_i(t)|\leq B\right) \\
	&\leq 2\left(\log_2\left(\frac{\mu\mathcal{N}(T+A)n}{Ax}+1\right)+1\right)e^{-x}\leq c_4'\log(nT)(nT)^{-(1+\alpha_4)}
\end{align*}
for a suitable constant $c_4'$. Using this in \eqref{eq:en_emp_proc2} yields
\begin{align*}
	&\scalebox{0.9}{$\displaystyle{\IP\Bigg(\sup_{\overline{\gamma}\in K_{\gamma,n,\eta}}\left|\sum_{i,j=1}^nC_{n,ij}^*\int_0^T\frac{2\int_0^{t-}\int_0^1\frac{d}{d\gamma}g(t-r;s\overline{\gamma}+(1-s)\gamma_n^*)dsdN_{n,j}(r)}{nT}dM_{n,i}(t)\right|>\frac{e_n}{2}, \Omega_{\mathcal{N}}\Bigg)}$} \nonumber \\
	&\hspace{2cm}\leq 2K_1\eta^{-1}c_4\log(nT)(nT)^{-(1+\alpha_4)}\leq c_4\frac{\log(nT)\left(1+c_n\right)}{(nT)^{\alpha_4}}
\end{align*}
for a suitable $c_4$.

For \eqref{eq:en_cont}, we again make the convention $d|M_{n,i}|(t)=dN_{n,i}(t)+\lambda_{n,i}(t)dt$. We then get,
\begin{align*}
	&\IP\Bigg(\sup_{\overline{\gamma}\in K_{\gamma}}\Bigg|\frac{2}{nT}\sum_{i,j=1}^nC_{n,ij}^*\int_0^T\int_0^{t-}\int_0^1\frac{d}{d\gamma}g(t-r;s\overline{\gamma}+(1-s)\gamma_n^*) \nonumber \\
	&\qquad\qquad\qquad\qquad-\frac{d}{d\gamma}g(t-r;sP_n(\overline{\gamma})+(1-s)\gamma_n^*)dsdN_{n,j}(r)dM_{n,i}(t)\Bigg|>\frac{e_n}{2},\Omega_{\mathcal{N}}\Bigg) \\
	&\leq \IP\left(\frac{D_g\mathcal{N}}{nT}\sum_{i,j=1}^nC_{n,ij}^*\int_0^Td|M_{n,i}|(t)\eta >\frac{e_n}{2},\Omega_{\mathcal{N}}\right) \\
	&\leq \IP\Bigg(\frac{D_g\mathcal{N}}{nT}\sum_{i,j=1}^nC_{n,ij}^*\Bigg(\int_0^TdN_{n,i}(t) \\
	&\qquad+\int_0^T\alpha_{n,i}\nu_0(X_{n,i}(t);\beta_n^*)+\sum_{j=1}^nC_{n,ij}^*\int_0^{t-}g(t-r;\gamma_n^*)dN_{n,j}(r)dt\Bigg)\eta >\frac{e_n}{2},\Omega_{\mathcal{N}}\Bigg) \\
	&\leq \IP\left(\frac{D_g\mathcal{N}}{nT}\sum_{i,j=1}^nC_{n,ij}^*\left(\frac{T+A}{A}\mathcal{N}+T\left(K_{\alpha}\overline{\nu}_i+\|C_{n,i\bdot}\|_1\overline{g}\mathcal{N}\right)\right)\eta >\frac{e_n}{2},\Omega_{\mathcal{N}}\right) \\
	&\leq \IP\left(\frac{D_g\mathcal{N}}{nT}\left(\|C_n^*\|_1\left(\frac{T+A}{A}\mathcal{N}+TK_{\alpha}\|\overline{\nu}\|_{\infty}\right)+\sum_{i=1}^n\|C_{n,i\cdot}^*\|_1^2\overline{g}\mathcal{N}T\right)\eta >\frac{e_n}{2},\Omega_{\mathcal{N}}\right) \\
	&\leq \IP\left(\frac{D_g\mathcal{N}c_n}{T}\left(\frac{T+A}{A}\mathcal{N}+TK_{\alpha}\|\overline{\nu}\|_{\infty}+c_n\overline{g}\mathcal{N}T\right)\eta >\frac{e_n}{2},\Omega_{\mathcal{N}}\right) \\
	&\leq \IP\left(c_4''c_n\left(1+c_n\right)\log^2(nT)\cdot\eta >\frac{e_n}{2},\Omega_{\mathcal{N}}\right),
\end{align*}
where the last inequality holds by definition of $c_4''$. But $\eta$ is chosen such that the above probability equals zero.

\section{Proofs of Section \ref{subsec:results_debiasing}}
\label{sup:results_debiasing}
\begin{proof}[Proof of Lemma \ref{lem:Rn}]
	Since, $R_{n,ab}=\Sigma_{n,ab}(C,\alpha,\theta)$ for some intermediate parameters that are the same within each row of $R_n$, we may ignore their dependence on the row and may simply study the row-wise difference between $\Sigma_n(C_1,\alpha_1,\theta_1)$ and $\Sigma_n(C,\alpha,\theta)$ for a parameter $(C,\alpha,\theta)$ that lies between $(C_1,\alpha_1,\theta_1)$ and $(C_2,\alpha_2,\theta_2)$. Recall that the entries of matrix $\Sigma_n$ (and thus the entries of $R_n$) are of the form
	\begin{align*}
		&\partial_a\partial_b\textrm{LS}_i(C,\alpha,\theta) \\
		&= \int_0^T2\partial_a\Psi_{n,i}(t;C,\alpha,\theta)\partial_b\Psi_{n,i}(t;C,\alpha,\theta)+2\Psi_{n,i}(t;C,\alpha,\beta)\partial_a\partial_b\Psi_{n,i}(t;C,\alpha,\theta)dt \\
		&\qquad-2\int_0^T\partial_a\partial_b\Psi_{n,i}(t;C,\alpha,\theta)dN_{n,i}(t),
	\end{align*}
	where $a$ and $b$ run through all the parameters. Also recall  the definition of $\Omega_{\mathcal{N}}$ from Lemma \ref{lem:omega_lemma}. 
	
	Proving the upper bound from the lemma is tedious but straight-forward. In order to focus on the rates, we let $K$ denote a constant that may change from line to line. It might depend on all constants mentioned in the lemma but not on $n$. We use the expressions for the derivatives of $\Psi_{n,i}$ that were computed at the beginning of the chapter and the fact that many second derivatives vanish. 
	
	On $\Omega_n$, we have by using Assumptions~(A1)-(A3) that for some constant $K>0$,
	\begin{align*}
		&\left|\Psi_{n,i}(t;C_1,\alpha_1,\theta_1)-\Psi_{n,i}(t;C,\alpha,\theta)\right| \\
		&= \Big|\alpha_{1,i}\nu_0(X_{n,i}(t);\beta_1)-\alpha_i\nu_0(X_{n,i}(t);\beta) \\
		&\qquad\qquad\qquad+\sum_{j=1}^n\int_0^{t-}\left(C_{1,ij}g(t-r;\gamma_1)-C_{ij}g(t-r;\gamma)\right)dN_{n,j}(r)\Big| \\
		&\leq K\big(|\alpha_{1,i}-\alpha_i|+\|\beta_1-\beta\|_1+\|C_{1,i\bdot}-C_{i\bdot}\|_1\mathcal{N}+|\gamma_1-\gamma|\mathcal{N}\|C_{1,i\bdot}\|_1\big),
	\end{align*}
	and 
	\begin{align*}
		&\left|\Psi_{n,i}(t;C,\alpha,\theta)\right|\leq K\left(1+\|C_{i\bdot}\|_1\mathcal{N}\right).
	\end{align*}
	
	Denote by $\nu_{0,r}$ the derivative of $\nu_0$ with respect to $\beta_r$, and by $\nu_{0,rq}$ the second derivative of $\nu_0$ with respect to $\beta_r$ and $\beta_q$. In the following, we will bound the expression inside the maximum from the formulation the lemma, assuming that $\Omega_{\mathcal{N}}$ holds, for each possible choice of $a$ and $b$. Since $\Sigma$ is symmetric and the arguments do not depend on the specifically chosen intermediate point, we may restrict to the upper triangular matrix.
	
	\underline{Let $a=\alpha_k$ and $b=\alpha_l$ for some $k,l\in\{1,...,n\}$:}
	
	We will give a more detailed argument in this case.
	Note that by the Assumptions~(A1) and~(A3), we have that
	\begin{align*}
		&\left|\nu_0(X_k(t);\beta_1)^2-\nu_0(X_k(t);\beta)^2\right| \\
		=&\left|\nu_0(X_k(t);\beta_1)-\nu_0(X_k(t);\beta)\right|\cdot\left|\nu_0(X_k(t);\beta_1)+\nu_0(X_k(t);\beta)\right| \\
		\leq&2\|\overline{\nu}\|_{\infty}L_\nu\cdot\|\beta_1-\beta\|_2\leq2\|\overline{\nu}\|_{\infty} L_\nu C\|\beta_1-\beta\|_1,
	\end{align*}
	where $C$ is the universal constant for which $\|x\|_2\leq C\|x\|_1$ for all $x\in\IR^p$. Then, with $K:=2\|\overline{\nu}\|_{\infty} L_\nu C$ we obtain
	\begin{align*}
		\left|R_{n,ab}-\Sigma_{n,ab}(C_1,\alpha_1,\theta_1)\right|&= \left|\frac{2}{nT}\int_0^T\nu_0(X_k(t);\beta_1)^2-\nu_0(X_k(t);\beta)^2dt\Ind(k=l)\right| \\
		&\leq \frac{K}{n}\|\beta_1-\beta\|_1.
	\end{align*}
	
	\underline{Let $a=\alpha_k$ and $b=C_{xy}$ for some $k,x,y\in\{1,...,n\}$:}
	\begin{align*}
		\left|R_{n,ab}-\Sigma_{n,ab}(C_1,\alpha_1,\theta_1)\right|&= \Bigg|\frac{2}{nT}\int_0^T\nu_0(X_k(t);\beta_1)\int_0^{t-}g(t-r;\gamma_1)dN_{n,y}(r) \\
		&\qquad-\nu_0(X_k(t);\beta)\int_0^{t-}g(t-r;\gamma)dN_{n,y}(r)dt\Ind(x=k)\Bigg| \\
		&\leq \frac{K\mathcal{N}}{n}\|\beta_1-\beta\|_1+\frac{K\mathcal{N}}{n}|\gamma_1-\gamma|.
	\end{align*}
	
	\underline{Let $a=\alpha_k$ and $b=\beta_r$ for some $k\in\{1,...,n\}$ and $r\in\{1,...,p\}$:}
	\begin{align*}
		&\left|R_{n,ab}-\Sigma_{n,ab}(C_1,\alpha_1,\theta_1)\right| \\
		&= \Bigg|\frac{2}{nT}\Bigg(\int_0^T\Big[\nu_0(X_k(t);\beta_1)\alpha_{1,k}\nu_{0,r}(X_k(t);\beta_1)-\nu_0(X_k(t);\beta)\alpha_k\nu_{0,r}(X_k(t);\beta)\Big]dt  \\
		&\qquad+\int_0^T\Big[\Psi_{n,k}(t;C_1,\alpha_1,\theta_1)\nu_{0,r}(X_k(t);\beta_1)-\Psi_{n,k}(t;C,\alpha,\theta)\nu_{0,r}(X_k(t);\beta)\Big]dt \\
		&\qquad-\int_0^T\Big[\nu_{0,r}(X_k(t);\beta_1)-\nu_{0,r}(X_k(t);\beta)\Big]dN_{n,k}(t) \Bigg)\Bigg| \\
		&\leq \frac{K\mathcal{N}(1+\|C_{1,k\bdot}\|_1)}{n}\|\beta_1-\beta\|_1+\frac{K}{n}|\alpha_{1,k}-\alpha_k|+\frac{K\mathcal{N}}{n}\|C_{1,k\cdot}-C_{k\bdot}\|_1+\frac{K\mathcal{N}\|C_{1,k\bdot}\|_1}{n}|\gamma_1-\gamma|
	\end{align*}
	
	\underline{Let $a=\alpha_k$ and $b=\gamma$:}
	\begin{align*}
		&\left|R_{n,ab}-\Sigma_{n,ab}(C_1,\alpha_1,\theta_1)\right| \\
		&= \Bigg|\frac{2}{nT}\int_0^T\nu_0(X_k(t);\beta_1)\sum_{j=1}^nC_{1,kj}\int_0^{t-}g'(t-r);\gamma_1)dN_{n,j}(r) \\
		&\hspace*{5cm}-\nu_0(X_k(t);\beta)\sum_{j=1}^nC_{kj}\int_0^{t-1}g'(t-r;\gamma)dN_{n,j}(r)dt\Bigg| \\
		&\leq \frac{K\|C_{1,k\bdot}\|_1\mathcal{N}}{n}\|\beta_1-\beta\|_1+\frac{K\|C_{1,k\bdot}\|_1\mathcal{N}}{n}|\gamma_1-\gamma|+\frac{K\mathcal{N}}{n}\|C_{1,k\bdot}-C_{k\bdot}\|_1.
	\end{align*}
	
	\underline{Let $a=C_{xy}$ and $b=C_{x'y'}$ for some $x,y,x',y'\in\{1,...,n\}$:}
	\begin{align*}
		&\left|R_{n,ab}-\Sigma_{n,ab}(C_1,\alpha_1,\theta_1)\right| \\
		&= \Bigg|\frac{2}{nT}\int_0^T\int_0^{t-}g(t-r;\gamma_1)dN_{n,y}(r)\int_0^{t-}g(t-r;\gamma_1)dN_{n,y'}(r) \\
		&\qquad-\int_0^{t-}g(t-r;\gamma)dN_{n,y}(r)\int_0^{t-}g(t-r;\gamma)dN_{n,y'}(r)dt\Ind(x=x')\Bigg|\leq\frac{K\mathcal{N}^2}{n}|\gamma_1-\gamma|
	\end{align*}
	
	\underline{Let $a=C_{xy}$ and $b=\beta_r$ for some $x,y\in\{1,...,n\}$ and $r\in\{1,...,p\}$:}
	\begin{align*}
		&\left|R_{n,ab}-\Sigma_{n,ab}(C_1,\alpha_1,\theta_1)\right| \\
		&= \Bigg|\frac{2}{nT}\int_0^T\int_0^{t-}g(t-r;\gamma_1)dN_{n,y}(r)\alpha_{1,x}\nu_{0,r}(X_x(t);\beta_1) \\
		&\qquad\qquad\qquad-\int_0^{t-}g(t-r;\gamma)dN_{n,y}(r)\alpha_x\nu_{0,r}(X_x(t);\beta)dt\Bigg| \\
		&\leq \frac{K\mathcal{N}}{n}|\gamma_1-\gamma|+\frac{K\mathcal{N}}{n}|\alpha_{1,x}-\alpha_x|+\frac{K\mathcal{N}}{n}\|\beta_1-\beta\|_1
	\end{align*}
	
	\underline{Let $a=C_{xy}$ and $b=\gamma$ for some $x,y\in\{1,...,n\}$:}
	\begin{align*}
		&\left|R_{n,ab}-\Sigma_{n,ab}(C_1,\alpha_1,\theta_1)\right| \\
		&= \Bigg|\frac{2}{nT}\Bigg(\int_0^T\int_0^{t-}g(t-r;\gamma_1)dN_{n,y}(r)\Big[\sum_{j=1}^nC_{1,xj}\int_0^{t-}g'(t-r;\gamma_1)\Big]dN_{n,j}(r) \\
		&\qquad\qquad\qquad-\int_0^{t-}g(t-r;\gamma)dN_{n,y}(r)\Big[\sum_{j=1}^nC_{xj}\int_0^{t-}g'(t-r;\gamma)\Big]dN_{n,j}(r)dt \\
		&\qquad+\int_0^T\Psi_{n,x}(t;C_1,\alpha_1,\theta_1)\int_0^{t-}g'(t-r;\gamma_1)dN_{n,y}(r) \\
		&\qquad\qquad\qquad-\Psi_{n,x}(t;C,\alpha,\theta)\int_0^{t-}g'(t-r;\gamma)dN_{n,y}(r)dt \\
		&\qquad-\int_0^T\int_0^{t-}g'(t-r;\gamma_1)dN_{n,y}(r)-\int_0^{t-}g'(t-r;\gamma)dN_{n,y}(r)dN_{n,x}(t)\Bigg)\Bigg|\Bigg) \\
		&\leq \frac{K\mathcal{N}^2(1+\|C_{1,x\bdot}\|_1)}{n}|\gamma_1-\gamma|+\frac{K\mathcal{N}^2}{n}\|C_{1,x\bdot}-C_{x\bdot}\|_1+\frac{K\mathcal{N}}{n}|\alpha_{1,x}-\alpha_x|+\frac{K\mathcal{N}}{n}\|\beta_1-\beta\|_1
	\end{align*}
	
	\underline{Let $a=\beta_r$ and $b=\beta_q$ for some $r,q\in\{1,...,p\}$:}
	\begin{align*}
		&\left|R_{n,ab}-\Sigma_{n,ab}(C_1,\alpha_1,\theta_1)\right| \\
		&= \scalebox{0.95}{$\displaystyle{\Bigg|\frac{2}{nT}\sum_{i=1}^n\Bigg(\int_0^T\Big[\alpha_{1,i}^2\nu_{0,r}(X_{n,i}(t);\beta_1)\nu_{0,q}(X_{n,i}(t);\beta_1)-\alpha_i^2\nu_{0,r}(X_{n,i}(t);\beta)\Big]\nu_{0,q}(X_{n,i}(t);\beta)dt}$} \\
		&\qquad+\int_0^T\Big[\Psi_{n,i}(t;C_1,\alpha_1,\theta_1)\alpha_{1,i}\nu_{0,rq}(X_{n,i}(t);\beta_1)-\Psi_{n,i}(t;C,\alpha,\theta)\alpha_i\nu_{0,rq}(X_{n,i}(t);\beta)\Big]dt \\
		&\qquad-\int_0^T\Big[\alpha_{1,i}\nu_{0,rq}(X_{n,i}(t);\beta_1)-\alpha_i\nu_{0,rq}(X_{n,i}(t);\beta)\Big]dN_{n,i}(t)\Bigg)\Bigg| \\
		&\leq \frac{K(1+\max_{i=1,...,n}\|C_{1,i\bdot}\|_1)\mathcal{N}}{n}\|\alpha_1-\alpha\|_1+K\left(1+\frac{1}{n}\|C_1\|_1\right)\mathcal{N}\|\beta_1-\beta\|_1 \\
		&\qquad\qquad+\frac{K\mathcal{N}}{n}\|C_1-C\|_1+\frac{K\mathcal{N}\|C_1\|_1}{n}|\gamma_1-\gamma|
	\end{align*}
	
	\underline{Let $a=\beta_r$ and $b=\gamma$ for some $r\in\{1,...,p\}$:}
	\begin{align*}
		&\left|R_{n,ab}-\Sigma_{n,ab}(C_1,\alpha_1,\theta_1)\right| \\
		&= \Bigg|\frac{2}{nT}\sum_{i=1}^n\int_0^T\alpha_{1,i}\nu_{0,r}(X_{n,i}(t);\beta_1)\sum_{j=1}^nC_{1,ij}\int_0^{t-}g'(t-r;\gamma_1)dN_{n,j}(r) \\
		&\hspace*{4cm}-\alpha_i\nu_{0,r}(X_{n,i}(t);\beta)\sum_{j=1}^nC_{ij}\int_0^{t-}g'(t-r;\gamma)dN_{n,j}(r)dt\Bigg|\Bigg) \\
		&\leq \frac{K}{n}\max_{i=1,...,n}\|C_{1,i\bdot}\|_1\mathcal{N}\|\alpha_1-\alpha\|_1+\frac{K\|C_1\|_1}{n}\mathcal{N}\|\beta_1-\beta\|_1+\frac{K\mathcal{N}}{n}\|C_1-C\|_1 \\
		&\qquad\qquad\qquad\qquad\qquad\qquad\qquad\qquad\qquad\qquad\qquad\qquad+\frac{K\|C_1\|_1}{n}\mathcal{N}|\gamma_1-\gamma|
	\end{align*}
	
	\underline{Let $a=\gamma$ and $b=\gamma$:}
	\begin{align*}
		&\left|R_{n,ab}-\Sigma_{n,ab}(C_1,\alpha_1,\theta_1)\right| \\
		&= \scalebox{0.9}{$\displaystyle{\Bigg|\frac{2}{nT}\sum_{i=1}^n\Bigg[\int_0^T\Big(\sum_{j=1}^nC_{1,ij}\int_0^{t-}g'(t-r;\gamma_1)dN_{n,j}(r)\Big)^2 -\Big(\sum_{j=1}^nC_{ij}\int_0^{t-}g'(t-r;\gamma)dN_{n,j}(r)\Big)^2dt}$} \\
		&\hspace*{2cm}+\int_0^T\Psi_{n,i}(t;C_1,\alpha_1,\theta_1)\sum_{j=1}^nC_{1,ij}\int_0^{t-}g''(t-r;\gamma_1)dN_{n,j}(r) \\
		&\hspace*{3cm}-\Psi_{n,i}(t;C,\alpha,\theta)\sum_{j=1}^nC_{ij}\int_0^{t-}g''(t-r;\gamma)dN_{n,j}(r)dt \\
		&-\int_0^T\sum_{j=1}^nC_{1,ij}\int_0^{t-}g''(t-r;\gamma_1)dN_{n,j}(r)-\sum_{j=1}^nC_{ij}\int_0^{t-}g''(t-r;\gamma)dN_{n,j}(r)dN_{n,i}(t)\Bigg]\Bigg| \\
		&\leq \frac{K\mathcal{N}\max_{i=1,...,n}(1,\mathcal{N}\|C_{1,i\bdot}\|_1,\mathcal{N}\|C_{i\bdot}\|_1)}{n}\|C_1-C\|_1 \\
		&+K\mathcal{N}^2\frac{1}{n}\sum_{i=1}^n\|C_{1,i\bdot}\|_1\max(\|C_{1,i\bdot}\|_1,\|C_{i\bdot}\|_1)|\gamma_1-\gamma|+\frac{K\mathcal{N}\max_{i=1,...,n}\|C_{i\bdot}\|_1}{n}\|\alpha_1-\alpha\|_1 \\
		&\qquad\qquad+K\mathcal{N}\frac{1}{n}\|C\|_1\|\beta_1-\beta\|_1.
	\end{align*}
	Combining all of these statements, we find that that on $\Omega_{\mathcal{N}}$, recalling that $(C,\alpha,\theta)$ lies between $(C_1,\alpha_1,\theta_1)$ and $(C_2,\alpha_2,\theta_2)$,
	\begin{align*}
		&\max_{a,b}\big|R_{n,ab}-\Sigma_{n,ab}(C_1,\alpha_1,\theta_1)\big| \\
		&\leq \frac{K\mathcal{N}\left(\max_{i=1,...,n}(\|C_{2,i\bdot}\|_1,\|C_{1,i\bdot}\|_1)+1\right)}{n}\|\alpha_1-\alpha_2\|_1 \\
		&\qquad\qquad\qquad+\frac{K\mathcal{N}^2\max_{i=1,...,n}(1,\|C_{2,i\bdot}\|_1,\|C_{1,i\bdot}\|_1)}{n}\|C_1-C_2\|_1 \\
		&\qquad\qquad\qquad+K\mathcal{N}\left(\frac{1}{n}\max(\|C_1\|_1,\|C_2\|_1)+1\right)\|\beta_1-\beta_2\|_1 \\
		&\qquad\qquad\qquad+\frac{K\mathcal{N}^2}{n}\left(\sum_{i=1}^n\|C_{1,i\bdot}\|_1\max(\|C_{1,i\bdot}\|_1,\|C_{2,i\bdot}\|_1,1)+1\right)|\gamma_1-\gamma_2| \\
		&\leq K\mathcal{N}^2\max_{i=1,...,n}(\|C_{1,i\bdot}\|_1,\|C_{2,i\bdot}\|_1,1) \\
		&\qquad\qquad\qquad\times\Big(\frac{1}{n}\|\alpha_1-\alpha_2\|_1+\frac{1}{n}\|C_1-C_2\|_1+\Big(\frac{1}{n}\sum_{i=1}^n\|C_{1,i\bdot}\|_1+1\Big)\|\theta_1-\theta_2\|_1\Big)
	\end{align*}
	Since $K$ does not depend on $n$ explicitly, the above bound holds uniformly. Recalling that for $\mathcal{N}\approx\log(nT)$, Lemma \ref{lem:omega_lemma} provides $\IP(\Omega_{\mathcal{N}})\to1$, the statement follows.
\end{proof}

\begin{proof}[Proof of Lemma \ref{lem:Dassump}]
	Recall the definition of $\Omega_{\mathcal{N}}$ from Lemma \ref{lem:omega_lemma} for $\mathcal{N}=\mathcal{N}_0\log(nT)$ with a suitable $\mathcal{N}_0>0$. By a union bound argument, we have for any sequence
	$$\epsilon_n:=c_0c_n\log(nT)$$
	($c_0$ will be chosen below),
	\begin{align}
		&\IP\Bigg(\Bigg\|\frac{1}{\sqrt{nT}}\sum_{i=1}^n \int_0^T\begin{pmatrix}
			\partial_{\theta} \\ \partial_{\alpha} \\ \partial_C
		\end{pmatrix}\Psi_{n,i}(t;C_{n,i\cdot}^*,\alpha_{n,i}^*,\theta_n^*)dM_{n,i}(t)\Bigg\|_{\infty}>\epsilon_n\Bigg) \nonumber \\
		&\leq n^2\max_{x,y=1,...,n}\IP\left(\,\left|\frac{1}{\sqrt{nT}}\sum_{i=1}^n \int_0^T\partial_{C_{xy}}\Psi_{n,i}(t;C_{n,i\bdot}^*,\alpha_{n,i}^*,\theta_n^*)dM_{n,i}(t)\right|>\epsilon_n,\Omega_{\mathcal{N}}\right) \label{eq:row1} \\
		&+n\max_{k=1,...,n}\IP\left(\,\left|\frac{1}{\sqrt{nT}}\sum_{i=1}^n \int_0^T\partial_{\alpha_k}\Psi_{n,i}(t;C_{n,i\bdot}^*,\alpha_{n,i}^*,\theta_n^*)dM_{n,i}(t)\right|>\epsilon_n,\Omega_{\mathcal{N}}\right) \label{eq:row2} \\
		&+p\max_{r=1,...,p}\IP\left(\,\left|\frac{1}{\sqrt{nT}}\sum_{i=1}^n \int_0^T\partial_{\beta_r}\Psi_{n,i}(t;C_{n,i\bdot}^*,\alpha_{n,i}^*,\theta_n^*)dM_{n,i}(t)\right|>\epsilon_n,\Omega_{\mathcal{N}}\right) \label{eq:row3} \\
		&+\IP\left(\,\left|\frac{1}{\sqrt{nT}}\sum_{i=1}^n \int_0^T\partial_{\gamma}\Psi_{n,i}(t;C_{n,i\bdot}^*,\alpha_{n,i}^*,\theta_n^*)dM_{n,i}(t)\right|>\epsilon_n,\Omega_{\mathcal{N}}\right)+\IP(\Omega_{\mathcal{N}}^c). \label{eq:row4}
	\end{align}
	Since $\IP(\Omega_{\mathcal{N}}^c)\to0$ by Lemma \ref{lem:omega_lemma}, we have left to prove that the remaining probabilities converge to zero. All expressions above can be handled using Theorem 3 in \cite{HRBR15} and using the computations of the derivatives of $\Psi_{n,i}$ from the beginning of this chapter. To this end, fix $\mu\in(0,3)$ as in Theorem 3 in \cite{HRBR15} (cf. also $\mu$ from Lemma \ref{lem:lambda_rate}). Let furthermore $\tau_n$ be defined as in the proof of Lemma \ref{lem:lambda_rate}. For \eqref{eq:row1}, we define
	$$H_i(t):=\frac{1}{\sqrt{nT}}\partial_{C_{xy}}\Psi_{n,i}(t;C_{n,i\bdot}^*,\alpha_{n,i}^*,\theta_n^*)=\frac{1}{\sqrt{nT}}\int_0^{t-}g(t-r;\gamma_n^*)dN_{n,y}(r)\Ind(i=x)\Ind(t\leq\tau_n)$$
	We see that only $H_x$ is different from zero. We may therefore set $M=1$ in Theorem 3 of \cite{HRBR15}. It holds that $|H_x(t)|\leq \frac{K\mathcal{N}}{\sqrt{nT}}:=B$ for some constant $K$ and the integral condition holds. We then define for $x=c\log(nT)$ ($c>2$)
	$$\hat{V}^{\mu}:=\frac{\mu}{\mu-\phi(\mu)}\int_0^{\tau_n}\frac{1}{nT}\left(\int_0^{t-}g(t-r;\gamma_n^*)dN_{n,y}(r)\right)^2dN_{n,x}(t)+\frac{K^2\mathcal{N}^2x}{(\mu-\phi(\mu))nT}.$$
	On $\Omega_{\mathcal{N}}$, it holds that $\tau_n=T$ and furthermore for a suitable $K'>0$
	$$w:=\frac{K^2\mathcal{N}^2x}{(\mu-\phi(\mu))nT}\leq\hat{V}^{\mu}\leq\frac{\mu K'\mathcal{N}^3}{(\mu-\phi(\mu))nA}+\frac{K^2\mathcal{N}^2x}{(\mu-\phi(\mu))nT}=:v.$$
	Note now that on $\Omega_{\mathcal{N}}$ for $n$ large enough and a suitable constant $K''$
	\begin{align*}
		2\sqrt{\hat{V}^{\mu}x}+\frac{K\mathcal{N}x}{3\sqrt{nT}}\leq K''\frac{\log^2 nT}{\sqrt{n}}\leq\epsilon_n.
	\end{align*}
	We then have by Theorem 3 of \citet{HRBR15} using $\epsilon=1$
	\begin{align*}
		&\IP\left(\left|\frac{1}{\sqrt{nT}}\sum_{i=1}^n \int_0^T\partial_{C_{xy}}\Psi_{n,i}(t;C_{n,i\bdot}^*,\alpha_{n,i}^*,\theta_n^*)dM_{n,i}(t)\right|>\epsilon_n,\Omega_{\mathcal{N}}\right)  \\
		&\leq 2\IP\left(\int_0^{\tau_n}H_x(t)dM_{n,x}(t)>2\sqrt{\hat{V}^{\mu}x}+\frac{K\mathcal{N}x}{3\sqrt{nT}},w\leq\hat{V}^{\mu}\leq v,\sup_{t\in[0,\tau_n]}|H_x(t)|\leq B\right) \\
		&\leq 4\left(\log_2\left(\frac{\mu K'\mathcal{N}T}{AK^2x}+1\right)+1\right)e^{-x}=4\left(\log_2\left(\frac{\mu K'\mathcal{N}_0T}{AK^2c}+1\right)+1\right)(nT)^{-c}.
	\end{align*}
	Since $c>2$, the above implies that \eqref{eq:row1} converges to zero. The argument for \eqref{eq:row2} goes similarly: Define
	$$H_i(t):=\frac{1}{\sqrt{nT}}\partial_{\alpha_k}\Psi_{n,i}(t;C_{n,i\bdot}^*,\alpha_{n,i}^*,\theta_n^*)=\frac{1}{\sqrt{nT}}\nu_0(X_{n,i}(t);\beta_n^*)\Ind(i=k).$$
	Thus, only $H_k$ is different from zero and we may again use $M=1$ in Theorem 3 of \cite{HRBR15}. It holds that $|H_k(t)|\leq \frac{K}{\sqrt{nT}}:=B$ for some constant $K$ for all $t\in[0,T]$. The integral condition of Theorem 3 is hence fulfilled. We then define
	$$\hat{V}^{\mu}:=\frac{\mu}{\mu-\phi(\mu)}\int_0^T\frac{1}{nT}\nu_0(X_k(t);\beta_n^*)^2dN_{n,k}(t)+\frac{K^2x}{(\mu-\phi(\mu))nT}.$$
	On $\Omega_{\mathcal{N}}$ it holds for a constant $K'>0$,
	$$w:=\frac{K^2x}{(\mu-\phi(\mu))nT}\leq\hat{V}^{\mu}\leq\frac{\mu K'\mathcal{N}}{(\mu-\phi(\mu))nA}+\frac{K^2x}{(\mu-\phi(\mu))nT}=:v.$$
	On $\Omega_{\mathcal{N}}$, for $n$ large enough and for a suitable constant $K''$,
	\begin{align*}
		2\sqrt{\hat{V}^{\mu}x}+\frac{Kx}{3\sqrt{nT}}\leq K''\frac{\log nT}{\sqrt{n}}\leq\epsilon_n.
	\end{align*}
	We then have by Theorem 3 of \citet{HRBR15} using $\epsilon=1$
	\begin{align*}
		&\IP\left(\left|\frac{1}{\sqrt{nT}}\sum_{i=1}^n \int_0^T\partial_{\alpha_k}\Psi_{n,i}(t;C_{n,i\bdot}^*,\alpha_{n,i}^*,\theta_n^*)dM_{n,i}(t)\right|>\epsilon_n,\Omega_{\mathcal{N}}\right)  \\
		&\leq 2\IP\left(\int_0^{\tau_n}H_k(t)dM_{n,k}(t)>2\sqrt{\hat{V}^{\mu}x}+\frac{Kx}{3\sqrt{nT}},w\leq\hat{V}^{\mu}\leq v,\sup_{t\in[0,T]}|H_k(t)|\leq B\right) \\
		&\leq 4\left(\log_2\left(\frac{\mu K'\mathcal{N}T}{AK^2x}+1\right)+1\right)e^{-x}=4\left(\log_2\left(\frac{\mu K'\mathcal{N}_0T}{AK^2c}+1\right)+1\right)(nT)^{-c}.
	\end{align*}
	Hence, \eqref{eq:row2} converges to zero because $c>2$. We study now \eqref{eq:row3}. The proof follows the same ideas. Recall that we denote the first derivative of $\nu_0$ with respect to $\beta_r$ by $\nu_{0,r}$. Let
	$$H_i(t):=\frac{1}{\sqrt{nT}}\partial_{\beta_r}\Psi_{n,i}(t;C_{n,i\bdot}^*,\alpha_{n,i}^*,\theta_n^*)=\frac{1}{\sqrt{nT}}\alpha_{n,i}^*\nu_{0,r}(X_{n,i}(t);\beta_n^*).$$
	This time all $H_i$ are different from zero and we use the multivariate version of Theorem 3 of \cite{HRBR15}, that is, $M=n$. It holds that $|H_i(t)|\leq \frac{K}{\sqrt{nT}}:=B$ for some constant $K$ for all $t\in[0,T]$. The integral condition of Theorem 3 is hence fulfilled. We then define
	$$\hat{V}^{\mu}:=\frac{\mu}{\mu-\phi(\mu)}\sum_{i=1}^n\int_0^T\frac{1}{nT}(\alpha_{n,i}^*)^2\nu_{0,r}(X_{n,i}(t);\beta_n^*)^2dN_{n,i}(t)+\frac{K^2x}{(\mu-\phi(\mu))nT}.$$
	On $\Omega_{\mathcal{N}}$ it holds for a constant $K'>0$,
	$$w:=\frac{K^2x}{(\mu-\phi(\mu))nT}\leq\hat{V}^{\mu}\leq\frac{\mu K'\mathcal{N}\|\alpha_n^*\|_1}{(\mu-\phi(\mu))nA}+\frac{K^2x}{(\mu-\phi(\mu))nT}=:v.$$
	On $\Omega_{\mathcal{N}}$ for $n$ large enough and a suitable constant $K''$, and a good choice of $c_0$
	\begin{align*}
		2\sqrt{\hat{V}^{\mu}x}+\frac{Kx}{3\sqrt{nT}}\leq K''\log nT\leq\epsilon_n.
	\end{align*}
	We then have by Theorem 3 of \citet{HRBR15} using $\epsilon=1$
	\begin{align*}
		&\IP\Bigg(\Big|\frac{1}{\sqrt{nT}}\sum_{i=1}^n\int_0^T\partial_{\beta_r}\Psi_{n,i}(t;C_{n,i\bdot}^*,\alpha_{n,i}^*,\theta_n^*)dM_{n,i}(t)\Big|>\epsilon_n,\Omega_{\mathcal{N}}\Bigg)  \\
		&\leq 2\IP\Bigg(\sum_{i=1}^n\int_0^TH_i(t)dM_{n,i}(t)>2\sqrt{\hat{V}^{\mu}x}+\frac{Kx}{3\sqrt{nT}},w\leq\hat{V}^{\mu}\leq v,\underset{i=1,...,n}{\sup_{t\in[0,T]}}|H_i(t)|\leq B\Bigg) \\
		&\leq 4\left(\log_2\left(\frac{\mu K'\mathcal{N}\|\alpha_n^*\|_1T}{AK^2x}+1\right)+1\right)e^{-x}=4\left(\log_2\left(\frac{\mu K'\mathcal{N}_0\|\alpha_n^*\|_1T}{AK^2c}+1\right)+1\right)(nT)^{-c}.
	\end{align*}
	The above implies that \eqref{eq:row3} converges to zero because $\|\alpha_n^*\|_1=O(n)$. Finally, we study \eqref{eq:row4}. Let this time
	$$H_i(t):=\frac{1}{\sqrt{nT}}\partial_{\gamma}\Psi_{n,i}(t;C_{n,i\bdot}^*,\alpha_{n,i}^*,\theta_n^*)=\frac{1}{\sqrt{nT}}\sum_{j=1}^nC_{n,ij}^*\int_0^{t-}g'(t-r;\gamma_n^*)dN_{n,j}(r)\Ind(t\leq\tau_n).$$
	Again, all $H_i$ are different from zero and we use $M=n$ in Theorem 3 of \cite{HRBR15}. We have $|H_i(t)|\leq \frac{KS\mathcal{N}}{\sqrt{nT}}:=B$ for $S:=(1\vee\max_{i=1,...,n}\|C_{n,i\bdot}^*\|_1)$ and some constant $K$ and hence the integral condition of Theorem 3 is fulfilled. Define
	\begin{align*}
		&\hat{V}^{\mu}:=\frac{\mu}{\mu-\phi(\mu)}\sum_{i=1}^n\int_0^{\tau_n}\frac{1}{nT}\left(\sum_{j=1}^nC_{n,ij}^*\int_0^{t-}g'(t-r;\gamma_n^*)dN_{n,j}(r)\right)^2dN_{n,i}(t) \\
		&\qquad\qquad\qquad\qquad\qquad\qquad\qquad\qquad\qquad\qquad\qquad\qquad\qquad+\frac{K^2S^2\mathcal{N}^2x}{(\mu-\phi(\mu))nT}.
	\end{align*}
	On $\Omega_{\mathcal{N}}$ it holds for a constant $K'>0$,
	$$w:=\frac{K^2S^2\mathcal{N}^2x}{(\mu-\phi(\mu))nT}\leq\hat{V}^{\mu}\leq\frac{\mu K'\mathcal{N}^3S^2}{(\mu-\phi(\mu))A}+\frac{K^2S^2\mathcal{N}^2x}{(\mu-\phi(\mu))nT}=:v.$$
	On $\Omega_{\mathcal{N}}$ for $n$ large enough and a suitable constant $K''$ and after possibly increasing $c_0$
	\begin{align*}
		2\sqrt{\hat{V}^{\mu}x}+\frac{KS\mathcal{N}x}{3\sqrt{nT}}\leq K''S\log nT\leq\epsilon_n.
	\end{align*}
	We then have by Theorem 3 of \citet{HRBR15} using $\epsilon=1$
	\begin{align*}
		&\IP\Bigg(\Big|\frac{1}{\sqrt{nT}}\sum_{i=1}^n\int_0^T\partial_{\gamma}\Psi_{n,i}(t;C_{n,i\bdot}^*,\alpha_{n,i}^*,\theta_n^*)dM_{n,i}(t)\Big|>\epsilon_n,\Omega_{\mathcal{N}}\Bigg)  \\
		&\leq 2\IP\Bigg(\sum_{i=1}^n\int_0^{\tau_n}H_i(t)dM_{n,i}(t)>2\sqrt{\hat{V}^{\mu}x}+\frac{KS\mathcal{N}x}{3\sqrt{nT}},w\leq\hat{V}^{\mu}\leq v,\underset{i=1,...,n}{\sup_{t\in[0,T]}}|H_i(t)|\leq B\Bigg) \\
		&\leq 4\left(\log_2\left(\frac{\mu K'\mathcal{N}nT}{AK^2x}+1\right)+1\right)e^{-x}=4\left(\log_2\left(\frac{\mu K'\mathcal{N}_0nT}{AK^2c}+1\right)+1\right)(nT)^{-c},
	\end{align*}
	which converges to zero. This completes the proof of the lemma.
\end{proof}

\section{Proofs of Section \ref{subsec:consistency}}
\label{sup:consistency}
\begin{proof}[Proof of Theorem \ref{thm:oracle}]
	The proof runs exactly along the lines of the proof of Theorem 6.2 in \citet{GB11}. However, for completeness we repeat it here in our setting. We write
	\begin{align*}
		\mathcal{E}_i(C,\alpha,\theta):&= \|\Psi_{n,i}(\cdot;C_{i\bdot},\alpha_i,\theta)-\lambda_{i}\|_2^2-\|\lambda_{n,i}(\cdot)\|_2^2 \\
		& =\textrm{LS}_i(C_{i\bdot},\alpha_i,\theta)+2\int_0^T\Psi_{n,i}(\cdot;C_{i\cdot},\alpha_i,\theta)dM_i(t),
	\end{align*}
	so that $\mathcal{E} (C,\nu,\theta)= \sum_{i=1}^n \mathcal{E}_i(C,\nu,\theta).$ We have
	\begin{align*}
		&\frac{1}{T}\mathcal{E}_i(\hat{C}_n(\theta),\hat{\alpha}_n(\theta),\theta)+2\omega_i\|\hat{C}_{n,i\bdot}(\theta)\|_1 \\
		&= \frac{1}{T}\textrm{LS}_i(\hat{C}_{n,i\bdot}(\theta),\hat{\alpha}_{n,i}(\theta),\theta)+\frac{2}{T}\int_0^T\Psi_{n,i}(t;\hat{C}_{n,i\bdot}(\theta),\hat{\alpha}_{n,i}(\theta),\theta)dM_{n,i}(t)+2\omega_i\|\hat{C}_{n,i\bdot}(\theta)\|_1 \\
		&= \frac{1}{T}\textrm{LS}_i(C_{n,i\bdot}^*(\theta),\alpha_{n,i}^*(\theta),\theta)+\frac{2}{T}\int_0^T\Psi_{n,i}(t;C_{n,i\bdot}^*(\theta),\alpha_{n,i}^*(\theta),\theta)dM_{n,i}(t)+2\omega_i\|C_{n,i\bdot}^*(\theta)\|_1 \\
		&\quad+\frac{1}{T}\textrm{LS}_i(\hat{C}_{n,i\bdot}(\theta),\hat{\alpha}_{n,i}(\theta),\theta)+2\omega_i\|\hat{C}_{n,i\bdot}(\theta)\|_1 \\
		&\quad\quad-\frac{1}{T}\textrm{LS}_i(C_{n,i\bdot}^*(\theta),\alpha_{n,i}^*(\theta),\theta)-2\omega_i\|C_{n,i\bdot}^*(\theta)\|_1 \\
		&\quad\quad+\frac{2}{T}\int_0^T\left[\Psi_{n,i}(t;\hat{C}_{n,i\bdot}(\theta),\hat{\alpha}_{n,i}(\theta),\theta)-\Psi_{n,i}(t;C_{n,i\bdot}^*(\theta),\alpha_{n,i}^*(\theta),\theta)\right]dM_{n,i}(t) \\
		&\leq \frac{1}{T}\mathcal{E}_i(C_n^*(\theta),\alpha_n^*(\theta),\theta)+2\omega_i\|C_{n,i\bdot}^*(\theta)\|_1 \\
		&\quad\quad+\frac{2}{T}\int_0^T\left[\Psi_{n,i}(t;\hat{C}_{n,i\bdot}(\theta),\hat{\alpha}_{n,i}(\theta),\theta)-\Psi_{n,i}(t;C_{n,i\bdot}^*(\theta),\alpha_{n,i}^*(\theta),\theta)\right]dM_{n,i}(t),
	\end{align*}
	where the last inequality uses the definition of $(\hat{C}_n(\theta),\hat{\alpha}_n(\theta))$. Furthermore, we have
	\begin{align*}
		&\left|\frac{2}{T}\int_0^T\left[\Psi_{n,i}(t;\hat{C}_{n,i\bdot}(\theta),\hat{\alpha}_{n,i}(\theta),\theta)-\Psi_{n,i}(t;C_{n,i\bdot}^*(\theta),\alpha_{n,i}^*(\theta),\theta)\right]dM_{n,i}(t)\right| \\
		&\leq \left|\frac{2}{T}\int_0^T\nu_0(X_{n,i}(t);\beta)dM_{n,i}(t)\left(\hat{\alpha}_{n,i}(\theta)-\alpha_{n,i}^*(\theta)\right)\right| \\
		&+\left|\sum_{j=1}^n\frac{2}{T}\int_0^T\int_0^{t-}g(t-r;\gamma)dN_{n,j}(r)dM_{n,i}(t)\left(\hat{C}_{n,ij}(\theta)-C_{n,ij}^*(\theta)\right)\right| \\
		&\leq \left|\frac{2}{T}\int_0^T\nu_0(X_{n,i}(t);\beta)dM_{n,i}(t)\right|\cdot\left|\hat{\alpha}_{n,i}(\theta)-\alpha_{n,i}^*(\theta)\right| \\
		&+\sup_{j=1,...,n}\left|\frac{2}{T}\int_0^T\int_0^{t-}g(t-r;\gamma)dN_{n,j}(r)dM_{n,i}(t)\right|\cdot\left\|\hat{C}_{n,i\bdot}(\theta)-C_{n,i\bdot}^*(\theta)\right\|_1.
	\end{align*}
	Both of the previous two displays together yield on $\mathcal{T}^{(i)}_n(a_n,d_{n,i})$ the following {\em basic inequality:}
	\begin{align}
		&\frac{1}{T}\mathcal{E}_i(\hat{C}_n(\theta),\hat{\alpha}_n(\theta),\theta)+2\omega_i\|\hat{C}_{n,i\bdot}(\theta)\|_1\leq\frac{1}{T}\mathcal{E}_i(C_n^*(\theta),\alpha_n^*(\theta),\theta) \nonumber \\
		&\qquad\qquad+2\omega_i\|C_{n,i\bdot}^*(\theta)\|_1+a_n\left|\hat{\alpha}_{n,i}(\theta)-\alpha_{n,i}^*(\theta)\right|+d_{n,i}\left\|\hat{C}_{n,i\bdot}(\theta)-C_{n,i\bdot}^*(\theta)\right\|_1. \label{eq:basicinequality}
	\end{align}
	We apply \eqref{eq:basicinequality} in the second inequality below. The first and fourth steps follow from $\omega_i\geq 3d_{n,i}$, and the last inequality is a consequence of the triangle inequality.
	\begin{align}
		&\frac{2}{T}\mathcal{E}_i(\hat{C}_n(\theta),\hat{\alpha}_n(\theta),\theta)+2\omega_i\|\hat{C}_{n,iS_i^c(C_n^*)}(\theta)\|_1 \nonumber \\
		&\leq \frac{2}{T}\mathcal{E}_i(\hat{C}_n(\theta),\hat{\alpha}_n(\theta),\theta)+4\omega_i\|\hat{C}_{n,i\bdot}(\theta)\|_1-4\omega_i\|\hat{C}_{n,iS_i(C_n^*)}(\theta)\|_1-6d_{n,i}\|\hat{C}_{n,iS_i^c(C^*)}(\theta)\| \nonumber \\
		&\leq \frac{2}{T}\mathcal{E}_i(C_n^*(\theta),\alpha_n^*(\theta),\theta)+4\omega_i\|C_{n,i\bdot}^*(\theta)\|_1-4\omega_i\|\hat{C}_{n,iS_i(C_n^*)}(\theta)\|_1-6d_{n,i}\|\hat{C}_{n,iS_i^c(C_n^*)}(\theta)\| \nonumber \\
		&+2a_n\left|\hat{\alpha}_{n,i}(\theta)-\alpha^*_{n,i}\right|+2d_{n,i}\left\|\hat{C}_{n,i\bdot}(\theta)-C^*_{n,i\bdot}(\theta)\right\|_1 \nonumber \\
		&\leq \frac{2}{T}\mathcal{E}_i(C_n^*(\theta),\alpha_n^*(\theta),\theta)+4\omega_i\|C_{n,i\bdot}^*(\theta)\|_1-4\omega_i\|\hat{C}_{n,iS_i(C_n^*)}(\theta)\|_1-2d_{n,i}\|\hat{C}_{n,iS_i^c(C_n^*)}(\theta)\| \nonumber \\
		&+2a_n\left|\hat{\alpha}_{n,i}(\theta)-\alpha^*_{n,i}\right|+2d_{n,i}\left\|\hat{C}_{n,i\bdot}(\theta)-C^*_{n,i\bdot}(\theta)\right\|_1 \nonumber \\
		&\leq \frac{2}{T}\mathcal{E}_i(C_n^*(\theta),\alpha_n^*(\theta),\theta)+4\omega_i\left(\|C_{n,i\bdot}^*(\theta)\|_1-\|\hat{C}_{n,iS_i(C_n^*)}(\theta)\|_1\right) \nonumber \\
		&+2a_n\left|\hat{\alpha}_{n,i}(\theta)-\alpha^*_{n,i}\right|+\frac{2}{3}\omega_i\left\|\hat{C}_{n,iS_i(C_n^*)}(\theta)-C^*_{n,iS_i(C_n^*)}(\theta)\right\|_1 \nonumber \\
		&\leq \frac{2}{T}\mathcal{E}_i(C_n^*(\theta),\alpha_n^*(\theta),\theta)+\frac{14}{3}\omega_i\left\|\hat{C}_{n,iS_i(C^*(\theta))}(\theta)-C^*_{n,iS_i(C_n^*)}(\theta)\right\|_1+2a_n\left|\hat{\alpha}_{n,i}(\theta)-\alpha^*_{n,i}\right|. \nonumber
	\end{align}
	By the definition of $\mathcal{E}_i(\cdot,\cdot,\cdot)$ this is equivalent to
	\begin{align}
		&\frac{2}{T}\left\|\Psi_{n,i}(\cdot;\hat{C}_{n,i\bdot}(\theta),\hat{\alpha}_{n,i}(\theta),\theta)-\lambda_{n,i}\right\|_T^2+2\omega_i\|\hat{C}_{n,iS_i^c(C_n^*)}(\theta)\|_1 \nonumber \\
		&\leq \frac{2}{T}\left\|\Psi_{n,i}(\cdot;C_{n,i\bdot}^*(\theta),\alpha_{n,i}^*(\theta),\theta)-\lambda_{n,i}\right\|_T^2 \nonumber \\
		&\qquad\qquad+\frac{14}{3}\omega_i\left\|\hat{C}_{n,iS_i(C^*(\theta))}(\theta)-C^*_{n,iS_i(C_n^*)}(\theta)\right\|_1+2a_n\left|\hat{\alpha}_{n,i}(\theta)-\alpha^*_{n,i}\right|. \label{eq:newbi}
	\end{align}
	We consider now two cases:
	
	\underline{Case I}: Suppose that
	\begin{align}
		&\frac{1}{T}\left\|\Psi_{n,i}(\cdot;C_{n,i\bdot}^*(\theta),\alpha_{n,i}^*(\theta),\theta)-\lambda_{n,i}\right\|_T^2 \nonumber \\
		&\qquad\qquad\qquad\geq\frac{14}{3}\omega_i\|\hat{C}_{n,iS_i(C_n^*)}(\theta)-C^*_{n,iS_i(C_n^*)}(\theta)\|_1+2a_n|\hat{\alpha}_{n,i}(\theta)-\alpha^*_{n,i}(\theta)|. \label{eq:case1}
	\end{align}
	In this case, \eqref{eq:newbi} implies
	\begin{align*}
		&\frac{2}{T}\left\|\Psi_{n,i}(\cdot;\hat{C}_{n,i\bdot}(\theta),\hat{\alpha}_{n,i}(\theta),\theta)-\lambda_{n,i}\right\|_T^2+2\omega_i\|\hat{C}_{n,iS_i^c(C_n^*)}(\theta)\|_1 \\
		&\leq \frac{3}{T}\left\|\Psi_{n,i}(\cdot;C_{n,i\bdot}^*(\theta),\alpha_{n,i}^*(\theta),\theta)-\lambda_{n,i}\right\|_T^2.
	\end{align*}
	Using this, again, together with \eqref{eq:case1}, we obtain
	\begin{align*}
		&\frac{2}{T}\left\|\Psi_{n,i}(\cdot;\hat{C}_{n,i\bdot}(\theta),\hat{\alpha}_{n,i}(\theta),\theta)-\lambda_{n,i}\right\|_T^2+2\omega_i\|\hat{C}_{n,i\bdot}(\theta)-C_{n,i\bdot}^*(\theta)\|_1 \\
		&\qquad\qquad+\frac{6}{7}a_n\left|\hat{\alpha}_{n,i}(\theta)-\alpha_{n,i}^*(\theta)\right|\leq\frac{24}{7}\cdot\frac{1}{T}\left\|\Psi_{n,i}(\cdot;C_{n,i\bdot}^*(\theta),\alpha_{n,i}^*(\theta),\theta)-\lambda_{n,i}\right\|_T^2
	\end{align*}
	and, hence, after multiplication with $7/3$
	\begin{align*}
		&\frac{14}{3T}\left\|\Psi_{n,i}(\cdot;\hat{C}_{n,i\bdot}(\theta),\hat{\alpha}_{n,i}(\theta),\theta)-\lambda_{n,i}\right\|_T^2+\frac{14}{3}\omega_i\|\hat{C}_{n,i\bdot}(\theta)-C_{n,i\bdot}^*(\theta)\|_1 \\
		&\qquad\qquad+2a_n\left|\hat{\alpha}_{n,i}(\theta)-\alpha_{n,i}^*(\theta)\right|\leq8\cdot\frac{1}{T}\left\|\Psi_{n,i}(\cdot;C_{n,i\bdot}^*(\theta),\alpha_{n,i}^*(\theta),\theta)-\lambda_{n,i}\right\|_T^2,
	\end{align*}
	which is stronger than the inequality we want to prove.
	
	\underline{Case II}: Suppose that
	\begin{align*}
		&\frac{1}{T}\left\|\Psi_{n,i}(\cdot;C_{n,i\bdot}^*(\theta),\alpha_{n,i}^*(\theta),\theta)-\lambda_{n,i}\right\|_T^2 \\
		<&\frac{14}{3}\omega_i\|\hat{C}_{n,iS_i(C_n^*)}(\theta)-C^*_{n,iS_i(C_n^*)}(\theta)\|_1+2a_n|\hat{\alpha}_{n,i}(\theta)-\alpha^*_{n,i}(\theta)|.
	\end{align*}
	In this case, we obtain from \eqref{eq:newbi}
	\begin{align}
		&\frac{2}{T}\left\|\Psi_{n,i}(\cdot;\hat{C}_{n,i\bdot}(\theta),\hat{\alpha}_{n,i}(\theta),\theta)-\lambda_{n,i}\right\|_T^2+2\omega_i\|\hat{C}_{n,iS_i^c(C_n^*)}(\theta)\|_1 \nonumber \\
		&\leq 14\omega_i\left\|\hat{C}_{n,iS_i(C^*(\theta))}(\theta)-C^*_{n,iS_i(C_n^*)}(\theta)\right\|_1+6a_n\left|\hat{\alpha}_{n,i}(\theta)-\alpha^*_{n,i}\right|. \label{eq:newbi_case2}
	\end{align}
	Since $L\geq\sup_{i=1,...,n}a_n/\omega_i$, this implies
	\begin{align*}
		\|\hat{C}_{n,iS_i^c(C_n^*)}(\theta)\|_1\leq7\left\|\hat{C}_{n,iS_i(C^*(\theta))}(\theta)-C^*_{n,iS_i(C_n^*)}(\theta)\right\|_1+3L\left|\hat{\alpha}_{n,i}(\theta)-\alpha^*_{n,i}\right|.
	\end{align*}
	Hence, $(\hat{C}_n(\theta),\hat{\alpha}_n(\theta))\in\tilde{\cR}_n(i;\theta,L)$. This means by using the Cauchy-Schwarz inequality
	\begin{align*}
		&16\omega_i\left\|\hat{C}_{n,iS_i(C_n^*)}(\theta)-C^*_{n,iS_i(C_n^*)}(\theta)\right\|_1+8a_n\left|\hat{\alpha}_{n,i}(\theta)-\alpha^*_{n,i}\right| \\
		&\leq 8\sqrt{4\omega_i^2|S_i(C_n^*)|+a_n^2}\cdot\sqrt{\|\hat{C}_{n,iS_i(C_n^*)}-C^*_{n,iS_i(C_n^*)}(\theta)\|_2^2+|\hat{\alpha}_{n,i}(\theta)-\alpha^*_{n,i}(\theta)|^2} \\
		&\leq \frac{8\sqrt{4\omega_i^2|S_i(C_n^*)|+a_n^2}}{\phi_{i,\textrm{comp}}(L;\theta)}\frac{1}{\sqrt{T}}\left\|\Psi_{n,i}(\cdot;\hat{C}_{n,i\bdot}(\theta),\hat{\alpha}_{n,i}(\theta),\theta)-\Psi_{n,i}(\cdot;C_{n,i\bdot}^*(\theta),\alpha_{n,i}^*(\theta),\theta)\right\|_T \\
		&\leq 18\cdot\frac{4\omega_i^2|S_i(C_n^*)|+a_n^2}{\phi_{i,\textrm{comp}}^2(L;\theta)}+\frac{1}{T}\left\|\Psi_{n,i}(\cdot;\hat{C}_{n,i\bdot}(\theta),\hat{\alpha}_{n,i}(\theta),\theta)-\lambda_{n,i}\right\|_T^2 \\
		&\qquad\qquad\qquad\qquad\qquad\qquad+\frac{8}{T}\left\|\Psi_{n,i}(\cdot;C^*_{n,i\bdot}(\theta),\alpha^*_{n,i}(\theta),\theta)-\lambda_{n,i}\right\|_T^2,
	\end{align*}
	where we used in the last inequality that for any numbers $x,y,z\in\IR$
	$$x(y+z)\leq\frac{9}{32}x^2+y^2+8z^2.$$
	Thus, we obtain from \eqref{eq:newbi_case2} that
	\begin{align}
		&\frac{2}{T}\left\|\Psi_{n,i}(\cdot;\hat{C}_{n,i\bdot}(\theta),\hat{\alpha}_{n,i}(\theta),\theta)-\lambda_{n,i}\right\|_T^2+2\omega_i\|\hat{C}_{n,i\bdot}(\theta)-C^*_{n,i\bdot}(\theta)\|_1 \nonumber \\
		&\qquad\qquad\qquad\qquad\qquad\qquad\qquad\qquad\qquad+2a_n|\hat{\alpha}_n(\theta)_{n,i}-\alpha^*_{n,i}(\theta)| \nonumber \\
		&\leq 14\omega_i\left\|\hat{C}_{n,iS_i(C_n^*)}(\theta)-C^*_{n,iS_i(C_n^*)}(\theta)\right\|_1+6a_n\left|\hat{\alpha}_{n,i}(\theta)-\alpha^*_{n,i}\right| \nonumber \\
		&\qquad\qquad+2\omega_i\|\hat{C}_{n,iS_i(C_n^*)}(\theta)-C^*_{n,iS_i(C_n^*}(\theta)\|_1+2a_n|\hat{\alpha}_n(\theta)_{n,i}-\alpha^*_{n,i}(\theta)|  \nonumber \\
		&= 16\omega_i\left\|\hat{C}_{n,iS_i(C_n^*)}(\theta)-C^*_{n,iS_i(C_n^*)}(\theta)\right\|_1+8a_n\left|\hat{\alpha}_{n,i}(\theta)-\alpha^*_{n,i}\right| \nonumber \\
		&\leq 18\cdot\frac{4\omega_i^2|S_i(C_n^*)|+a_n^2}{\phi_{i,\textrm{comp}}^2(L;\theta)}+\frac{1}{T}\left\|\Psi_{n,i}(\cdot;\hat{C}_{n,i\bdot}(\theta),\hat{\alpha}_{n,i}(\theta),\theta)-\lambda_{n,i}\right\|_T^2 \nonumber \\
		&\qquad\qquad\qquad\qquad\qquad\qquad+\frac{8}{T}\left\|\Psi_{n,i}(\cdot;C^*_{n,i}(\theta),\alpha_{n,i}^*(\theta),\theta)-\lambda_{n,i}\right\|_T^2, \nonumber
	\end{align}
	which in turn implies
	\begin{align*}
		&\frac{1}{T}\left\|\Psi_{n,i}(\cdot;\hat{C}_{n,i\bdot}(\theta),\hat{\alpha}_{n,i}(\theta),\theta)-\lambda_{n,i}\right\|_T^2+2\omega_i\|\hat{C}_{n,i\bdot}(\theta)-C^*_{n,i\bdot}(\theta)\|_1 \\
		&\qquad\qquad\qquad\qquad\qquad\qquad\qquad\qquad\qquad+2a_n|\hat{\alpha}_n(\theta)_{n,i}-\alpha^*_{n,i}(\theta)| \\
		&\leq 18\cdot\frac{4\omega_i^2|S_i(C_n^*)|+a_n^2}{\phi_{i,\textrm{comp}}^2(L;\theta)}+\frac{8}{T}\left\|\Psi_{n,i}(\cdot;C^*_{n,i\bdot}(\theta),\alpha_{n,i}^*(\theta),\theta)-\lambda_{n,i}\right\|_T^2.
	\end{align*}
\end{proof}

\begin{proof}[Proof of Lemma \ref{lem:diff}]
	We begin the proof with two observations. For the first observation, let
	\begin{align*}
		v_i(t;\theta):&= \begin{pmatrix}
			\nu_0(X_{n,i}(t);\beta), & \int_0^{t-}g(t-r;\gamma)dN_{n,1}(r), & \cdots & \int_0^{t-}g(t-r;\gamma)dN_{n,n}(r)
		\end{pmatrix}^T,
	\end{align*}
	so that 
	$$\Psi_{n,i}(t;c,a,\theta)=v_i(t;\theta)^T\begin{pmatrix}
		a \\ c
	\end{pmatrix},\quad\lambda_{n,i}(t)=v_i(t;\theta_n^*)^T\begin{pmatrix}
		\alpha_{n,i}^* \\ \big(C_{n,i\bdot}^*\big)^T
	\end{pmatrix}.$$
	Using this, we can rewrite the criterion function from Lemma \ref{lem:od} as
	\begin{align}
		&\left\|\Psi_{n,i}(\cdot;c,a,\theta)-\lambda_{n,i}\right\|_T^2 \nonumber \\
		&= \int_0^T\Bigg(v_i(t;\theta)^T\begin{pmatrix}
			a \\ c
		\end{pmatrix}-v_i(t;\theta_n^*)^T\binom{
			\alpha_{n,i}^*}{ \big(C_{n,i\bdot}^*\big)^T}
		\Bigg)^2dt \nonumber \\
		&= \binom{a}{c}^T\Big[\int_0^Tv_i(t;\theta)v_i(t;\theta)^Tdt\Big]\begin{pmatrix}
			a \\ c
		\end{pmatrix}-2\begin{pmatrix}
			a \\ c
		\end{pmatrix}^T\Big[\int_0^Tv_i(t;\theta)v_i(t;\theta_n^*)^Tdt\Big]\begin{pmatrix}
			\alpha_{n,i}^* \\ \big(C_{n,i\bdot}^*\big)^T
		\end{pmatrix} \nonumber \\
		&\qquad+\begin{pmatrix}
			\alpha_{n,i}^* \\ \big(C_{n,i\bdot}^*\big)^T
		\end{pmatrix}^T\Big[\int_0^Tv_i(t;\theta_n^*)v_i(t;\theta_n^*)^Tdt\Big]\begin{pmatrix}
			\alpha_{n,i}^* \\ \big(C_{n,i\bdot}^*\big)^T
		\end{pmatrix}. \nonumber
	\end{align}
	This, in turn, implies
	\begin{align*}
		&\left|\frac{1}{T}\left\|\Psi_{n,i}(\cdot;c,a,\theta)-\lambda_{n,i}\right\|_T^2-\frac{1}{T}\left\|\Psi_{n,i}(\cdot;c,a,\theta_n^*)-\lambda_{n,i}\right\|_T^2\right| \\
		&\leq \Bigg|\begin{pmatrix}
			a \\ c
		\end{pmatrix}^T\Big[\frac{1}{T}\int_0^Tv_i(t;\theta)v_i(t;\theta)^Tdt\Big]\begin{pmatrix}
			a \\ c
		\end{pmatrix}-2\begin{pmatrix}
			a \\ c
		\end{pmatrix}^T\Big[\frac{1}{T}\int_0^Tv_i(t;\theta)v_i(t;\theta_n^*)^Tdt\Big]\begin{pmatrix}
			\alpha_{n,i}^* \\ \big(C_{n,i\bdot}^*\big)^T
		\end{pmatrix} \nonumber \\
		&+\begin{pmatrix}
			\alpha_{n,i}^* \\ \big(C_{n,i\bdot}^*\big)^T
		\end{pmatrix}^T\Big[\frac{1}{T}\int_0^Tv_i(t;\theta_n^*)v_i(t;\theta_n^*)^Tdt\Big]\begin{pmatrix}
			\alpha_{n,i}^* \\ \big(C_{n,i\bdot}^*\big)^T
		\end{pmatrix} \\
		&-\begin{pmatrix}
			a \\ c
		\end{pmatrix}^T\Big[\frac{1}{T}\int_0^Tv_i(t;\theta_n^*)v_i(t;\theta_n^*)^Tdt\Big]\begin{pmatrix}
			a \\ c
		\end{pmatrix}+2\begin{pmatrix}
			a \\ c
		\end{pmatrix}^T\Big[\frac{1}{T}\int_0^Tv_i(t;\theta_n^*)v_i(t;\theta_n^*)^Tdt\Big]\begin{pmatrix}
			\alpha_{n,i}^* \\ \big(C_{n,i\bdot}^*\big)^T
		\end{pmatrix} \nonumber \\
		&-\begin{pmatrix}
			\alpha_{n,i}^* \\ \big(C_{n,i\bdot}^*\big)^T
		\end{pmatrix}^T\Big[\frac{1}{T}\int_0^Tv_i(t;\theta_n^*)v_i(t;\theta_n^*)^Tdt\Big]\begin{pmatrix}
			\alpha_{n,i}^* \\ \big(C_{n,i\bdot}^*\big)^T
		\end{pmatrix}\Bigg| \\
		&= \Bigg|\begin{pmatrix}
			a \\ c
		\end{pmatrix}^T\Big[\frac{1}{T}\int_0^T\left(v_i(t;\theta)-v_i(t;\theta_n^*)\right)\left(v_i(t;\theta)-v_i(t;\theta_n^*)\right)^Tdt\Big]\begin{pmatrix}
			a \\ c
		\end{pmatrix} \\
		&+2\begin{pmatrix}
			a \\ c
		\end{pmatrix}^T\Big[\frac{1}{T}\int_0^Tv_i(t;\theta)v_i(t;\theta_n^*)^Tdt\Big]\begin{pmatrix}
			a-\alpha_{n,i}^* \\ c-\big(C_{n,i\bdot}^*\big)^T
		\end{pmatrix} \nonumber \\
		&-\begin{pmatrix}
			a \\ c
		\end{pmatrix}^T\Big[\frac{2}{T}\int_0^Tv_i(t;\theta_n^*)v_i(t;\theta_n^*)^Tdt\Big]\begin{pmatrix}
			a \\ c
		\end{pmatrix}+\begin{pmatrix}
			a \\ c
		\end{pmatrix}^T\Big[\frac{2}{T}\int_0^Tv_i(t;\theta_n^*)v_i(t;\theta_n^*)^Tdt\Big]\begin{pmatrix}
			\alpha_{n,i}^* \\ \big(C_{n,i\bdot}^*\big)^T
		\end{pmatrix}\Bigg| \\
		&= \Bigg|\begin{pmatrix}
			a \\ c
		\end{pmatrix}^T\Big[\frac{1}{T}\int_0^T\left(v_i(t;\theta)-v_i(t;\theta_n^*)\right)\left(v_i(t;\theta)-v_i(t;\theta_n^*)\right)^Tdt\Big]\begin{pmatrix}
			a \\ c
		\end{pmatrix} \\
		&\qquad+2\begin{pmatrix}
			a \\ c
		\end{pmatrix}^T\Big[\frac{1}{T}\int_0^T\left(v_i(t;\theta)-v_i(t;\theta_n^*)\right)v_i(t;\theta_n^*)^Tdt\Big]\begin{pmatrix}
			a-\alpha_{n,i}^* \\ c-\big(C_{n,i\bdot}^*\big)^T
		\end{pmatrix}\Bigg|
	\end{align*}
	Let $\mathcal{N}_0>0$ be such that $\IP(\Omega_{\mathcal{N}})\to1$ for $\mathcal{N}=6\mathcal{N}_0\log(nT)$ according to Lemma \ref{lem:omega_lemma}.
	Note that, on  $\Omega_{\mathcal{N}}$, we have that every entry of the difference $v_i(t;\theta)-v_i(t;\theta_n^*)$ can be uniformly (in $t$ and $i$) be bounded by $\log(nT)\cdot\|\theta-\theta_n^*\|_2$ times a constant.
	Let us assume for the remainder of the proof that we are on the event $\Omega_{\mathcal{N}}$.
	Let now $\theta\in\Theta$ be arbitrary, and let $(c,a)$ be such that $c$ is $i$-th row of a matrix $C\in[0,\infty)^{n\times n}$, and $a$ is the $i$-th entry of a vector $\alpha\in(0,\infty)^n$ with $(C,\alpha)\in\mathcal{H}_n(\theta)\cup\mathcal{H}_n(\theta_n^*)$ and $c_{S_i(C_n^*)^c}=0$.
	In particular, then $a\in[0,K_\alpha]$.
	For such $(a,c)$, we hence obtain on $\Omega_{\mathcal{N}}$ for suitable constants $\mathcal{C}_1,\mathcal{C}_2>0$ that collect all previously mentioned constants
	\begin{align}
		&\left|\frac{1}{T}\left\|\Psi_{n,i}(\cdot;c,a,\theta)-\lambda_{n,i}\right\|_T^2-\frac{1}{T}\left\|\Psi_{n,i}(\cdot;c,a,\theta_n^*)-\lambda_{n,i}\right\|_T^2\right| \nonumber \\
		\leq\,&\mathcal{C}_1(1+\|c\|_1)^2\log^2(nT)\cdot\|\theta-\theta_n^*\|_2^2 \nonumber \\
		&\qquad+\mathcal{C}_2(1+\|c\|_1)\log^2(nT)\cdot\|\theta-\theta_n^*\|_2\cdot\left\|\begin{pmatrix}
			a-\alpha_{n,i}^* \\ c-\left(C_{n,i\bdot}^*\right)^T
		\end{pmatrix}\right\|_1. \label{eq:FO}
	\end{align}
	This was the first observation. For the second observation, recall that $C_n^*=C_n^*(\theta_n^*)$ and $\alpha_n^*=\alpha_n^*(\theta_n^*)$. For any $a>0$ and $c\in[0,\infty)^n$ with $c_{S_i(C_n^*)^c}=0$, we have that $(c,a)\in\tilde{\mathcal{R}}_n(i;\theta_n^*,L)$ (recall the definition of the individual compatibility constant) and, hence,
	\begin{align}
		\frac{1}{T}\left\|\Psi_{n,i}(\cdot;c,a,\theta_n^*)-\lambda_{n,i}\right\|_T^2&= \frac{1}{T}\left\|\Psi_{n,i}(\cdot;c,a,\theta_n^*)-\Psi_{n,i}(\cdot;C_{n,i\bdot}^*(\theta_n^*),\alpha_{n,i}^*(\theta_n^*),\theta_n^*)\right\|_T^2 \nonumber \\
		\geq&\phi_{i,\textrm{comp}}(L;\theta_n^*)^2\left(\left\|c_{S_i(C_n^*)}-C_{n,iS_i(C_n^*)}^*(\theta_n^*)\right\|_2^2+\left|a-\alpha_{n,i}^*(\theta_n^*)\right|^2\right) \nonumber \\
		\geq&\phi_{i,\textrm{comp}}(L;\theta_n^*)^2\frac{1}{1+|S_i(C_n^*)|}\left\|\begin{pmatrix}
			a-\alpha_{n,i}^*(\theta_n^*) \\ c-C_{n,i\bdot}^*(\theta_n^*)
		\end{pmatrix}\right\|_1^2. \label{eq:SO}
	\end{align}
	Now, using these two observations, we can make the following argument. Firstly, we note that \eqref{eq:FO} provides an upper bound on the value of the criterion function for $\overline{\theta}_n$ at the minimizer because \eqref{eq:FO} shows that the criterion function for $\overline{\theta}_n$ lies close to the criterion function for $\theta_n^*$. More precisely, by definition of $(\alpha_n^*(\theta),C_n^*(\theta))$ and recalling that $\alpha_n^*=\alpha_n^*(\theta_n^*)$ and $C_n^*=C_n^*(\theta_n^*)$, we obtain by (A2) that
	\begin{align*}
		&\frac{1}{T}\left\|\Psi_{n,i}(\cdot;C_{n,i\bdot}^*(\overline{\theta}_n),\alpha_{n,i}^*(\overline{\theta}_n),\overline{\theta}_n)-\lambda_{n,i}\right\|_T^2 \\
		&\leq \frac{1}{T}\left\|\Psi_{n,i}(\cdot;C_{n,i\bdot}^*(\theta_n^*),\alpha_{n,i}^*(\theta_n^*),\overline{\theta}_n)-\lambda_{n,i}\right\|_T^2 \\
		&\leq \frac{1}{T}\left\|\Psi_{n,i}(\cdot;C_{n,i\bdot}^*(\theta_n^*),\alpha_{n,i}^*(\theta_n^*),\theta_n^*)-\lambda_{n,i}\right\|_T^2+\mathcal{C}_1(1+\|C_{n,i\bdot}^*\|_1)^2\log^2(nT)\cdot\left\|\overline{\theta}_n-\theta_n^*\right\|_2^2 \\
		&= \mathcal{C}_1(1+\|C_{n,i\bdot}^*\|_1)^2\log^2(nT)\cdot\left\|\overline{\theta}_n-\theta_n^*\right\|_2^2\leq\mathcal{C}_1(1+c_n)^2\log^2(nT)\cdot\left\|\overline{\theta}_n-\theta_n^*\right\|_2^2
	\end{align*}
	Moreover, again applying \eqref{eq:FO},
	\begin{align*}
		&\frac{1}{T}\left\|\Psi_{n,i}(\cdot;C_{n,i\bdot}^*(\overline{\theta}_n),\alpha_{n,i}^*(\overline{\theta}_n),\overline{\theta}_n)-\lambda_{n,i}\right\|_T^2 \\
		\geq&\frac{1}{T}\left\|\Psi_{n,i}(\cdot;C_{n,i\bdot}^*(\overline{\theta}_n),\alpha_{n,i}^*(\overline{\theta}_n),\theta_n^*)-\lambda_{n,i}\right\|_T^2-\mathcal{C}_1(1+\|C_{n,i\bdot}^*(\overline{\theta}_n)\|_1)^2\log^2(nT)\cdot\|\overline{\theta}_n-\theta_n^*\|_2^2 \\
		&\qquad-\mathcal{C}_2(1+\|C_{n,i\bdot}^*(\overline{\theta}_n)\|_1)\log^2(nT)\cdot\|\overline{\theta}_n-\theta_n^*\|_2\cdot\left\|\begin{pmatrix}
			\alpha_{n,i}^*(\overline{\theta}_n)-\alpha_{n,i}^* \\ C_{n,i\bdot}^*(\overline{\theta_n})-\big(C_{n,i\bdot}^*\big)^T
		\end{pmatrix}\right\|_1.
	\end{align*}
	Putting both of the previous displays together, we obtain
	\begin{align*}
		&\frac{1}{T}\left\|\Psi_{n,i}(\cdot;C_{n,i\bdot}^*(\overline{\theta}_n),\alpha_{n,i}^*(\overline{\theta}_n),\theta_n^*)-\lambda_{n,i}\right\|_T^2 \\
		&\leq \log^2(nT)\cdot\left\|\overline{\theta}_n-\theta_n^*\right\|_2\Bigg(\mathcal{C}_1\left((1+c_n)^2+(1+\|C_{n,i\bdot}^*(\overline{\theta}_n)\|_1)^2\right)\|\overline{\theta}_n-\theta_n^*\|_2 \\
		&\qquad+\mathcal{C}_2(1+\|C_{n,i\bdot}^*(\overline{\theta}_n)\|_1)\cdot\left\|\begin{pmatrix}
			\alpha_{n,i}^*(\overline{\theta}_n)-\alpha_{n,i}^* \\ C_{n,i\bdot}^*(\overline{\theta_n})-\left(C_{n,i\bdot}^*\right)^T
		\end{pmatrix}\right\|_1\Bigg).
	\end{align*}
	Using \eqref{eq:SO}, we conclude
	\begin{align}
		&\frac{\phi_{i,\textrm{comp}}(L;\theta_n^*)^2}{1+|S_i(C_n^*)|}\left\|\begin{pmatrix}
			\alpha_{n,i}^*(\overline{\theta}_n)-\alpha_{n,i}^*(\theta_n^*) \\ C_{n,i\bdot}^*(\overline{\theta}_n)-C_{n,i\bdot}^*(\theta_n^*)
		\end{pmatrix}\right\|_1^2 \nonumber \\
		&\leq \log^2(nT)\cdot\left\|\overline{\theta}_n-\theta_n^*\right\|_2\Bigg(\mathcal{C}_1\left((1+c_n)^2+(1+\|C_{n,i\bdot}^*(\overline{\theta}_n)\|_1)^2\right)\|\overline{\theta}_n-\theta_n^*\|_2 \nonumber \\
		&\qquad+\mathcal{C}_2(1+\|C_{n,i\bdot}^*(\overline{\theta}_n)\|_1)\cdot\left\|\begin{pmatrix}
			\alpha_{n,i}^*(\overline{\theta}_n)-\alpha_{n,i}^* \\ C_{n,i\bdot}^*(\overline{\theta_n})-\left(C_{n,i\bdot}^*\right)^T
		\end{pmatrix}\right\|_1\Bigg). \label{eq:SFO}
	\end{align}
	Now, if
	$$\left\|\overline{\theta}_n-\theta_n^*\right\|_2\leq\left\|\begin{pmatrix}
		\alpha_{n,i}^*(\overline{\theta}_n)-\alpha_{n,i}^* \\ C_{n,i\bdot}^*(\overline{\theta}_n)-\left(C_{n,i\bdot}^*\right)^T
	\end{pmatrix}\right\|_1,$$
	we obtain from \eqref{eq:SFO} that
	\begin{align*}
		&\left\|\begin{pmatrix}
			\alpha_{n,i}^*(\overline{\theta}_n)-\alpha_{n,i}^*(\theta_n^*) \\ C_{n,i\bdot}^*(\overline{\theta}_n)-C_{n,i\bdot}^*(\theta_n^*)
		\end{pmatrix}\right\|_1\leq\frac{1+|S_i(C_n^*)|}{\phi_{i,\textrm{comp}}(L;\theta_n^*)^2}\log^2(nT)\cdot\left\|\overline{\theta}_n-\theta_n^*\right\|_2 \\
		&\qquad\qquad\qquad\qquad\times\left(\mathcal{C}_1(1+c_n)^2+\mathcal{C}_1(1+\|C_{n,i\bdot}^*(\overline{\theta}_n)\|_1)^2+\mathcal{C}_2\left(1+\|C_{n,i\bdot}^*(\overline{\theta}_n)\|_1\right)\right).
	\end{align*}
	The other case,
	\begin{equation}
		\label{eq:other_case}
		\left\|\overline{\theta}_n-\theta_n^*\right\|_2\geq\left\|\begin{pmatrix}
			\alpha_{n,i}^*(\overline{\theta}_n)-\alpha_{n,i}^* \\ C_{n,i\bdot}^*(\overline{\theta}_n)-\left(C_{n,i\bdot}^*\right)^T
		\end{pmatrix}\right\|_1,
	\end{equation}
	is stronger than what we wanted to prove.
	To see this, we observe the following for any $(C,\alpha)\in\tilde{\mathcal{R}}_n(i;\theta_n^*,L)$.
	By Assumptions~(A1) and~(A2) and the definition of $\tilde{\mathcal{R}}_n(i;\theta_n^*,L)$, we have on $\Omega_{\mathcal{N}}$ that
	\begin{align*}
		&\frac{1}{T}\left\|\Psi_{n,i}(\cdot;C_{i\cdot},\alpha_i,\theta_n^*)-\Psi_{n,i}(\cdot;C_{n,i\cdot}^*,\alpha_{n,i}^*,\theta_n^*)\right\|_T^2 \\
		&= \frac{1}{T}\int_0^T\left((\alpha_i-\alpha_{n.i}^*)\nu_0(X_{n,i}(t);\beta_n^*)+\sum_{j=1}^n(C_{ij}-C_{n,ij}^*)\int_0^{t-}g(t-r;\gamma_n^*)dN_{n,j}(r)\right)^2dt \\
		&\leq \frac{2}{T}\int_0^T(\alpha_i-\alpha_{n.i}^*)^2\|\overline{\nu}\|_{\infty}^2 \\
		&\qquad+\overline{g}^2\log^2(nT)\left(\left\|C_{iS_i(C_n^*)}-C_{n,iS_i(C_n^*)}^*\right\|_1+\left\|(C_{iS_i(C_n^*)^c}-C_{n,iS_i(C_n^*)^c}^*\right\|_1\right)^2dt \\
		&\leq 2\Bigg((\alpha_i-\alpha_{n.i}^*)^2\|\overline{\nu}\|_{\infty}^2 \\
		&+\overline{g}^2\log^2(nT)\left(\left\|C_{iS_i(C_n^*)}-C_{n,iS_i(C_n^*)}^*\right\|_1+7\left\|(C_{iS_i(C_n^*)}-C_{n,iS_i(C_n^*)}^*\right\|_1+3L|\alpha_i-\alpha_{n,i}^*|\right)^2\Bigg) \\
		&\leq \overline{k}\log^2(nT)\left((\alpha_i-\alpha_{n.i}^*)^2+\left\|C_{iS_i(C_n^*)}-C_{n,iS_i(C_n^*)}^*\right\|_1^2\right) \\
		& \leq\overline{k}\log^2(nT)\left((\alpha_i-\alpha_{n.i}^*)^2+\left\|C_{iS_i(C_n^*)}-C_{n,iS_i(C_n^*)}^*\right\|_2^2|S_i(C_n^*)|\right) \\
		&\leq \overline{k}\log^2(nT)\left((\alpha_i-\alpha_{n.i}^*)^2+\left\|C_{iS_i(C_n^*)}-C_{n,iS_i(C_n^*)}^*\right\|_2^2\right)(1+|S_i(C_n^*)|),
	\end{align*}
	where $\overline{k}$ is a suitably chosen constant.
	We hence obtain
	\begin{align*}
		\phi_{i,\text{comp}}(L;\theta_n^*)^2\leq&\overline{k}\log^2(nT)(1+|S_i(C_n^*)|).
	\end{align*}
	The above, together with Inequality~\eqref{eq:other_case}, implies the statement that we wanted to prove on $\Omega_{\mathcal{N}}$.
	The proof is complete because $\IP(\Omega_{\mathcal{N}})\to1$.
\end{proof}

\section{Presentation of additional simulation results}
\label{sup:emp}
\subsection{More details on the simulation presented in the main text}
\label{sup:additional_simulation}
In this section, we collect further results of the simulation study presented in Section \ref{subsec:simulations}. Figure~\ref{fig:alpha} shows a visual account of the estimation of $\alpha_n^*$. The slim-oracle model which omits the covariate shows here a strong bias.

\begin{figure}
	\centering
	\includegraphics[width=\textwidth]{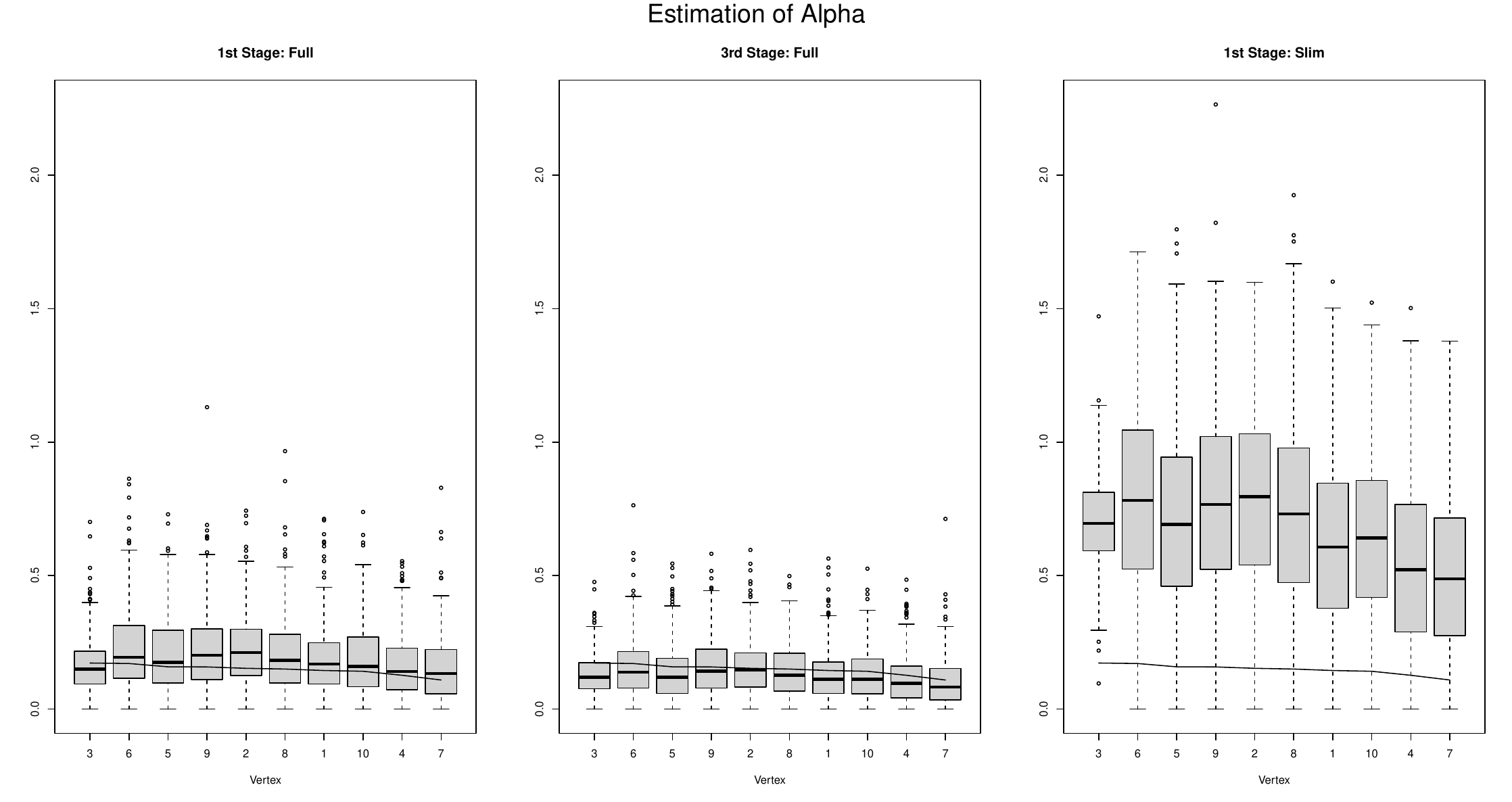}
	\caption{Box plots of the estimated values for $\alpha_{n,i}$ in all scenarios. The solid lines show the true values of $\alpha_{n,i}$ in decreasing order. De-biasing is used together with a thresholding step.}
	\label{fig:alpha}
\end{figure}

Figure~\ref{fig:sparsity} shows that the in the third stage (including the thresholding step), considerably fewer edges are collected.
In the main text, we argue that this effect primarily reduces the false positives and does not so much lead to any missed detections.
This can also be seen in Figure~\ref{fig:avgC} which shows the average estimates in the three models.
However, the root mean squared error (RMSE) is not so different in the three cases, cf. Figure~\ref{fig:Cmse}

\begin{figure}
	\centering
	\includegraphics[width=\textwidth]{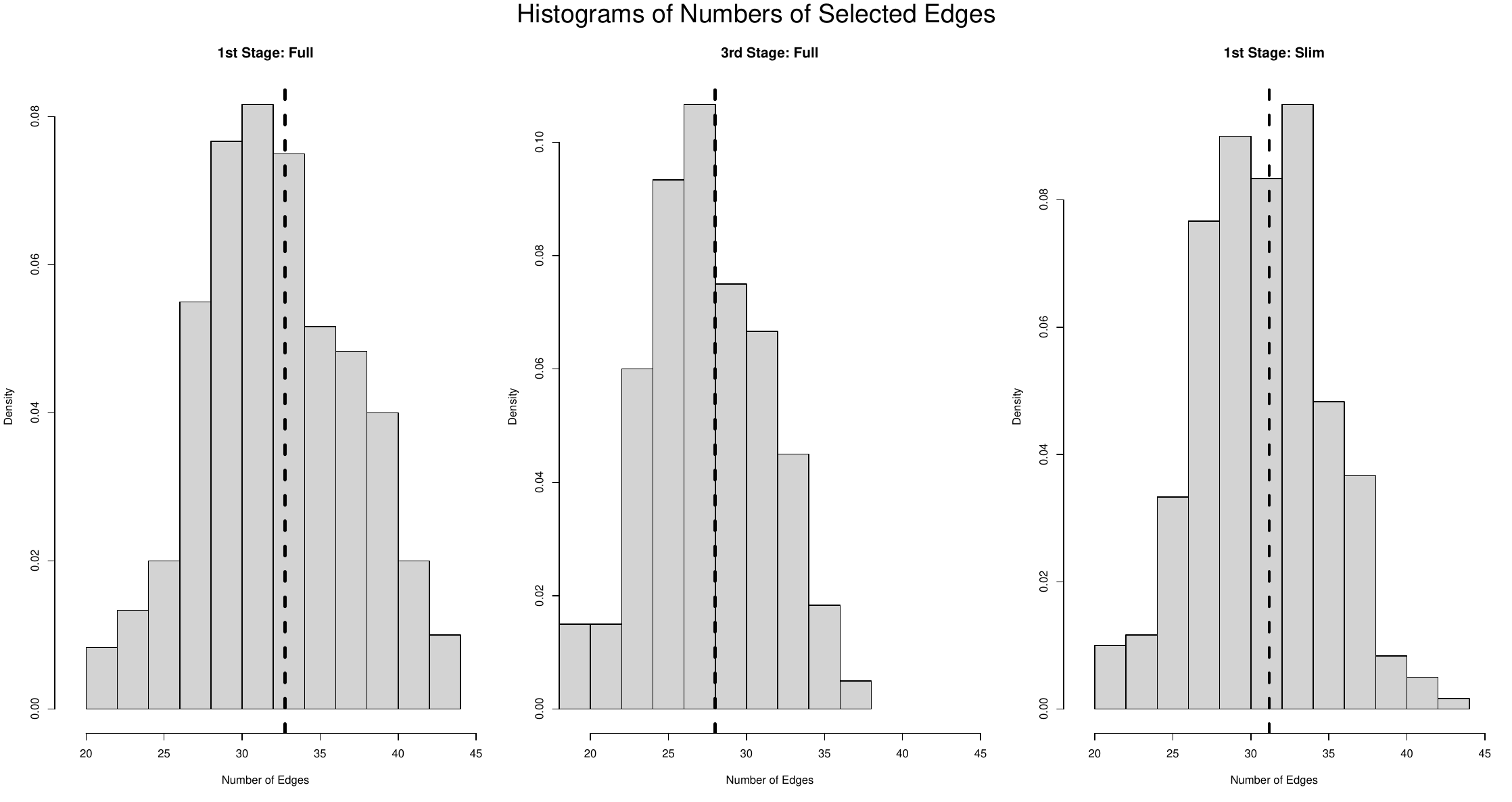}
	\caption{Histograms of the estimated number of edges. Solid lines show the true number of edges (13), and the dashed lines show the average of the estimated number of edges in each set-up. De-biasing is used together with a thresholding step.}
	\label{fig:sparsity}
\end{figure}

\begin{figure}
	\centering
	\includegraphics[width=\textwidth]{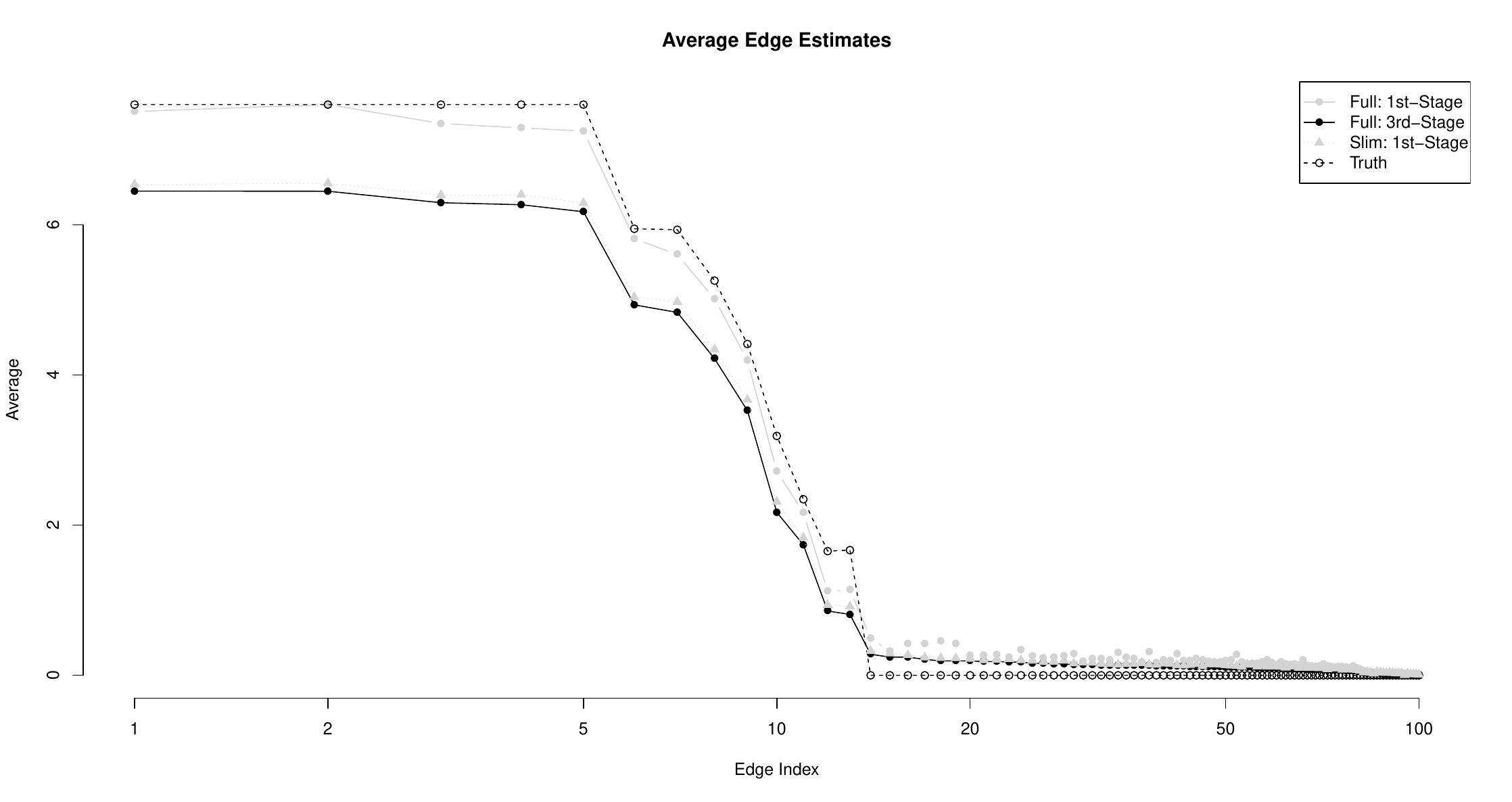}
	\caption{Average estimates of all entries of the matrix $C_n$ in the respective settings. The dotted line shows the true values. De-biasing is used together with a thresholding step.}
	\label{fig:avgC}
\end{figure}

\begin{figure}
	\centering
	\includegraphics[width=\textwidth]{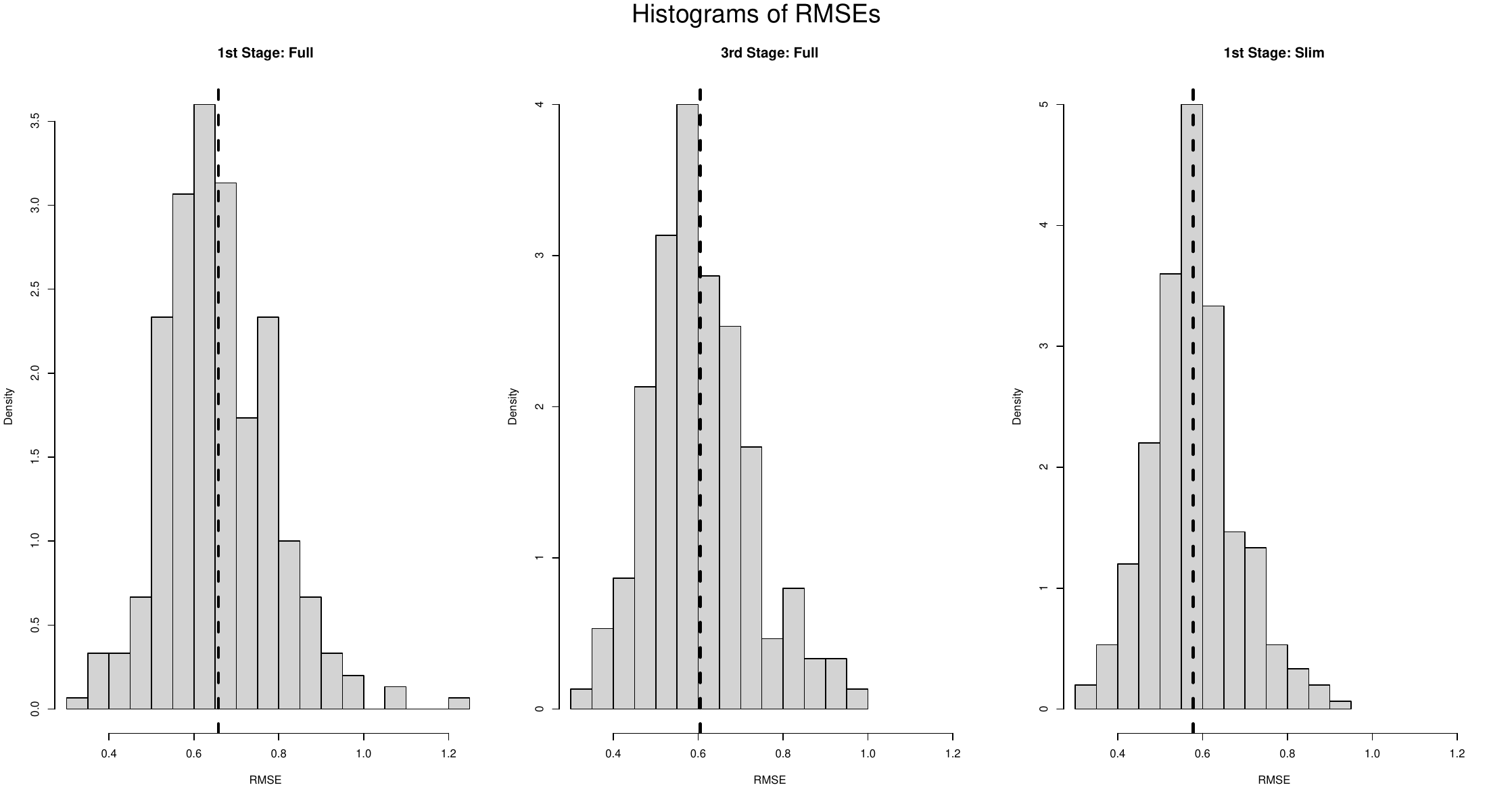}
	\caption{Histograms of mean squared errors for the entire network. The dashed lines show the averages in the respective situation. De-biasing is used together with a thresholding step.}
	\label{fig:Cmse}
\end{figure}

\subsection{Larger network size}
\label{subsec:large_network}

In order to illustrate the performance of our estimator in a more challenging situation, we have also considered the situation of $n=30$ vertices that are observed over a time horizon of $T=16$ days.
In other words, we have now more vertices and a shorter time window.
The network parameter $C_n^*$ has $43$ non-zero edges (out of $900$).
It is shown in the left panel of Figure~\ref{fig:estimated_large_networks}.
As for the other example, we consider $300$ Monte Carlo replicates and choose the bandwidth parameters once via cross-validation (reusing the same values in all Monte Carlo steps).

Figures~\ref{fig:beta_large} and~\ref{fig:gamma_large} show the estimators before and after the de-biasing step (including the thresholding) for $\beta_n^*$ and $\gamma_n^*$.
Our general theory predicts that estimation of $\beta_n^*$ and $\gamma_n^*$ should work equally well as for the smaller model because $\sqrt{nT}$ is in both cases roughly the same.
We see that the estimation of $\beta_n^*$ indeed is very similar as for the smaller model.
For the estimation of $\gamma_n^*$, however, we see that the first stage estimator seems unstable, and very often a result close to the boundary of the optimization region is obtained.
On the positive side, we see that the de-biasing step can counteract these defect relatively well.
In addition to a reduced (but non-vanishing bias), the estimates for $\gamma_n^*$ look close to a normal distribution.
Nevertheless, we suspect that the initial choice of penalty was not ideal or that the cross-validation bandwidth from one dataset cannot be directly transferred to a different dataset.

\begin{figure}
	\centering
	\includegraphics[width=\textwidth]{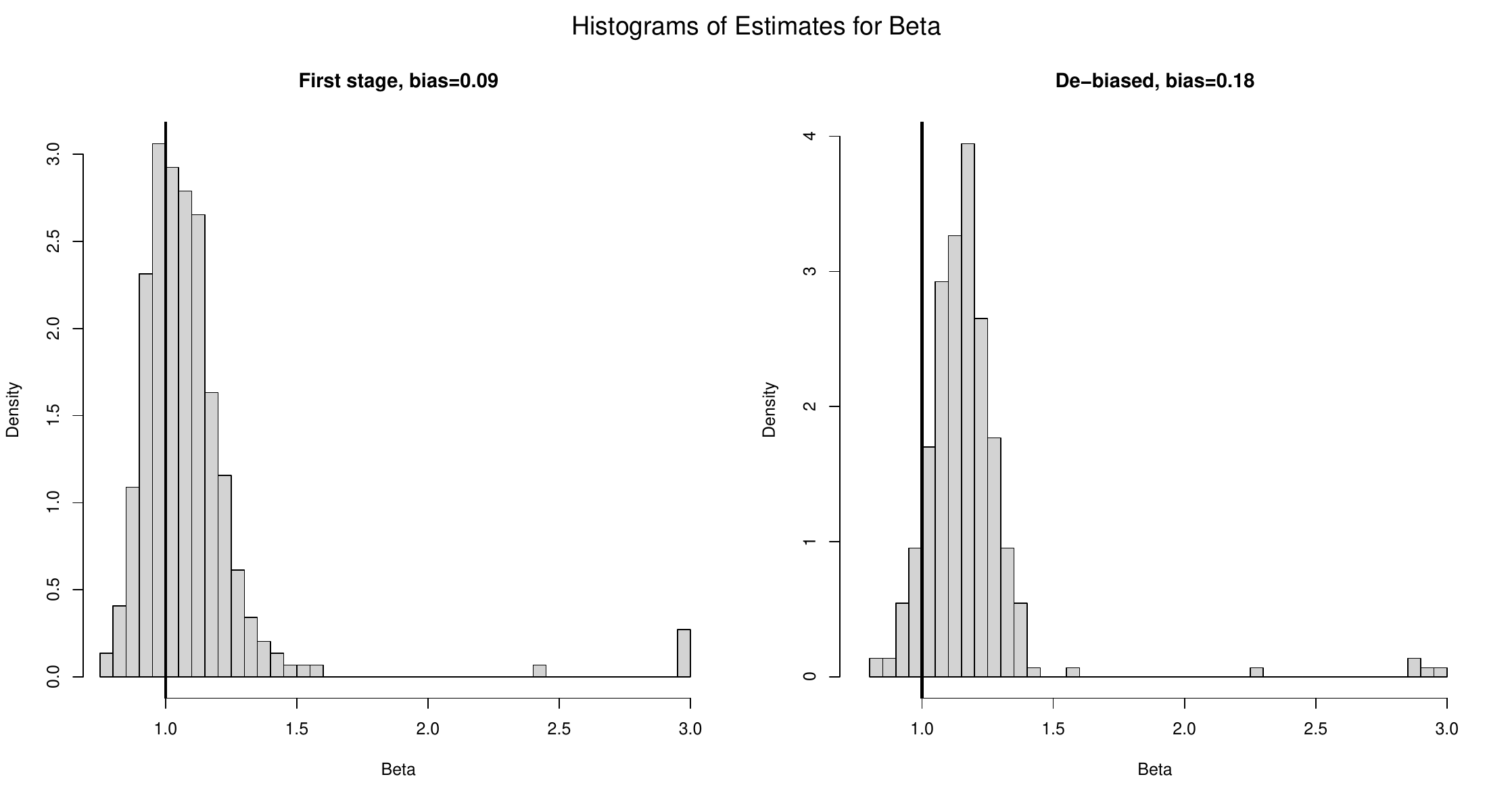}
	\caption{Histograms of the estimators for $\beta_n^*$ before and after de-biasing (including the thresholding); the true value is shown as solid line.}
	\label{fig:beta_large}
\end{figure}

\begin{figure}
	\centering
	\includegraphics[width=\textwidth]{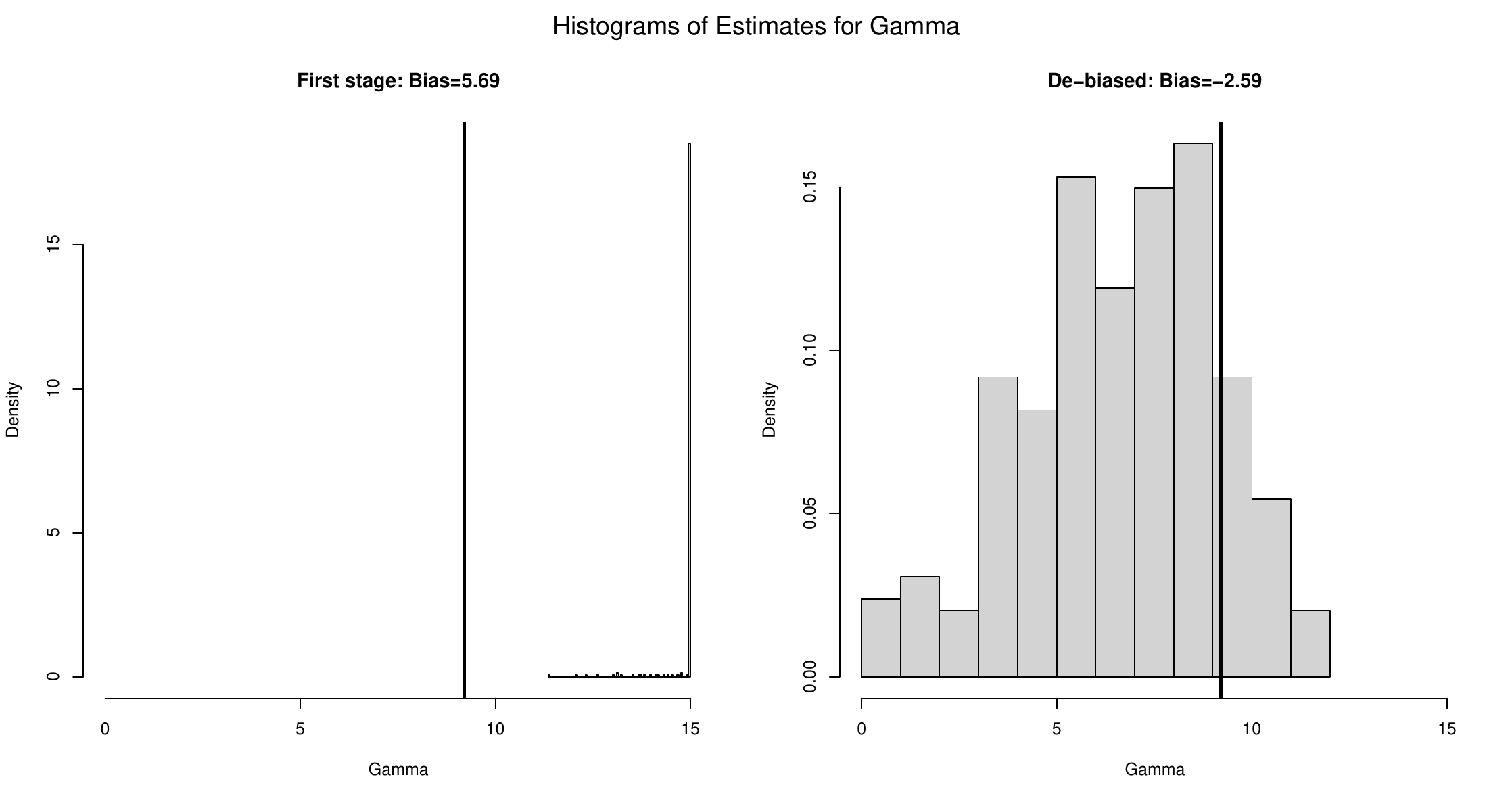}
	\caption{Histograms of the estimators for $\gamma_n^*$ before and after de-biasing (including the thresholding); the true value is shown as solid line.}
	\label{fig:gamma_large}
\end{figure}

When it comes to estimation of the network, we expect a lower performance of the method compared to the smaller network because the amount of data is reduced while the parameter number increases.
We see in Figure~\ref{fig:estimated_large_networks} on the right hand side those edges that received the highest average influence estimates in the $300$ Monte Carlo repetitions.
We have restricted to the $43$ largest estimates, the number of true edges.
It can be seen that many edges are detected correctly, and in particular the disjoint partitions of the true network are not linked by large estimates.
To look at this in more detail, we show in Figure~\ref{fig:detection_large} how often each edge was detected.
We see that about $30$ of the frequently detected edges are detected more often than non-existing edges (by the 3rd stage estimator).
But, compared to the small model, there is a high-number of false positives.
This is to be expected, because the estimation task is now more difficult.
Interestingly, the frequently detected edges are not necessarily those that have the highest true influence.
Similarly as for the smaller network, we observe that the 3rd stage estimator after the de-biasing shows fewer false positives compared to the other methods.
This indicated that inclusion of covariates and de-biasing can have a benficial effect on network estimation.

\begin{figure}
	\centering
	\includegraphics[width=\textwidth]{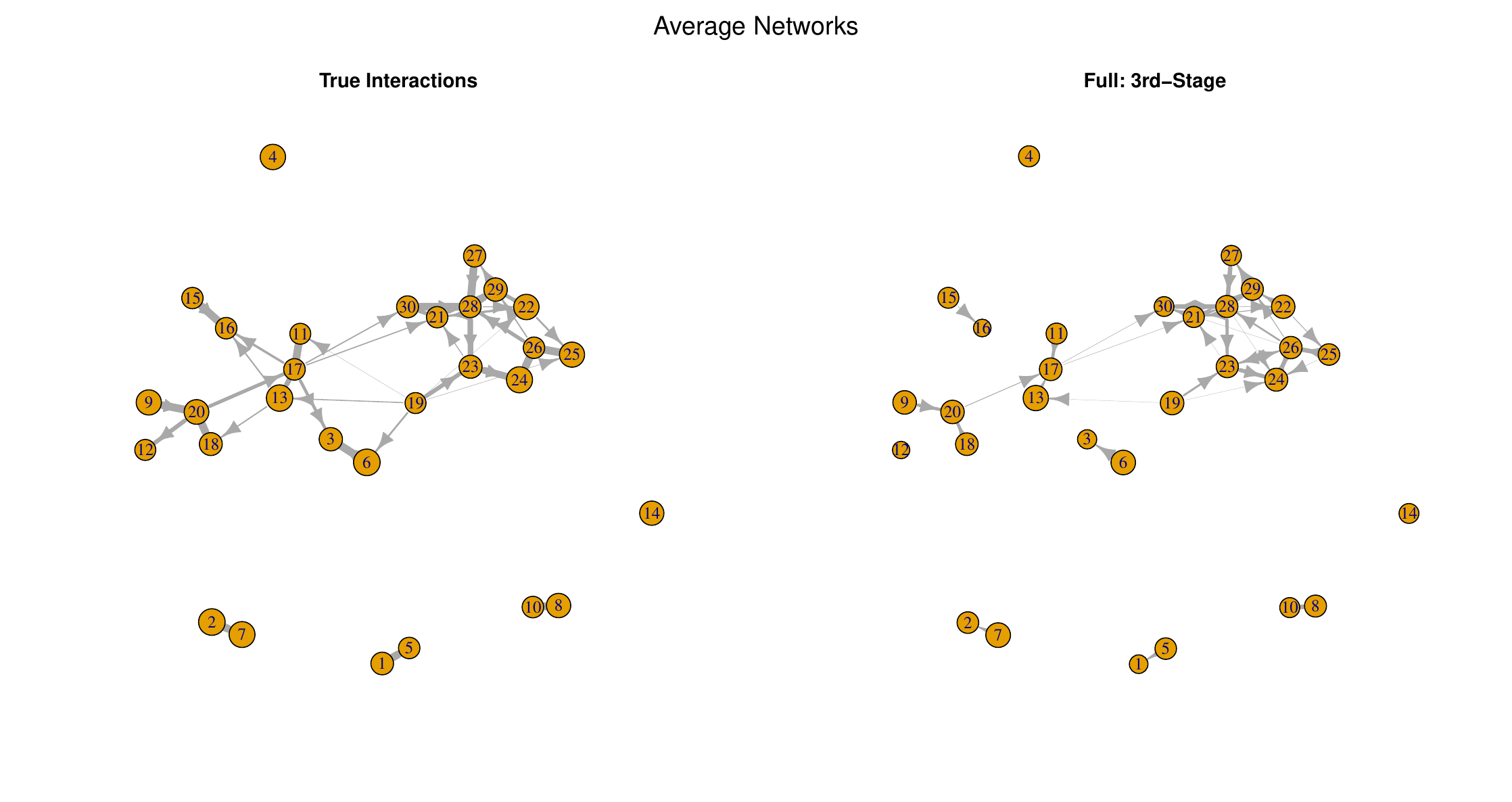}
	\caption{True network on the left and the $43$ edges that have the highest average estimate of influences (on the right).}
	\label{fig:estimated_large_networks}
\end{figure}

\begin{figure}
	\centering
	\includegraphics[width=\textwidth]{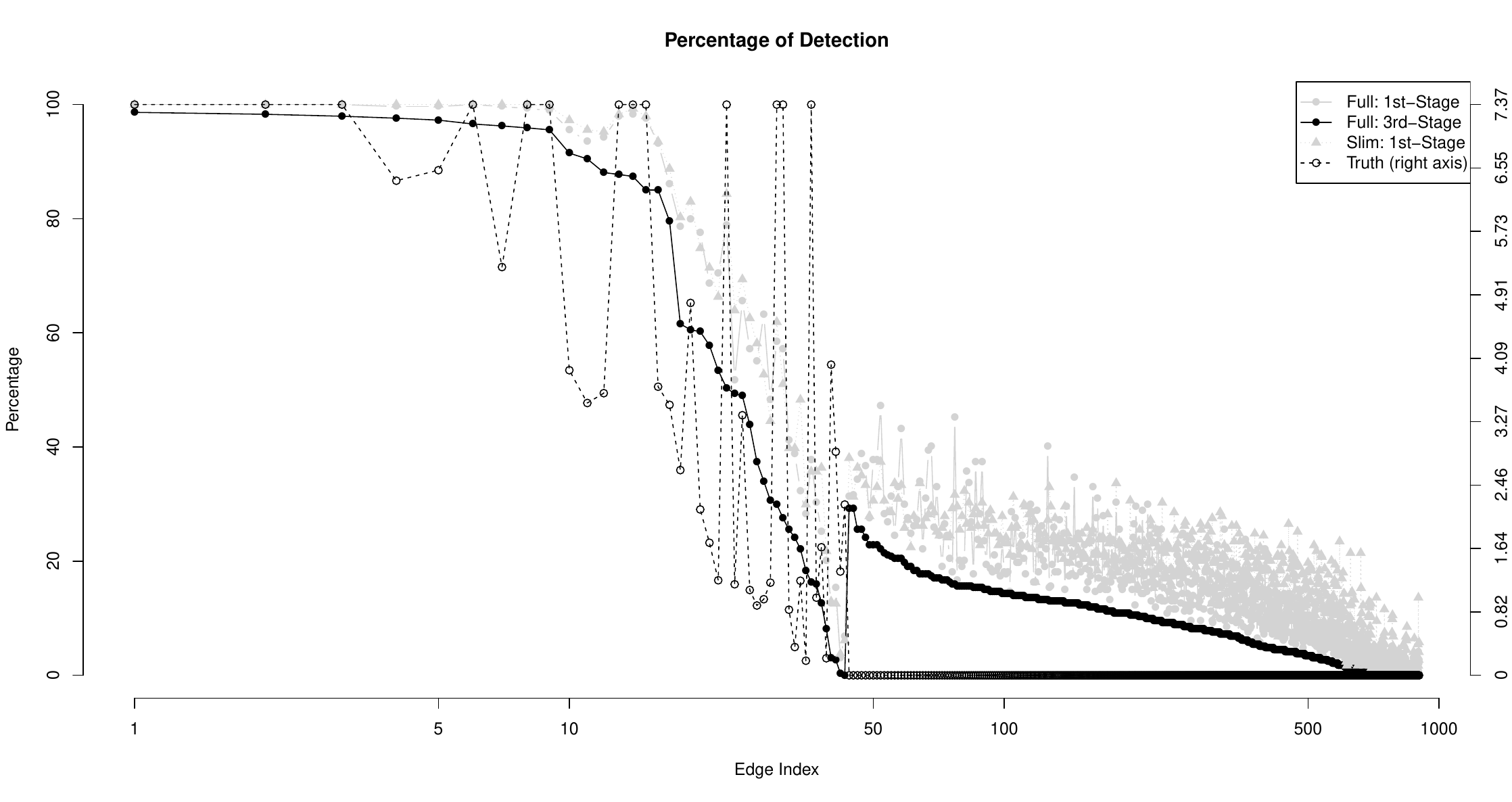}
	\caption{Percentage of detections per edge. The x-axis shows the indices of the edges (in decreasing frequency of detection in the 3rd stage); the three solid lines corresponding to the models show how often the respective edge is detected. The dashed line shows the actual value of the influence of the corresponding edge (right axis).}
	\label{fig:detection_large}
\end{figure}
\end{appendix}

\FloatBarrier
\newpage

\bibliographystyle{plainnat}
\bibliography{mybib}	

\begin{thebibliography}{30}
\providecommand{\natexlab}[1]{#1}
\providecommand{\url}[1]{\texttt{#1}}
\expandafter\ifx\csname urlstyle\endcsname\relax
  \providecommand{\doi}[1]{doi: #1}\else
  \providecommand{\doi}{doi: \begingroup \urlstyle{rm}\Url}\fi

\bibitem[Andersen et~al.(1993)Andersen, Borgan, Gill, and Keiding]{ABGK93}
P.~K. Andersen, O.~Borgan, R.~D. Gill, and N.~Keiding.
\newblock \emph{Statistical Models Based on Counting Processes}.
\newblock Springer, 1993.

\bibitem[Bacry et~al.(2020)Bacry, Bompaire, Gaiffas, and Muzy]{BBGM20}
E.~Bacry, M.~Bompaire, S.~Gaiffas, and J.-F. Muzy.
\newblock Sparse and low-rank multivariate hawkes processes.
\newblock \emph{Journal of Machine Learning Research}, 21\penalty0
  (50):\penalty0 1--32, 2020.
\newblock URL \url{http://jmlr.org/papers/v21/15-114.html}.

\bibitem[Bacry et~al.(2015)Bacry, Mastromatteo, and Muzy]{BMM15}
Emmanuel Bacry, Iacopo Mastromatteo, and Jean-Fran{\c c}ois Muzy.
\newblock Hawkes processes in finance.
\newblock \emph{Market Microstructure and Liquidity}, 01\penalty0
  (01):\penalty0 1550005, 2015.
\newblock \doi{10.1142/S2382626615500057}.

\bibitem[Bertsekas(1995)]{B95}
Dimitri~P. Bertsekas.
\newblock \emph{Nonlinear Programming}.
\newblock Athena Scientific, Belmont, MA, 1995.

\bibitem[Bickel et~al.(2009)Bickel, Ritov, and Tsybakov]{BYT09}
P.~J. Bickel, Y.~Ritov, and A.~B. Tsybakov.
\newblock Simultaneous analysis of lasso and dantzig selector.
\newblock \emph{Ann. Statist.}, 37\penalty0 (4):\penalty0 1705--1732, 08 2009.

\bibitem[Br{\'e}maud and Massouli{\'e}(1996)]{BM96}
Pierre Br{\'e}maud and Laurent Massouli{\'e}.
\newblock {Stability of nonlinear Hawkes processes}.
\newblock \emph{The Annals of Probability}, 24\penalty0 (3):\penalty0 1563 --
  1588, 1996.
\newblock URL \url{https://doi.org/10.1214/aop/1065725193}.

\bibitem[B\"{u}hlmann and van~de Geer(2011)]{GB11}
P.~B\"{u}hlmann and Sara van~de Geer.
\newblock \emph{Statistics for High-Dimensional Data}.
\newblock Springer, 2011.

\bibitem[Cai et~al.(2022)Cai, Zhang, and Guan]{CZG22}
B.~Cai, J.~Zhang, and Y.~Guan.
\newblock Latent network structure learning from high-dimensional multivariate
  point processes.
\newblock \emph{Journal of the American Statistical Association}, 119:\penalty0
  1--14, 2022.
\newblock URL \url{https://doi.org/10.1080/01621459.2022.2102019}.

\bibitem[Chen et~al.(2017)Chen, Witten, and Shojaie]{CWS17}
S.~Chen, D.~Witten, and A.~Shojaie.
\newblock {Nearly assumptionless screening for the mutually-exciting
  multivariate Hawkes process}.
\newblock \emph{Electronic Journal of Statistics}, 11\penalty0 (1):\penalty0
  1207 -- 1234, 2017.
\newblock URL \url{https://doi.org/10.1214/17-EJS1251}.

\bibitem[Efron et~al.(2004)Efron, Hastie, Johnstone, and Tibshirani]{EHJT04}
B.~Efron, T.~Hastie, I.~Johnstone, and R.~Tibshirani.
\newblock Least angle regression.
\newblock \emph{Ann. Statist.}, 32\penalty0 (2):\penalty0 407--499, 2004.

\bibitem[Fang et~al.(2023)Fang, Xu, Xu, Zhu, and Guan]{FXXZG23}
G.~Fang, G.~Xu, H.~Xu, X.~Zhu, and Y.~Guan.
\newblock Group network hawkes process.
\newblock \emph{Journal of the American Statistical Association}, 119:\penalty0
  2328 -- 2344, 2023.
\newblock URL \url{https://doi.org/10.1080/01621459.2022.2102019}.

\bibitem[Hansen et~al.(2015)Hansen, Reynaud-Bouret, and Rivoirard]{HRBR15}
N.~R. Hansen, P.~Reynaud-Bouret, and V.~Rivoirard.
\newblock Lasso and probabilistic inequalities for multivariate point
  processes.
\newblock \emph{Bernoulli}, 21\penalty0 (1):\penalty0 83--143, 02 2015.

\bibitem[Hastie et~al.(2015)Hastie, Tibshirani, and Wainwright]{HTW15}
T.~Hastie, R.~Tibshirani, and M.~Wainwright.
\newblock \emph{Statistical Learning with Sparsity}.
\newblock Chapman and Hall/CRC, 1 edition, 2015.

\bibitem[Hawkes and Oakes(1974)]{HO71}
Alan~G. Hawkes and David Oakes.
\newblock A cluster process representation of a self-exciting process.
\newblock \emph{Journal of Applied Probability}, 11\penalty0 (3):\penalty0
  493--503, 1974.

\bibitem[Lemonnier et~al.(2017)Lemonnier, Scaman, and Kalogeratos]{LSK17}
R.~Lemonnier, K.~Scaman, and A.~Kalogeratos.
\newblock Multivariate hawkes processes for large-scale inference.
\newblock \emph{Proceedings of the AAAI Conference on Artificial Intelligence},
  31\penalty0 (1), 2017.
\newblock \doi{https://doi.org/10.1609/aaai.v31i1.10846}.

\bibitem[Li and Cui(2024)]{LC24}
Chenlong Li and Kaiyan Cui.
\newblock Multivariate hawkes processes with spatial covariates for
  spatiotemporal event data analysis.
\newblock \emph{Annals of the Institute of Statistical Mathematics},
  76:\penalty0 535 -- 578, 2024.

\bibitem[Mammen and M\"uller(2023)]{MM23}
E.~Mammen and M.~M\"uller.
\newblock Nonparametric estimation of locally stationary hawkes processes.
\newblock \emph{Bernoulli}, 29:\penalty0 2062 -- 2083, 2023.

\bibitem[Pearl(2000)]{P00}
J.~Pearl.
\newblock \emph{Causality}.
\newblock Cambridge University Press, 1 edition, 2000.

\bibitem[Pitkin et~al.(2024)Pitkin, Manolopoulou, and Ross]{PMR24}
J.~Pitkin, I.~Manolopoulou, and G.~Ross.
\newblock {Bayesian hierarchical modelling of sparse count processes in retail
  analytics}.
\newblock \emph{The Annals of Applied Statistics}, 18\penalty0 (2):\penalty0
  946 -- 965, 2024.
\newblock URL \url{https://doi.org/10.1214/23-AOAS1811}.

\bibitem[Reynaud-Bouret and Schbath(2010)]{RBS10}
P.~Reynaud-Bouret and S.~Schbath.
\newblock Adaptive estimation for hawkes processes; application to genome
  analysis.
\newblock \emph{Ann. Statist.}, 38\penalty0 (5):\penalty0 2781--2822, 2010.

\bibitem[Rizoiu et~al.(2017)Rizoiu, Xie, Sanner, Cebrian, Yu, and
  Van~Hentenryck]{RXSCYvH17}
M.-A. Rizoiu, L.~Xie, S.~Sanner, M.~Cebrian, H.~Yu, and P.~Van~Hentenryck.
\newblock Expecting to be hip: Hawkes intensity processes for social media
  popularity.
\newblock In \emph{Proceedings of the 26th International Conference on World
  Wide Web}, WWW '17, pages 735--744. International World Wide Web Conferences
  Steering Committee, 2017.

\bibitem[Sulem et~al.(2024)Sulem, Rivoirard, and Rousseau]{SRR24}
D.~Sulem, V.~Rivoirard, and J.~Rousseau.
\newblock {Bayesian estimation of nonlinear Hawkes processes}.
\newblock \emph{Bernoulli}, 30\penalty0 (2):\penalty0 1257 -- 1286, 2024.
\newblock URL \url{https://doi.org/10.3150/23-BEJ1631}.

\bibitem[Tang and Li(2023)]{TL23}
X.~Tang and L.~Li.
\newblock Multivariate temporal point process regression.
\newblock \emph{Journal of the American Statistical Association}, 118:\penalty0
  830 -- 845, 2023.
\newblock URL \url{https://doi:10.1080/01621459.2021.1955690}.

\bibitem[van~de Geer et~al.(2014)van~de Geer, B{\"u}hlmann, Ritov, and
  Dezeure]{vdGBRD14}
S.~van~de Geer, P.~B{\"u}hlmann, Y.~Ritov, and R.~Dezeure.
\newblock {On asymptotically optimal confidence regions and tests for
  high-dimensional models}.
\newblock \emph{The Annals of Statistics}, 42\penalty0 (3):\penalty0 1166 --
  1202, 2014.
\newblock URL \url{https://doi.org/10.1214/14-AOS1221}.

\bibitem[Wang and Shojaie(2021)]{WS21}
X.~Wang and A.~Shojaie.
\newblock Causal discovery in high-dimensional point process networks with
  hidden nodes.
\newblock \emph{Entropy}, 23\penalty0 (12), 2021.
\newblock URL \url{https://doi.org/10.3390/e23121622}.

\bibitem[Wang et~al.(2024)Wang, Kolar, and Shojaie]{WKS24}
X.~Wang, M.~Kolar, and A.~Shojaie.
\newblock Statistical inference for networks of high-dimensional point
  processes.
\newblock \emph{Journal of the American Statistical Association}, 120:\penalty0
  1 -- 11, 2024.
\newblock URL \url{https://doi.org/10.1080/01621459.2024.2392907}.

\bibitem[Yuan et~al.(2021)Yuan, Schoenberg, and Bertozzi]{YSB21}
B.~Yuan, F.~P. Schoenberg, and A.~L. Bertozzi.
\newblock Fast estimation of multivariate spatiotemporal hawkes processes and
  network reconstruction.
\newblock \emph{Annals of the Institute of Statistical Mathematics},
  73\penalty0 (6):\penalty0 1127--1152, 2021.
\newblock URL \url{https://doi.org/10.1007/s10463-020-00780-1}.

\bibitem[Zhao et~al.(2015)Zhao, Erdogdu, He, Rajaraman, and Leskovec]{ZEHRL15}
Q.~Zhao, M.~A. Erdogdu, H.~Y. He, A.~Rajaraman, and J.~Leskovec.
\newblock Seismic: A self-exciting point process model for predicting tweet
  popularity.
\newblock In \emph{Proceedings of the 21th ACM SIGKDD International Conference
  on Knowledge Discovery and Data Mining}, page 1513–1522, 2015.
\newblock \doi{https://doi.org/10.1145/2783258.2783401}.

\bibitem[Zheng et~al.(2021)Zheng, Raskutti, Willett, and Mark]{ZRWM21}
Lili Zheng, Garvesh Raskutti, Rebecca Willett, and Benjamin Mark.
\newblock Context-dependent networks in multivariate time series: Models,
  methods, and risk bounds in high dimensions.
\newblock \emph{Journal of Machine Learning Research}, 22\penalty0
  (216):\penalty0 1--88, 2021.
\newblock URL \url{http://jmlr.org/papers/v22/20-244.html}.

\bibitem[Zhou(2025)]{Z25}
Shuheng Zhou.
\newblock Thresholded lasso for high dimensional variable selection.
\newblock \emph{Annals of the Institute of Statistical Mathematics}, 2025.
\newblock published online.

\end{thebibliography}
\end{document}